\documentclass[11pt,letterpaper]{article}

\usepackage[left=2.2cm,tmargin=2.5cm,bmargin=2.5cm,right=2.2cm]{geometry}

\usepackage{amsfonts}
\usepackage{amssymb}
\usepackage{amsmath}
\usepackage{amsthm}
\usepackage{graphicx}
\usepackage{tikz}
\usepackage{enumerate, color}
\usepackage{biblatex}
\usepackage[all]{xy}

\usetikzlibrary{arrows.meta, positioning, shapes, calc}

\newtheorem{theorem}{Theorem}[section]
\newtheorem{manualtheoreminner}{Theorem}
\newenvironment{manualtheorem}[1]{%
  \IfBlankTF{#1}
    {\renewcommand{\themanualtheoreminner}{\unskip}}
    {\renewcommand\themanualtheoreminner{#1}}%
  \manualtheoreminner
}{\endmanualtheoreminner}

\newtheorem{corollary}[theorem]{Corollary}
\newtheorem{cor}[theorem]{Corollary}

\newtheorem{lemma}[theorem]{Lemma}

\newtheorem{prop}[theorem]{Proposition}

\newtheorem{non-theorem}[theorem]{Non-theorem}
\theoremstyle{definition}

\theoremstyle{remark}

\newtheorem{example}[theorem]{Example}

\newtheorem{question}[theorem]{Question}

\numberwithin{equation}{section}
\newtheoremstyle{dotless}{}{}{}{}{\bfseries}{}{ }{}
\theoremstyle{dotless}

\newcommand{\C}{\mathbb{C}}

\newcommand{\Z}{\mathbb{Z}}
\newcommand{\N}{\mathbb {N}}
\newcommand{\F}{\mathbb{F}}

\newcommand{\ord}{\mathrm{ord} \,}

\newcommand{\res}[2]{\left(\frac{#1}{#2}\right)}
\newcommand{\Frob}{\operatorname{Frob}}
\newcommand{\Gr}{\operatorname{Gr}}

\newcommand{\hchi}{\psi} 
\newcommand{\gene}{\zeta_1} 
\newcommand{\qq}{\mathbf{ \mathsf{q}}} 
\newcommand{\qv}{v}

\newcommand{\Tr}{{\mathrm{Tr}}}
\newcommand{\Res}{{\mathrm{Res}}}
\newcommand{\Disc}{{\mathrm{Disc}}}

\renewcommand{\tilde}{\widetilde}
\newcommand{\LieH}{H}
\newcommand{\Ext}{\operatorname{Ext}}
\newcommand \coker {\operatorname{coker}}
\newcommand{\tr}{\operatorname{tr}}

\begin{document}
\title{Functional equations of axiomatic multiple Dirichlet series, Weyl groupoids, and quantum algebra}
\author{Will Sawin and Ian Whitehead}

\maketitle

\begin{abstract}
    We prove functional equations for multiple Dirichlet series defined by a collection of five geometric axioms. We find functional equations of two types: one modeled on the functional equations of Dirichlet $L$-functions, and another modeled on the functional equations of Kubota $L$-series with Gauss sums as coefficients. These functional equations generate groupoid structures, which we relate to the Weyl groupoids of arithmetic root systems. From the known classification of arithmetic root systems, we obtain a complete classification of multiple Dirichlet series which can be used to compute moments of $L$-functions via established analytic techniques. Our classification includes all moments of $L$-functions which have appeared in the multiple Dirichlet series literature previously, alongside some new moments. Finally, we give applications of our functional equations to quantum algebra, specifically the cohomology of Nichols algebras. 
\end{abstract}

\tableofcontents

\pagebreak 

\textbf{Index of Notation} \\
$q$: odd prime power \\
$\F_q$: finite field \\
$\chi$: character of $\F_q^*$ \\
$n$: even integer, order of $\chi$ \\
$\xi$: quadratic character of $\F_q^*$ \\
$T$: variable for polynomial ring over $\F_q$ \\
$\mathrm{Res}$: resultant \\
$\mathrm{Disc}$: discriminant \\
$\res{f}{g}_\chi$ residue symbol \\
$f$, $f_i$, $g$, $g_i$, $h$, $h_i$: monic polynomials in $\F_q[T]$\\
$\pi$: prime polynomial in $\F_q[T]$\\
$J$: set of indices \\
$c_j$: integer multiplicity \\
$\alpha_j$: Weil number \\
$d$: weight of Weil number \\
$g_\chi$: finite field Gauss sum \\
$g_\chi(f_1, f_2)$ function field Gauss sum \\
$\exp$: function field additive character \\
$c_{-1}$ Laurent series coefficient of $T^{-1}$ \\
$r$: number of variables for multiple Dirichlet series, rank of arithmetic root system \\
$M=(M_{ij})$: symmetric $r\times r$ integer matrix \\
$x$, $x_i$: variables in $\C$ \\
$\vec{x}=(x_1, \ldots x_r)$: tuple of variables in $\C^r$ \\
$i$, $j$, (rarely) $h$: indices, usually from $1$ to $r$\\
$k$, $k_i$, $\ell$, $\ell_i$, (rarely) $m$: integer mod $n$ or $n_i$ \\
$\vec{k}=(k_1, \ldots k_r)$: tuple of integers mod $n$ \\
$Z(\vec{x}, \vec{k}; q, \chi, M)$: global multiple Dirichlet series \\
$Z_\pi(\vec{x}, \vec{k}; q, \chi, M)$: global multiple Dirichlet series \\
$a(f_1, \ldots f_r; q, \chi, M)$: multiple Dirichlet series coefficient \\
$d$, $d_i$, (rarely) $e$: degree of $f$, $f_i$ or power of $\pi$ \\
$n_i$ : $\frac{n}{\gcd(n, M_{ii}+n/2)}$ \\
$n_{ij}$: least nonnegative integer such that $n_{ij}(M_{ii}+n/2) \equiv -M_{ij} \bmod n$ \\
$\%$: programmer's modulo, i.e. $m\% n$ is the smallest nonnegative integer congruent to $m$ modulo $n$ \\
$\vec{f}$: $r-1$ tuple of polynomials $(f_1, \ldots f_{r-1})$  \\
$F$: $f_1^{n_{r1}}\cdots f_{r-1}^{n_{r \, (r-1)}}$ in Sections 2, 3, or $f_1^{e_{r1}}\cdots f_{r-1}^{e_{r \, (r-1)}}$ in Section 4\\
$\tilde{F}$: $f_1^{M_{1r}}\cdots f_{r-1}^{M_{(r-1) \, r}}$ \\
$\varepsilon$: when $f_1\cdots f_{r-1}$ is squarefree, $\prod_{i=1}^{r-1} \res{f_i'}{f_i}_{\chi}^{M_{ii}} \, \prod_{1 \leq i < j \leq r-1} \res{f_i}{f_j}_{\chi}^{M_{ij}}$  \\
$D(x, \vec{f}; q, \chi, M)$: single-variable subseries of $Z$ \\
$D_\pi(x, \vec{f}; q, \chi, M)$: single-variable subseries of $Z_\pi$ \\
$D^{(\pi)}(x, \vec{f}; q, \chi, M)$: single-variable subseries with one prime removed \\
$\sigma_{i; q, \chi, M}$: transformation in functional equation \\
$\tau_i$: transformation of matrix $M$ \\
$\mu$: M\"obius function \\
$\phi$: Euler phi function \\
$\Lambda_{\pi}(d,e)$: count of monic coprime polynomials \\
$o_{ij}$: multiple of $n$ in equation defining $n_{ij}$  \\
$p_{ij}$: remainder in equation defining $n_{ij}$ \\
$\bar{a}$: hypothetical formula for coefficients in Section 2 \\
$f_\pi$, $\vec{f}_\pi$: largest $\pi$-power factor of $f$, largest $\pi$-power factor of each $f_i$  \\
$f^{(\pi)}$, $\vec{f}^{(\pi)}$: largest $\pi$-coprime factor of $f$, largest $\pi$-coprime factor of each $f_i$  \\
$v_\pi$: valuation at $\pi$ \\
$\varepsilon_\pi$: $\prod\limits_{i=1}^{r-1} \res{f_i^{(\pi)}}{(f_i)_\pi}_{\chi}^{M_{ii}} \res{(f_i)_\pi}{f_i^{(\pi)}}_{\chi}^{M_{ii}} \prod\limits_{i=1}^{r-2}\prod\limits_{j=1+i}^{r-1}\res{f_i^{(\pi)}}{(f_j)_\pi}_{\chi}^{M_{ij}} \res{(f_i)_\pi}{f_j^{(\pi)}}_{\chi}^{M_{ij}}$ \\
$A$: degree or valuation of $f_1\cdots f_{r-1}$ in inductive proofs \\
$\left( \Gamma_{k, \ell}(x, d) \right)_{k, \ell}$: scattering matrix for Kubota-type functional equation \\
$\left( \Gamma_{\pi, k, \ell}(x, d) \right)_{k, \ell}$: scattering matrix for local Kubota-type functional equation \\
$M_m$: augmented matrix \\
$S^{k, n}$ power series projection operator \\
$E_m^n$: matrix expansion \\
$L$: Dirichlet $L$-function \\
$\omega_1$: root number \\
$\omega$: sign part of root number \\
$\omega_\pi$: local sign part of root number \\
$b$: 1 if $n|\deg F$, $0$ otherwise \\
$b_{\pi}$: 1 if $n|v_\pi \tilde{F}$, $0$ otherwise \\
$K$: $\deg F$ mod $n$, i.e. $\sum n_{ij} k_j$ in Section 3, $\sum e_{ij} k_j$ in Section 4 \\
$v$: $\sum_{j \neq i} M_{ij} k_j$, also sometimes used for an arbitrary integer mod $n$ or $n_r$ \\
$v_i$: $\sum_{j > i} M_{ij} k_j$ \\
$\left(\Theta_{k, \ell}(x, K, v)\right)_{k, \ell}$: scattering matrix for Dirichlet-type functional equation \\
$\left(\Theta_{\pi, k, \ell}(x, K, v)\right)_{k, \ell}$: scattering matrix for Dirichlet-type local functional equation \\
$\mu_\infty$: set of all roots of unity in $\mathbb C$ \\
$\hchi$: bicharacter \\
$E$: basis of $\mathbb Z^r$ \\
$s_{i,E}$: reflection in arithmetic root system \\
$m_{ij}$: coefficient used in reflection $s_{i,E}$, equal to $n_{ij}$ in the Kubota case and $e_{ij}$ in the Dirichlet case \\
$B_{\hchi,E}$: set of bases in arithmetic root system \\
$\mathbf \Delta$: set of roots of arithmetic root system \\
$\mathfrak g$: simple Lie algebra \\
$v$, $\qq$: parameters in arithmetic root system associated to a simple Lie algebra, with $\qq=v^2$ \\
$\gene$: root of unity of order $n$ \\
$M^{\hchi,E,\gene}$: symmetric matrix arising from arithmetic root system \\
$C_E$: cone of vectors whose dot products with basis vectors in $E$ is nonnegative \\
$\langle, \rangle_E$: bilinear form used to characterize the relationship between the transformation $\sigma_i$ and $s_{i,E}$ \\
$d^{\mathrm{min}}, d^{'\mathrm{min}}$: minimum exponent of monomial with nonzero coefficient \\
$\tilde{d}$: $\sum_{k\neq i} d_k m_{ik}$\\
$V_{E,\hchi}$: diagonal braided vector space \\
$V'_{E,\hchi}$: twist of $V_{E,\hchi}$ obtained by negating the braiding  \\
$\mathfrak B(V)$: Nichols algebra of braided vector space $V$ \\
$\mathbb A^{d_i}$: affine space of dimension $d_i$, always viewed as moduli space of monic polynomials of degree $d_i$ \\
$F_{d_1,\dots,d_r}$: polynomial function used to define perverse sheaf \\
$\hat{d}$: $\sum_{i=1}^r d_i$ \\
$\mathcal L_\chi$: Kummer sheaf \\
$\mathbb Q_\ell$: $\ell$-adic numbers \\
$K_{d_1,\dots,d_r}$: perverse sheaf used to obtain multiple Dirichlet series coefficients \\
$\operatorname{Sym}^d (\mathbb C)$: $d$th symmetric power of the complex plane \\
$\operatorname{Conf}_{d_1,\dots,d_r} (\mathbb C)$: space parameterizing configurations of $\hat{d}$ colored points in the plane with $d_i$ points of the $i$th color \\
$S_d$: symmetric group \\
$\mathcal L\{W\}$: local system arising from a representation \\
$W(x,u, \vec{d},k)$: series obtained from weight-extraction on $D(x, T^{\vec{d}},k)$ defined as a sum of Betti-numbers
$\overline{W}(x,u, \vec{d},k)$: obtained from $W(x,u, \vec{d},k)$ by a change of variables that simplifies the functional equation\\
$\Omega(x, u, \tilde{d})$: scattering matrix obtained from weight-extraction on $\Gamma(x, \tilde{d})$ \\
$\overline{\Omega}(x, u, \tilde{d})$: obtained from $\Omega(x, u, \tilde{d})$ by a change of variables that simplifies the functional equation 

\pagebreak

\section{Introduction} \label{SectionIntro}

\subsection{Background}

Multiple Dirichlet series are several-variable analogues of zeta and $L$-functions. They are important objects of study for their applications to moments of $L$-functions, their connections to metaplectic Eisenstein series, and their intrinsically interesting combinatorial and analytic properties. The applications to moments of $L$-functions originate in work of Goldfeld and Hoffstein \cite{GoldfeldHoffstein}; a general approach is laid out in \cite{DiaconuGoldfeldHoffstein}. Successful applications include the first three moments of quadratic Dirichlet $L$-functions \cite{DiaconuGoldfeldHoffstein, Diaconu, DiaconuWhitehead}, and the fourth moment in the rational function field setting \cite{BucurDiaconu, DiaconuPasolPopa}. There are also important applications to more general families of $L$-functions, including the first moment of $L$-functions of $n$th-order Dirichlet characters \cite{FriedbergHoffsteinLieman}, and the second moment of absolute values of these  \cite{Diaconu2004}. Another major application is to the second moment of cubic Dirichlet $L$-functions, or the first moment of cubic twists of $\mathrm{GL}(2)$ automorphic $L$-functions \cite{BrubakerThesis, BrubakerFriedbergHoffstein}.

In all these applications, the principal analytic tools used to study multiple Dirichlet series are their functional equations. The multiple Dirichlet series studied in this article and in the above references all possess several functional equations which are analogues of the functional equation of a single-variable zeta or $L$-function. These functional equations can be composed to form a group. The functional equations give substantial information about the analytic behavior of the series, including meromorphic continuation, growth estimates, poles, and residues. This information is then used to extract moments of $L$-functions from multiple Dirichlet series by contour integration methods. 

Groups of functional equations play a central role in the theory of Weyl group multiple Dirichlet series. This type of multiple Dirichlet series is constructed to have a group of functional equations isomorphic to the Weyl group of a finite root system. There are several constructions in the literature, e.g. \cite{BrubakerBumpFriedberg, ChintaGunnellsJAMS}. The local, prime-power coefficients of Weyl group multiple Dirichlet series exhibit an intricate combinatorial structure related to deformations of the Weyl character formula \cite{ChintaOffen} and crystal bases \cite{BrubakerBumpFriedberg2}. A major result is that Weyl group multiple Dirichlet series for a given root system coincide with the Fourier-Whittaker coefficients of Eisenstein series on metaplectic covers of the corresponding algebraic groups \cite{McNamara, PatnaikPuskas, Chen}. 

However, many applications to moments of $L$-functions require multiple Dirichlet series which do not fit in the Weyl group framework. The functional equations of Weyl group multiple Dirichlet series relate the series to itself. But in applications to families of L-functions involving cubic and higher-order characters, we see a different phenomenon: the computations use a finite set of multiple Dirichlet series which are related to one another by functional equations. In this case, as we argue in Section \ref{SectionGroup} below, the functional equations are best understood as a groupoid or small category with inverses, with objects given by the individual multiple Dirichlet series and morphisms given by the functional equations between them. Our work gives a general framework to associate a multiple Dirichlet series to its groupoid of functional equations. 

We prove functional equations in the general setting of axiomatic multiple Dirichlet series. This class of series includes the Weyl group multiple Dirichlet series, all the multiple Dirichlet series that have been used to establish moments of $L$-functions in past work, and many more. We establish the functional equations of these series directly from their axiomatic characterization. The groupoid of functional equations can be read off from the twisted multiplicativity property satisfied by the coefficients of the series, or the heuristic form of the coefficients. We classify all axiomatic multiple Dirichlet series with finite groupoids of functional equations--these are the cases where established analytic techniques give meromorphic continuation and allow for the computation of moments. In fact, we find that the classification of multiple Dirichlet series with finite groupoids of functional equations essentially coincides with a known classification of arithmetic root systems due to Heckenberger \cite{HeckenbergerClassification}. Thus we have a complete list of these series, and of moments that can be computed with known multiple Dirichlet series techniques. This list includes all the examples mentioned above, and some new moments. 

The axiomatic approach to multiple Dirichlet series originates in work of Diaconu and Pa\c{s}ol for series associated to higher moments of quadratic $L$-functions \cite{DiaconuPasol}, and work of the second author for finite and affine root systems \cite{WhiteheadThesis, WhiteheadTypeA}. These constructions were limited to quadratic characters and $L$-functions. Recent work of the first author extends the axiomatic approach much further, to arbitrary characters and more general root systems \cite{s-amds}. The idea behind this approach is to work over the rational function field $\F_q(t)$ and identify the coefficients of multiple Dirichlet series as traces of Frobenius on the cohomology of certain perverse sheaves. Five axioms encode the expected properties of these coefficients. Rather than constructing a multiple Dirichlet series based on its expected functional equations, \cite{s-amds} proves that there exists a unique multiple Dirichlet series satisfying the axioms. The axioms do not mention functional equations, and \cite{s-amds} does not prove them. The motivation for the axiomatic approach is to specify properties which pin down the multiple Dirichlet series in cases where they are not uniquely determined by their functional equations, notably in cases where the groupoid of functional equations is infinite. The infinite groupoid case is relevant to higher moments. The axiomatic approach handles finite and infinite groupoids uniformly (though meromorphic continuation of the series remains a challenge in the infinite case). Moreover, the axiomatic approach allows us to define multiple Dirichlet series without predicting the functional equations in advance. 

The main drawback of the axiomatic approach is that it is limited to the rational function field. However, the axiomatic method can also be used to construct the local, prime parts of multiple Dirichlet series. These local parts are power series which have the same shape regardless of the choice of global field and prime. Then the local parts can be assembled into global multiple Dirichlet series using a twisted multiplicativity rule dependent on the arithmetic of the chosen global field. In this indirect way, the axiomatic approach is relevant to all global fields. Moreover, moment problems in function fields are an important part of the broader theory of moments of $L$-functions. Moments of Dirichlet $L$-functions over function fields have been heavily studied, including the recent proof of the Conrey-Farmer-Keating-Rubinstein-Snaith conjectures for quadratic Dirichlet $L$-functions of rational function fields in \cite{BDPW}, \cite{MPPRW}, as well as other recent results~\cite{David2021,David2022,David2025,Gao2023,Gao2025}. We hope that the geometric tools used here will also suggest new approaches in the number field setting. 

\subsection{Axiomatic Multiple Dirichlet Series}

To state the axioms and our main results, we must introduce some notation. Let $q$ be an odd prime power and $\F_q$ the finite field of order $q$. Let $\chi:\F_q^* \to \C^*$ be a character of even order $n$, and let $\xi:\F_q^* \to \{ \pm 1 \}$ be the unique character of order 2. Let $\F_q[T]$ be the polynomial ring over $\F_q$, and $\F_q[T]^+$ the set of monic polynomials. We call a polynomial prime if it is monic and irreducible. The resultant of $f, \, g \in \F_q[T]$ is defined as 
\begin{equation*}
\Res(f, g) = a^{\deg f} b^{\deg g} \prod_{\substack{ \alpha \in \overline{\F_q} \\ f(\alpha)=0}} \prod_{\substack{ \beta \in \overline{\F_q} \\ g(\beta)=0}} (\beta-\alpha)
\end{equation*}
where $a$ is the leading coefficient of $f$ and $b$ is the leading coefficient of $g$. The discriminant of $f \in \F_q[T]$ is $\Disc(f)=a^{-1}(-1)^{\deg f(\deg f - 1)/2}\Res(f', f)$. The resultant and discriminant lie in $\F_q$; the resultant has a symmetry $\Res(f, g) = (-1)^{(\deg f) (\deg g)} \Res(g, f)$. We define the $n$th power residue symbol as $\res{f}{g}_{\chi} = \chi(\Res(f,g/b))$. This has the expected properties: it is multiplicative in both $f$ and $g$, is zero if and only if $f$ and $g$ have a common factor, and for fixed prime $\pi$, we have $\res{f}{\pi}_\chi=1$ if and only if $f$ is an $n$th power modulo $\pi$. It satisfies the $n$th power reciprocity law, for $f, g \in \F_q[T]^+$:
\begin{equation*}
\res{f}{g}_{\chi} = \chi(-1)^{(\deg f)(\deg g)} \res{g}{f}_\chi .
\end{equation*}

We will need the notion of a compatible system of sums of $q$-Weil numbers. This is a kind of function which models the dependence on $q$ and $\chi$ seen in cohomological formulas. A $q$-Weil number of weight $d$ is simply a complex number $\alpha$ of absolute value $|\alpha|=q^{d/2}$ for some nonnegative integer $d$. For any natural number $e$, we can form the unique field extension of degree $e$, $\F_{q^e}/\F_q$, and we can define a character on $\chi_e$ on $\F_{q^e}^*$ by composing $\chi$ with the norm map $\mathrm{N}_{\F_{q^e}/\F_q}$. A compatible system of $q$-Weil numbers is a function $a$ from pairs $(q, \chi)$ to $q$-Weil numbers such that if $a(q, \chi)=\alpha$ then $a(q^e, \chi_e) = \alpha^e$ for all $e \in \N$. More generally, a compatible system of sums of $q$-Weil numbers is a function $a$ from pairs $(q, \chi)$ to $\C$, such that for each $(q, \chi)$, it takes the values
\begin{equation*}
a(q^e, \chi_e) = \sum_{j \in J} c_j \alpha_j^{e}
\end{equation*}
for all $e \in \N$, where $J$ is a finite set of indices, each $c_j$ is an integer, and each $\alpha_j$ is a $q$-Weil number. $J$, $c_j$, and $\alpha_j$ all depend on $(q, \chi)$, subject to the compatibility property for extensions.

For example, $q^{d/2}$ for any $d \geq 0$ is a compatible system of $q$-Weil numbers, and any polynomial in $\sqrt{q}$ with integer coefficients is a compatible system of sums. For a more interesting example, consider the finite field Gauss sums
\begin{equation*}
g_\chi = \sum_{a \in \F_q} \chi(b) e^{2 \pi i \Tr(a)/p}
\end{equation*}
where $\Tr$ denotes the trace from $\F_q$ to a field of prime order $p$. This is a complex number of absolute value $\sqrt{q}$. The definition extends to $(q^e, \chi_e)$. The Hasse-Davenport lifting relation is 
\begin{equation*}
-g_{\chi_e} = (-g_{\chi})^e
\end{equation*}
which means that $-g_\chi$ is a compatible system of $q$-Weil numbers (though $g_\chi$ is not). This further implies that any polynomial in $g_\chi$ with integer coefficients is a compatible system of sums. Note that for systems of sums, it is not possible to check compatibility or determine the integers $c_j$ and Weil numbers $\alpha_j$ from finitely many values of the function; this requires examining infinitely many extensions of $(q, \chi)$.

We now state the five axioms. In addition to the finite field $\F_q$ and character $\chi$ of even order $n$, our multiple Dirichlet series depend on an $r \times r$ symmetric matrix $M$ with integer entries $M_{ij}$. We set $\vec{x}=(x_1, \ldots x_r) \in \C^r$. The multiple Dirichlet series we study will be denoted 
\begin{equation}
Z(\vec{x}; q, \chi, M) = \sum_{f_1, \ldots f_r \in \F_q[t]^+} a(f_1, \ldots f_r; q, \chi, M) x_1^{\deg f_1}\cdots x_r^{\deg f_r}.
\end{equation}
For a prime $\pi \in \F_q[T]$, the $\pi$-part of the multiple Dirichlet series will be denoted
\begin{equation}
Z_\pi(\vec{x}; q, \chi, M) = \sum_{d_1, \ldots d_r \geq 0} a(\pi^{d_1}, \ldots \pi^{d_r}; q, \chi, M) x_1^{d_1 \deg \pi}\cdots x_r^{d_r \deg \pi}.
\end{equation}
In later sections we will introduce an additional parameter $\vec{k} \in \Z^r$, which limits the terms appearing in the series based on congruences modulo $n$. We will often omit the parameters $q, \, \chi, \, M$ after the semicolon when they are clear from context. Note that in the function field setting, Dirichlet series are equivalent to power series. The variables $x_i$ should be considered as analogues of $q^{-s_i}$ in the number field case. 

The five axioms characterize the coefficients $a(f_1, \ldots f_r; q, \chi, M)$. They are as follows:

\begin{enumerate}
\item \label{Axiom1} \emph{(Twisted Multiplicativity)} If $\gcd(f_1\cdots f_r, g_1\cdots g_r)=1$, then we have 
\begin{equation}
\begin{split}
    &a(f_1g_1, \ldots f_rg_r; q, \chi, M) \\
    &= a(f_1, \ldots f_r; q, \chi, M)a(g_1, \ldots g_r; q, \chi, M) \prod_{i=1}^r \prod_{j=i}^r \res{f_i}{g_j}_\chi^{M_{ij}} \res{g_i}{f_j}_\chi^{M_{ij}}.
    \end{split}
\end{equation}
\item \label{Axiom2} \emph{(Normalization)} We have $a(1, \ldots 1; q, \chi, M) = 1 $ and $a(1, \ldots, \pi, \ldots 1; q, \chi, M)=1$ for all linear polynomials $\pi$. 
\item \label{Axiom3} \emph{(Prime Power Coefficients)} For all primes $\pi \in \F_q[T]$, we have 
\begin{equation}
a(\pi^{d_1}, \ldots \pi^{d_r}; q, \chi, M) = \res{\pi'}{\pi}_\chi^{\sum d_i M_{ii}} \sum_{j \in J(d_1, \ldots d_r; q, \chi, M)} c_j \alpha_j^{\deg \pi} 
\end{equation}
where the sum $\sum\limits_{j \in J(d_1, \ldots d_r; q, \chi, M)} c_j \alpha_j$, considered as a function of $(q, \chi)$, is a compatible system of sums of $q$-Weil numbers.

\item \label{Axiom4} \emph{(Local-to-Global Principle)} We have 
\begin{equation}
 \sum_{\substack{f_1, \ldots f_r \in \F_q[T]^+ \\ \deg f_i = d_i}} a(f_1, \ldots f_r; q, \chi, M) = \sum_{j \in J(d_1, \ldots d_r; q, \chi, M)} c_j \frac{q^{\sum d_i}}{\overline{\alpha_j}}.
\end{equation}

\item \label{Axiom5} \emph{(Dominance Principle)} For each $j \in J(d_1, \ldots d_r; q, \chi, M)$ with $\sum d_i \geq 2$, $\alpha_j$ is a $q$-Weil number of weight less than $\sum d_i - 1$. 
\end{enumerate}

Notice that the only dependence on the matrix $M$ is in Axiom \ref{Axiom1}: $M$ and $\chi$ govern the twisted multiplicativity of the coefficients. The entries of $M$ can be taken as integers modulo $n$. The restriction to $\chi$ of even order is for convenience only; if $\chi$ has odd order, we can replace it with a square root and multiply each entry of $M$ by 2. Axioms \ref{Axiom1}, \ref{Axiom2}, and the simplest case of Axiom \ref{Axiom3} where $\sum d_i = 1$ are sufficient to compute $a(f_1, \ldots f_r; q, \chi, M)$ when $f_1\cdots f_r$ is squarefree. In this case, we have 
\begin{equation} \label{SquarefreeCoefficients}
a(f_1, \ldots f_r; q, \chi, M) = \prod_{i=1}^{r} \res{f_i'}{f_i}_{\chi}^{M_{ii}} \prod_{1\leq i<j \leq r} \res{f_i}{f_j}_{\chi}^{M_{ij}}.
\end{equation}
This is the heuristic form of a multiple Dirichlet series coefficient. The remaining axioms determine the more complicated coefficients which contain higher powers of primes. Some of these are nonzero even when the corresponding product of $n$th power residue symbols would vanish. 

Axiom \ref{Axiom3} states that $a(\pi^{d_1}, \ldots \pi^{d_r}; q, \chi, M)$ for $\pi$ linear, which is a local coefficient of $Z_\pi(\vec{x}; q, \chi, M)$, is a compatible system of sums. Axiom \ref{Axiom4} states that 
\begin{equation*}
\sum\limits_{\substack{f_1, \ldots f_r \in \F_q[T]^+ \\ \deg f_i = d_i}} a(f_1, \ldots f_r; q, \chi, M)
\end{equation*} 
which is a global power series coefficient of $Z(\vec{x}; q, \chi, M)$, is a compatible system of sums as well, and that these two functions are related by a change of variables. Once $a(\pi^{d_1}, \ldots \pi^{d_r}; q, \chi, M)$ is known for $\pi$ linear, it can be extended to arbitrary prime $\pi$ using Axiom \ref{Axiom3}. The prime-power coefficients then determine all coefficients $a(f_1, \ldots f_r; q, \chi, M)$ by Axiom \ref{Axiom1}. Axiom \ref{Axiom5} ensures that any prime-power coefficient $a(\pi^{d_1}, \ldots \pi^{d_r}; q, \chi, M)$ is relatively small compared to the global coefficients it contributes to. 

As a simple example of the axioms, take $M=(0)$ with $q, \chi$ arbitrary. Then the expression $a(f_1; q, \chi, M)=1$ for all $f_1$ satisfies the axioms. $a(\pi; q, \chi, M)=1$ is a compatible system of $q$-Weil numbers of weight 0. The local-to-global principle holds in this case because 
\begin{equation*}
\sum_{\substack{f_1 \in \F_q[T]^+ \\ \deg f_1 = d_1}} a(f_1; q, \chi, M) = \sum_{\substack{f_1 \in \F_q[T]^+ \\ \deg f_1 = d_1}} 1 = q^d.
\end{equation*}
The multiple Dirichlet series in this case is the function field zeta function 
\begin{equation*}
Z(x_1; q, \chi, M) = (1-qx_1)^{-1} = \prod_{\pi \text{ prime}} (1-x_1^{\deg \pi})^{-1}.
\end{equation*}
The next simplest example is for $\chi$ nontrivial and $M=(1+\tfrac{n}{2})$. In this case, Proposition \ref{PropCoeff1} implies 
\begin{equation*}
a(\pi^d; q, \chi, M) = \left\lbrace \begin{array}{cc} \xi(-1)^{d(d-1)/2}q^{d-d/n} g_\chi^{-d} & d \equiv 0 \bmod n \\\xi(-1)^{d(d-1)/2}q^{d-1-(d-1)/n}g_\chi^{1-d} & d \equiv 1 \bmod n \\ 0 & \text{otherwise} \end{array} \right.
\end{equation*}
for $\pi \in \F_q[T]^+$ linear. This is a compatible system of $q$-Weil numbers of weight $d-\frac{2d}{n}$ when $n|d$, or $d-1-\frac{2(d-1)}{n}$ when $n|d-1$. We then obtain a formula for $a(f_1; q, \chi, M)$ for all $f_1 \in \F_q[T]^+$ using Axioms \ref{Axiom1} and \ref{Axiom3}. The evaluation of the corresponding global coefficients 
\begin{equation*}
\sum_{\substack{f_1 \in \F_q[T]^+ \\ \deg f_1 = d_1}} a(f_1; q, \chi, M) = \left\lbrace \begin{array}{cc} \xi(-1)^{d(d-1)/2}q^{d+d/n} g_\chi^{-d} & d \equiv 0 \bmod n \\\xi(-1)^{d(d-1)/2}q^{d+(d-1)/n}g_\chi^{1-d} & d \equiv 1 \bmod n \\ 0 & \text{otherwise} \end{array} \right.
\end{equation*}
confirming Axiom \ref{Axiom4}, is a nontrivial identity of Gauss sums. 

In \cite[Theorem 1.1]{s-amds}, the first author proves the following general theorem:
\begin{theorem}
For any $M$, there exists a unique system of coefficients $a(\vec{f}; q, \chi, M)$ satisfying the axioms.
\end{theorem}
\noindent The existence part of the theorem is proven by explicit construction of a perverse sheaf on $\mathbb{A}^{d_1+\cdots+d_r}$, viewed as the parameter space of tuples $\vec{f}=(f_1, \ldots f_r)$ with $\deg f_i=d_i$. The trace of Frobenius on the stalk at $\vec{f}$ gives $a(\vec{f}; q, \chi, M)$. Then the axioms translate into statements about the cohomology of this sheaf; in particular, Axiom \ref{Axiom4} is a Verdier duality statement and Axiom \ref{Axiom5} is a cohomological purity statement. The perverse sheaf construction will not be used in this paper to prove anything about the multiple Dirichlet series (though in Section \ref{SectionTopology} we will use it to deduce topological and algebraic consequences from the functional equations satisfied by these series). The uniqueness part of the theorem is proven by a combinatorial induction, similar to the arguments in Section \ref{SectionCoeffs} below.

\cite[Theorem 1.1]{s-amds} generalizes earlier work of Diaconu and Pasol~\cite{DiaconuPasol}, who proved a similar result restricted to the case when $\chi$ is a quadratic character and $M$ has a certain special form. Their existence proof relies on the combinatorics and \'{e}tale cohomology of compactifications of moduli spaces of hyperelliptic curves. This geometry and combinatorics is all encapsulated in the definition of the relevant perverse sheaves. Since the series of \cite{DiaconuPasol} are a special case of the series we study, our results apply, though the functional equations in this case were obtained earlier in \cite{WhiteheadThesis}.

\subsection{Outline and Main Results}

Our goal in this paper is to establish and study the functional equations of the axiomatic multiple Dirichlet series $Z(\vec{x}; q, \chi, M)$. The structure of the paper is as follows. In Section \ref{SectionCoeffs}, we generalize the formula \eqref{SquarefreeCoefficients} to the case when $f_1\cdots f_{r-1}$ is squarefree. It suffices to determine the local coefficients $a(1, \ldots 1, \pi^d; q, \chi, M)$ and $a(1, \ldots \pi, \ldots 1, \pi^d; q, \chi, M)$ for $d \in \N$--then all the coefficients are determined by twisted multiplicativity. The general formula is in Proposition \ref{PropCoeff4}. The method of proof is induction on $d$, using a comparison of local and global coefficients via Axioms \ref{Axiom4} and \ref{Axiom5} to complete the inductive step. We do not state the formula in full generality here--see \eqref{Coeff4}. Instead, we highlight the two cases used in subsequent sections. First, if $M_{rr}=0$, then Equation \eqref{SquarefreeCoefficients} is still valid, i.e. 
\begin{equation*}
a(f_1, \ldots f_r; q, \chi, M) = \prod_{i=1}^{r} \res{f_i'}{f_i}_{\chi}^{M_{ii}} \prod_{1\leq i<j \leq r} \res{f_i}{f_j}_{\chi}^{M_{ij}}
\end{equation*}
for $f_1\cdots f_{r-1}$ squarefree and $f_r$ arbitrary. This is shown to follow from Proposition \ref{PropCoeff4} in Proposition \ref{PropDirichletCoeff}. This means that, as a function of $f_r$, $a(f_1, \ldots f_r; q, \chi, M)$ is a constant times a Dirichlet character. The second case is a bit more complicated. For $i$ between $1$ and $r$, let $n_i=\frac{n}{\gcd(n, M_{ii}+\tfrac{n}{2})}$. For $j\neq i$, let $n_{ij}$ be the least nonnegative integer such that $n_{ij}(M_{ii}+\tfrac{n}{2}) \equiv -M_{ij} \bmod n$, if such an integer exists. Assume that $n_{ri}$ exists, i.e. $\gcd(M_{rr}+\tfrac{n}{2}, n)|M_{ri}$ for $i=1, \ldots r-1$. Let $F=f_1^{n_{r1}}\cdots f_{r-1}^{n_{r \, (r-1)}}$ and 
\begin{equation*}
\varepsilon= \prod_{i=1}^{r-1} \res{f_i'}{f_i}_{\chi}^{M_{ii}} \, \prod_{1 \leq i < j \leq r-1} \res{f_i}{f_j}_{\chi}^{M_{ij}}.
\end{equation*}
Then we have
\begin{equation*}
a(f_1, \ldots f_r) = \, \varepsilon \, \frac{\xi(-1)^{\deg f_r(\deg f_r -1)/2}}{g_{\xi \chi^{M_{rr}}}^{\deg f_r}} \sum_{\substack{u \in \F_q[T]^+ \\ u^{n_r}|f_r}} q^{(n_r-1)\deg u} g_{\xi \chi^{M_{rr}}}(F, f_r/u^{n_r})
\end{equation*}
where $g_{\xi \chi^{M_{rr}}}(F, f_r/u^{n_r})$ is a function field Gauss sum defined in equation \eqref{ff-gauss-sum} below. In this case, as a function of $f_r$, $a(f_1, \ldots f_r; q, \chi, M)$ behaves like a Gauss sum modulo $f_r$. 

In Section \ref{SectionKubota}, we prove functional equations in the setting where $a(f_1, \ldots f_r; q, \chi, M)$ behaves like a Gauss sum. Assume that $q \equiv 1 \bmod 4$, and $M$ is such that $\gcd(n, M_{rr}+n/2) | M_{ri}$, for all $i$. Setting $\vec{f}=(f_1, \ldots f_{r-1})$, we define the single variable subseries
\begin{align}
&D(x, \vec{f}; q, \chi, M) = \sum_{f_r \in \F_q[T]^+} a(\vec{f}, f_r; q, \chi, M) x^{\deg f_r} \\
&D_\pi(x, \vec{f}; q, \chi, M) = \sum_{j\geq 0} a(\vec{f}, \pi^j; q, \chi, M) x^{j\deg \pi}.
\end{align}
Later, we will introduce an additional parameter $k \in \Z$, which limits the terms appearing in the series based on congruences modulo $n$. Because the coefficients are essentially Gauss sums, $D(x, \vec{f})$ for $f_1\cdots f_{r-1}$ squarefree is essentially a Kubota $L$-series, as studied in the function field case by \cite{Hoffstein} and \cite{Patterson}, with a functional equation in $x \mapsto \frac{g_{\xi \chi^{M_{rr}}}^2}{q^2x}$. Theorem \ref{TheoremFE} extends this functional equation to arbitrary $\vec{f}$. The known functional equations of Kubota $L$-series become the input for an induction on $\deg f_1\cdots f_{r-1}$. A comparison of the local series $D_\pi$ and the global series $D$, using Axioms \ref{Axiom4} and \ref{Axiom5}, completes the inductive step. We derive a multivariable functional equation from the single variable version. For this introduction, we will state the multivariable functional equations in an abridged form, emphasizing the underlying transformation of $\vec{x}$. The full functional equations are given in vector form with an explicit scattering matrix, using the parameter $\vec{k}$ which we have not introduced yet. We expect $r$ different functional equations, each one transforming some $x_i$ to a constant divided by $x_i$. We will describe the $i$th functional equation, assuming that $\gcd(n, M_{ii}+n/2) | M_{ij}$, for all $j$. Define a transformation $\sigma_{i; q, \chi, M}: \C^r \to \C^r$ as follows:
\begin{equation}
\sigma_{i; q, \chi, M}(\vec{x})_j = \left\lbrace \begin{array}{cc} 
\frac{g_{\xi \chi^{M_{ii}}}^2}{q^2x_i} & j=i \\
x_j\left(\frac{q x_i}{g_{\xi \chi^{M_{ii}}}}\right)^{n_{ij}} & j \neq i \end{array} \right. 
\end{equation}
with $n_{ij}$ as defined in the previous paragraph. We prove the following:
\begin{manualtheorem}{\ref{TheoremMultiFE}}
Assume that $q \equiv 1 \bmod 4$, and $M$ is such that $\gcd(n, M_{ii}+n/2) | M_{ij}$, for all $j$. There is an explicit functional equation relating $Z(\vec{x}; q, \chi, M)$ to $Z(\sigma_i(\vec{x}); q, \chi, M)$.
\end{manualtheorem}
\noindent Theorem \ref{TheoremLocalMultiFE} gives an analogous local functional equation.

In Section \ref{SectionDirichlet}, we prove functional equations in the setting where $a(f_1, \ldots f_r; q, \chi, M)$ behaves like a Dirichlet character. Assume that $M_{rr}=0$. In this case, $D(x, \vec{f})$ for $f_1\cdots f_{r-1}$ squarefree is essentially a Dirichlet $L$-function, with a functional equation in $x \mapsto \frac{1}{qx}$. Theorem \ref{TheoremDirichletFE} extends this functional equation to arbitrary $\vec{f}$, with a similar proof structure to Theorem \ref{TheoremFE}. Again, we derive a multivariable functional equation from the single variable version. For this introduction, we will state the $i$th multivariable functional equation in abridged form. Assume $M_{ii}=0$. Let $e_{ij}=1$ if $M_{ij} \neq 0$ and $e_{ij}=0$ if $M_{ij}=0$. Define a transformation $\sigma_{i;q, \chi, M}:\C^r \to \C^r$ as follows:
\begin{equation}
(\sigma_{i;q, \chi, M}(\vec{x}))_j=\left\lbrace \begin{array}{cc}
\frac{1}{qx_i} & j=i \\ (g_{\chi^{M_{ij}}}x_i)^{e_{ij}}x_j & j \neq i \end{array} \right. .
\end{equation}
Define a transformation $\tau_i$ on $r \times r$ symmetric integer matrices as follows: 
\begin{equation}
\begin{split}
&(\tau_i(M))_{ii}=M_{ii}=0 \\
&(\tau_i(M))_{ij}=-M_{ij} \text{ for all } j\neq i \\
&(\tau_i(M))_{jj} = M_{jj}+e_{ij}(M_{ij}+n/2) \text{ for all } j\neq i \\
&(\tau_i(M))_{hj} = M_{hj}+e_{ih}e_{ij}(M_{ih}+M_{ij}) \text{ for all } h\neq j \neq i .
\end{split}
\end{equation}
We prove the following:
\begin{manualtheorem}{\ref{dirichlet-multivariable-global}}
Suppose $M_{ii}=0$. There is an explicit functional equation relating $Z(\vec{x}; q, \chi, M)$ and $Z(\sigma_i(\vec{x}); q, \chi, \tau_i(M))$. 
\end{manualtheorem}
\noindent Note that the Dirichlet functional equation transforms the matrix $M$ as well as the variables $\vec{x}$. Theorem \ref{dirichlet-multivariable-local} gives an analogous local functional equation. A similar functional equation was proven, by a very different method, by Hase-Liu \cite{HaseLiu24}. The functional equations proven in Theorems \ref{TheoremMultiFE} and \ref{dirichlet-multivariable-global} coincide in the case when $M_{ii}=0$ and $\frac{n}{2} | M_{ij}$ for all $j \neq i$. In this case, the characters and Gauss sums appearing in the coefficients are quadratic. This is the setting in which the second author proved functional equations in \cite{WhiteheadThesis}. Sections \ref{SectionKubota} and \ref{SectionDirichlet} are generalizations of this proof in two different directions. 

Section \ref{SectionGroup} explains the algebraic structure of the functional equations we have proven and suggests potential applications to moments. Because the Dirichlet functional equation transforms between multiple Dirichlet series for different matrices $M$, we don't just have a group of symmetries relating one series to itself. Instead, we have a groupoid of symmetries acting on a set of multiple Dirichlet series, with functional equations relating one series to another. The groupoids that arise from axiomatic multiple Dirichlet series have already been tabulated in a different context. They are the Weyl groupoids of arithmetic root systems, which arise from the classification of Nichols algebras of diagonal type \cite{HeckenbergerClassification, HeckenbergerRank2}. In \cite{HeckenbergerRank2}, Heckenberger defines a generalization of ordinary root systems, using a bicharacter of $\Z^r \times \Z^r$ in place of a bilinear form. When the resulting collection of roots and Weyl groupoid is finite, he calls this object an arithmetic root system. We give the full definition in Section \ref{SectionGroup}, and explain how to translate the data $(q, \chi, M)$ for a multiple Dirichlet series into a bicharacter. Theorem \ref{translation-key-step} establishes that our functional equations correspond to simple reflections in the Weyl groupoid. We do not have access to all simple reflections because we can only prove the Dirichlet and Kubota functional equations described above. But we conjecture that each multiple Dirichlet series satisfies a functional equation for every simple reflection. 

In the situation where the generalized root system is finite and therefore is an arithmetic root system, and all simple reflections correspond to Dirichlet or Kubota functional equations, we obtain strong results. We will paraphrase these results here to avoid introducing additional notation. 
\begin{manualtheorem}{\ref{rational-fn}} Let $M$ be a symmetric integer matrix corresponding to an arithmetic root system, such that each simple reflection in the Weyl groupoid corresponds to a functional equation of Kubota or Dirichlet type. Then $Z ( \vec{x}; q, \chi, M) $ is in fact a rational function of the variables $x_1,\dots,x_r$.
\end{manualtheorem}
\begin{manualtheorem}{\ref{uniquely-determined}} Under the same hypotheses on $M$, the power series $Z ( \vec{x}; q, \chi, M) $ is determined up to a constant by the functional equations it satisfies. 
\end{manualtheorem}
\noindent This is the situation in which established analytic techniques for multiple Dirichlet series are most successful. For example, in the number field case, we would expect the analogous multiple Dirichlet series to have meromorphic continuation to $\C^r$. We give a complete list of such series in Theorem \ref{ClassificationTheorem}.

Examination of the list shows that members on it include function field analogues of almost all the multiple Dirichlet series which have been used to compute moments of $L$-functions, as listed at the beginning of this introduction. In the initial works, creative ideas and new techniques were required to construct each of these series, but we derive them all uniformly from the same classification. Moreover, it seems likely that any series whose existence, functional equations, and meromorphic continuation can be proven by the same set of methods, based on analysis of Dirichlet characters and Gauss sums, used in these prior works, would also fit into the axiomatic framework and therefore appear on our list (or appear as a twisted form of series on our list, as \cite{BrubakerFriedbergHoffstein} is a twisted form of \cite{BrubakerThesis}). We construct all at once all the series that could have been produced one at a time by these classical methods. Many of the new series we find are applicable to moments which have not been computed previously.  Our uniform framework should enable future results of interest to be, like Theorems \ref{rational-fn} and \ref{uniquely-determined}, proven simultaneously for all these old and new series.


A simple example will illustrate the results of Sections \ref{SectionCoeffs}-\ref{SectionGroup}. This example was already worked out in \cite[Prop. 4.1]{s-amds} and is discussed as part of Example \ref{super-An} below. Take $q$ and $\chi$ arbitrary, and $M=\begin{pmatrix} 0 & 1 \\ 1 & 0 \end{pmatrix}$. This gives rise to a multiple Dirichlet series with the heuristic form 
\begin{equation} \label{ChintaMohlerExample}
Z(x_1, x_2; q, \chi, M) \approx \sum_{f_1, f_2 \in \F_q[T]^+} \res{f_1}{f_2}_\chi x_1^{\deg f_1} x_2^{\deg f_2} = \sum_{f_2 \in \F_q[T]^+} L\left(x_1, \res{*}{f_2}_\chi \right) x_2^{\deg f_2}.
\end{equation}
which is a generating function for Dirichlet $L$-functions associated to characters of order $n$. This multiple Dirichlet series was used to compute the first moment of such $L$-functions in the number field case by Friedberg, Hoffstein, and Lieman in \cite{FriedbergHoffsteinLieman}, and in the function field case by Chinta and Mohler in \cite{ChintaMohler}. Note that equation \eqref{ChintaMohlerExample} is not a true equality because $a(f_1, f_2; q, \chi M)$ is not equal to $\res{f_1}{f_2}_\chi$ in certain cases where $f_1f_2$ is not squarefree. This leads to extra terms which must be sieved out to obtain a moment formula. 

Because $M_{11}=M_{22}=0$, the series $Z(x_1, x_2; q, \chi, M)$ satisfies two functional equations of Dirichlet type, with underlying transformations $\sigma_1, \, \sigma_2$. They relate the series to two other series associated to matrices $\tau_1(M)$, $\tau_2(M)$. Each of these new series satisfies an inverse Dirichlet functional equation relating it back to the original $Z(x_1, x_2; q, \chi, M)$, and another Kubota functional equation relating it to itself. The full groupoid of functional equations is illustrated in Figure \ref{FigureGroupoid}. This is the Weyl groupoid of a rank 2 arithmetic root system with three Dynkin diagrams. Because every simple reflection in the groupoid corresponds to a Dirichlet or Kubota functional equation, Theorem \ref{rational-fn} and Proposition \ref{uniquely-determined} apply. The series $Z(x_1, x_2; q, \chi, M)$ is a rational function, and it can be computed from the functional equations that 
\begin{equation}
Z(x_1, x_2; q, \chi, M)=\frac{1-q^2 x_1 x_2}{(1-q x_1) (1-q x_2) \left(1-q^{n+1} x_1^n x_2^n\right)}.
\end{equation}
This matches the result of \cite{ChintaMohler}. Our work determines the precise set of multiple Dirichlet series for which this type of computation can be carried out.

\begin{figure}
    \centering
    \begin{tikzpicture}
[
    node distance=1cm,
    box/.style={
        draw,
        rectangle,
        rounded corners,
        minimum width=3.8cm,
        minimum height=1cm,
        align=center,
    },
    arrow/.style={-Stealth, thick}
]

\node[box] (node1) {$\tau_1(M)=\begin{pmatrix} 0 & -1 \\ -1 & 1+\tfrac{n}{2} \end{pmatrix}$};
\node[box] (node2) [below=of node1] {$M=\begin{pmatrix} 0 & 1 \\ 1 & 0 \end{pmatrix}$};
\node[box] (node3) [below=of node2] {$\tau_2(M)=\begin{pmatrix} 1+\tfrac{n}{2} & -1 \\ -1 & 0 \end{pmatrix}$};

\draw[arrow] ($(node1.south)-(0.3,0)$) -- node[left,xshift=-2pt] {$\sigma_1:(x_1, x_2) \mapsto \left(\frac{1}{qx_1}, g_{\chi^{-1}} x_1 x_2 \right)$} ($(node2.north)-(0.3,0)$);
\draw[arrow] ($(node2.north)+(0.3,0)$) -- node[right,xshift=2pt] {$\sigma_1:(x_1, x_2) \mapsto \left(\frac{1}{qx_1}, g_\chi x_1 x_2 \right)$} ($(node1.south)+(0.3,0)$);

\draw[arrow] ($(node2.south)-(0.3,0)$) -- node[left,xshift=-2pt] {$\sigma_2:(x_1, x_2) \mapsto \left(g_\chi x_1 x_2, \frac{1}{qx_2} \right)$} ($(node3.north)-(0.3,0)$);
\draw[arrow] ($(node3.north)+(0.3,0)$) -- node[right,xshift=2pt] {$\sigma_2:(x_1, x_2) \mapsto \left(g_{\chi^{-1}} x_1 x_2, \frac{1}{qx_2} \right)$} ($(node2.south)+(0.3,0)$);

\draw[arrow] ($(node1.north)+(0.3,0)$) to[out=90,in=90,looseness=2] 
              node[above] {$\sigma_2:(x_1, x_2) \mapsto \left( \frac{qx_1 x_2}{g_\chi}, \frac{g_\chi^2 }{q^2x_2} \right)$}
              ($(node1.north)-(0.3,0)$);

\draw[arrow] ($(node3.south)-(0.3,0)$) to[out=270,in=270,looseness=2] 
              node[below] {$\sigma_1:(x_1, x_2) \mapsto \left(\frac{g_\chi^2 }{q^2x_1}, \frac{qx_1 x_2}{g_\chi} \right)$}
              ($(node3.south)+(0.3,0)$);

\end{tikzpicture}
    \caption{Weyl groupoid of functional equations.}
    \label{FigureGroupoid}
\end{figure}
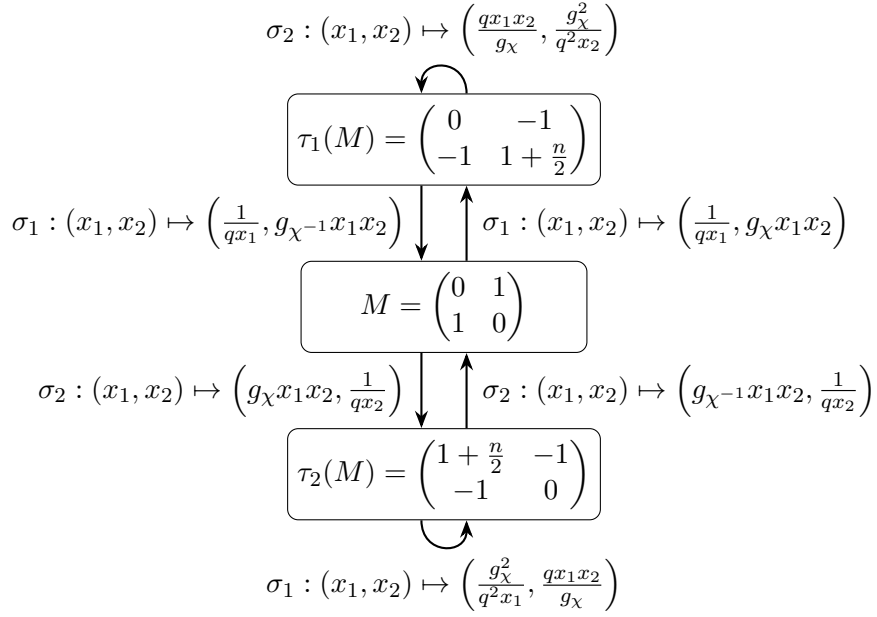

We expect that there should exist a hybrid functional equation unifying the Dirichlet and Kubota types, though we do not conjecture an exact form for this question. Rather than being based on Dirichlet or Kubota $L$-series, this hybrid functional equation should be based on a single-variable series $D(x, \vec{f}; q, \chi, M)$ whose coefficient at $f_r$ is a Dirichlet character of $f_r$ times a Gauss sum modulo $f_r$. To our knowledge, functional equations for these series have not been established. Because the proofs in Sections \ref{SectionKubota} and \ref{SectionDirichlet} require known single-variable functional equations as inputs, we cannot simply imitate these proofs for the hybrid functional equation. However, if the single-variable functional equation of $D(x, \vec{f}; q, \chi, M)$ for $f_1\cdots f_{r-1}$ squarefree can be established by other means, then the same strategy of Sections \ref{SectionKubota} and \ref{SectionDirichlet} should establish general single-variable and multivariable functional equations.

\subsection{Relations with topology and quantum algebra}

The last parts of this paper, which is not necessary for the main results, investigates the connection of our multiple Dirichlet series to several other topics, obtained by a series of correspondences. These topics, expressed with some overlap, include quantum groups, intersection cohomology of local systems on configuration space, Nichols algebras, and perverse sheaves. Our functional equations lead to perhaps-unexpected relations between the dimensions of certain cohomology groups in these different settings, often enabling the computation of these dimensions. We summarize these briefly here, but the interested reader should examine section 6, where these topics are discussed in more depth and, in particular, relevant definitions are given.

One correspondence is due to Kapranov and Schechtman~\cite{KapranovSchechtman}. Kapranov and Schechtman studied perverse sheaves on symmetric powers of the complex plane arising from intermediate extensions of local systems on the configuration space of $n$ points on the plane. They showed an isomorphism between the cohomology groups of the perverse sheaf and cohomology groups ($\operatorname{Ext}$ and $\operatorname{Tor}$) of a Nichols algebra. \cite{s-amds} expressed the coefficients of axiomatic multiple Dirichlet series as traces of Frobenius on perverse sheaves on symmetric powers of the affine line over finite fields. The affine line over finite fields is analogous to the affine line over the complex numbers, whose underlying topological space is the plane. The construction of perverse sheaves in \cite{s-amds} makes equal sense over the complex numbers, and it is possible to relate the perverse sheaves produced in the finite field and complex number setting. The complex analogues of the perverse sheaves of \cite{s-amds} turn out to to correspond, via \cite{KapranovSchechtman}, to the Nichols algebras of diagonal braided vector spaces.

Combining all these correspondences, axiomatic multiple Dirichlet series may be related to the Nichols algebras of diagonal braided vector spaces. It is not immediately obvious which properties of the multiple Dirichlet series correspond to which properties of the Nichols algebra, but the functional equations of the multiple Dirichlet series apparently correspond to the Weyl groupoid of the algebra, and the multiple Dirichlet series may be meromorphic everywhere whenever the algebra is finite-dimensional.

An important special case of Nichols algebras is the subalgebra of Lusztig's small quantum group generated by the positive roots. In this special case, a version of the Kapranov-Schechtman correspondence was obtained earlier by Bezrukavnikov, Finkelberg, and Schechtman~\cite{Bezrukavnikov1998}. The corresponding perverse sheaves are relevant to the Weyl group multiple Dirichlet series. The fact that quantum groups are connected to Weyl group multiple Dirichlet series is well-known, but this seems to give a different connection, because it involves dimensions of cohomology groups on the quantum group side instead of other invariants, and because it involves only a single multiple Dirichlet series at a time and not the whole space of Whittaker functions.

Using our proofs of the functional equations, we prove certain linear relations between the dimensions of cohomology groups in all of these settings: cohomology of Nichols algebras, intersection  cohomology of certain local systems on configuration spaces, and cohomology of certain perverse sheaves. In the case of an arithmetic root system such that each simple reflection corresponds to a functional equation of Kubota or Dirichlet type, these linear relations enable the computation of the dimensions of these cohomology groups. In particular, this applies to quantum groups, where we give a combinatorial expression for the dimensions of the associated graded components of the Ext groups of the trivial module with itself over the subalgebra of Lusztig's small quantum group generated by the positive roots, partially generalizing \cite[Theorem 5.6.1]{Drupieski2012}. 

It would be interesting to understand if there exists a purely algebraic proof of these relations, and we give positive and negative results in this direction.

\subsection{Further Questions}

Our work suggests future problems in the analysis, geometry, and algebra of axiomatic multiple Dirichlet series. 

\begin{enumerate} 

\item \emph{Prove a hybrid functional equation}. As discussed above, this requires proving the functional equation of a hybrid $L$-series whose coefficients are Dirichlet characters multiplied by Gauss sums. The functional equation might arise from an automorphic interpretation of this series.

\item \emph{Construct multiple Dirichlet series over number fields analogous to the ones constructed here over $\mathbb F_q[T]$.}  The local coefficients of these series should match the local coefficients of our series, at least at primes not dividing the number of roots of unity of the base field. One could also ask for analogous series over function fields of higher genus. The strongest possible result in this direction would be the statement that for all finite arithmetic root systems and all global fields containing the necessary roots of unity, the global multiple Dirichlet series has meromorphic continuation to $\C^r$, with an explicit list of poles. 

We will briefly explain why it is reasonable to expect that one can prove functional equations and meromorphic continuation to $\C^r$ for the analogues over arbitrary global fields of the series listed in Theorem \ref{ClassificationTheorem}. The global functional equations will follow from the known local functional equations, via analogues of Propositions \ref{LocalImpliesGlobal} and \ref{LocalDirichletImpliesGlobal} -- similar statements are proven for Weyl group multiple Dirichlet series in \cite{ChintaGunnellsJAMS}. It is also necessary to establish an initial region of convergence and a simple meromorphic continuation using convexity bounds for $L$-functions, so that each functional equation gives an identity of meromorphic functions on overlapping regions of $\C^r$. Then, following \cite{DiaconuGoldfeldHoffstein}, Bochner's principle in several complex variables immediately gives meromorphic continuation to $\C^r$. Each step requires substantial work, but the techniques required are well-established in the multiple Dirichlet series literature.

\item \emph{Extract moment formulas from Examples \ref{new-example-1}, \ref{new-example-2}, and \ref{new-example-3}}. These multiple Dirichlet series are computable but have not appeared in the literature before. Each one has a new arithmetic application, which should be provable over function fields, and, if problem (2) is solved, over number fields as well.

Again, the techniques to extract moments are rather technical, but well-established in the multiple Dirichlet series literature. Moment formulas are obtained from the multiple Dirichlet series by contour integration or Tauberian arguments. This requires explicit computation of the poles and residues, using the functional equations and local information. The final ingredient is a technical sieving step to remove $L$-functions with nonsquarefree conductors. This requires analysis specific to the symmetry type, but has been carried out in many cases, e.g. \cite{Diaconu, DiaconuWhitehead} for the root system $D_4$. 

\item \emph{Extend the axiomatic definition to include twisted multiple Dirichlet series}. The most general Weyl group multiple Dirichlet series constructions allow for twisting. The heuristic form of a twisted multiple Dirichlet series has the coefficient at $(f_1, \ldots f_r)$ multiplied by fixed Dirichlet characters of $f_1, \ldots f_r$. Twisted Weyl group multiple Dirichlet series are general Fourier-Whittaker coefficients of metaplectic Eisenstein series; the untwisted version is the first nonconstant coefficient. At present, the axioms do not allow for twisting, but it is an interesting problem to generalize them and to prove functional equations in this setting.

\item \emph{Show meromorphic continuation of axiomatic multiple Dirichlet series to as large a domain as possible when the groupoid of functional equations is infinite}. Hase-Liu establishes a domain of convergence for all axiomatic multiple Dirichlet series in \cite{HaseLiu24} -- can this be extended using improved estimates for the coefficients? Since each coefficient is a sum of $q$-Weil numbers of bounded weight, the main problem is to bound the number of terms in the sum. This is analogous to estimating the dimension of the underlying cohomology groups, rather than the trace of Frobenius. 

Predicting the poles is also an interesting problem. In the case of a finite groupoid, the poles are completely determined by applying the functional equations to an initial set of poles; they correspond to roots in the arithmetic root system. From the list of poles, we can construct the denominator of the rational function $Z(\vec{x}; q, \chi, M)$ as a product of function field zeta functions. The zeta functions are Eulerian and involve only powers of $x_i^n$, so the numerator of $Z$ is a multivariable Dirichlet polynomial satisfying the same twisted multiplicativity as the full rational function. In the case of an infinite groupoid, we expect to see additional poles besides the ones whose existence is implied by the functional equations. Determining theses poles and their multiplicities is a subtle problem, related to the problem of determining imaginary root multiplicities in Kac-Moody algebras. The denominator of $Z$ will be an infinite product of zeta functions. The numerator will not be a polynomial, but we expect it to be a power series which absolutely converges on the largest possible domain. The bounds of Axiom \ref{Axiom5} are tailored for this convergence. Meromorphic continuation of multiple Dirichlet with infinite groupoids of functional equations would make many new moment results possible.

\item  \emph{Classify affine arithmetic root systems}. One can define an analogue of arithmetic root systems where, as with affine root systems, rather than assuming that the root system has finitely many Weyl chambers, one only assumes that the Weyl chambers form finitely many orbits under an infinite abelian group of linear transformations. Meromorphic continuation for certain multiple Dirichlet series associated to affine root systems was proved by the second author in \cite{WhiteheadThesis}, and it is possible that the method can be generalized to a multiple Dirichlet series defined using an arbitrary affine arithmetic root system, in which case it would be interesting to have a list of these. This would generalize the rank two case handled in \cite{Cuntz2018}.

\end{enumerate}

\textbf{Acknowledgements:} We thank Iv\'an Angiono, Valentin Buciumas, Adrian Diaconu, Michael Finkelberg, Holley Friedlander, Dorian Goldfeld, Matthew Hase-Liu, Jeff Hoffstein, Ivan Losev, Mikhail Kapranov, Dinakar Muthiah, Daniel Nakano, Anna Pusk\'{a}s, Steve Sawin, and Vadim Schechtman for helpful conversations related to this project. While working on this project, Will Sawin was supported by NSF grants DMS-2101491 and DMS-2502029 and a Sloan Research Fellowship.

 \section{Explicit Evaluation of Coefficients} \label{SectionCoeffs}

 \subsection{Further Notation and Lemmas on Gauss Sums}

 Fix an odd prime power $q$. Fix a character $\chi$ of $\F_q^*$, with even order $n$, and let $\xi:\F_q^* \to \lbrace \pm 1 \rbrace$ be the unique character of order 2. Fix an $r\times r$ symmetric integer matrix $M=(M_{ij})$ whose entries we assume lie in $[0,n)$. 

 We have already introduced the finite field Gauss sum
\begin{equation}\label{ff-gauss-sum}
g_\chi=\sum_{a \in \F_q} \chi(b) e^{2 \pi i \mathrm{Tr}(a)/p}
\end{equation}
where $\mathrm{Tr}$ denotes the trace from $\F_q$ to a field of prime order $p$. We will also employ Gauss sums defined in the function field setting as follows, for $f_1, f_2 \in \F_q[T]$:
\begin{equation} \label{GaussSumDefn}
g_\chi(f_1, f_2)=\sum_{h \in \F_q[T]/(f_2)} \res{h}{f_2}_\chi \exp(h f_1/f_2) 
\end{equation}
where $\exp(h f_1/f_2) = e^{2 \pi i \mathrm{Tr}(c_{-1})/p}$, and $c_{-1}$ is the coefficient of $T^{-1}$ in the expansion of $h f_1/f_2$ as a Laurent series in $T^{-1}$. By definition, we have $g_\chi(f_1, 1)=1$ and $g_{\chi}(1, T)=g_{\chi}$. 

We collect some important properties of $g_\chi(f_1, f_2)$, referring to \cite{s-amds} for proofs. First is a twisted multiplicativity property:
\begin{lemma} \label{GaussSumTwistedMult} \cite[Lemma 2.7]{s-amds} 
If $\gcd(f_1f_2, h_1h_2)=1$, then 
\begin{equation*}
g_\chi(f_1h_1, f_2h_2)=g_\chi(f_1, f_2) g_\chi(h_1, h_2) \res{f_1}{h_2}^{-1}_\chi \res{h_1}{f_2}_\chi^{-1} \res{f_2}{h_2}_\chi \res{h_2}{f_2}_\chi.
\end{equation*}
\end{lemma}
\noindent Next is a formula, a consequence of the Hasse-Davenport relations, which relates function field Gauss sums to the underlying finite field Gauss sums:
\begin{lemma} \label{GaussSumLifting} \cite[Lemma 2.4]{s-amds} 
For $f_2$ squarefree and $\gcd(f_1, f_2)=1$, we have
\begin{equation*}
g_\chi(f_1, f_2) = \xi(-1)^{\deg f_2(\deg f_2-1)/2} \res{f_1}{f_2}_\chi^{-1} \res{f_2'}{f_2}_{\xi \chi} g_\chi^{\deg f_2}.
\end{equation*}
\end{lemma}
\noindent The appearance of $\res{f_2'}{f_2}_{\xi \chi}$ in this formula is some initial evidence that Gauss sums will appear in axiomatic multiple Dirichlet series coefficients when the diagonal entries of the matrix $M$ are nonzero. Finally, we will use a version of Pellet's formula:
\begin{lemma} \label{Pellet} \cite[Lemma 2.3]{s-amds} 
We have
\begin{equation*}
\xi(-1)^{\deg f(\deg f -1)/2} \res{f'}{f}_\xi = \xi(\mathrm{Disc}(f))= (-1)^{\deg f} \mu(f)
\end{equation*}
where $\mu$ is the function field M\"{o}bius function.
\end{lemma}

We will need two more lemmas on Gauss sums to prove the main result of this section.

\begin{lemma}\label{GaussSumEval}
Suppose $f\in \F_q[T]^+$ has degree $d$ and $\chi$ has order $n$. Then
\begin{equation}
g_{\chi}(1, f) = \left\lbrace \begin{array}{cc} g_{\chi^d} \sum\limits_{\deg \nu = d-1} \res{\nu}{f}_{\chi} & n \nmid d  \\ 
-q \sum\limits_{\deg \nu = d-1} \res{\nu}{f}_{\chi} & n\mid d, \, f \neq f_0^n \\  
\phi(f) -q \sum\limits_{\deg \nu = d-1} \res{\nu}{f}_{\chi} & f = f_0^n \\  
\end{array} \right.
\end{equation}
where the sums are over monic polynomials $\nu$, $\phi(f)$ is Euler's totient function, and $f_0$ is the $n$th root of $f$ if $f$ is an $n$th power. 
\end{lemma}
\begin{proof}
By definition
\begin{equation*}
g_{\chi}(1, f) = \sum_{h \bmod f} \res{h}{f} \exp(h/f).
\end{equation*}
Writing each $h$ as  $a \nu$ with $a\in\F_q^*$ and $\nu \in \F_q[T]^+$ of degree $i<d$, we have
\begin{equation*}
g_{\chi}(1, f) = \sum_{i=0}^{d -2} \sum_{\deg \nu=i} \sum_{a \in \F_q^*} \res{a \nu}{f}_{\chi} +\sum_{\deg \nu=d -1} \sum_{a \in \F_q^*} \res{a \nu}{f}_{\chi} \exp(a/T).
\end{equation*}
If $n \nmid d$, then $\chi^d$ is a nontrivial character, and the sum over $a$ in the first part vanishes, while the sum in the second part is $g_{\chi^d}$. In this case,
\begin{equation*} 
g_{\chi}(1, f) = g_{\chi^d} \sum_{\deg \nu=d -1} \res{\nu}{f}_{\chi}.
\end{equation*}
If $n \mid d$, then the sum over $a$ in the first part is $q-1$, while the sum in the second part is $-1$. In this case,
\begin{equation*}
g_{\chi}(1, f) = (q-1)\sum_{i=0}^{d -2} \sum_{\deg \nu=i} \res{\nu}{f}_{\chi} -\sum_{\deg \nu=d -1} \res{\nu}{f}_{\chi}.
\end{equation*}
If $f$ is not an $n$th power, then $\res{*}{f}_{\chi}$ is a nontrivial character of $(\F_q[T]/f)^*$, and $\sum\limits_{\nu \bmod f} \res{\nu}{f}_{\chi} = 0$, giving 
\begin{equation*}
g_{\chi}(1, f) = -q \sum_{\deg \nu=d -1} \res{\nu}{f}_{\chi}.
\end{equation*} 
If $f$ is an $n$th power, then $\res{*}{f}_{\chi}$ is the trivial character of $(\F_q[T]/f)^*$, and $\sum\limits_{\nu\bmod f} \res{\nu}{f}_{\chi} = \phi(f)$, giving 
\begin{equation*}
g_{\chi}(1, f) = \phi(f) -q \sum_{\deg \nu=d -1} \res{\nu}{f}_{\chi}.
\end{equation*} 
\end{proof}

\begin{lemma}\label{GaussSumSum}
Suppose that $\chi$ has order $n$. For $0 \leq m < n-1$ and $\pi \in \F_q[t]^+$ linear, we have 
\begin{equation}
\begin{split}
&\sum_{\substack{f\in \F_q[t]^+ \\ \deg f = d}} g_{\chi}(\pi^{m}, f) = \left\lbrace\begin{array}{cc} 
1 & d=0 \\
q^{1+m}g_{\chi^{1+m}} & d=1+m \\
q^{d+d/n}(1-q^{-1}) & d>0, \, d\equiv 0 \bmod n \\
q^{d+(d-m-1)/n}(1-q^{-1})g_{\chi^{1+m}} & d>1+m, \, d \equiv 1+m \bmod n \\
0 & d \not\equiv 0, 1+m \bmod n \end{array}\right. .
\end{split}
\end{equation}
In the case $m=n-1$,
\begin{equation}
\sum_{\substack{f\in \F_q[t]^+ \\ \deg f = d}} g_{\chi}(\pi^{n-1}, f) = \left\lbrace\begin{array}{cc} 
1 & d=0 \\
-q^n & d=n \\
0 & d \neq 0, n
\end{array}\right. .
\end{equation}
\end{lemma}
\begin{proof}
By Lemma \ref{GaussSumTwistedMult}, if $\pi \nmid f$,  we have $g_{\chi}(\pi^{m}, f) = \res{\pi}{f}_{\chi}^{-m} g_{\chi}(1, f)$. If $f=\pi^{1+m}f^{(\pi)}$ and $\pi \nmid f^{(\pi)}$, we have
\begin{equation*}
\begin{split}
g_{\chi}(\pi^{m}, f) &=\res{\pi}{f^{(\pi)}}_{\chi}^{-m} \res{\pi}{f^{(\pi)}}_{\chi}^{1+m}\res{f^{(\pi)}}{\pi}_{\chi}^{1+m}g_{\chi}(\pi^{m}, \pi^{1+m})g_{\chi}(1, f^{(\pi)}) \\
&=\chi(-1)^{(\deg f-1-m)(1+m)} \res{\pi}{f^{(\pi)}}_{\chi}^{2+m} q^{m} g_{\chi^{1+m}}(1, \pi)g_{\chi}(1, f^{(\pi)}).
\end{split}
\end{equation*}
Outside these two cases, $g_{\chi}(\pi^{m}, f)=0$. Replacing $f^{(\pi)}$ by $f$ in the second case, we write:
\begin{equation} \label{GSSParts}
\begin{split}
&\sum_{\deg f = d} g_{\chi}(\pi^{m}, f) = \sum_{\substack{\deg f = d \\ \pi \nmid f}} \res{\pi}{f}_{\chi}^{-m} g_{\chi}(1, f) \\
&+\chi(-1)^{(d-1-m)(1+m)} q^{m} g_{\chi^{1+m}} \sum_{\substack{\deg f = d-1-m \\ \pi \nmid f}} \res{\pi}{f}_{\chi}^{2+m}  g_{\chi}(1, f).
\end{split}
\end{equation}

Next, we apply Lemma \ref{GaussSumEval} to replace $g_{\chi}(1,f)$ in both sums. After doing this, we will need to evaluate several character sums to complete the proof. We let $\Lambda_{\pi}(d,e)$ denote the number of pairs of monic polynomials $(f, \nu)$ with $\deg f = d$, $\deg \nu =e$, and $\gcd(f, \pi\nu)=1$. For convenience, we set $\Lambda_{\pi}(d,e)=0$ if $d$ or $e$ is negative or is not an integer. We have the double sum evaluation
\begin{equation*}
\sum_{\substack{\deg f = d \\ \deg \nu = d-1}} \res{\pi}{f}_{\chi}^{-m} \res{\nu}{f}_{\chi} = \Lambda_{\pi}\left(d, \frac{d-1- (m \% n)}{n}\right)
\end{equation*}
where $m \% n$ is the smallest nonnegative integer congruent to $m$ modulo $n$. 
This holds because the expression $\res{\pi}{f}_{\chi}^{-m} \res{\nu}{f}_{\chi}$ is a character of $f$ modulo $\pi \nu$, so the sum over $f$ vanishes unless $\pi^{-m} \nu$ is an $n$-th power. This requires that $\nu = \pi^{m \%  n} \nu_0^n$ for some $\nu_0$ of degree $\frac{d-1- (m \% n)}{n}$. In this case, the double sum is equal to the number of pairs $(f, \nu_0)$ with $\gcd(f, \pi \nu_0)=1$. We also have the evaluation 
\begin{equation*}
\sum_{\substack{\deg f = d \\ f=f_0^n}} \res{\pi}{f}_{\chi}^{-m} \phi(f)  = \Lambda_{\pi}\left(\frac{d}{n}, d\right)
\end{equation*}
as this sum is equal to the number of pairs $(f_0, \nu)$ with $\deg f_0 = d/n$, $\deg \nu_0=d$, and $\gcd(f_0, \pi \nu)=1$.

If $d \not \equiv 0, \, 1+m \bmod n$, then both sums in \eqref{GSSParts} vanish. If $d \equiv 0 \not\equiv 1+m \bmod n$, then \eqref{GSSParts} becomes
\begin{equation*}
\Lambda_{\pi}\left(\frac{d}{n}, d\right) + q^{1+m} \Lambda_{\pi}\left(d-1-m, \frac{d}{n}-1\right).
\end{equation*}
Only the terms with $f=f_0^n$ contribute in the first sum, while the second sum becomes a nonvanishing double sum. If $d \equiv 1+m \not \equiv 0 \bmod n$, then \eqref{GSSParts} becomes
\begin{equation*}
g_{\chi^d} \Lambda_{\pi}\left(d, \frac{d-1-m}{n}\right)+q^{m}g_{\chi^d} \Lambda_{\pi}\left(\frac{d-1-m}{n}, d-1-m\right).
\end{equation*}
The first sum becomes a nonvanishing double sum, while only the terms with $f=f_0^n$ contribute in the second sum. Finally, in the special case $m=n-1$, $d \equiv 0 \equiv 1+m \bmod n$, \eqref{GSSParts} becomes
\begin{equation*}
\begin{split}
& \Lambda_{\pi}\left(\frac{d}{n}, d\right) - q\Lambda_{\pi}\left(d, \frac{d-n}{n}\right) -q^{n-1}\Lambda_{\pi}\left(\frac{d-n}{n}, d-n\right)+q^{n}\Lambda_{\pi}\left(d-n, \frac{d}{n}-2\right).
\end{split}
\end{equation*}
In this case, both terms have a contribution from $f=f_0^n$ and a nonvanishing double sum. 

The expression $\Lambda_{\pi}(d,e)$ can be evaluated explicitly from its generating function, as 
\begin{equation*}
\Lambda_{\pi}(d,e) = \left\lbrace \begin{array}{cc} q^e & d=0 \\ \frac{q^{d+e}(1-q^{-1})(1-q^{-2d})}{1+q^{-1}} & e \geq d >0 \\ \frac{q^{d+e}(1-q^{-1})(1+q^{-2e-1})}{1+q^{-1}} & d>e \end{array}\right. .
\end{equation*}
Applying this formula leads to the desired results. 
\end{proof}

\subsection{Coefficient Formulas}
 
Our goal is to prove a formula for the axiomatic multiple Dirichlet series coefficients $a(f_1, \ldots f_r; q, \chi, M)$ for $f_1\cdots f_{r-1}$ squarefree. These involve local coefficients $a(1, \ldots 1, \pi^d; q, \chi, M)$, $a(1, \ldots \pi \ldots 1, \pi^d; q, \chi M)$ for arbitrary $d$. Series with these coefficients form the base case in our inductive proofs of functional equations in the following sections. Our formula generalizes \eqref{SquarefreeCoefficients}, which is valid for $f_1\cdots f_r$ squarefree, but is more complicated and requires a more involved proof. 

Based on our fixed $M$ and $\chi$, for $1 \leq i \leq r$ we set 
\begin{equation} \label{ni}
n_i= \ord \xi \chi^{M_{ii}}=\frac{n}{\gcd(n, M_{ii}+\frac{n}{2})} .
\end{equation} 
Multiplying by $\xi$ does not change the order of a character if that order is $0$ mod $4$. It multiplies or divides the order by $2$ if the order is odd or $2$ mod $4$ respectively. For $i \neq j$, we choose $n_{ij}$, $o_{ij}$ and $p_{ij}$ such that 
\begin{equation}\label{nij}
-M_{ij} = n_{ij}(M_{ii} + \frac{n}{2}) +o_{ij}n+p_{ij}
\end{equation}
and $0 \leq n_{ij} < n_i$, $0 \leq p_{ij} <  \gcd(n, M_{ii}+\frac{n}{2})$. These inequalities make the choice unique.  

Given $f_1, \ldots ,f_{r-1}$ with $f_1\cdots f_{r-1}$ squarefree, let $F=f_1^{n_{r1}}\cdots f_{r-1}^{n_{r \, (r-1)}}$ and
\begin{equation} \label{GlobalRoot}
\varepsilon = \prod_{i=1}^{r-1} \res{f_i'}{f_i}_{\chi}^{M_{ii}} \, \prod_{i=1}^{r-2}\prod_{j=i+1}^{r-1} \res{f_i}{f_j}_{\chi}^{M_{ij}}.
\end{equation}
Warning: $F$ will be defined differently in Section \ref{SectionDirichlet} below. The main result of this section is the following formula:

\begin{prop} \label{PropCoeff4}
For $f_1, \ldots f_{r-1}$ with $f_1\cdots f_{r-1}$ squarefree, we have
\begin{equation} \label{Coeff4}
\begin{split}
a(f_1, \ldots f_r) =& \varepsilon \res{f_1^{p_{r1}}\cdots f_{r-1}^{p_{r \, (r-1)}}}{f_r}_{\chi}^{-1}  \frac{\xi(-1)^{\deg f_r(\deg f_r -1)/2}}{g_{\xi \chi^{M_{rr}}}^{\deg f_r}} \\ &\sum_{\substack{u \in \F_q[T]^+ \\ u^{n_r}|f_r}} q^{(n_r-1)\deg u} g_{\xi \chi^{M_{rr}}}(F, f_r/u^{n_r}).
\end{split}
\end{equation}
\end{prop}

One can check that the formula $\eqref{Coeff4}$ recovers \eqref{SquarefreeCoefficients} in the case when $f_1 \cdots f_r$ is squarefree. This follows from Lemma 2.4 in \cite{s-amds} and some character computations similar to those in the proof of the next proposition.

Let $\bar{a}(f_1, \ldots f_r)$ denote the formula on the right side of \eqref{Coeff4}. If \ref{PropCoeff4} stated a formula valid for all tuples of polynomials $f_1,\ldots, f_r$, to prove \ref{PropCoeff4} it would suffice to verify that $\bar{a}$ satisfies the axioms, as \cite[Theorem 1.1]{s-amds} ensures that the axiom uniquely determines the series. Since this formula is not valid for an arbitrary polynomial, we cannot directly apply \cite[Theorem 1.1]{s-amds}, but we will check $a(f_1,\ldots, f_r) = \bar{a}(f_1,\ldots, f_r)$ by an inductive argument similar to the proof of \cite[Theorem 1.1]{s-amds}. We begin by checking Axiom \ref{Axiom1}.

\begin{prop} \label{PropTwistedMult}
The coefficients $\bar{a}(f_1, \ldots f_r)$ satisfy the twisted multiplicativity axiom, namely 
\begin{equation}\label{TwistedMult}
\bar{a}(f_1g_1, \ldots f_rg_r) = \bar{a}(f_1, \ldots f_r)\bar{a}(g_1, \ldots g_r) \prod_{i=1}^r \prod_{j=i}^r \res{f_i}{g_j}_{\chi}^{M_{ij}} \res{g_i}{f_j}_{\chi}^{M_{ij}}
\end{equation}
for $\gcd(f_1f_2\cdots f_r, g_1g_2\cdots g_r)=1$ and $f_1g_1f_2g_2 \cdots f_{r-1}g_{r-1}$ squarefree. 
\end{prop}

\begin{proof}
First we have 
\begin{equation*}
\begin{split}
\varepsilon(f_1g_1, \ldots f_{r-1}g_{r-1}) &= \prod_{i=1}^{r-1} \res{(f_ig_i)'}{f_ig_i}_{\chi}^{M_{ii}} \, \prod_{i=1}^{r-2}\prod_{j=i+1}^{r-1} \res{f_ig_i}{f_jg_j}_{\chi}^{M_{ij}} \\
&=\prod_{i=1}^{r-1} \res{f_i'g_i}{f_i}_{\chi}^{M_{ii}}\res{f_ig_i'}{g_i}_{\chi}^{M_{ii}} \, \prod_{i=1}^{r-2}\prod_{j=i+1}^{r-1} \res{f_ig_i}{f_j}_{\chi}^{M_{ij}}\res{f_ig_i}{g_j}_{\chi}^{M_{ij}}  \\
&=\prod_{i=1}^{r-1}\prod_{j=i}^{r-1} \res{g_i}{f_j}_{\chi}^{M_{ij}}\res{f_i}{g_j}_{\chi}^{M_{ij}} \varepsilon(f_1, \ldots f_{r-1})\varepsilon(g_1, \ldots g_{r-1})
\end{split}
\end{equation*}
so the product of residue symbols $\varepsilon$ accounts for some of the twisted multiplicativity. 

We also have
\begin{equation*}
\res{(f_1g_1)^{p_{r1}}\cdots (f_{r-1}g_{r-1})^{p_{r \, (r-1)}}}{f_rg_r}_{\chi}^{-1} = \res{f_1^{p_{r1}}\cdots f_{r-1}^{p_{r \, (r-1)}}}{f_r}_{\chi}^{-1} \res{g_1^{p_{r1}}\cdots g_{r-1}^{p_{r \, (r-1)}}}{g_r}_{\chi}^{-1} \prod_{i=1}^r \res{f_i}{g_r}^{-p_{ri}}_{\chi} \res{g_i}{f_r}^{-p_{ri}}_{\chi}.
\end{equation*}

Next, we use Lemma \ref{GaussSumTwistedMult}. Letting $F=f_1^{n_{r1}}\cdots f_{r-1}^{n_{r \, (r-1)}}$ and $G=g_1^{n_{r1}}\cdots g_{r-1}^{n_{r \, (r-1)}}$, we have 
\begin{equation*}
g_{\xi \chi^{M_{rr}}}(FG, f_rg_r)=g_{\xi \chi^{M_{rr}}}(F, f_r)g_{\xi \chi^{M_{rr}}}(G, g_r)\res{F}{g_r}_{\xi \chi^{M_{rr}}}^{-1}\res{G}{f_r}_{\xi \chi^{M_{rr}}}^{-1} \res{f_r}{g_r}_{\xi \chi^{M_{rr}}}\res{g_r}{f_r}_{\xi \chi^{M_{rr}}}.
\end{equation*} 
This twisted multiplicativity extends to the sums over $u$ because each $u \in \F_q[T]^+$ with $u^{n_r}|f_rg_r$ factors uniquely as $vw$ with $v^{n_r}|f_r$, $w^{n_r}|g_r$, and 
\begin{equation*} 
\begin{split}
&q^{(n_r-1)\deg u} g_{\xi \chi^{M_{rr}}}(FG, f_rg_r/u^{n_r}) \\
&= q^{(n_r-1)\deg v}g_{\xi \chi^{M_{rr}}}(F, f_r/v^{n_r})q^{(n_r-1)\deg w}g_{\xi \chi^{M_{rr}}}(G, g_r/w^{n_r}) \\
&\cdot \res{F}{g_r}_{\xi \chi^{M_{rr}}}^{-1} \res{G}{f_r}_{\xi \chi^{M_{rr}}}^{-1} \res{f_r}{g_r}_{\xi \chi^{M_{rr}}} \res{g_r}{f_r}_{\xi \chi^{M_{rr}}}.
\end{split}
\end{equation*}
Furthermore, we can expand
\begin{equation*}
\res{F}{g_r}_{\xi \chi^{M_{rr}}}^{-1}\res{G}{f_r}_{\xi \chi^{M_{rr}}}^{-1} = \prod_{i=1}^{r-1} \res{f_i}{g_r}_{\xi \chi^{M_{rr}}}^{-n_{ri}} \res{g_i}{f_r}_{\xi \chi^{M_{rr}}}^{-n_{ri}}.
\end{equation*}

Finally, we have 
\begin{equation*}
\begin{split}
&\xi(-1)^{(\deg f_rg_r)(\deg f_rg_r-1)/2} \\
&= \xi(-1)^{(\deg f_r)(\deg f_r-1)/2}\xi(-1)^{(\deg g_r)(\deg g_r-1)/2} \res{f_r}{g_r}_{\xi}\res{g_r}{f_r}_{\xi}.
\end{split}
\end{equation*}

Note that the formulas above involve two residue symbols of $f_i$ and $g_r$ (or symmetrically $g_i$ and $f_r$): one with character $(\chi)^{-p_{ri}}$ and one with character $(\xi \chi^{M_{rr}})^{-n_{ri}}$. By definition we have $-p_{ri}-n_{ri}(M_{rr}+n/2) \equiv M_{ri} \bmod n$ so that
\begin{equation} \label{CharacterEquality}
(\chi)^{-p_{ri}} (\xi \chi^{M_{rr}})^{-n_{ri}} = \chi^{M_{ri}}.
\end{equation}

Then combining all the formulas above yields the desired twisted multiplicativity formula \eqref{TwistedMult}. 
\end{proof}

We now prove formula \eqref{Coeff4} in the case of a one-by-one, and then a two-by-two matrix $M$. These are sufficient to deduce the general case. 

\begin{prop}\label{PropCoeff1}
Suppose that $M=(M_{11})$. We have 
\begin{equation}\label{Coeff1}
a(f) = \frac{\xi(-1)^{\deg f(\deg f -1)/2}}{g_{\xi \chi^{M_{11}}}^{\deg f}} \sum_{\substack{u \in \F_q[T]^+ \\ u^{n_1}|f}} q^{(n_1-1)\deg u} g_{\xi \chi^{M_{11}}}(1, f/u^{n_1}).
\end{equation}
\end{prop}
Formula \eqref{Coeff1} also holds for $a(1, \ldots 1, f)$ with the indices of $1$ replaced by $r$. 
\begin{proof}
Let $a(f)$ denote the coefficients satisfying the axioms, whose existence and uniqueness is proven in \cite[Theorem 1.1]{s-amds}. Let $\bar{a}(f)$ denote the formulas on the right side of equation \eqref{Coeff1}, so we must prove that $a(f)=\bar{a}(f)$. We will use Axioms \ref{Axiom1} and \ref{Axiom3} to reduce to the case when $f=\pi^d$ for a linear polynomial $\pi$. Then we will prove that $a(\pi^d)=\bar{a}(\pi^d)$ for $\pi$ linear by induction on $d$.

The twisted multiplicativity Axiom \ref{Axiom1} holds for $\bar{a}(f)$ by Proposition \ref{PropTwistedMult}.

For $\pi$ prime, the formula \eqref{Coeff1} specializes to:
\begin{equation}\label{Coeff1Prime}
\bar{a}(\pi^d) = \frac{\xi(-1)^{d \deg \pi(d \deg \pi -1)/2}}{g_{\xi \chi^{M_{11}}}^{d \deg \pi}} \left\lbrace \begin{array}{cc} q^{(d-d/n_1)\deg \pi} & d \equiv 0 \bmod n_1 \\ q^{(d-1-(d-1)/n_1) \deg \pi}g_{\xi \chi^{M_{11}}}(1, \pi) & d \equiv 1 \bmod n_1 \\ 0 & d \not\equiv 0, 1 \bmod n_1 \end{array} \right. .
\end{equation}

In the $d \equiv 0 \bmod n_1$ case, we use the identity $\res{\pi'}{\pi}_{ \xi \chi^{M_{11}}}^d=1$ and Lemma \ref{Pellet} to rewrite formula \eqref{Coeff1Prime} as:
\begin{equation*}
    \bar{a}(\pi^d) =(-1)^d \res{\pi'}{\pi}_{\chi}^{dM_{11}} \left(\frac{\xi(-1)^{d(d-1)/2} q^{d-d/n_1}}{(-g_{\xi \chi^{M_{11}}})^{d}}\right)^{\deg \pi} .
\end{equation*}
This expression matches the form of Axiom \ref{Axiom3}, with a single Weil number of weight $d-\frac{2d}{n_1}$. 
If $d \equiv 1 \bmod n_1$, we may apply those same formulas along with Lemma \ref{GaussSumLifting} to obtain
\begin{equation*}
    \bar{a}(\pi^d) = (-1)^{d-1} \res{\pi'}{\pi}_{\chi}^{dM_{11}} \left( \frac{\xi(-1)^{d(d-1)/2}q^{d-1-(d-1)/n_1}}{(-g_{\xi \chi^{M_{11}}})^{d-1}} \right)^{\deg \pi} .
\end{equation*}

This expression matches the form of Axiom \ref{Axiom3}, with a single Weil number of weight $d-1-\frac{2(d-1)}{n_1}$.

Because the sets of coefficients $a(\pi^d)$ and $\bar{a}(\pi^d)$ both satisfy Axiom \ref{Axiom3}, it suffices to prove $a(\pi^d)=\bar{a}(\pi^d)$ for $\pi$ linear. Then this automatically extends to $\pi$ of arbitrary degree.

For $\pi$ linear and $d=0,1$ we have $a(\pi^d)=1$ by Axiom \ref{Axiom2}. This matches $\bar{a}(\pi^d)$ by \eqref{Coeff1Prime}.

By Lemma \ref{GaussSumSum}, we have
\begin{equation}
\sum_{\deg f = d} g_{\xi \chi^{M_{11}}} (1, f) = \left\lbrace \begin{array}{cc} 
1 & d=0 \\
qg_{\xi \chi^{M_{11}}} & d = 1 \\
q^{d+d/n_1}(1-q^{-1}) & d>0, \, d \equiv 0 \bmod n_1  \\
q^{d+(d-1)/n_1}(1-q^{-1}) g_{\xi \chi^{M_{11}}}  & d>1, \, d \equiv 1 \bmod n_1  \\
0 & d \not \equiv 0, 1 \bmod n_1
\end{array} \right. .
\end{equation}

By an interchange of summation, we obtain
\begin{equation*}
\sum_{\deg f = d} \sum_{u^{n_1}|f} q^{(n_1-1)\deg u} g_{\xi \chi^{M_{11}}}(1, f/u^{n_1}) = \sum_{j=0}^{\lfloor d/{n_1} \rfloor} q^{j n_1 } \sum_{\deg f = d-j n_1} g_{\xi \chi^{M_{11}}}(1, f).
\end{equation*}
If $d \equiv 0 \bmod n_1$, the double sum evaluates to $q^{d+d/n_1}$. If $d \equiv 1 \bmod n_1$, it evaluates to $q^{d+(d-1)/n_1}g_{\xi \chi^{M_{11}}}$. It is $0$ otherwise. Therefore,
\begin{equation*}
\sum_{\deg f = d} \bar{a}(f)=\frac{\xi(-1)^{d(d-1)/2}}{g_{\xi \chi^{M_{11}}}^d} \left\lbrace \begin{array}{cc} 
q^{d+d/n_1} & d \equiv 0 \bmod n_1 \\
q^{d+(d-1)/n_1}g_{\xi \chi^{M_{11}}} & d \equiv 1 \bmod n_1 \\
0 & d \not \equiv 0, 1 \bmod n_1
\end{array} \right. .
\end{equation*}

Now suppose that $a(\pi^{d'})=\bar{a}(\pi^{d'})$ is known for all $d'<d$. By twisted multiplicativity, this implies that $a(f)=\bar{a}(f)$ for all $f$ of degree $d$, except when $f=\pi^d$ for $\pi$ linear. Thus 
\begin{equation*}
\begin{split}
&\sum_{\deg f = d} a(f) - \sum_{\deg \pi = 1} a(\pi^d)  =\sum_{\deg f = d} \bar{a}(f) -\sum_{\deg \pi = 1} \bar{a}(\pi^d) \\
&=\frac{\xi(-1)^{d(d-1)/2}}{g_{\xi \chi^{M_{11}}}^d} 
\left\lbrace \begin{array}{cc} 
q^{d+d/n_1}-q^{d+1-d/n_1} & d \equiv 0 \bmod n_1 \\
q^{d+(d-1)/n_1}g_{\xi \chi^{M_{11}}}-q^{d-(d-1)/n_1}g_{\xi \chi^{M_{11}}} & d \equiv 1 \bmod n_1 \\
0 & d \not \equiv 0, 1 \bmod n_1
\end{array} \right. .
\end{split}
\end{equation*}
By Axioms \ref{Axiom4} and \ref{Axiom5}, assuming $d\geq 2$, $\sum_{\deg f = d} a(f)$ is a linear combination of Weil numbers of weights greater than $d+1$ and $\sum_{\deg \pi = 1} a(\pi^d) = q a(T^d)$ is a linear combination of Weil numbers of weights less than $d+1$, so there is no cancellation between these terms.  We conclude that for $\deg \pi =1$, 
\begin{equation*}
a(\pi^d)  = \bar{a}(\pi^d) = \frac{\xi(-1)^{d(d-1)/2}}{g_{\xi \chi^{M_{11}}}^d} 
\left\lbrace \begin{array}{cc} 
q^{d-d/n_1} & d \equiv 0 \bmod n_1 \\
q^{(d-1-(d-1)/n_1)}g_{\xi \chi^{M_{11}}} & d \equiv 1 \bmod n_1 \\
0 & d \not \equiv 0, 1 \bmod n_1
\end{array} \right.
\end{equation*}
which completes the proof.
\end{proof} 

\begin{prop}\label{PropCoeff2}
Suppose that $M=\begin{pmatrix} M_{11} & M_{21}  \\ M_{21} & M_{22} \end{pmatrix}$. For $f_1$ squarefree, we have 
\begin{equation}\label{Coeff2}
\begin{split}
a(f_1, f_2) = &\res{f_1'}{f_1}^{M_{11}}_{\chi} \res{f_1}{f_2}_{\chi}^{-p_{21}} \frac{\xi(-1)^{\deg f_2(\deg f_2 -1)/2}}{g_{\xi \chi^{M_{22}}}^{\deg f_2}}  \sum_{\substack{u \in \F_q[T]^+ \\ u^{n_2}|f_2}} q^{(n_2-1)\deg u} g_{\xi \chi^{M_{22}}}(f_1^{n_{21}}, f_2/u^{n_2}).
\end{split}
\end{equation}
\end{prop}

Formula \eqref{Coeff2} also holds for $a(1, \ldots f_i, \ldots 1, f_r)$ with the index 1 relabeled as $i$ and the index $2$ relabeled as $r$. 

\begin{proof} 
The proof proceeds as before. Let $\bar{a}(f_1, f_2)$ denote the formula on the right side of \eqref{Coeff2}. We will show $a(f_1, f_2)=\bar{a}(f_1, f_2)$. The twisted multiplicativity property for $\bar{a}(f_1, f_2)$ follows from Proposition \ref{PropTwistedMult}.

Now it suffices to prove $a(\pi, \pi^d)=\bar{a}(\pi, \pi^d)$ for $\pi$ prime. We have already shown that $a(1, \pi^d)=\bar{a}(1, \pi^d)$ by Proposition \ref{PropCoeff1}.

For $p_{21}=0$ and $n_{21} < n_2-1$, formula \eqref{Coeff2} specializes to:
\begin{equation}\label{Coeff2Prime}
\begin{split}
\bar{a}(\pi, \pi^d) = & \res{\pi'}{\pi}_{\chi}^{M_{11}} \frac{\xi(-1)^{d \deg \pi(d \deg \pi -1)/2}}{g_{\xi \chi^{M_{22}}}^{d \deg \pi}} \\
& \cdot \left\lbrace \begin{array}{cc} 
q^{(d-d/n_2) \deg \pi} & d \equiv 0 \bmod n_2 \\ 
q^{(d-1 - (d-n_{21}-1)/n_2) \deg \pi }g_{(\xi \chi^{M_{22}})^{n_{21}+1}}(1, \pi) & d \equiv 1+n_{21} \bmod n_2 \\ 
0 & d \not\equiv 0, 1+n_{21} \bmod n_2 
\end{array} \right. .
\end{split}
\end{equation}
In the case $p_{21}=0$, $n_{21} = n_2-1$, formula \eqref{Coeff2} specializes to: 
\begin{equation} \label{Coeff3Prime}
\bar{a}(\pi, \pi^{d}) = \res{\pi'}{\pi}_{\chi}^{M_{11}} \cdot \left\lbrace \begin{array}{cc} 1 & d=0 \\  0 & d>0 \end{array}\right. .
\end{equation}
Finally, in the case $p_{21}>0$, note that the character $(\chi )^{-p_{21}}$ is nontrivial since $0<p_{21}<n$. Therefore the same formula \eqref{Coeff3Prime} holds in this case. 

Using twisted multiplicativity, formula \eqref{Coeff3Prime} implies that $a(f_1, f_2) = 0$ if $\gcd(f_1, f_2) \neq 1$ and $a(f_1, f_2) = \res{f_1'}{f_1}_{\chi}^{M_{11}} \res{f_1}{f_2}_{\chi}^{M_{21}} a(1, f_2)$ if $\gcd(f_1, f_2) = 1$.

In the $p_{21}=0$, $n_{21}<n_2-1$, $d \equiv 0 \bmod n_2$ case, we use the identity $\res{\pi'}{\pi}_{ \xi \chi^{M_{22}}}^d=1$ and Lemma \ref{Pellet} to rewrite formula \eqref{Coeff2Prime} as:
\begin{equation*} 
\bar{a}(\pi, \pi^{d}) = (-1)^d \res{\pi'}{\pi}_{\chi}^{M_{11}+dM_{22}} \left(\frac{\xi(-1)^{d(d-1)/2}q^{d-d/n_2}}{(-g_{\xi \chi^{M_{22}}})^d}\right)^{\deg \pi}.
\end{equation*}
This expression matches the form of Axiom \ref{Axiom3}, with a single Weil number of weight $d-\frac{2d}{n_2}$. 

In the $p_{21}=0$, $n_{21}<n_2-1$, $d \equiv 1+n_{21} \bmod n_2$ case, we may apply the same formulas along with Lemma \ref{GaussSumLifting} to obtain
\begin{equation*} (-1)^{d-1}\res{\pi'}{\pi}_{\chi}^{M_{11}+d M_{22}}
\left( \frac{-g_{(\xi \chi^{M_{22}})^{n_{21}+1}} \xi(-1)^{d(d-1)/2} q^{d-1-(d-n_{21}-1)/n_2}}{(-g_{\xi \chi^{M_{22}}})^d} \right)^{\deg \pi}
\end{equation*}
which matches the form of Axiom \ref{Axiom3}, with a single Weil number of weight $d-1-\frac{2(d-n_{21}-1)}{n_2}$. 

In the case when $n_{21}=n_2-1$ or $p_{21}=0$, formula \eqref{Coeff3Prime} matches the form of Axiom \ref{Axiom3}, with a single Weil number of weight $0$ if $d=0$ and no Weil numbers if $d>0$.

Because the sets of coefficients $a(\pi, \pi^d)$ and $\bar{a}(\pi, \pi^d)$ both satisfy Axiom \ref{Axiom3}, it suffices to prove $a(\pi, \pi^d)=\bar{a}(\pi, \pi^d)$ for $\pi$ linear. Then this automatically extends to $\pi$ of arbitrary degree.

For $\pi$ linear, we have $a(1, 1)=a(1, \pi)=a(\pi, 1)=1$ by Axiom \ref{Axiom2}. This matches $\bar{a}(1, 1)=\bar{a}(1, \pi)=\bar{a}(\pi, 1)=1$ by equations \eqref{Coeff1Prime}, \eqref{Coeff2Prime}, and \eqref{Coeff3Prime}.

For the inductive step, we compute 
\begin{equation*}
\sum_{\substack{\deg f_1 = 1, \\ \deg f_2 = d}} \bar{a}(f_1, f_2).
\end{equation*}
The case $\deg f_1=0$ is already done in Proposition \ref{PropCoeff1}, with the result that equation \eqref{Coeff2} holds when $f_1=1$. Combining this with twisted multiplicativity and the simple evaluation $a(f_1, 1)=\res{f_1'}{f_1}^{M_{11}}$ for all $f_1$ squarefree, we conclude that equation \eqref{Coeff2} holds whenever $\gcd(f_1, f_2)=1$. 

We will now complete the argument in the case $p_{21}=0$, leaving the more technical $p_{21}>0$ case for the end.

By Proposition \ref{GaussSumSum}, if $n_{21}<n_2-1$, 
\begin{equation}
\begin{split}
&\sum_{\deg f_2 = d} g_{\xi \chi^{M_{22}}}(\pi^{n_{21}}, f_2) \\
&= \left\lbrace \begin{array}{cc} 
1 & d=0 \\
q^{1+n_{21}} g_{(\xi \chi^{M_{22}})^{1+n_{21}}} & d = 1+n_{21} \\
q^{d(n_2+1)/n_2}(1-q^{-1}) & d>0, d \equiv 0 \bmod n_2 \\
q^{d+(d-n_{21}-1)/n}(1-q^{-1})g_{(\xi \chi^{M_{22}})^{1+n_{21}}} & d>1+n_{21}, d \equiv 1+n_{21} \bmod n_2 \\
0 & d \not \equiv 0, 1+n_{21} \bmod n_2
\end{array}\right.
\end{split}
\end{equation}
and if $n_{21}=n_2-1$, 
\begin{equation}
\sum_{\deg f_2 = d} g_{\xi \chi^{M_{22}}}(\pi^{n_{21}}, f_2) = \left\lbrace \begin{array}{cc} 
1 & d=0 \\
-q^{n_2} & d = n_2 \\
0 & d \neq 0, n_2
\end{array}\right. .
\end{equation}

By an interchange of summation, we obtain
\begin{equation*}
\sum_{\deg f_2 = d} \sum_{u^{n_2}|f_2} q^{(n_2-1)\deg u} g_{\xi \chi^{M_{22}}}(\pi^{n_{21}}, f_2/u^{n_2}) 
= \sum_{j=0}^{\lfloor d/n_2 \rfloor} q^{j n_2} \sum_{\deg f_2 = d-j n_2} g_{\xi \chi^{M_{22}}}(\pi^{n_{21}}, f_2).
\end{equation*}
If $n_{21}<n_2-1$, $d \equiv 0 \bmod n_2$, the double sum evaluates to $q^{d+d/n_2}$. If $n_{21}<n_2-1$, $d \equiv 1+n_{21} \bmod n_2$, it evaluates to $q^{d+(d-n_{21}-1)/n}g_{(\xi \chi^{M_{22}})^{1+n{21}}}$. If $n_{21}=n_2-1$, $d=0$, it is $1$. It is $0$ otherwise. 

All together, we find that for $n_{21}<n_2-1$,
\begin{equation*}
\begin{split}
& \sum_{\deg f_1 = 1} \sum_{\deg f_2=d} \bar{a}(f_1, f_2) \\
&=\frac{\xi(-1)^{d(d-1)/2}}{g_{\xi \chi^{M_{22}}}^{d}} \left\lbrace \begin{array}{cc} 
q^{d+1+d/n_2} & d \equiv 0 \bmod n_2 \\
q^{d+1+(d-n_{21}-1)/n_2}g_{(\xi \chi^{M_{22}})^{1+n_{21}}} & d \equiv 1+n_{21} \bmod n_2 \\
0 & d_2 \not \equiv 0, 1+n_{21} \bmod n_2
\end{array} \right.
\end{split}
\end{equation*}
and for $n_{21}=n_2-1$,
\begin{equation*}
\sum_{\deg f_1 = 1} \sum_{\deg f_2=d} \bar{a}(f_1, f_2)=\left\lbrace \begin{array}{cc} q & d=0 \\ 0 & d>0 \end{array} \right. .
\end{equation*}

Now suppose that $a(\pi, \pi^{d'})=\bar{a}(\pi, \pi^{d'})$ is known for all $d'<d$. By twisted multiplicativity, this implies that $a(f_1, f_2)=\bar{a}(f_1, f_2)$ for all $f_1$ linear and $f_2$ of degree $d$, except when $f=\pi$ and $f_2=\pi^{d}$ for $\pi$ linear. Thus for $n_{21}<n_2-1$,
\begin{equation*}
\begin{split}
&\sum_{\deg f_1 = 1} \sum_{\deg f_2 = d} a(f_1, f_2) - \sum_{\deg \pi = 1} a(\pi, \pi^{d})  =\sum_{\deg f_1 = 1} \sum_{\deg f_2 = d} \bar{a}(f_1, f_2) - \sum_{\deg \pi = 1} \bar{a}(\pi, \pi^{d})  \\
&=\frac{\xi(-1)^{d(d-1)/2}}{g_{\xi \chi^{M_{22}}}^{d}} \left\lbrace \begin{array}{cc} 
q^{d+1+d/n_2}-q^{d+1-d/n_2} & d \equiv 0 \bmod n_2 \\
(q^{d+1+(d-n_{21}-1)/n_2}-q^{d-(d-n_{21}-1)/n_2})g_{(\xi \chi^{M_{22}})^{1+n_{21}}} & d_2 \equiv 1+n_{21} \bmod n_2 \\
0 & d \not \equiv 0, 1+n_{21} \bmod n_2
\end{array} \right. .
\end{split}
\end{equation*}

By Axioms \ref{Axiom4} and \ref{Axiom5}, assuming $d\geq 1$, the first term is a linear combination of Weil numbers of weights greater than $d+2$ and the second term is a linear combination of Weil numbers of weights less than $d+2$, so there is no cancellation between them.  We conclude that for $\deg \pi =1$, 
\begin{equation*}
\begin{split}
&a(\pi, \pi^{d})  = \bar{a}(\pi, \pi^{d}) = \frac{\xi(-1)^{d(d-1)/2}}{g_{\xi \chi^{M_{22}}}^{d}} \left\lbrace \begin{array}{cc} 
q^{d-d/n_2} & d \equiv 0 \bmod n_2 \\
q^{d-1-(d_2-n_{21}-1)/n_2}g_{(\xi \chi^{M_{22}})^{1+n_{21}}} & d \equiv 1+n_{21} \bmod n_2 \\
0 & d_2 \not \equiv 0, 1+n_{21} \bmod n_2
\end{array} \right. .
\end{split}
\end{equation*}

Similarly, for $n_{21}=n_2-1$, 
\begin{equation*}
\sum_{\deg f_1 = 1} \sum_{\deg f_2 = d} a(f_1, f_2) - \sum_{\deg \pi = 1} a(\pi, \pi^{d})  =\sum_{\deg f_1 = 1} \sum_{\deg f_2 = d} \bar{a}(f_1, f_2) - \sum_{\deg \pi = 1} \bar{a}(\pi, \pi^{d})  =0.
\end{equation*}
Assuming $d \geq 1$, there is no cancellation between the two terms, so we conclude that for $\deg \pi=1$,
\begin{equation*}
a(\pi, \pi^{d})=\bar{a}(\pi, \pi^{d})=0
\end{equation*}
as desired.

We now turn to the case $p_{21}>0$. We will modify the proof of Lemma \ref{GaussSumSum} to compute the sum 
\begin{equation*} 
\sum_{\deg f_2=d} \res{\pi}{f_2}_{\chi}^{-p_{21}} g_{\xi \chi^{M_{22}}}(\pi^{n_{21}}, f_2) = \sum_{\deg f_2=d} \res{\pi}{f_2}_{\chi}^{M_{21}} g_{\xi \chi^{M_{22}}}(1, f_2).
\end{equation*}
This equality is due to the twisted multiplicativity of the Gauss sum and the formula $\chi^{-p_{21}}(\xi \chi^{M_{22}})^{-n_{21}} = \chi^{M_{21}}$ of \eqref{CharacterEquality}.

If $n_2 \nmid d$ we use Lemma \ref{GaussSumEval} to rewrite this sum as
\begin{equation*}
g_{(\xi \chi^{M_{22}})^d} \sum_{\deg f_2=d} \sum_{\deg \nu = d-1} \res{\pi}{f_2}_{\chi}^{M_{21}} \res{\nu}{f_2}_{\xi \chi^{M_{22}}}.
\end{equation*}
Letting $\nu=\pi^e \nu^{(\pi)}$ with $\pi \nmid \nu^{(\pi)}$, we have
\begin{equation*}
g_{(\xi \chi^{M_{22}})^d} \sum_{\deg f_2=d} \sum_{\deg \nu = d-1} \res{\pi}{f_2}_{\xi^e \chi^{M_{21}+eM_{22}}} \res{\nu^{(\pi)}}{f_2}_{\xi \chi^{M_{22}}}.
\end{equation*}
The character $\xi^e \chi^{M_{21}+e M_{22}}$ cannot be trivial, as this would imply that $ M_{21} +e(M_{22}+n/2) \equiv 0 \bmod n$, contradicting our assumption that $p_{21} \neq 0$. Then the polynomials $f_2$ with fixed residue modulo $\nu^{(\pi)}$ are equidistributed modulo $\pi$ by the Chinese remainder theorem, so the $f_2$ sum vanishes.

Similarly, if $n_2|d$, we use Lemma \ref{GaussSumEval} to rewrite the sum as 
\begin{equation*} 
\sum_{\substack{\deg f_2 = d \\ f_2=f_0^{n_2}}} \res{\pi}{f_2}_{\chi}^{M_{21}} \phi(f) - q \sum_{\deg f = d} \sum_{\deg \nu = d-1} \res{\pi}{f}_{\chi}^{M_{21}} \res{\nu}{f_2}_{\xi \chi^{M_{22}}}.
\end{equation*}
The double sum part of this expression vanishes for the same reason as above, leaving
\begin{equation*} 
\sum_{\deg f_0 = d/n_2} \res{\pi}{f_0}_{\chi}^{n_2 M_{21}} q^{d-d/n_2} \phi(f_0).
\end{equation*}
The character $\chi^{n_2 M_{21}}$ cannot be trivial, as this would imply that $n | n_2 M_{21}$, which implies $\gcd(n, M_{22}+n/2)|M_{21}$, again contradicting our assumption that $p_{21} \neq 0$. We rewrite the sum as
\begin{equation*} 
\sum_{\deg f_0 = d/n_2} \res{\pi}{f_0}_{\chi}^{n_2 M_{21}} q^{d-d/n_2} \sum_{\substack{\nu \bmod f_0 \\ \gcd(\nu, f_0)=1}} 1
\end{equation*}
and interchange summation. For any fixed $\nu$ of degree less than $d/n_2$, the polynomials $f_0$ coprime to $\nu$ are equidistributed modulo $\pi$ by the Chinese remainder theorem, so the sum vanishes. In the special case $d=0$, the sum is $1$.

We conclude that if $p_{21}>0$,
\begin{equation*} 
\sum_{\deg f_2=d} \res{\pi}{f_2}_{\chi}^{-p_{21}} g_{\xi \chi^{M_{22}}}(\pi^{n_{21}}, f_2) = \left\lbrace \begin{array}{cc} 1 & d=0 \\ 0 & d>0 \end{array}\right. .
\end{equation*}

By an interchange of summation, we obtain
\begin{equation*}
\begin{split}
&\sum_{\deg f_2 = d} \res{\pi}{f_2}_{\chi}^{-p_{21}} \sum_{u^{n_2}|f_2} q^{(n_2-1)\deg u} g_{\xi \chi^{M_{22}}}(\pi^{n_{21}}, f_2/u^{n_2}) \\
&= \sum_{j=0}^{\lfloor d/n_2 \rfloor} q^{j(n_2-1)} \sum_{\deg u = j}  \sum_{\deg f_2 = d-jn_2} \res{\pi}{u^{n_2}f_2}_{\chi}^{-p_{21}} g_{\xi \chi^{M_{22}}}(\pi^{n_{21}}, f_2).
\end{split}
\end{equation*}
The character $\chi^{-p_{21}n_2}$ cannot be trivial, as this would imply $n| -p_{21}n_2$, which in turn implies $\gcd(n, M_{22}+n/2) | p_{21}$, contradicting our assumption that $p_{21}>0$. Therefore the $u$ sum vanishes unless $j=0$, leaving the previous formula of $1$ if $d=0$, $0$ if $d>0$. Thus
\begin{equation*}
\sum_{\deg f_1 = 1} \sum_{\deg f_2=d} \bar{a}(f_1, f_2)=\left\lbrace \begin{array}{cc} q & d=0 \\ 0 & d>0 .\end{array} \right.
\end{equation*}
The rest of the proof is the same as in the $p_{21}=0$, $n_{21}=n_2-1$ case.

\end{proof}

\begin{proof}[Proof of Proposition \ref{PropCoeff4}]
Using Proposition \ref{PropTwistedMult}, it suffices to check formula \eqref{Coeff4} when $f_1, \ldots f_r$ are all powers of the same prime $\pi$. Since $f_1\cdots f_{r-1}$ is squarefree, this means evaluating either $a(1, \ldots 1, \pi^d)$ or $a(1, \ldots \pi, \ldots 1, \pi^d)$. If some $f_i=1$, then the axiomatic coefficient $a(f_1, \ldots f_r)$ matches the coefficient with $f_i$ removed and the $i$th row and column deleted from the matrix $M$; this follows from the uniqueness of coefficients satisfying the axioms. Thus in the case of  $a(1, \ldots 1, \pi^d)$, the evaluation follows from Proposition \ref{PropCoeff1}. In the case of $a(1, \ldots \pi, \ldots 1, \pi^d)$, it follows from Proposition \ref{PropCoeff2}. 
\end{proof}

\section{Kubota Functional Equation} \label{SectionKubota}
\subsection{Single Variable Functional Equation}

Now we are ready to prove the functional equations of axiomatic multiple Dirichlet series. To find these functional equations, we use single-variable subseries with single-variable functional equations. The setup remains the same: $q$ is an odd prime, $\chi$ is a character of even order $n$ and $M=(M_{ij})$ is an $r \times r$ integer symmetric matrix, but for this section we make the assumption that $p_{ri}=0$, i.e. $\gcd(n, M_{rr}+n/2) | M_{ri}$, for all $i$. 

We also make the assumption that $q\equiv 1 \bmod 4$ in order to match the Kubota $L$-series studied in \cite{Hoffstein} and \cite{Patterson}. Hoffstein makes this assumption to simplify his calculations. Patterson does not make it explicitly, but he cites Hoffstein for the functional equation of the Kubota $L$-series. This assumption means $\xi(-1)=1$, so it removes the powers of $\xi(-1)$ from formulae \eqref{Coeff4}, \eqref{Coeff1}, \eqref{Coeff2}. 

Let $a(f_1, \ldots f_r; q, \chi, M)$ be axiomatic multiple Dirichlet series coefficients. Fix $\vec{f}=(f_1, \ldots f_{r-1})$ and let
\begin{align}
&D(x, \vec{f}, k; q, \chi, M) = \sum_{\substack{f_r \in \F_q[T]^+ \\ \deg f_r \equiv k \bmod n_r}} a(\vec{f}, f_r; q, \chi, M) x^{\deg f_r} \\
&D_\pi(x, \vec{f}, k; q, \chi, M) = \sum_{\substack{d\geq 0 \\ d \equiv k \bmod n_r}} a(\vec{f}, \pi^d; q, \chi, M) x^{d \deg \pi}
\end{align}
for $k \bmod n_r$ and a prime $\pi \in \F_q[T]^+$. We will often factor $f\in \F_q[T]^+$ as $f^{(\pi)} f_\pi$, where $f^{(\pi)}$ is coprime to $\pi$ and $f_\pi$ is a power of $\pi$. We also extend this notation to vectors $\vec{f}_\pi$ and $\vec{f}^{(\pi)}$. Since the Gauss sum $g_{\xi \chi^{M_{rr}}}$ appears frequently throughout this section, we abbreviate it as $g$. Set $F=f_1^{n_{r1}}\cdots f_{r-1}^{n_{r \, (r-1)}}$, as in the previous section. Again, we caution the reader that $F$ will be redefined in the next section.

Assume that $f_1\cdots f_{r-1}$ is squarefree. Let
\begin{equation*}
\varepsilon= \prod_{i=1}^{r-1} \res{f_i'}{f_i}_{\chi}^{M_{ii}} \, \prod_{i=1}^{r-2}\prod_{j=i+1}^{r-1} \res{f_i}{f_j}_{\chi}^{M_{ij}}
\end{equation*}
as in equation \eqref{GlobalRoot}. Proposition \ref{PropCoeff4} implies that we can factor $D(x, \vec{f}, k)$ as 
\begin{equation}\label{GlobalKubotaD}
\varepsilon \left( 1-\frac{q^{n_r} x^{n_r}}{g^{n_r}} \right)^{-1}\left( \sum_{\substack{f_r \in \F_q[T]^+ \\ \deg f_r \equiv k \bmod n_r}}  \frac{g_{\xi \chi^{M_{rr}}}(F, f_r) x^{\deg f_r} }{g^{\deg f_r}} \right).
\end{equation}
This expression is substantially simplified by our $q \equiv 1 \bmod 4$ assumption. The latter factor is Kubota's $L$-series, analyzed in \cite{Hoffstein} and \cite{Patterson} in the rational function field case. 

For any fixed $\vec{f}$ and prime $\pi$, let
\begin{equation}\label{LocalRoot}
\varepsilon_\pi=\prod_{i=1}^{r-1} \res{f_i^{(\pi)}}{(f_i)_\pi}_{\chi}^{M_{ii}} \res{(f_i)_\pi}{f_i^{(\pi)}}_{\chi}^{M_{ii}} \prod_{i=1}^{r-2}\prod_{j=1+i}^{r-1}\res{f_i^{(\pi)}}{(f_j)_\pi}_{\chi}^{M_{ij}} \res{(f_i)_\pi}{f_j^{(\pi)}}_{\chi}^{M_{ij}}.
\end{equation}
By twisted multiplicativity, we have
\begin{equation}\label{LocalSimplification}
D_{\pi}(x, \vec{f}, k) = \varepsilon_\pi a(\vec{f}^{(\pi)}, 1) \res{F^{(\pi)}}{\pi^k}_{\xi \chi^{M_{rr}}}^{-1} D_{\pi}(x, \vec{f}_\pi, k).
\end{equation}
Since $f_1 \cdots f_{r-1}$ has at most one factor of $\pi$, and each $n_{ri}<n_r$, we have $v_{\pi} F<n_r$. If $v_{\pi} F<n_r-1$, formulas \eqref{Coeff1Prime}, \eqref{Coeff2Prime} imply:
\begin{equation} \label{LocalKubotaD1}
\begin{split}
D_\pi(x, \vec{f}_\pi, k)=&\left(1 -\frac{q^{(n_r-1)\deg \pi}x^{n_r \deg \pi}}{g^{n_r \deg \pi}}\right)^{-1} \\
&\cdot \left\lbrace \begin{array}{cc} 1 & k \equiv 0 \bmod n_r \\ \frac{q^{v_{\pi} F \deg \pi} g_{(\xi \chi^{M_{rr}})^{v_{\pi} F+1}}(1, \pi)x^{(v_{\pi} F+1)\deg \pi}}{g^{(v_{\pi} F+1)\deg \pi}} & k \equiv v_{\pi} F+1 \bmod n_r \\ 0 & k \not\equiv 0, v_{\pi} F+1 \bmod n_r \end{array}\right. .
\end{split}
\end{equation}
If $v_{\pi}(F)=n_r-1$, formula \eqref{Coeff3Prime} implies:
\begin{equation}\label{LocalKubotaD2}
D_\pi(x, \vec{f}_\pi, k) = \left\lbrace \begin{array}{cc} 1 & k \equiv 0 \bmod n_r \\ 0 & k \not\equiv 0 \bmod n_r \end{array}\right. .
\end{equation}

The functional equations of $D(x, \vec{f}, k)$ for $f_1\cdots f_r$ squarefree, and $D_\pi(x, \vec{f}_\pi, k)$ for $v_{\pi} (f_1\cdots f_r) \leq 1$ give the base case for our inductive proof. 

\begin{prop} \label{PropGlobalFE}
For $f_1\cdots f_{r-1}$ squarefree, $D(x, \vec{f}, k)$ is a rational function of $x$ with denominator dividing $1-q\left(\frac{q x}{g}\right)^{n_r}$. We have the vector functional equation:
\begin{equation} \label{GlobalFE}
\left( D(x, \vec{f}, k) \right)_k = \left(\frac{qx}{g}\right)^{\deg F} \left( \Gamma_{k, \ell}(x, \deg F) \right)_{k, \ell} \left( D\left(\frac{g^2}{q^2x}, \vec{f}, \ell\right) \right)_{\ell}
\end{equation}
where $\left( \Gamma_{k, \ell}(x, \deg F) \right)_{k, \ell}$ is the $n_r \times n_r$ scattering matrix with
\begin{equation*}
\Gamma_{k,k}(x, \deg F) = \left(\frac{qx}{g}\right)^{1 - (\deg F +1-2k)\% n_r} \frac{1-q}{1-q\left(\frac{q x}{g}\right)^{n_r}}
\end{equation*}
\begin{equation*}
\begin{split}
\Gamma_{k,\ell}(x, \deg F)=&(\xi \chi^{M_{rr}})^{\ell(1+\deg F)}(-1) g_{(\xi \chi^{M_{rr}})^{2k-\deg F -1}} \left(\frac{qx}{g}\right) \frac{1-\left(\frac{qx}{g}\right)^{-n_r}}{1-q\left(\frac{qx}{g}\right)^{n_r}} 
\end{split}
\end{equation*}
for $\ell \equiv 1+\deg F -k \bmod n_r$, and all other entries equal to $0$. In addition, there are special entries when $k \equiv \ell \equiv 1+\deg F -k \bmod n_r$. In this case 
\begin{equation*}
\Gamma_{k,\ell}(x, \deg F)=\left(\frac{qx}{g}\right)^{1-n_r}.
\end{equation*}
\end{prop}
\begin{proof}
The rational function statement follows from a general property of Kubota $L$-series \cite[bottom of page 245]{Patterson}. The functional equations of the Kubota $L$-series \cite[top of page 251]{Patterson} imply, for all $k$:
\begin{equation}\label{KubotaFE}
\begin{split}
 &D(x, \vec{f}, k) = \left(\frac{qx}{g}\right)^{\deg F + 1 - (\deg F +1-2k)\% n_r} \frac{1-q}{1-q\left(\frac{q x}{g}\right)^{n_r}} D\left(\frac{g^2}{q^2x}, \vec{f}, k\right) \\
&+(\xi \chi^{M_{rr}})^{\ell(1+\deg F)}(-1) g_{(\xi \chi^{M_{rr}})^{2k-\deg F -1}} \left(\frac{qx}{g}\right)^{\deg F + 1}  \frac{1-\left(\frac{qx}{g}\right)^{-n_r}}{1-q\left(\frac{qx}{g}\right)^{n_r}} D\left(\frac{g^2}{q^2x}, \vec{f}, \ell \right)
\end{split}
\end{equation}
where $\ell=1+\deg F -k$. Rewriting these functional equations in vector form gives the desired statement. The special entries of the scattering matrix come from adding the two terms together. 
\end{proof}

Note that there is a slight discrepancy between the functional equation given in \cite{Patterson} and the functional equation \eqref{KubotaFE} as written: the power of $(\xi\chi^{M_{rr}})(-1)$ does not match. This is because of a subtlety in the definition of the Kubota $L$-series. Patterson uses a different additive character from ours, so that his function field Gauss sum $g(r, \varepsilon, c)$ (with underlying character $\chi$ of order $n$) is our $\chi(-1)^{\deg c} g_{\chi}(r, c)$. This leads to a small difference in the functional equation. 

Importantly for us, $\Gamma_{k, \ell}(x, \deg F)$ depends on $\deg F \bmod n_r$, and the full functional equation depends on $\deg F$, but not on $f_1, \ldots f_{r-1}$ themselves. 

The local generating series $D_{\pi}(x, \vec{f}, k)$ satisfy corresponding local functional equations. 

\begin{prop}\label{PropLocalFE}
For $v_{\pi}(f_1\cdots f_{r-1}) \leq 1$, $D_{\pi}(x, \vec{f}, k)$ is a rational function of $x$ with denominator dividing $1-q^{-\deg \pi} \left( \frac{qx}{g}\right)^{n_r \deg \pi}$. We have the vector functional equation:
\begin{equation} \label{LocalFE}
\left( D_{\pi}(x, \vec{f}, k) \right)_k = \left(\frac{qx}{g} \right)^{v_{\pi}(F) \deg \pi} \left( \res{F^{(\pi)}}{\pi}_{\xi\chi^{M_{rr}}}^{k+\ell} \Gamma_{\pi, k, \ell}(x, v_{\pi} F) \right)_{k, \ell} \left( D_{\pi}\left(\frac{g^2}{q^2x}, \vec{f}, \ell \right) \right)_{\ell}
\end{equation}
where $\left( \Gamma_{\pi, k, \ell}(x, v_{\pi} F) \right)_{k, \ell}$ is the $n_r \times n_r$ scattering matrix with
\begin{equation*}
\begin{split}
\Gamma_{\pi, k,k}(x, v_{\pi} F) =& \left( \frac{qx}{g}\right)^{(1-(v_{\pi} F +1-2k)\% n_r)\deg \pi} \frac{1-q^{-\deg \pi}}{1-q^{-\deg \pi}\left( \frac{q x}{g}\right)^{n_r\deg \pi}}
\end{split}
\end{equation*}
\begin{equation*}
\begin{split}
\Gamma_{\pi, k,\ell}(x, v_{\pi} F) =&(\xi \chi^{M_{rr}})^{\ell(1+v_{\pi} F))\deg \pi}(-1)\frac{g_{(\xi \chi^{M_{rr}})^{2k-v_{\pi} F -1}}(1, \pi)}{q^{\deg \pi}}  \left( \frac{qx}{g}\right)^{\deg \pi} \frac{1-\left(\frac{q x}{g}\right)^{-n_r\deg \pi}}{1-q^{-\deg \pi}\left(\frac{q x}{g}\right)^{n_r\deg \pi}}
 \end{split}
\end{equation*}for $\ell \equiv 1+v_{\pi} F -k \bmod n_r$, and all other entries equal to $0$. In addition, there are special entries in the matrix when $k \equiv \ell \equiv 1+v_\pi F -k \bmod n_r$. In this case 
\begin{equation*}
\Gamma_{\pi, k,\ell}(x, v_{\pi} F)=\left(\frac{qx}{g}\right)^{(1-n_r)\deg \pi}.
\end{equation*}
\end{prop}

Proposition \ref{PropLocalFE} can be checked by direct computation from the evaluations of $D_{\pi}(x, \vec{f}, k)$ above. 

We remark that the local scattering matrix $\left( \Gamma_{\pi, k, \ell}(x, d)\right)_{k, \ell}$ is obtained from the global scattering matrix $\left( \Gamma_{k, \ell}(x, d) \right)_{k, \ell}$ by the same change of variables which relates the axiomatic local and global multiple Dirichlet series coefficients. The fact that $x \mapsto \frac{g^2}{q^2 x}$ commutes with this change of variables is crucial to the proof below. 

The next Lemma gives the inverse of the scattering matrix. This lemma is not used in the rest of this section, but is used in the proof of Theorem \ref{uniquely-determined}.

\begin{lemma}\label{scattering-matrix-inverse} We have the identity \[(\Gamma(x,K))^{-1}= \Gamma(\frac{g^2}{q^2x} ,K).\]\end{lemma}

The formula should hold because, when we apply the functional equation twice, we transform a multiple Dirichlet series to itself. The lemma does not follow immediately from the functional equation \eqref{GlobalFE} because we lose information when we multiply by a vector. Instead we verify the lemma by a direct computation.

\begin{proof} 

The matrix $\Gamma(x,K)$ is conjugate under a permutation matrix to a block diagonal matrix with $2 \times 2 $ and $1\times 1$ blocks. The $1\times 1$ blocks have unique entry $\left(\frac{ qx}{g)}\right)^{1-n_i}$ with inverse $\left(\frac{ qx}{g}\right)^{n_i-1}$ which can be obtained from $\left(\frac{ qx}{g}\right)^{1-n_i}$ by the substitution $x \mapsto \frac{ g}{q^2x^2}$.

The two-by-two block consisting of the columns and rows $k,l$ where $k+l \equiv 1 + K \bmod n_i$ has determinant

$$ \left( \frac{qx}{g} \right)^{ 2- (K+1-2k) \% n_i - (K+1-2\ell)\% n_i}\frac{ (1-q)^2}{ \left( 1 -  \left(\frac{qx}{g}\right)^{n_i}\right)^2} $$ $$- (\xi \chi^{M_{ii}})^{ (k+\ell) (1+K)} (-1) g_{ (\xi \chi^{M_{ii}} )^{2k-K-1}}(1,T)  g_{ (\xi \chi^{M_{ii}} )^{2\ell-K-1}}(1,T)\left( \frac{qx}{g}\right)^2  \frac{ \left( 1- \left( \frac{qx}{g} \right)^{-n_i}\right)^2}{ \left(1 - q \left( \frac{qx}{g} \right)^{n_i} \right)^2 } $$

The exponent $(k+\ell) (1+K)$ simplifies to $(1+K)^2 \equiv 1+K \bmod 2$ while $g_{ (\xi \chi^{M_{ii}} )^{2k-K-1}}  g_{ (\xi \chi^{M_{ii}} )^{2\ell-K-1}}$ simplifies to $ q (\xi \chi^{M_{ii}})^{ k-\ell} (-1)$ and the characters cancel since $k-\ell \equiv k+\ell \equiv 1+K \bmod 2$. The expressions $K+1-2k$ and $K+1-2\ell$ are additive inverses mod $n_i$ and not both zero so taking $\% n_i$ and summing gives $n_i$. This gives the determinant 

$$ \left( \frac{qx}{g} \right)^{ 2- n_i}\frac{ (1-q)^2}{ \left( 1 -  \left(\frac{qx}{g}\right)^{n_i}\right)^2} - q\left( \frac{qx}{g}\right)^2  \frac{ \left( 1- \left( \frac{qx}{g} \right)^{-n_i}\right)^2}{ \left(1 - q \left( \frac{qx}{g} \right)^{n_i} \right)^2 } = \frac{  \left(\frac{qx}{g}\right)^{2-n_i} \left( 1 -q \left( \frac{qx}{g}\right)^{-n_i} \right)}{ \left(1 - q \left( \frac{qx}{g} \right)^{n_i} \right). }$$

Dividing the $k,k$ entry $ \left(\frac{qx}{g}\right)^{1 - (K +1-2k)\% n_i} \frac{1-q}{1-q\left(\frac{q x}{g}\right)^{n_i}}$ by that determinant, we obtain the $\ell,\ell$ entry of the inverse matrix
$$ \left(\frac{qx}{g}\right)^{n_i -1 - (K +1-2k)\% n_i} \frac{1-q}{1-q\left(\frac{q x}{g}\right)^{-n_i}}$$

The substitution $x \mapsto \frac{ g^2}{q^2x}$ applied to the $\ell,\ell$ entry  $ \left(\frac{qx}{g}\right)^{1 - (K +1-2l)\% n_i} \frac{1-q}{1-q\left(\frac{q x}{g}\right)^{n_i}}$ gives $$ \left(\frac{qx}{g}\right)^{((K+1-2\ell)\%n_i -1 } \frac{1-q}{1-q\left(\frac{q x}{g}\right)^{-n_i}}$$ and these are equal since $(K+1-2\ell)\%n_i  = n_i - (K +1-2k)\%n_i$. An identical argument works for the $k,k$ entry.

Dividing minus the $k,\ell$ entry $(\xi \chi^{M_{ii}})^{\ell (1+K)}(-1) g_{(\xi \chi^{M_{ii}})^{2k-K -1}}\left(\frac{qx}{g}\right) \frac{1-\left(\frac{qx}{g}\right)^{-n_i}}{1-q\left(\frac{qx}{g}\right)^{n_i}} $ by the determinant gives the $k,\ell$ entry of the inverse matrix

$$ - (\xi \chi^{M_{ii}})^{\ell (1+K)}(-1) g_{(\xi \chi^{M_{ii}})^{2k-K -1}} \left(\frac{qx}{g}\right)^{-1 +n_i}  \frac{1-\left(\frac{qx}{g}\right)^{-n_i}}{1-q\left(\frac{qx}{g}\right)^{-n_i}} $$

$$ =  (\xi \chi^{M_{ii}})^{\ell (1+K)}(-1) g_{(\xi \chi^{M_{ii}})^{2k-K -1}} \left(\frac{qx}{g}\right)^{-1 }  \frac{1-\left(\frac{qx}{g}\right)^{n_i}}{1-q\left(\frac{qx}{g}\right)^{-n_i}}$$
while the substitution $x \mapsto \frac{ g^2}{q^2x}$ applied to the $k,\ell$ entry gives
$$(\xi \chi^{M_{ii}})^{\ell (1+K)}(-1) g_{(\xi \chi^{M_{ii}})^{2k-K -1}} \left(\frac{qx}{g}\right)^{-1} \frac{1-\left(\frac{qx}{g}\right)^{n_i}}{1-q\left(\frac{qx}{g}\right)^{-n_i}} $$ and these are equal. An identical argument works for the $\ell,k$ entry. \end{proof}

The local scattering matrix $(\Gamma_{\pi, k, \ell}(x, K))_{k, \ell}$ has the same property. 

We will need some further notation to separate out the contribution of one prime $\pi$ to the series $D(x, \vec{f}, k)$. The part of $D$ coprime to $\pi$, with a twist by the character $\res{\pi^\ell}{*}_{\xi \chi^{M_{rr}}}$ for $\ell \bmod n_r$, is 
\begin{equation}
D^{(\pi)}(x, \vec{f}, k, \pi^\ell) = \sum_{\substack{ f_r \in \F_q[t]^+ \\ \deg f_r \equiv k \bmod n \\ \pi \nmid f_r}} a(\vec{f},f_r) \res{\pi^\ell}{f_r}_{\xi \chi^{M_{rr}}} x^{\deg f_r}.
\end{equation}

The following lemma gives a formula to combine the $\pi$-parts and $\pi$-coprime parts of the series $D$.

\begin{lemma}\label{LemmaFactorization}
For any prime $\pi$, we have the following equality of vectors:
\begin{equation} \label{Factorization}
\begin{split}
&\varepsilon_\pi \left( (\xi \chi^{M_{rr}})^{(k-\ell\deg \pi)(\ell \deg \pi)} (-1)  D^{(\pi)}(x, \vec{f}^{(\pi)}, k-\ell \deg \pi, \pi^{2\ell}F_\pi^{-1}) \right)_{k, \ell} \\
&\cdot \left( \res{F^{(\pi)}}{\pi^\ell}_{\xi \chi^{M_{rr}}}^{-1} D_\pi(x, \vec{f}_\pi, \ell) \right)_\ell \\
&= \left( D(x, \vec{f}, k) \right)_k
\end{split}
\end{equation}
where $\varepsilon_\pi$ is the $n$th root of unity defined in equation \eqref{LocalRoot}.
\end{lemma}

\begin{proof}
By twisted multiplicativity and the character identity $(\xi \chi^{M_{rr}})^{-n_{ri}} = \chi^{M_{ri}}$, 
\begin{equation} \label{RemovePrime}
\begin{split}
a(f_1, \ldots f_{r-1}, f_r) = & \varepsilon_\pi \res{f_r^{(\pi)}}{(f_r)_\pi}_{\chi}^{M_{rr}}\res{(f_r)_\pi}{f_r^{(\pi)}}_{\chi}^{M_{rr}} \\
&\res{F_\pi}{f_r^{(\pi)}}_{\xi \chi^{M_{rr}}}^{-1} a(f_1^{(\pi)}, \ldots f_r^{(\pi)}) \\
& \res{F^{(\pi)}}{(f_r)_\pi}_{\xi \chi^{M_{rr}}}^{-1} a((f_1)_\pi, \ldots (f_r)_\pi).
\end{split}
\end{equation}
Note that $\varepsilon_\pi$ is independent of $f_r$, the second line only depends on $f_r^{(\pi)}$, and the third line only depends on $(f_r)_\pi$. Using power reciprocity and the $q \equiv 1 \bmod 4$ assumption, the residues in the first line can be simplified to $(\xi \chi^{M_{rr}})^{(\deg f_r^{(\pi)})(\deg (f_r)_\pi)}(-1) \res{(f_r)_\pi}{f_r^{(\pi)}}^2_{\xi \chi^{M_{rr}}}$. This depends on both $(f_r)_\pi$ and $f_r^{(\pi)}$, but really only on $v_{\pi}(f_r) \bmod n_r$, $\res{\pi}{f_r^{(\pi)}}_{\xi \chi^{M_{rr}}}$, and $\deg f_r^{(\pi)} \bmod n_r$.

Therefore 
\begin{equation*}
\begin{split}
&\varepsilon_\pi \sum_{\ell=0}^{n-1} \, (\xi \chi^{M_{rr}})^{(k-\ell\deg \pi)(\ell \deg \pi)} (-1) D^{(\pi)}(x, \vec{f}^{(\pi)}, k-\ell \deg \pi, \pi^{2 \ell} F_\pi^{-1}) \res{F^{(\pi)}}{\pi^\ell}_{\xi \chi^{M_{rr}}}^{-1} D_\pi(x, \vec{f}_\pi, \ell) \\
& =\sum_{\ell=0}^{n-1} \, \sum_{\substack{ f_r^{(\pi)} \in \F_q[t]^+ \\ \deg f_r^{(\pi)} \equiv k-\ell \deg \pi \bmod n_r \\ \pi \nmid f_r^{(\pi)}}} \, \sum_{\substack{(f_r)_\pi = \pi^d \\ d\geq 0 \\ d \equiv \ell \bmod n_r}} \, \varepsilon_\pi (\xi \chi^{M_{rr}})^{(k-\ell\deg \pi)(\ell \deg \pi)} (-1) \res{\pi^{2\ell}}{f_r^{(\pi)}}_{\xi \chi^{M_{rr}}} \\ 
& \qquad \res{F^{(\pi)}}{\pi^\ell}_{\xi \chi^{M_{rr}}}^{-1} \res{F_\pi}{f_r^{(\pi)}}_{\xi \chi^{M_{rr}}}^{-1} a(\vec{f}^{(\pi)}, f_r^{(\pi)}) a(\vec{f}_\pi, (f_r)_\pi) x^{\deg f_r^{(\pi)}+\deg (f_r)_\pi}\\
&= \sum_{\substack{f_r \in \F_q[t]^+ \\ \deg f_r \equiv k \bmod n_r}} a(\vec{f}, f_r) x^{\deg f_r} = D(x, \vec{f}, k).
\end{split}
\end{equation*}
This proves formula \eqref{Factorization}.
\end{proof}

Note that the vector $\left( \res{F^{(\pi)}}{\pi^\ell}_{\xi \chi^{M_{rr}}}^{-1} D_\pi(x, \vec{f}_\pi, \ell) \right)_\ell$ satisfies the same functional equation \eqref{LocalFE} as $\left(D_\pi(x, \vec{f}, \ell) \right)_\ell$, because of equation \eqref{LocalSimplification}.

We will need a matrix identity which generalizes Lemma \ref{LemmaFactorization}, using a collection of axiomatic Dirichlet series constructed from different matrices. Fix $m \in \mathbb Z$; let $M_{00}=\cdots=M_{0 \,(r-1)}=0$, and $M_{0r}= m(M_{rr} +n/2)\%n$. Let $M_{m(M_{rr} +n/2)}$ be the $r+1 \times r+1$ symmetric matrix constructed by augmenting $M$ with a zeroth row and column, with entries $M_{0j}$. The $n_{r0}$ defined in equation \eqref{nij} for the matrix $M_{m(M_{rr} +n/2)}$ will be congruent to $-m$ modulo $n_{r}$. We define
\begin{align*}
&D(x, (f_0, \vec{f}), k; M_{m(M_{rr} +n/2)}) = \sum_{\substack{f_r \in \F_q[T]^+ \\ \deg f_r \equiv k \bmod n}} a(f_0, \vec{f}, f_r; M_{m(M_{rr} +n/2)}) x^{\deg f_r} \\
&D_\pi(x, (f_0, \vec{f}), k; M_{m(M_{rr} +n/2)}) = \sum_{\substack{j\geq 0 \\ j \equiv k \bmod n}} a(f_0, \vec{f}, \pi^j; M_{m(M_{rr} +n/2)}) x^{j\deg \pi} \\
&D^{(\pi)}(x, (f_0, \vec{f}), k, \pi^\ell; M_{m(M_{rr} +n/2)}) = \sum_{\substack{ f_r \in \F_q[t]^+ \\ \deg f_r \equiv k \bmod n \\ \pi \nmid f_r}} a(f_0, \vec{f},f_r; M_{m(M_{rr} +n/2)}) \res{\pi^\ell}{f_r}_{\xi \chi^{M_{rr}}} x^{\deg f_r}
\end{align*}
where the coefficients $a(f_0, \vec{f}, f_r; M_{m(M_{rr} +n/2)})$ are constructed axiomatically from matrix $M_{m(M_{rr} +n/2)}$. This notation makes the dependence on the underlying matrix explicit. We will continue to omit writing down this matrix when it is the original fixed matrix $M$. 

\begin{lemma}\label{LemmaStrongFactorization}
For any prime $\pi$ and $v \in \Z$, we have the following equality of $n_r \times n_r$ matrices:
\begin{equation} \label{StrongFactorization}
\begin{split}
&\varepsilon_\pi \left( (\xi \chi^{M_{rr}})^{(k-\ell\deg \pi)\ell \deg \pi} (-1)  D^{(\pi)}(x, \vec{f}^{(\pi)}, k-\ell \deg \pi, \pi^{2\ell+v}F_\pi^{-1}) \right)_{k, \ell} \\
&\cdot \left((\xi \chi^{M_{rr}})^{\ell m \deg \pi}(-1) \res{F^{(\pi)}}{\pi^{\ell-m}}_{\xi \chi^{M_{rr}}}^{-1}D_\pi(x, \pi, \vec{f}_\pi, \ell-m; M_{(v+2m)(M_{rr}+n/2)}) \right)_{\ell,m} \\
&= \left( (\xi \chi^{M_{rr}})^{k m \deg \pi}(-1)D(x, \pi, \vec{f}, k-m\deg \pi; M_{(v+2m)(M_{rr}+n/2)}) \right)_{k,m}.
\end{split}
\end{equation}
\end{lemma}

\begin{proof}
Lemma \ref{LemmaFactorization} for the matrix $M_{(v+2m)(M_{rr}+n/2)}$ gives:
\begin{equation*}
\begin{split}
&\varepsilon_\pi \left( (\xi \chi^{M_{rr}})^{(k-\ell\deg \pi)\ell \deg \pi} (-1)  D^{(\pi)}(x, 1, \vec{f}^{(\pi)}, k-\ell \deg \pi, \pi^{2\ell+v+2m}F_\pi^{-1}; M_{(v+2m)(M_{rr}+n/2)}) \right)_{k, \ell} \\
&\cdot \left( \res{F^{(\pi)}}{\pi^\ell}_{\xi \chi^{M_{rr}}}^{-1}D_\pi(x, \pi, \vec{f}_\pi, \ell; M_{(v+2m)(M_{rr}+n/2)}) \right)_\ell \\
&= \left( D(x, \pi, \vec{f}, k; M_{(v+2m)(M_{rr}+n/2)}) \right)_k
\end{split}
\end{equation*}
with the same $\varepsilon_\pi$ used for $M$. Simplifying $D^{(\pi)}(x, 1, \vec{f}^{(\pi)}, k-\ell \deg \pi, \pi^{2\ell+v+2m}F_\pi^{-1}; M_{(v+2m)(M_{rr}+n/2)})$ to $D^{(\pi)}(x, \vec{f}^{(\pi)}, k-\ell \deg \pi, \pi^{2\ell+v+2m}F_\pi^{-1})$, making the substitutions $k \mapsto k-m\deg \pi$ and $\ell \mapsto \ell-m$, and rearranging powers of $(\xi\chi^{M_{rr}})(-1)$, we have:
\begin{equation*}
\begin{split}
&\varepsilon_\pi \left( (\xi\chi^{M_{rr}})^{(k-\ell\deg \pi)\ell \deg \pi} (-1)  D^{(\pi)}(x, \vec{f}^{(\pi)}, k-\ell \deg \pi, \pi^{2\ell+v}F_\pi^{-1}) \right)_{k, \ell} \\
&\cdot \left((\xi \chi^{M_{rr}})^{\ell m \deg \pi}(-1)\res{F^{(\pi)}}{\pi^{\ell-m}}_{\xi \chi^{M_{rr}}}^{-1} D_\pi(x, \pi, \vec{f}_\pi, \ell-m; M_{(v+2m)(M_{rr}+n/2)}) \right)_\ell \\
&= \left( (\xi \chi^{M_{rr}})^{k m \deg \pi}(-1)D(x, \pi, \vec{f}, k-m\deg \pi; M_{(v+2m)(M_{rr}+n/2)}) \right)_k .
\end{split}
\end{equation*}
The first matrix is independent of $m$. Extending the latter two vectors to matrices indexed by $m$ gives the result. 
\end{proof}

The simplest case of Lemma \ref{LemmaStrongFactorization} is for $\vec{f}=\vec{1}$ and $\deg \pi =1$:
\begin{equation} \label{StrongFactorizationExample}
\begin{split}
&\left( (\xi \chi^{M_{rr}})^{(k-\ell)\ell} (-1)  D^{(\pi)}(x, \vec{1}, k-\ell, \pi^{2\ell+v}) \right)_{k, \ell} \\
&\cdot \left((\xi \chi^{M_{rr}})^{\ell m}(-1) D_\pi(x, \pi, \vec{1}, \ell-m; M_{(v+2m)(M_{rr}+n/2)}) \right)_{\ell,m} \\
&= \left( (\xi \chi^{M_{rr}})^{k m}(-1)D(x, \pi, \vec{1}, k-m; M_{(v+2m)(M_{rr}+n/2)}) \right)_{k,m}.
\end{split}
\end{equation}
These three matrices can be computed explicitly. The second matrix is evaluated using the formulas for local generating series \eqref{LocalKubotaD1}, \eqref{LocalKubotaD2}. For $m=\ell$, 
\begin{equation*} 
D_\pi(x, \pi, \vec{1}, \ell-m; M_{(v+2m)(M_{rr}+n/2)}) = \left( 1- q^{-1}\left(\frac{qx}{g}\right)^{n_r}\right)^{-1}.
\end{equation*}
For $m \equiv 1-v-\ell \bmod n_r$,
\begin{equation*} 
\begin{split}
 &D_\pi(x, \pi, \vec{1}, \ell-m; M_{(v+2m)(M_{rr}+n/2)}) \\
 &= \left( 1- q^{-1}\left(\frac{qx}{g}\right)^{n_r}\right)^{-1} \left(\frac{g_{(\xi \chi^{M_{rr}})^{1-v-2m}}(1, \pi)}{q}\right) \left(\frac{q x}{g}\right)^{(1-v-2m)\%n_r}.
 \end{split}
\end{equation*}
And for $ m \not\equiv \ell, 1-v-\ell \bmod n_r$, $D_\pi(x, \pi, \vec{1}, \ell-m; M_{(v+2m)(M_{rr}+n/2)}) = 0$. In addition, there are special entries in the matrix when $m \equiv \ell \equiv 1-v-\ell \bmod n_r$. In this case $D_\pi(x, \pi, \vec{1}, \ell-m; M_{(v+2m)(M_{rr}+n/2)}) = 1$.

The matrix $\left((\xi \chi^{M_{rr}})^{\ell m}(-1) D_\pi(x, \pi, \vec{1}, \ell-m; M_{(v+2m)(M_{rr}+n/2)}) \right)_{\ell,m}$ is conjugate under a permutation matrix to a block diagonal matrix of $2 \times 2$ blocks, where each nontrivial $2 \times 2$ block has determinant equal to $(\xi \chi^{M_{rr}})^{1-v}(-1)\left( 1- q^{-1}\left(\frac{qx}{g}\right)^{n_r}\right)^{-1}$. Therefore the inverse matrix has entries which are monomials in $x$, with no poles. 

The third matrix is related to the second one by Axiom \ref{Axiom4}. Note that all $\pi$ of degree $1$ are related by a linear change of variables $T \mapsto T-c$. If a set of coefficients $a(f_1(T), \ldots f_n(T))$ satisfy the five axioms, then so do $a(f_1(T-c), \ldots f_n(T-c))$. By the uniqueness part of \cite[Theorem 1.1]{s-amds}, the coefficients are invariant under $T \mapsto T-c$. Therefore the matrix $\left( (\xi \chi^{M_{rr}})^{k m}(-1)D(x, \pi, \vec{1}, k-m; M_{(v+2m)(M_{rr}+n/2)}) \right)_{k,m}$ is the same for all $\pi$ of degree $1$. Summing these matrices over all $\pi$ of degree $1$ is the same as multiplying by $q$. Thus, by Axiom \ref{Axiom4}, if we replace $q$ with $1/q$, replace each $q$-Weil number $\alpha$ with $1/\bar{\alpha}$, and replace $x$ with $qx$, we transform back and forth between $\left( (\xi \chi^{M_{rr}})^{k m}(-1)D(x, \pi, \vec{1}, k-m; M_{(v+2m)(M_{rr}+n/2)}) \right)_{k,m}$ and $\left((\xi \chi^{M_{rr}})^{\ell m}(-1) D_\pi(x, \pi, \vec{1}, \ell-m; M_{(v+2m)(M_{rr}+n/2)}) \right)_{\ell,m}$.

This means that, for $m\equiv k \bmod n_r$ 
\begin{equation*}
D(x, \pi, \vec{1}, k-m; M_{(v+2m)(M_{rr}+n/2)}) = \left( 1- q\left(\frac{qx}{g}\right)^{n_r}\right)^{-1}. 
\end{equation*}
For $m \equiv 1-v-k \bmod n_r$, 
\begin{equation*} 
\begin{split}
&D(x, \pi, \vec{1}, k-m; M_{(v+2m)(M_{rr}+n/2)}) \\
&= \left( 1- q \left(\frac{qx}{g}\right)^{n_r}\right)^{-1} g_{(\xi \chi^{M_{rr}})^{1-v-2m}}(1, \pi) \left(\frac{q x}{g}\right)^{(1-v-2m)\%n_r}.
\end{split}
\end{equation*}
And for $ m \not\equiv k, 1-v-k \bmod n_r$, $D(x, \pi, \vec{1}, \ell-m; M_{(v+2m)(M_{rr}+n/2)}) = 0$. In addition, there are special entries in the matrix when $m \equiv k \equiv 1-v-k \bmod n_r$. In this case, $D(x, \pi, \vec{1}, k-m; M_{(v+2m)(M_{rr}+n/2)})=1$.

This matrix is also conjugate to a block-diagonal matrix of $2 \times 2$ blocks, where each nontrivial $2 \times 2$ block has determinant $(\xi \chi^{M_{rr}})^{1-v}(-1)\left(1- q\left(\frac{qx}{g}\right)^{n_r}\right)^{-1}$. The inverse matrix has entries which are monomials in $x$, with no poles.

The first matrix $\left( (\xi \chi^{M_{rr}})^{(k-\ell)\ell} (-1)  D^{(\pi)}(x, \vec{1}, k-\ell, \pi^{2\ell+v}) \right)_{k, \ell}$ can be computed from the other two. It has the interesting property of being mapped to its inverse by the change of variables $q \mapsto 1/q$ $\alpha \mapsto 1/\bar{\alpha}$, $x \mapsto qx$. 

We now give functional equations generalizing \eqref{GlobalFE} and \eqref{LocalFE} for the matrices appearing in Lemma \ref{LemmaStrongFactorization}. 

\begin{lemma}\label{LemmaMatrixFE}
For $\pi$ prime, suppose that $\pi f_1\cdots f_{r-1}$ is squarefree, or more generally that the functional equations \eqref{GlobalFE} hold for $\left(D(x, \pi, \vec{f}, k; M_{m(M_{rr}+n/2)} \right)_k$ for all $m$. Then for $v \in \Z$,
\begin{equation}\label{MatrixFE}
\begin{split}
&\left( (\xi \chi^{M_{rr}})^{k m \deg \pi}(-1)D(x, \pi, \vec{f}, k-m\deg \pi; M_{(v+2m)(M_{rr}+n/2)}) \right)_{k,m} \\
&= \left( \frac{qx}{g} \right)^{\deg F} \left( \Gamma_{k, \ell}(x, \deg F - v \deg \pi) \right)_{k, \ell} \\
& \phantom{=}\cdot \left( (\xi \chi^{M_{rr}})^{\ell m \deg \pi}(-1)  D\left(\frac{g^2}{q^2x}, \pi, \vec{f}, \ell-m\deg \pi; M_{(v+2m)(M_{rr}+n/2)} \right) \right)_{\ell, m} \\
& \phantom{=}\cdot \left(\Delta_{m,m}(x, v)^{\deg \pi} \right)_{m,m} 
\end{split}
\end{equation}
where $\left( \Gamma_{k, \ell}(x, \deg F - v\deg \pi) \right)_{k, \ell}$ is the $n_r \times n_r$ scattering matrix defined in Proposition \ref{PropGlobalFE}, and $\left(\Delta_{m,m}(x,v)^{\deg \pi} \right)_{m,m}$ is a diagonal matrix with entries $\Delta_{m,m}(x,v)^{\deg \pi}=\left(\frac{qx}{g}\right)^{((-v-2m)\% n_r)\deg \pi}$.
\end{lemma}
Note that the $m=0$ column of the matrix identity \eqref{MatrixFE} recovers the functional equation \eqref{GlobalFE} for $\left(D(x, \pi, \vec{f}, k; M_{v(M_{rr}+n/2)} \right)_k$. 
\begin{proof}
Equation \ref{KubotaFE} for the matrix $M_{(v+2m)(M_{rr}+n/2)}$ gives:
\begin{equation*}
\begin{split}
&D(x, \pi, \vec{f}, k-m\deg \pi; M_{(v+2m)(M_{rr}+n/2)}) \\
& = \left(\frac{qx}{g}\right)^{\deg F + ((-v-2m)\%n_r)\deg \pi + 1 - (\deg F -v\deg \pi +1-2k)\% n_r} \frac{1-q}{1-q\left(\frac{q x}{g}\right)^{n_r}} \\
& \cdot D\left(\frac{g^2}{q^2x}, \pi, \vec{f}, k-m\deg \pi; M_{(v+2m)(M_{rr}+n/2)}\right) \\
&+(\xi \chi^{M_{rr}})^{(\ell- m\deg \pi)(1+\deg F-v\deg \pi)}(-1) g_{(\xi \chi^{M_{rr}})^{2k-\deg F +v\deg \pi-1}} \left(\frac{qx}{g}\right)^{\deg F+ ((-v-2m)\%n_r)\deg \pi + 1} \\
&\cdot \frac{1-\left(\frac{qx}{g}\right)^{-n_r}}{1-q\left(\frac{qx}{g}\right)^{n_r}} D\left(\frac{g^2}{q^2x}, \pi, \vec{f}, \ell - m \deg \pi; M_{(v+2m)(M_{rr}+n/2)} \right)
\end{split}
\end{equation*}
where $\ell=1+\deg F -v\deg \pi -k$. Multiplying through by $(\xi \chi^{M_{rr}})^{km\deg \pi}(-1)$, and rewriting these functional equations in matrix form gives the desired statement.
\end{proof}

\begin{lemma}\label{LemmaLocalMatrixFE}
Fix $\pi$ prime, and suppose that $\pi \nmid f_1\cdots f_{r-1}$, or more generally that the functional equations \eqref{LocalFE} hold for $\left(\res{F^{(\pi)}}{\pi^k}^{-1} D_{\pi}^*(x, \pi, \vec{f}_{\pi}, k; M_{m(M_{rr}+n/2)}) \right)_k$ for all $m$. Then for $v \in \Z$,
\begin{equation}\label{LocalMatrixFE}
\begin{split}
&\left( (\xi \chi^{M_{rr}})^{k m \deg \pi}(-1)\res{F^{(\pi)}}{\pi^{k-m}}_{\xi \chi^{M_{rr}}}^{-1} D_{\pi}(x, \pi, \vec{f}_{\pi}, k-m; M_{(v+2m)(M_{rr}+n/2)}) \right)_{k,m} \\
&= \left(\frac{qx}{g}\right)^{v_{\pi}(F) \deg \pi} \left( \res{F^{(\pi)}}{\pi}_{\xi \chi^{M-{rr}}}^{k+\ell} \Gamma_{\pi, k, \ell}(x, v_{\pi}(F)-v) \right)_{k, \ell} \\
&\phantom{=} \cdot \left( (\xi \chi^{M_{rr}})^{\ell m \deg \pi}(-1) \res{F^{(\pi)}}{\pi^{\ell-m}}_{\xi \chi^{M_{rr}}}^{-1} D_{\pi}\left(\frac{g^2}{q^2x}, \pi, \vec{f}_{\pi}, \ell-m; M_{(v+2m)(M_{rr}+n/2)} \right) \right)_{\ell, m}\\
&\phantom{=}\cdot\left(\Delta_{m,m}(x, v)^{\deg \pi} \right)_{m,m} 
\end{split}
\end{equation}
where $\left( \res{F^{(\pi)}}{\pi}_{\xi \chi^{M-{rr}}}^{k+\ell} \Gamma_{\pi, k, \ell}(x, v_{\pi}(F)-v \right)_{k, \ell}$ is the $n_r \times n_r$ scattering matrix defined in Proposition \ref{LocalFE}, as applied to $M_{v(M_{rr}+n/2)}$, and $\left(\Delta_{m,m}(x, v)^{\deg \pi} \right)_{m,m} $ is as defined in Lemma \ref{LemmaMatrixFE}.
\end{lemma}

This lemma follows from Proposition \ref{PropLocalFE} in the same way that Lemma \ref{LemmaMatrixFE} follows from Proposition \ref{PropGlobalFE}.

Our goal is to prove that the same functional equations hold for all $\vec{f}$, whether or not $f_1\cdots f_{r-1}$ is squarefree. We first prove a technical statement showing the compatibility of local and global functional equations.

\begin{lemma} \label{LocalImpliesGlobal}
Fix $\vec{f}$, $\pi$ prime, $v\equiv -v_{\pi}(F) \bmod n_r$. Suppose $D(x, \pi, \vec{f}^{(\pi)}, k-m\deg \pi; M_{(v+2m)(M_{rr}+n/2)})$ is a rational function of $x$ with denominator dividing $1-q\left(\frac{qx}{g}\right)^{n_r}$ for all $k, m$, and that the global matrix functional equation \eqref{MatrixFE} holds for$\left( (\xi \chi^{M_{rr}})^{k m \deg \pi}(-1)D(x, \pi, \vec{f}^{(\pi)}, k-m\deg \pi; M_{(v+2m)(M_{rr}+n/2)}) \right)_{k,m}$. Suppose also that $D_{\pi}(x, \vec{f}_\pi, k)$ is a rational function of $x$ with denominator dividing $1-q^{-\deg \pi}\left(\frac{qx}{g}\right)^{n_r\deg \pi}$ for all $k$, and that the local vector functional equation \eqref{LocalFE} holds for $\left( D_{\pi}(x, \vec{f}_\pi, k) \right)_k$. Then $D(x, \vec{f}, k)$ is a rational function of $x$ with denominator dividing $1-q\left(\frac{qx}{g}\right)^{n_r}$ for all $k$, and the global vector functional equation \eqref{GlobalFE} holds for $\left( D(x, \vec{f}, k) \right)_k$.
\end{lemma}

\begin{proof}
Lemma \ref{LemmaStrongFactorization} applied to $\vec{f}^{(\pi)}$ gives 
\begin{equation*}
\begin{split}
&\left( (\xi \chi^{M_{rr}})^{(k-\ell\deg \pi)\ell \deg \pi} (-1)  D^{(\pi)}(x, \vec{f}^{(\pi)}, k-\ell \deg \pi, \pi^{2\ell+v}) \right)_{k, \ell} \\
&\cdot \left((\xi \chi^{M_{rr}})^{\ell m \deg \pi}(-1) \res{F^{(\pi)}}{\pi^{\ell-m}}_{\xi \chi^{M_{rr}}}^{-1}D_\pi(x, \pi, \vec{1}, \ell-m; M_{(v+2m)(M_{rr}+n/2)}) \right)_{\ell,m} \\
&= \left( (\xi \chi^{M_{rr}})^{k m \deg \pi}(-1)D(x, \pi, \vec{f}^{(\pi)}, k-m\deg \pi; M_{(v+2m)(M_{rr}+n/2)}) \right)_{k,m}.
\end{split}
\end{equation*}
The second matrix is computed explicitly in the discussion following Lemma \ref{LemmaStrongFactorization}, for $F^{(\pi)}=1$ and $\deg \pi =1$. The computation can be extended to all $\pi$ and $F^{(\pi)}$ using Axiom \ref{Axiom3}. The result is that this matrix is conjugate to a block-diagonal matrix of $2 \times 2$ blocks, where each nontrivial $2 \times 2$ block has determinant  equal to $(\xi \chi^{M_{rr}})^{(1-v)\deg \pi}(-1)\left( 1- q^{-\deg \pi}\left(\frac{qx}{g}\right)^{n_r \deg \pi}\right)$. Its inverse is a matrix of monomials in $x$.

Lemma \ref{LemmaFactorization} applied to $\vec{f}$ gives 
\begin{equation*}
\begin{split}
&\varepsilon_\pi \left( (\xi \chi^{M_{rr}})^{(k-\ell\deg \pi)(\ell \deg \pi)} (-1)  D^{(\pi)}(x, \vec{f}^{(\pi)}, k-\ell \deg \pi, \pi^{2\ell-v_{\pi}(F)}) \right)_{k, \ell} \\
&\cdot \left( \res{F^{(\pi)}}{\pi^\ell}_{\xi \chi^{M_{rr}}}^{-1} D_\pi(x, \vec{f}_\pi, \ell) \right)_\ell \\
&= \left( D(x, \vec{f}, k) \right)_k.
\end{split}
\end{equation*}

We can combine the two equations:
\begin{equation} \label{LocalGlobalCombo}
\begin{split}
& \left( D(x, \vec{f}, k) \right)_k \\
&=\varepsilon_\pi \left( (\xi \chi^{M_{rr}})^{k m \deg \pi}(-1)D(x, \pi, \vec{f}^{(\pi)}, k-m\deg \pi; M_{(v+2m)(M_{rr}+n/2)}) \right)_{k,m}  \\
&\cdot \left((\xi \chi^{M_{rr}})^{\ell m \deg \pi}(-1) \res{F^{(\pi)}}{\pi^{\ell-m}}_{\xi \chi^{M_{rr}}}^{-1}D_\pi(x, \pi, \vec{1}, \ell-m; M_{(v+2m)(M_{rr}+n/2)}) \right)_{\ell,m}^{-1} \\
&\cdot \left( \res{F^{(\pi)}}{\pi^\ell}_{\xi \chi^{M_{rr}}}^{-1} D_\pi(x, \vec{f}_\pi, \ell) \right)_\ell.
\end{split}
\end{equation}

By hypothesis, the vector in line four of \eqref{LocalGlobalCombo} 
satisfies a functional equation with scattering matrix $\left(\frac{qx}{g} \right)^{v_{\pi}(F) \deg \pi} \left( \res{F^{(\pi)}}{\pi}_{\xi \chi^{M-{rr}}}^{k+\ell} \Gamma_{\pi, k, \ell}(x, v_{\pi}(F)) \right)_{k, \ell}$. By Lemma \ref{LemmaLocalMatrixFE}, the matrix with inverse in line three of \eqref{LocalGlobalCombo}
satisfies a functional equation with scattering matrix $\left( \res{F^{(\pi)}}{\pi}_{\xi \chi^{M-{rr}}}^{k+\ell} \Gamma_{\pi, k, \ell}(x, -v) \right)_{k, \ell}$ and a diagonal matrix $\left(\Delta_{m,m}(x, v)^{\deg \pi} \right)_{m,m}$ on the right. Since $v\equiv -v_{\pi}(F) \bmod n$, we have $ \Gamma_{\pi, k, \ell}(x, v_{\pi}(F))=\Gamma_{\pi, k, \ell}(x, -v)$. Thus the product of lines three and four satisfies a functional equation in $x\mapsto \frac{g^2}{q^2 x}$, with scattering matrix $\left(\left(\frac{qx}{g} \right)^{v_{\pi}(F) \deg \pi} \Delta_{m,m}(x, v)^{-\deg \pi} \right)_{m,m}$. This scattering matrix is a diagonal matrix of monomials in $x$.

The product of lines three and four of \eqref{LocalGlobalCombo} is a priori a vector of rational functions with denominators dividing $1-q^{-\deg \pi}\left(\frac{qx}{g}\right)^{n_r\deg \pi}$. However, if any of these functions had any poles at $\left(\frac{qx}{g}\right)^{n_r}=q$, the functional equation would transform them to poles at $\left(\frac{qx}{g}\right)^{n_r}=q^{-1}$, which cannot exist. Therefore the product must in fact be a vector of polynomials. Then from our hypothesis on the rationality of $D(x, \pi, \vec{f}^{(\pi)}, k-m\deg \pi; M_{(v+2m)(M_{rr}+n/2)})$, we have that $\left( D(x, \vec{f}, k) \right)_k$ is a vector of rational functions with denominators dividing $1-q\left(\frac{qx}{g}\right)^{n_r}$.

By hypothesis, $\left( (\xi \chi^{M_{rr}})^{k m \deg \pi}(-1)D(x, \pi, \vec{f}^{(\pi)}, k-m\deg \pi; M_{(v+2m)(M_{rr}+n/2)}) \right)_{k,m}$ satisfies a functional equation in $x\mapsto \frac{g^2}{q^2 x}$, with scattering matrix $\left( \frac{qx}{g} \right)^{\deg F^{(\pi)}} \left( \Gamma_{k, \ell}(x, \deg F^{(\pi)} - v \deg \pi) \right)_{k, \ell}$ and a diagonal matrix $\left(\Delta_{m,m}(x, v)^{\deg \pi} \right)_{m,m}$ on the right. So all together, $\left( D(x, \vec{f}, k) \right)_k$ satisfies a functional equation in $x\mapsto \frac{g^2}{q^2 x}$, with scattering matrix
\begin{equation*}
\left(\frac{qx}{g}\right)^{\deg F^{(\pi)}+v_{\pi}(F) \deg \pi} \left( \Gamma_{k, \ell}(x, \deg F^{(\pi)} - v \deg \pi) \right)_{k, \ell} = \left(\frac{qx}{g}\right)^{\deg F} \left( \Gamma_{k, \ell}(x, \deg F)\right)_{k, \ell}
\end{equation*}
as desired.
\end{proof}

Local generating series $D_{\pi}(x, \vec{f}_\pi, k)$ satisfying the functional equation \eqref{LocalFE} have been constructed in many cases by combinatorial means. Then Lemma \ref{LocalImpliesGlobal} can be used to show that the global series assembled from these local ones by twisted multiplicativity satisfy the functional equation \eqref{GlobalFE}. The method of proof is an induction of the number of square factors of $f_1 \cdots f_{r-1}$. This type of proof is carried out in \cite{ChintaGunnellsJAMS} over number fields and in \cite{HFriedlander23} over the rational function field. It allows global Weyl group multiple Dirichlet series to be constructed from their local parts. The axiomatic approach does not require a separate combinatorial construction of the local parts, and it verifies the local and global functional equations simultaneously. Our main theorem is proven by a different type of induction. 

Let $D(x)$ be a power series in $x$, whose coefficients are linear combinations of $q$-Weil numbers. We say that $D(x)$ is ``sharp'' if each coefficient of $x^j$ in $D(x)$ is a linear combination of $q$-Weil numbers of weights greater than $j$, and ``flat'' if each coefficient of $x^j$ in $D(x)$ is a linear combination of $q$-Weil numbers of weights less than $j$. Sharp series and flat series are closed under linear combinations, products, and inverses.

By Axiom \ref{Axiom5}, for $d_1+\cdots+d_{r-1} \geq 2$, any series of the form
\begin{equation*}
q^{\frac{-\sum d_i-1}{2}} \sum_{\deg f_1 = d_1, \ldots \deg f_{r-1}=d_{r-1}} D(x, \vec{f}, k)
\end{equation*} 
is sharp, while any series of the form 
\begin{equation*}
q^{\frac{-\sum d_i+1}{2}}D_\pi(x, (\pi^{d_1}, \ldots \pi^{d_{r-1}}), k)
\end{equation*} is flat.

\begin{theorem} \label{TheoremFE}
For all $\vec{f}=(f_1, \ldots f_{r-1}) \in (\F_q[T]^+)^{r-1}$, $D(x, \vec{f}, k)$ is a rational function of $x$ with denominator dividing $1-q\left(\frac{q x}{g}\right)^{n_r}$ and the global functional equation \eqref{GlobalFE} is satisfied. $D_{\pi}(x, \vec{f}, k)$ is a rational function of $x$ with denominator dividing $1-q^{-\deg \pi} \left( \frac{qx}{g}\right)^{n_r \deg \pi}$ and the local functional equation \eqref{LocalFE} is satisfied.
\end{theorem}

\begin{proof}
The proof is by induction on $\deg f_1 + \cdots + \deg f_{r-1}$ and $v_{\pi} f_1+\cdots+v_{\pi} f_{r-1}$. Note that the matrix $M$ is not fixed in the argument, though we always assume $\gcd(n, M_{rr}+n/2) | M_{ri}$, for all $i$. At each step in the induction, the global functional equation is proven for all matrices $M$ and all $\vec{f}$ with $\deg f_1 + \cdots + \deg f_{r-1}$ fixed, and the local functional equation is proven for all matrices $M$, all $\pi$, and all $\vec{f}$ with $v_{\pi} f_1+\cdots+v_{\pi} f_{r-1}$ fixed.

If $\deg f_1 + \cdots + \deg f_{r-1} =0$ or $1$, then the rationality property and global functional equation hold by Proposition \ref{PropGlobalFE}. If $v_\pi f_1 + \cdots + v_\pi f_{r-1} =0$ or $1$, then the rationality property and local functional equation hold by Proposition \ref{PropLocalFE}. 

For the inductive step, suppose that the global rationality property and functional equation are known for all $M$ and all $\vec{f}$ with $\deg f_1 + \cdots + \deg f_{r-1} < A$, and the local rationality property and functional equation are known for all $M$ and all $\vec{f}$ with $v_\pi f_1 + \cdots + v_\pi f_{r-1} < A$. If $\vec{f}=(f_1,\ldots f_{r-1})$ has $\deg f_1 + \cdots + \deg f_{r-1} = A$ and $f_1\cdots f_{r-1}$ is squarefree, then the global rationality property and functional equation are known for $\vec{f}$ by Proposition \ref{GlobalFE}. If $\pi^2 | f_1\cdots f_{r-1}$ for some prime $\pi$, we apply Lemma \ref{LocalImpliesGlobal}. The vector $(\pi, \vec{f}^{(\pi)})$ has $\deg \pi + \deg f_1^{(\pi)} + \cdots + f_{r-1}^{(\pi)}< A$, so by induction, the global rationality property and functional equation hold for each extended matrix $M_{(v+2m)(M_{rr}+n/2)}$ with $v \equiv -v_{\pi}(F)$ and $m=0, \ldots n-1$. Moreover, unless $\pi$ is linear and each $f_i$ is a power of $\pi$, we have $v_\pi f_1+\cdots + v_\pi f_{r-1} < A$, so the local rationality property and functional equation hold for $\vec{f}$ at $\pi$. Therefore by Lemma \ref{LocalImpliesGlobal}, the global rationality property and functional equation hold for $\vec{f}$. This establishes the global rationality property and functional equation for all $\vec{f}$ except when the $f_i$ are all powers of the same linear polynomial.  

We now consider the sum
\begin{equation}\label{TwoSums}
\sum_{\deg f_1 = d_1, \ldots \deg f_{r-1}=d_{r-1}} \left(D(x, \vec{f}, k)\right)_k - \sum_{\deg \pi =1} \left(D(x, (\pi^{d_1}, \ldots \pi^{d_{r-1}}), k)\right)_k
\end{equation}
which satisfies the global rationality property and functional equation \eqref{GlobalFE}. Note that, by Axiom \ref{Axiom3}, the second sum is $q\left(D(x, (\pi^{d_1}, \ldots \pi^{d_{r-1}}), k)\right)_k$ for an arbitrary degree 1 monic polynomial $\pi$. 

We use Lemma \ref{LemmaFactorization} to write the second sum as 
\begin{equation*}
\begin{split}
& q\left(D(x, (\pi^{d_1}, \ldots \pi^{d_{r-1}}), k)\right)_k \\
&= q\left((\xi \chi^{M_{rr}})^{(k-\ell)\ell} (-1)  D^{(\pi)}(x, \vec{1}, k-\ell, \pi^{2\ell + v}) \right)_{k, \ell} \left(D_\pi(x, (\pi^{d_1}, \ldots \pi^{d_{r-1}}), \ell) \right)_\ell
\end{split}
\end{equation*}
where $v\equiv -d_1 n_{r1} - \cdots - d_{r-1} n_{r \, (r-1)} \bmod n_r$. Further, by equation \eqref{StrongFactorizationExample}, 
\begin{equation*}
\begin{split}
&q\left((\xi \chi^{M_{rr}})^{(k-\ell)\ell} (-1)  D^{(\pi)}(x, \vec{1}, k-\ell, \pi^{2\ell + v}) \right)_{k, \ell} \\
&= \left( (\xi \chi^{M_{rr}})^{k m}(-1)D(x, \pi, \vec{1}, k-m; M_{(v+2m)(M_{rr}+n/2)}) \right)_{k,m} \\
&\cdot \left((\xi \chi^{M_{rr}})^{\ell m}(-1) D_\pi(x, \pi, \vec{1}, \ell-m; M_{(v+2m)(M_{rr}+n/2)}) \right)_{\ell,m}^{-1}
\end{split}
\end{equation*}
and so we may multiply \eqref{TwoSums} by $q^{\frac{-A-1}{2}}\left( (\xi \chi^{M_{rr}})^{k m}(-1)D(x, \pi, \vec{1}, k-m; M_{(v+2m)(M_{rr}+n/2)}) \right)_{k,m}^{-1}$ to obtain 
\begin{equation}\label{NewTwoSums}
\begin{split}
&\left( (\xi \chi^{M_{rr}})^{k m}(-1)D(x, \pi, \vec{1}, k-m; M_{(v+2m)(M_{rr}+n/2)}) \right)_{k,m}^{-1} \left( q^{\frac{-A-1}{2}} \sum_{\deg f_i = d_i} D(x, \vec{f}, k)\right)_k \\
&- \left((\xi \chi^{M_{rr}})^{k m}(-1) D_\pi(x, \pi, \vec{1}, k-m; M_{(v+2m)(M_{rr}+n/2)}) \right)_{k,m}^{-1} \left(q^{\frac{-A+1}{2}} D_\pi(x, (\pi^{d_1}, \ldots \pi^{d_{r-1}}), k)\right)_k.
\end{split}
\end{equation}
By the discussion following Lemma \ref{LemmaFactorization}, $q^{\frac{-A-1}{2}}\left( (\xi \chi^{M_{rr}})^{k m}(-1)D(x, \pi, \vec{1}, k-m; M_{(v+2m)(M_{rr}+n/2)}) \right)_{k,m}^{-1}$ is a matrix of monomials in $x$. From the rationality property of \eqref{TwoSums}, we have that \eqref{NewTwoSums} is a vector of rational functions with denominators dividing $1-q\left(\frac{q x}{g}\right)^{n_r}$.

We know that formula \eqref{TwoSums} satisfies a functional equation in $x \mapsto  \frac{g^2}{q^2 x}$ with scattering matrix $\left(\frac{qx}{g}\right)^{\deg F}\left( \Gamma_{k, \ell}(x, \deg F)\right)_{k, \ell}$. By Lemma \ref{LemmaMatrixFE}, $\left( (\xi \chi^{M_{rr}})^{k m}(-1)D(x, \pi, \vec{1}, k-m; M_{(v+2m)(M_{rr}+n/2)}) \right)_{k,m}$ satisfies a functional equation with scattering matrix $\left( \Gamma_{k, \ell}(x, -v) \right)_{k, \ell}$ and diagonal matrix $\left(\Delta_{m,m}(x,v)\right)_{m,m}$ on the right. Since $v \equiv -\deg F \bmod n$, we have $\Gamma_{k, \ell}(x, \deg F)=\Gamma_{k, \ell}(x, -v)$. Thus, \eqref{NewTwoSums} satisfies a functional equation with scattering matrix $\left(\left(\frac{qx}{g}\right)^{\deg F}\Delta_{m,m}(x,v)^{-1}\right)_{m,m}$. This scattering matrix is a diagonal matrix of monomials in $x$.  

If any of the entries of \eqref{NewTwoSums} had any poles at $\left(\frac{qx}{g}\right)^{n_r}=q^{-1}$, the functional equation would transform them to poles at $\left(\frac{qx}{g}\right)^{n_r}=q$, which cannot exist. Therefore \eqref{NewTwoSums} must in fact be a vector of polynomials.

By the discussion preceding this theorem, the series $D(x, \pi, \vec{1}, k-m; M_{(v+2m)(M_{rr}+n/2)})$ is sharp (except for a constant term of $1$ when $k=m$), and $q^{\frac{-A-1}{2}} \sum D(x, \vec{f}, k)$ is sharp. Furthermore, $D_\pi(x, \pi, \vec{1}, k-m; M_{(v+2m)(M_{rr}+n/2)})$ is flat (except for a constant term of $1$ when $k=m$), and $q^{\frac{-A+1}{2}} D_\pi(x, (\pi^{d_1}, \ldots \pi^{d_{r-1}}), k)$ is flat. It follows that the full first summand of \eqref{NewTwoSums} is sharp and the full second summand is flat. Therefore, there is no cancellation between these summands. 

We conclude that each summand of \eqref{NewTwoSums} has entries which are polynomials in $x$. After multiplying the first summand by $q^{\frac{A+1}{2}}\left( (\xi \chi^{M_{rr}})^{k m}(-1)D(x, \pi, \vec{1}, k-m; M_{(v+2m)(M_{rr}+n/2)}) \right)_{k,m}$, we see that each $\sum D(x, \vec{f}, k)$ is a rational function of $x$ with denominator dividing $1-q\left(\frac{q x}{g}\right)^{n_r}$. After multiplying the second summand by $q^{\frac{A-1}{2}} \left((\xi \chi^{M_{rr}})^{k m}(-1) D_\pi(x, \pi, \vec{1}, k-m; M_{(v+2m)(M_{rr}+n/2)}) \right)_{k,m}$, we see that each $D_\pi(x, (\pi^{d_1}, \ldots \pi^{d_{r-1}}), k)$ is a rational function of $x$ with denominator dividing $1-q^{-1} \left( \frac{qx}{g}\right)^{n_r}$.

Since mapping $x \mapsto  \frac{g^2}{q^2 x}$ and multiplying by powers of $\frac{qx}{g}$ preserve both sharp and flat polynomials, the two summands of \eqref{NewTwoSums} must satisfy the same functional equation individually. 

Consider the second summand of \eqref{NewTwoSums}. This satisfies a functional equation with scattering matrix $\left(\left(\frac{qx}{g}\right)^{\deg F}\Delta_{m,m}(x,v)^{-1}\right)_{m,m}$.
By Lemma \ref{LemmaLocalMatrixFE}, $\left((\xi \chi^{M_{rr}})^{k m}(-1) D_\pi(x, \pi, \vec{1}, k-m; M_{(v+2m)(M_{rr}+n/2)}) \right)_{k,m}$ has a functional equation with scattering matrix $\left( \Gamma_{\pi, k, \ell}(x, -v) \right)_{k, \ell}$ and diagonal matrix $\left(\Delta_{m,m}(x,v)\right)_{m,m}$ on the right. Combining these functional equations, and recalling that $v\equiv -\deg F \bmod n_r$, we find that $\left(D_\pi(x, (\pi^{d_1}, \ldots \pi^{d_{r-1}}), k)\right)_k$ satisfies a functional equation with scattering matrix
\begin{equation*}
\left(\frac{qx}{g}\right)^{\deg F} \left( \Gamma_{\pi, k, \ell}(x, \deg F) \right)_{k, \ell}
\end{equation*}
as desired. 

The local rationality property and functional equation of $\left(D_\pi(x, (\pi^{d_1}, \ldots \pi^{d_{r-1}}), k)\right)_k$ imply the global rationality property and functional equation of $\left(D(x, (\pi^{d_1}, \ldots \pi^{d_{r-1}}), k)\right)_k$ by Lemma \ref{LocalImpliesGlobal}. This finishes the verification of the global rationality property and functional equations for all $\vec{f}$ with $\deg f_i=d_i$. 

For $\pi$ of degree $1$, we have already verified the local rationality property and functional equation of $\left(D_\pi(x, (\pi^{d_1}, \ldots \pi^{d_{r-1}}), k)\right)_k$. To extend to $\pi$ of arbitrary degree, we use Axiom \ref{Axiom3}, which expresses the coefficients of $\left(D_\pi(x, (\pi^{d_1}, \ldots \pi^{d_{r-1}}), k)\right)_k$ as a sum of the $\deg \pi$th powers of the same Weil numbers appearing in the coefficients of $\left(D_T(x, (T^{d_1}, \ldots T^{d_{r-1}}), k)\right)_k$, times $\left(\frac{\pi'}{\pi}\right)_\chi^{\sum_{i=1}^{r} d_i M_{ii}}$ where $d_r$ is the coefficient of $x^r$. This carries over to the functional equations. Set $d=v_{\pi}(F)= d_1 n_{r1} + \cdots + d_{r-1} n_{r \, (r-1)}$. The scattering matrix $\left(\Gamma_{\pi, k, \ell}(x, d)\right)_{k, \ell}$ is obtained from $\left(\Gamma_{T, k, \ell}(x, d) \right)_{k, \ell}$ by raising $x$ and each Weil number appearing in each entry to the power $\deg \pi$ and multiplying the $k,\ell$ entry by $\left(\frac{\pi'}{\pi}\right)_\chi^{(k-\ell)M_{ii}}$. Examining the formulas in Proposition \ref{LocalFE}, we see that $\Gamma_{\pi,k,k}(x, d)$ is obtained from $\Gamma_{T,k,k}(x, d)$ by raising $x$ and each Weil number to the power $\deg \pi$ while for $k\neq \ell$, $\Gamma_{\pi,k,\ell}(x, d)$ is obtained from $\Gamma_{T,k,\ell}(x, d)$ by raising $x$ and each Weil number to the power $\deg \pi$, except in the $g_{(\xi \chi^{M_{rr}})^{2k- d-1}}(1,\pi)$ term, but this term exactly cancels the $\left(\frac{\pi'}{\pi}\right)_\chi^{(k-\ell)M_{ii}}$ factor since $2k - d- 1 \equiv k-\ell \bmod n$. 

Further, the local rationality property and functional equation extend to $\left(D_\pi(x, \vec{f}, k)\right)_k$ with $v_\pi(f_i)=d_i$ by equation \eqref{LocalSimplification} and by including the necessary residue symbols in the scattering matrix entries. This finishes the verification of the local rationality property and functional equations for all $\vec{f}$ with $v_\pi f_i=d_i$. 
\end{proof}

\subsection{Multivariable Functional Equation}\label{ss-k-multi}

The functional equation of Theorem \ref{TheoremFE} applies to single-variable generating series, but this leads to a functional equation for the multivariable series which are our main objects of study. 

Let $a(f_1, \ldots f_r)$ be the axiomatic multiple Dirichlet series coefficients determined by the data of $q$, $\chi$, $M$ as before. We continue with the same assumptions on these parameters. Write $\vec{x}=(x_1, \ldots x_r) \in \C^r$, $\vec{k}=(k_1, \ldots k_r) \in \Z^r$ and let 
\begin{align}
&Z(\vec{x}, \vec{k}) = \sum_{\substack{f_1, \ldots f_r \in \F_q[T]^+ \\ \deg f_i \equiv k_i \bmod n}} a(f_1,\ldots f_r) x_1^{\deg f_1} \cdots x_r^{\deg f_r} \\
&Z_{\pi}(\vec{x}, \vec{k}) = \sum_{\substack{j_1, \ldots j_r \geq 0 \\ j_i \equiv k_i \bmod n}} a(\pi^{j_1}, \ldots \pi^{j_r}) x_1^{j_1\deg \pi} \cdots x_r^{j_r \deg \pi}.
\end{align}

We now introduce notation relevant for the multivariable functional equations. Fix $i \in \{1, \ldots r\}$ and assume that $\gcd(n, M_{ii}+n/2) | M_{ij}$, for all $j$. For fixed $k_1, \ldots k_{i-1}, k_{i+1}, \ldots k_r$, let $K=k_1 n_{i1} + \cdots +k_{i-1} n_{i \, (i-1)} + k_{i+1} n_{i \, (i+1)} +\cdots + k_r n_{ir}$, and let $v_i=k_{i+1} M_{i \,(i+1)} +\cdots + k_r M_{i r}$.

Define a transformation $\sigma_{i; q, \chi, M}: \C^r \to \C^r$ as follows:
\begin{equation} \label{sigma-i-kubota}
\sigma_{i; q, \chi, M}(\vec{x})_j = \left\lbrace \begin{array}{cc} 
\frac{g^2}{q^2x_i} & j=i \\
x_j\left(\frac{q x_i}{g}\right)^{n_{ij}} & j \neq i \end{array} \right. .
\end{equation}

For any power series $D(x)$, let $S^{k, n} D(x)$ be the series consisting of all terms in $D(x)$ whose powers of $x$ are congruent to $k$ mod $n$. This can be computed as $\frac1n\sum\limits_{j=0}^{n-1} \zeta^{-jk} D(\zeta^j x)$ where $\zeta$ is a primitive $n$-th root of unity. $S^{k, n}$ is a $\C$-linear operation on power series, which satisfies $S^{k, n} x^K D(x) = x^K S^{k-K, n}D(x)$, and more generally
\begin{equation*}
S^{k, n} (D_1 D_2) = \sum_{j=0}^{n-1} (S^{k-j, n}D_1)(S^{j, n}D_2).
\end{equation*}
If $D(x)$ is the power series of a rational function, then so is $S^{k, n} D(x)$.

We need to define an expansion operation that transforms the $n_r \times n_r$ scattering matrices from the previous section into $n\times n$ scattering matrices, using the fact that $n_r$ divides $n$. Let $m$ be a natural number dividing $n$. Let $\Gamma(x,K)$ be an $m\times m$ matrix with entries power series in a variable $x$, depending on an integer $K$, such that the exponent of $x$ in each nonzero term in the entry $\Gamma_{k,\ell} (x,K)$ is congruent to $k + \ell- K \bmod m$. Then let $E_m^n \Gamma(x,K)$ be the $n\times n$ matrix whose entries are given by 
\begin{equation*}
(E_m^n \Gamma)_{k,\ell}(x,K) = S^{k+\ell-K, n} \Gamma_{k \% m, \ell \% m} (x,K).
\end{equation*} 

\begin{theorem}\label{TheoremMultiFE}
Assume that $q \equiv 1 \bmod 4$, and $M$ is such that $\gcd(n, M_{ii}+n/2) | M_{ij}$, for all $j$. The vector $(Z(\vec{x}, (k_1, \ldots k_r))_{k_i}$ satisfies a functional equation 
\begin{equation}\label{eq-TheoremMultiFE}
(Z(\vec{x}, (k_1, \ldots k_r)))_{k_i} = \left( \chi^{v_i (k_i+\ell_i)}(-1) E_{n_i}^n \Gamma_{k_i, \ell_i}(x_i, K) \right)_{k_i, \ell_i} (Z(\sigma_i(\vec{x}), (k_1, \ldots k_{i-1}, \ell_i, k_{i+1}, \ldots k_r)))_{\ell_i}
\end{equation}
where $\left( E_{n_i}^n \Gamma_{k, \ell}(x, K) \right)_{k, \ell}$ is the $n \times n$ scattering matrix obtained by applying the expansion operation above to the scattering matrix from Proposition \ref{PropGlobalFE}, and Theorem \ref{TheoremFE}, namely
\begin{equation*}
\begin{split}
E_{n_i}^n\Gamma_{k,\ell}(x, K) =& q^{(n-n_i-(n-n_i+1+K-k-\ell)\% n+(1+K-k-\ell)\%n_i)/n_i} \left(\frac{qx}{g}\right)^{n-n_i+1 - (n-n_i+1+K-k-\ell)\% n}\\
&\frac{1-q}{1-q^{n/n_i}\left(\frac{q x}{g}\right)^n}
\end{split}
\end{equation*}
when $k \equiv \ell \bmod n_i$,
\begin{equation*}
\begin{split}
E_{n_i}^n\Gamma_{k,\ell}(x, K)=&(\xi \chi^{M_{rr}})^{\ell(1+K)}(-1) g_{(\xi \chi^{M_{rr}})^{2k-K -1}} \left(\frac{qx}{g}\right) \\
&\frac{q^{((k+\ell-K-1)\%n)/n_i}\left(\frac{qx}{g}\right)^{(k+\ell-K-1)\%n}-q^{((k+\ell-K-1+n_i)\%n)/n_i} \left(\frac{qx}{g}\right)^{-n_i+(k+\ell-K-1+n_i)\%n}}{1-q^{n/n_i}\left(\frac{qx}{g}\right)^n} 
\end{split}
\end{equation*}
for $\ell \equiv 1+K -k \bmod n_i$, and all other entries equal to $0$. In addition, there are special entries when $k \equiv \ell \equiv 1+K -k-n_i \bmod n$. In this case 
\begin{equation*}
\Gamma_{k,\ell}(x, K)=\left(\frac{qx}{g}\right)^{1-n_i}.
\end{equation*}
\end{theorem}
\begin{proof} First we will prove the statement for $i=r$. We can write $Z( \vec{x}, (k_1, \ldots k_r))$ (with $k_r$ mod $n$) in terms of the $D(x_r, (f_1, \ldots f_{r-1}), k_r)$ (with $k_r$ mod $n_r$) from the previous section as follows:
\begin{equation*}
Z( \vec{x}, (k_1, \ldots k_r)) = \sum_{\substack{f_1, \ldots f_{r-1} \in \F_q[T]^+ \\ \deg f_i \equiv k_i \bmod n}}   x_1^{\deg f_1}\cdots x_{r-1}^{\deg f_{r-1}} S^{k_r, n} D(x_r, (f_1, \ldots f_{r-1}), k_r).
\end{equation*}
Here we understand $k_r$ as an integer mod $n$. After applying the functional equation \eqref{KubotaFE}, we obtain
\begin{equation*}
\begin{split}
&Z( \vec{x}, (k_1, \ldots k_r)) = \sum_{\substack{f_1, \ldots f_{r-1} \in \F_q[T]^+ \\ \deg f_i \equiv k_i \bmod n}}  x_1^{\deg f_1}\cdots x_{r-1}^{\deg f_{r-1}} S^{k_r, n} \sum_{\ell_r=1}^{n_r} \left(\frac{qx_r}{g}\right)^K
\Gamma_{k_r, \ell_r}(x_r, K) D(\frac{g^2}{q^2 x_r}, (f_1, \ldots f_{r-1}), \ell_r) \\
&=\sum_{\substack{f_1, \ldots f_{r-1} \in \F_q[T]^+ \\ \deg f_i \equiv k_i \bmod n}}  x_1^{\deg f_1}\cdots x_{r-1}^{\deg f_{r-1}} \sum_{\ell_r=1}^{n_r} \sum_{j=0}^n \left(S^{k_r-j, n} \left(\frac{qx_r}{g}\right)^K
\Gamma_{k_r, \ell_r}(x_r, K)\right)\left(S^{j, n} D(\frac{g^2}{q^2 x_r}, (f_1, \ldots f_{r-1}), \ell_r) \right).
\end{split}
\end{equation*}

Because $D(\frac{g^2}{q^2 x_r}, (f_1, \ldots f_{r-1}), \ell_r)$ is a rational function whose power series expression only contains powers of $x_r$ congruent to $-\ell_r$ modulo $n_r$, the terms of the $j$ sum vanish except when $j \equiv -\ell_r \bmod n_r$. We can thus reparametrize the double sum over $j$ and $\ell_r \bmod n_r$ as a single sum over $\ell_r \bmod n$, with $j=-\ell_r$. This leads to the formula 
\begin{equation*}
\begin{split}
Z( \vec{x}, (k_1, \ldots k_r)) = &\sum_{\substack{f_1, \ldots f_{r-1} \in \F_q[T]^+ \\ \deg f_i \equiv k_i \bmod n}}  \left(\left(\frac{qx_r}{g}\right)^{n_{r 1}} x_1\right)^{\deg f_1}\cdots \left(\left(\frac{qx_r}{g}\right)^{n_{r \, (r-1)}} x_{r-1}\right)^{\deg f_{r-1}} \\
&\sum_{\ell_r=1}^{n} \left(S^{k_r+\ell_r-K, n} 
\Gamma_{k_r, \ell_r}(x_r, K)\right)\left(S^{-\ell_r, n} D(\frac{g^2}{q^2 x_r}, (f_1, \ldots f_{r-1}), \ell_r) \right)
\end{split}
\end{equation*}
which is equivalent to the desired vector functional equation
\begin{equation}
(Z(\vec{x}, (k_1, \ldots k_r)))_{k_r} = \left( E_{n_r}^n \Gamma_{k_r, \ell_r}(x_r, K) \right)_{k_r, \ell_r} (Z(\sigma_r(\vec{x}), (k_1, \ldots k_{r-1}, \ell_r)))_{\ell_r}.
\end{equation}
This completes the proof when $i=r$. 

For an arbitrary $i$ with $\gcd(n, M_{ii}+n/2) | M_{ij}$, for all $j$, let $\rho$ be the matrix representing the permutation $(i \, r \, r-1 \, \ldots \, i+1)$. If $a(f_1, \ldots f_r; M)$ are axiomatic coefficients for the matrix $M$, then we have
\begin{equation*} 
a(\rho(f_1, \ldots f_r); \rho M \rho^{-1}) = a(f_1, \ldots f_r; M) \chi(-1)^{\deg f_i (M_{i \, (i+1)} \deg f_{i+1} + \cdots + M_{ir}\deg f_r)}.
\end{equation*}
It is straightforward to check that this expression satisfies the axioms for the conjugate matrix $\rho M \rho^{-1}$. The power of $\chi(-1)$ is introduced because of the reciprocity law, since the order of the $f_j$ terms changes. Thus 
\begin{equation*}
(\chi(-1)^{k_i v_i} Z(\vec{x}, (k_1, \ldots k_r); M))_{k_i} = Z(\rho(\vec{x}), \rho(k_1, \ldots k_r); \rho M \rho^{-1})_{k_i}.
\end{equation*}
The $\sigma_r$ functional equation on the right side corresponds to a $\sigma_i$ functional equation on the left. To obtain the $\sigma_i$ functional equation for $Z(\vec{x}, (k_1, \ldots k_r); M)_{k_i}$ on its own, conjugate the scattering matrix by a diagonal matrix whose entry at the $(k_i, k_i)$ position is $\chi(-1)^{k_i v_i}$. This introduces the powers of $\chi(-1)^{v_i}$ into the formulas for the scattering matrix entries.
\end{proof}

We also have the analogous local functional equation.

\begin{theorem} \label{TheoremLocalMultiFE}
Assume that $q \equiv 1 \bmod 4$, $n$ is even, and $M$ is such that $\gcd(n, M_{ii}+n/2) | M_{ij}$, for all $j$. The vector $(Z_{\pi}(\vec{x}, (k_1, \ldots k_r))_{k_i}$ satisfies a functional equation 
\begin{equation}
(Z_{\pi}(\vec{x}, (k_1, \ldots k_r))_{k_i} = (\chi^{v_i (k_i+\ell_i)\deg \pi}(-1) E_{n_i}^n\Gamma_{\pi, k_i, \ell_i}(x_i, K))_{k_i, \ell_i} (Z_{\pi}(\sigma_i(\vec{x}), (k_1, \ldots k_{i-1}, \ell_i, k_{i+1, \ldots k_r}))_{\ell_i}
\end{equation}
where $\left( E_{n_i}^n \Gamma_{\pi, k, \ell}(x, K) \right)_{k, \ell}$ is the $n \times n$ scattering matrix obtained by applying the expansion operation above to the scattering matrix from Proposition \ref{PropLocalFE} and Theorem \ref{TheoremFE}, namely
\begin{equation*}
\begin{split}
E_{n_i}^n\Gamma_{k,\ell}(x, K) = & q^{(n-n_i-(n-n_i+1+K-k-\ell)\% n+(1+K-k-\ell)\%n_i)\deg \pi/n_i} \left(\frac{qx}{g}\right)^{(n-n_i+1 - (n-n_i+1+K-k-\ell)\% n)\deg \pi} \\
&\frac{1-q^{\deg \pi}}{1-q^{n\deg \pi/n_i}\left(\frac{q x}{g}\right)^{n\deg \pi}}
\end{split}
\end{equation*}
when $k \equiv \ell \bmod n_i$,
\begin{equation*}
\begin{split}
&E_{n_i}^n\Gamma_{k,\ell}(x, K)=(\xi \chi^{M_{rr}})^{\ell(1+K)\deg \pi}(-1) \frac{g_{(\xi \chi^{M_{rr}})^{2k-K -1}}(1, \pi) }{q^{\deg \pi}} \left(\frac{qx}{g}\right)^{\deg \pi} \\
&\frac{q^{((k+\ell-K-1)\%n)\deg \pi/n_i}\left(\frac{qx}{g}\right)^{ ((k+\ell-K-1)\%n)\deg \pi}-q^{((k+\ell-K-1+n_i)\%n)\deg \pi/n_i} \left(\frac{qx}{g}\right)^{(-n_i+(k+\ell-K-1+n_i)\%n)\deg \pi}}{1-q^{n \deg \pi /n_i}\left(\frac{qx}{g}\right)^{n \deg \pi}} 
\end{split}
\end{equation*}
for $\ell \equiv 1+K -k \bmod n_i$, and all other entries equal to $0$. In addition, there are special entries when $k \equiv \ell \equiv 1+K -k-n_i \bmod n$. In this case 
\begin{equation*}
\Gamma_{k,\ell}(x, K)=\left(\frac{qx}{g}\right)^{(1-n_i)\deg \pi}.
\end{equation*}
\end{theorem}
The proof is the same as that of Theorem \ref{TheoremMultiFE}, using the local functional equation in Theorem \ref{TheoremFE}.

\section{Dirichlet Functional Equation}\label{SectionDirichlet}
\subsection{Single Variable Functional Equation}
In this section, we prove a functional equation for axiomatic multiple Dirichlet series based on functional equations of Dirichlet L-functions. Unlike the functional equation constructed in the previous section, this functional equation usually relates multiple Dirichlet series constructed from two different matrices. The one exception is when the relevant Dirichlet characters are quadratic, where this functional equation relates a multiple Dirichlet series to itself, and matches the previously constructed functional equation. This special case will be discussed at the end of this section.

Again, we begin with a single-variable subseries and its functional equation. We still take $q$ an odd prime power (with no assumption that $q\equiv 1 \bmod 4$ in this section), $\chi$ a character of even order $n$, and $M=(M_{ij})$ an $r \times r$ integer matrix $M=(M_{ij})$, whose entries we assume lie in $[0, n)$. For this section, we make the assumption that $M_{rr} =0$. This implies $n_r=2$, each $n_{ri}=0$ or $1$, and $p_{ri}=(-M_{ri})\% \frac{n}{2}$.

The formula for coefficients $a(f_1, \ldots f_r)$ with $f_1 \cdots f_{r-1}$ squarefree becomes substantially simpler in this case. It matches formula \eqref{SquarefreeCoefficients}, which had the stronger hypothesis of $f_1 \cdots f_{r}$ squarefree.

\begin{prop} \label{PropDirichletCoeff}
Assume $M_{rr}=0$. Fix $f_1, \ldots ,f_{r-1}$ with $f_1\cdots f_{r-1}$ squarefree. Then for all $f_r$,
\begin{equation}
a(f_1, \ldots f_r; q, \chi, M)=\prod_{i=1}^{r-1} \res{f_i'}{f_i}_{\chi}^{M_{ii}} \, \prod_{i=1}^{r-1}\prod_{j=i+1}^{r} \res{f_i}{f_j}_{\chi}^{M_{ij}}.
\end{equation}
\end{prop}

\begin{proof}
We specialize Proposition \ref{PropCoeff4}. Recall the notation of that Proposition: $F=f_1^{n_{ri}}\cdots f_{r-1}^{n_{r \, (r-1)}}$ and
\begin{equation*}
\varepsilon= \prod_{i=1}^{r-1} \res{f_i'}{f_i}_{\chi}^{M_{ii}} \, \prod_{i=1}^{r-2}\prod_{j=i+1}^{r-1} \res{f_i}{f_j}_{\chi}^{M_{ij}}.
\end{equation*} 
If $M_{rr}=0$, then  $F$ is squarefree and $f_1^{p_{r1}}\cdots f_{r-1}^{p_{r \, (r-1)}}$ is $n$-th power free. Moreover, if a prime $\pi$ divides $F$ and $f_r$, then by formula \eqref{Coeff3Prime}, $a(f_1, \ldots f_r)$ vanishes. Thus, the formula holds for $\gcd(F, f_r) \neq 1$ and we may assume that $\gcd(F, f_r)=1$. 

The Gauss sums in equation \eqref{Coeff4} are quadratic in this case. We recall the following evaluations, for $\xi$ the unique character of order $2$ on $\F_q^*$. For $\gcd(F, f_r)=1$, we have $g_{\xi}(F, f_r) = \res{F}{f_r}_{\xi}^{-1} g_{\xi}(1, f_r)$. Gauss showed that
\begin{equation} \label{QuadGaussSum}
g_{\xi} = \left\lbrace \begin{array}{cc} \sqrt{q} & q \equiv 1 \bmod 4 \\ i\sqrt{q} & q \equiv -1 \bmod 4 \end{array}\right. .
\end{equation}
And finally, for arbitrary $f_r$, 
\begin{equation*}
g_{\xi}(1, f_r) = \left\lbrace \begin{array}{cc} \xi(-1)^{\deg f_r(\deg f_r-1)/2} g_{\xi}^{\deg f_r} & f_r \text{ squarefree} \\ 0 & \text{ otherwise} \end{array}\right. .
\end{equation*}
The first case is by Lemma \ref{GaussSumLifting}, and the second is an elementary evaluation. Therefore in equation \eqref{Coeff4}, the only term which contributes to the sum over $u$ is where $f_r/u^2$ is squarefree. Fixing this $u$ value, we obtain
\begin{equation*} 
\begin{split}
a(f_1, \ldots f_r) &= \varepsilon \res{f_1^{p_{r1}}\cdots f_{r-1}^{p_{r \, (r-1)}}}{f_r}^{-1}_{\chi} \res{F}{f_r}_{\xi}^{-1} \frac{\xi(-1)^{\deg f_r(\deg f_r -1)/2}}{g_{\xi}^{\deg f_r}} q^{\deg u} g_{\xi}(1, f_r/u^2) \\
&=\varepsilon \res{f_1^{p_{r1}}\cdots f_{r-1}^{p_{r \, (r-1)}}}{f_r}^{-1}_{\chi} \res{F}{f_r}_{\xi}^{-1} \frac{q^{\deg u} \xi(-1)^{\deg u}}{g_{\xi}^{2 \deg u}} .
\end{split}
\end{equation*}
The last factor is $1$ by \eqref{QuadGaussSum}. For each $i$, we have $\chi^{-p_{ri}}\xi^{-n_{ri}} = \chi^{M_{ri}}$ by equation \eqref{CharacterEquality}. This leads to 
\begin{equation*}
\res{f_1^{p_{r1}}\cdots f_{r-1}^{p_{r \, (r-1)}}}{f_r}^{-1}_{\chi} \res{F}{f_r}_{\xi}^{-1} = \prod_{i=1}^{r-1} \res{f_i}{f_r}_{\chi}^{M_{ri}}
\end{equation*}
which completes the proof. 
\end{proof}

We will set $e_{ij}=1$ if $M_{ij}\neq 0$, $e_{ij}=0$ if $M_{ij}=0$. For the rest of this section, we set $\tilde{F}=f_1^{M_{1r}}\cdots f_{r-1}^{M_{(r-1) \, r}}$, and $F = f_1^{e_{1r}}\cdots f_{r-1}^{e_{(r-1) \, r}}$. Note that this definition of $F$ does not match the definition used in Sections \ref{SectionCoeffs}, \ref{SectionKubota}, or in the proof of Proposition \ref{PropDirichletCoeff}. The reason for the redefinition is that this $F$ plays the same role in the Dirichlet functional equation that the previous $F$ played in the Kubota functional equation. 

For fixed $\vec{f}=(f_1, \ldots f_{r-1})$ and prime $\pi$, we define 
\begin{align}
&D(x, \vec{f}) = \sum_{f_r \in \F_q[T]^+} a(\vec{f}, f_r) x^{\deg f_r} \\
&D_{\pi}(x, \vec{f}) = \sum_{d \geq 0} a(\vec{f}, \pi^d)x^{d \deg \pi} .
\end{align}

Then for $f_1\cdots f_{r-1}$ squarefree, Proposition \ref{PropDirichletCoeff} implies that 
\begin{equation} \label{DirichletGlobalEval}
D(x, \vec{f}) = \varepsilon L\left(x, \res{\tilde{F}}{*}_{\chi}\right) .
\end{equation} 
Where $L\left(x, \res{\tilde{F}}{*}_{\chi}\right)$ is the Dirichlet $L$-function 
\begin{equation*}
L\left(x, \res{\tilde{F}}{*}_{\chi}\right)=\sum_{f_r \in \F_q[T]^+} \res{\tilde{F}}{f_r}_{\chi} x^{\deg f_r} .
\end{equation*}
By power reciprocity, $L\left(x, \res{\tilde{F}}{*}_{\chi}\right)= L\left(\chi(-1)^{\deg \tilde{F}} x, \res{*}{\tilde{F}}_{\chi}\right)$.
We are assuming that all $M_{ij}$ are between $0$ and $n-1$, so $\tilde{F}$ is $n$-th power free. $F$, which is the conductor of the $L$-function, is squarefree. Unless $\tilde{F}=1$, $L\left(x, \res{*}{\tilde{F}}_{\chi}\right)$ is a polynomial of degree $\deg F-1$. In the special case $\tilde{F}=1$, $L\left(x, \res{*}{\tilde{F}}_{\chi}\right) = (1-qx)^{-1}$ is the function field zeta function.

Locally, Proposition \ref{PropDirichletCoeff} implies that 
\begin{equation} \label{DirichletLocalEval}
D_{\pi}(x, \vec{f}) = \varepsilon \sum_{d \geq 0} \res{\tilde{F}}{\pi}_{\chi}^d x^{d \deg \pi} = \frac{\varepsilon}{1-\res{\tilde{F}}{\pi}_{\chi} x^{\deg \pi}} = \frac{\varepsilon}{1-\res{\pi}{\tilde{F}}_{\chi} (\chi(-1)^{\deg \tilde{F}} x)^{\deg \pi}}.
\end{equation} 
This becomes just $\varepsilon$ if $\pi | \tilde{F}$. 

The Dirichlet $L$-functions have a symmetry in $x \mapsto \frac{1}{qx}$. Let $b=1$ if $n|\deg \tilde{F}$ and $b=0$ otherwise, so that the Euler factor at infinity is $(1-x)^{-b}$. Then 
\begin{equation}
(1-x)^{-b} L\left(x, \res{*}{\tilde{F}}_{\chi} \right)=\omega_1 (\sqrt{q}x)^{\deg F -1-b} \left(1-\frac{1}{qx}\right)^{-b}L\left(\frac{1}{qx}, \res{*}{\tilde{F}}_{\bar{\chi}}\right)
\end{equation}
where the root number $\omega_1$ is 
\begin{equation*}
\omega_1= \left( \sum_{\nu \bmod F} \res{\nu}{\tilde{F}}_{\chi} \exp \left(\nu/F \right) \right) \left\lbrace \begin{array}{cc} q^{(1- \deg F)/2}g_{\chi^{\deg \tilde{F}}}^{-1} & n \nmid \deg \tilde{F} \\ q^{- \deg F/2} & n | \deg \tilde{F} \end{array} \right. .
\end{equation*}

We can expand the root number further using the twisted multiplicativity of of the Gauss sum:
\begin{equation*}
\begin{split}
\sum_{\nu \bmod F} \res{\nu}{\tilde{F}}_{\chi} \exp \left(\nu/F\right) &= \prod_{\substack{1 \leq i \leq r-1 \\ M_{ir} \neq 0}} g_{\chi^{M_{ir}}}(1, f_i) \prod_{\substack{1 \leq i<j \leq r-1 \\ M_{ir}, M_{jr} \neq 0}} \res{f_i}{f_j^{M_{jr}}}_\chi \res{f_j}{f_i^{M_{ir}}}_\chi \\
&=\omega  \prod_{\substack{1 \leq i \leq r-1 \\ M_{ir} \neq 0}} g_{\chi^{M_{ir}}}^{\deg f_i} \res{f_i'}{f_i}_\chi^{M_{ir}+n/2} \prod_{\substack{1 \leq i<j \leq r-1 \\ M_{ir}, M_{jr} \neq 0}} \res{f_i}{f_j}_{\chi}^{M_{ir}+M_{jr}} 
\end{split}
\end{equation*}
where 
\begin{equation} \label{RootNumberSign}
\omega =\omega(\vec{f}; M)= \prod_{\substack{1 \leq i \leq r-1 \\ M_{ir} \neq 0}} \xi(-1)^{\deg f_i(\deg f_i-1)/2} \prod_{\substack{1 \leq i<j \leq r-1 \\ M_{ir}, M_{jr} \neq 0}} \chi(-1)^{(\deg f_i)(\deg f_j) M_{ir} } = \pm 1.
\end{equation}

Introduce a new $r\times r$ integer matrix $M'$ with entries as follows: $M'_{rr}=M_{rr}=0$, $M'_{ir}=-M_{ir}$ for $i\neq r$, $M'_{ii}=M_{ii}+M_{ir}+\frac{n}{2}$ for $i \neq r$ and $M_{ir} \neq 0$, $M'_{ii}=M_{ii}$ for $M_{ir} = 0$, $M'_{ij}=M_{ij}$ for $i\neq j \neq r$ if $M_{ir}=0$ or $M_{jr}=0$, $M'_{ij}=M_{ij}+M_{ir}+M_{jr}$ for $i\neq j \neq r$ if $M_{ir}\neq0$ and $M_{jr}\neq 0$. By Proposition \ref{PropDirichletCoeff} the axiomatic coefficients $a(f_1, \ldots f_{r-1}, f_r; q, \chi, M')$ constructed from the matrix $M'$ are as follows
\begin{equation*}
a(f_1, \ldots f_{r-1}, f_r; M')= \varepsilon \prod_{i=1}^{r-1} \res{f_i}{f_r}_{\chi}^{-M_{ir}} \prod_{\substack{1 \leq i \leq r-1 \\ M_{ir} \neq 0}} \res{f_i'}{f_i}_{\chi}^{M_{ir}+n/2}
\prod_{\substack{1 \leq i<j \leq r-1 \\ M_{ir}, M_{jr} \neq 0}} \res{f_i}{f_j}_{\chi}^{M_{ir}+M_{jr}}.
\end{equation*}
And the single-variable series for $M'$ is 
\begin{equation*}
D(x, \vec{f}; M')= \epsilon \prod_{\substack{1 \leq i \leq r-1 \\ M_{ir} \neq 0}} \res{f_i'}{f_i}_{\chi}^{M_{ir}+n/2}
\prod_{\substack{1 \leq i<j \leq r-1 \\ M_{ir}, M_{jr} \neq 0}} \res{f_i}{f_j}_{\chi}^{M_{ir}+M_{jr}} L \left(\chi(-1)^{\deg \tilde{F}} x, \res{*}{\tilde{F}}_{\bar{\chi}}\right) .
\end{equation*}

Then the above calculations lead to the following functional equation, which relates $D(x, \vec{f}; M)$ to $D(x, \vec{f}; M')$
\begin{prop} \label{PropDirichletFE}
Suppose that $M_{rr}=0$ and $f_1 \cdots f_{r-1}$ is squarefree. Then $D(x, \vec{f}; M)$ is a rational function of $x$ with denominator dividing $1-qx$ if $\tilde{F}$ is an $n$th power in $\F_q[T]$, and a polynomial otherwise. It satisfies the functional equation 
\begin{equation} \label{DirichletFE}
\begin{split}
D(x, \vec{f}; M) =& \omega \, g_{\chi^{\deg \tilde{F}}}^{b-1} \, (\chi(-1)^{\deg \tilde{F}} x)^{\deg F-1} \prod_{\substack{1 \leq i \leq r-1 \\ M_{ir} \neq 0}} g_{\chi^{M_{ir}}}^{\deg f_i} \, \left(\frac{x-1}{1-qx}\right)^b D\left(\frac{1}{qx}, \vec{f}; M'\right)
\end{split}
\end{equation}
with $\omega$ defined in equation \eqref{RootNumberSign}.
\end{prop}
\noindent Since we are assuming $f_1 \cdots f_{r-1}$ is squarefree, the only way $\tilde{F}$ can be an $n$th power is if it is $1$. But we will show in Theorem \ref{TheoremFE} that the statement of the proposition actually holds for all $f$, without the squarefree hypothesis. $F$ will still be the conductor, determining the degree of the polynomial, even when it is not squarefree. 

Next we give the corresponding local functional equation. The local series $D_{\pi}(x, \vec{f}; M')$, for $f_1\cdots f_{r-1}$ squarefree, is as follows:
\begin{equation*}
D_{\pi}(x, \vec{f}, M')= \frac{\varepsilon \prod\limits_{\substack{1 \leq i \leq r-1 \\ M_{ir} \neq 0}} \res{f_i'}{f_i}_{\chi}^{M_{ir}+n/2}
\prod\limits_{\substack{1 \leq i<j \leq r-1 \\ M_{ir}, M_{jr} \neq 0}} \res{f_i}{f_j}_{\chi}^{M_{ir}+M_{jr}}}{1-\res{\pi}{\tilde{F}}^{-1}_{\chi} (\chi(-1)^{\deg \tilde{F}} x)^{\deg \pi}}.
\end{equation*}
The denominator becomes $1$ if $\pi | \tilde{F}$.

We will state the local functional equation only when $\vec{f}$ consists of powers of $\pi$. 
\begin{prop} \label{PropDirichletLocalFE}
Suppose that $M_{rr}=0$ and $f_1 \cdots f_{r-1}=1$ or $\pi$. Then $D_\pi(x, \vec{f}; M)$ is a rational function of $x$ with denominator dividing $1-x^{\deg \pi}$ if $n | v_\pi \tilde{F}$, and a polynomial otherwise. It satisfies the functional equation 
\begin{equation} \label{DirichletLocalFE}
\begin{split}
D_{\pi}(x, \vec{f}; M) &= \prod_{\substack{1 \leq i \leq r-1 \\ M_{ir} \neq 0}} \res{\pi'}{\pi}_{\chi}^{-v_\pi f_i(M_{ir}+n/2)} \omega_\pi \left(\frac{-(-g_{\chi^{v_{\pi} \tilde{F}}})^{\deg \pi}}{q^{\deg\pi}}\right)^{b_{\pi}-1} (\chi(-1)^{v_\pi\tilde{F}}qx)^{\deg F-\deg \pi} \\
&\prod_{\substack{1 \leq i \leq r-1 \\ M_{ir} \neq 0}} \left(\frac{-(-g_{\chi^{M_{ir}}})^{\deg \pi}}{q^{\deg \pi}}\right)^{v_\pi f_i} \left( \frac{(qx)^{\deg \pi}-1}{1-x^{\deg \pi}}\right)^{b_{\pi}} D_{\pi}\left(\frac{1}{qx}, \vec{f}; M'\right)
\end{split}
\end{equation}
where $b_{\pi}=1$ if $n | v_{\pi} \tilde{F}$ and $0$ otherwise, and 
\begin{equation} \label{LocalRootNumberSign}
\omega_\pi =\omega_\pi(\vec{f}; M)= \prod_{\substack{1 \leq i \leq r-1 \\ M_{ir} \neq 0}} \xi(-1)^{\deg \pi (v_\pi f_i)(v_\pi f_i-1)/2} \prod_{\substack{1 \leq i<j \leq r-1 \\ M_{ir}, M_{jr} \neq 0}} \chi(-1)^{\deg \pi(v_\pi f_i)(v_\pi f_j) M_{ir} } = \pm 1.
\end{equation}
\end{prop}
\noindent Again, Theorem \ref{TheoremDirichletFE} will establish that this statement holds for all $\vec{f}=(\pi^{d_1}, \ldots \pi^{d_{r-1}})$. Proposition \ref{PropDirichletLocalFE} is checked by direct computation when $\vec{f}=(1, \ldots 1)$ or $\vec{f} = (1, \ldots \pi, \ldots 1)$. Some terms in Equation \eqref{DirichletLocalFE} trivialize in these cases, but this is the correct form of the functional equation for general $\vec{f}$.


Note that the form of the functional equation \eqref{DirichletFE} depends on $\deg f_1, \ldots \deg f_{r-1}$ but not on $f_1, \ldots f_{r-1}$ themselves. Furthermore, the local and global functional equations \eqref{DirichletFE} and \eqref{DirichletLocalFE} are compatible under the axioms. Axiom \ref{Axiom5} gives the change of variables between these two functional equations. Axiom \ref{Axiom3} explains the extra term involving $\res{\pi'}{\pi}_\chi$ in \eqref{DirichletLocalFE}, which arises from the discrepancy between $M'$ and $M$.

The series $D(x,\vec{f})$ has an Euler product expression. We can see this from the explicit formulas above when $f_1\cdots f_{r-1}$ is squarefree, but it holds for arbitrary $\vec{f}$ because of the twisted multiplicativity axiom with $M_{rr}=0$. We will need the analogue of Lemma \ref{LemmaFactorization} to separate out the contribution of one prime. Let 
\begin{equation}
    D^{(\pi)}(x, \vec{f}, \pi^{\ell}; q, \chi M) = \sum_{\substack{f_r \in \F_q[T]^+\\ \pi \nmid f_r}} a(\vec{f}, f_r; q, \chi, M) \res{\pi^{\ell}}{f_r}_{\chi} x^{\deg f_r}
\end{equation}
and for $\pi \nmid h$, let 
\begin{equation}
    D_{\pi}(x, \vec{f}, h; q, \chi, M) = \sum_{d \geq 0} a(\vec{f}, \pi^d; q, \chi, M) \res{h}{\pi^d}_{\chi} x^{d \deg \pi}.
\end{equation}
Note that $D_{\pi}(x, \vec{f}, h) = D_{\pi}\left(\res{h}{\pi}_{\chi}^{1/\deg\pi} x, \vec{f}\right)$, so this latter expression is just a change of variables of the $D_{\pi}(x, \vec{f})$ defined earlier.

\begin{lemma} \label{LemmaDirichletFactorization}
For any prime $\pi$, we have the following equality
\begin{equation}
    D(x, \vec{f}) = \varepsilon_\pi D_{\pi}(x, \vec{f}_{\pi}, \tilde{F}^{(\pi)}) D^{(\pi)}(x, \vec{f}^{(\pi)}, \tilde{F}_{\pi})
\end{equation}
where $\varepsilon_\pi=\varepsilon_\pi(\vec{f}; M)$ is defined in equation \eqref{LocalRoot}.
\end{lemma}
\begin{proof}
By the twisted multiplicativity axiom, we have
\begin{equation*}
a(f_1, \ldots f_{r-1}, f_r) = \varepsilon_\pi a((f_1)_\pi,\ldots (f_{r-1})_\pi, (f_r)_\pi) \res{\tilde{F}^{(\pi)}}{(f_r)_{\pi}}_{\chi}  a(f_1^{(\pi)},\ldots f_{r-1}^{(\pi)}, f_r^{(\pi)}) \res{\tilde{F}_\pi}{f_r^{(\pi)}}_{\chi} .
\end{equation*}
Then the lemma follows by splitting the sum over $f_r$ in $D(x, \vec{f})$ into a sum over $(f_r)_\pi$ and a sum over $f_r^{(\pi)}$.
\end{proof}

For a given $r\times r$ matrix $M$, we will employ an augmented matrix $M_v$ with an additional zeroth row and column. The extra entries are given by $M_{00}=\cdots=M_{0\, (r-1)}=0$ and $M_{0r}=v \% n$.

\begin{prop} \label{LocalDirichletImpliesGlobal} Suppose $D(x, (\pi, \vec{f}^{(\pi)}); M_{v_{\pi} \tilde{F}})$ is a polynomial, or a rational function with denominator dividing $1-qx$ if $\tilde{F}$ is an $n$th power, and the functional equation \eqref{DirichletFE} holds for this series. Further, suppose $D_{\pi}(x, \vec{f}_{\pi}; M)$ is a polynomial, or a rational function with denominator dividing $1-x^{\deg \pi}$ if $\tilde{F}_\pi$ is an $n$th power, and the local functional equation \eqref{DirichletLocalFE} holds for this series. Then $D(x, \vec{f}; M)$ is a polynomial, or a rational function with denominator dividing $1-qx$ if $\tilde{F}$ is an $n$th power, and the functional equation \eqref{DirichletFE} holds for this series.
\end{prop}
\begin{proof}
By Lemma \ref{LemmaDirichletFactorization} we have 
\begin{equation*}
    D(x, \vec{f}; M) = \varepsilon_\pi(\vec{f}; M) D_{\pi}(x, \vec{f}_{\pi}, \tilde{F}^{(\pi)}; M) D^{(\pi)}(x, \vec{f}^{(\pi)}, \tilde{F}_{\pi}; M)
\end{equation*}
and also
\begin{equation*}
\begin{split}
    D(x, (\pi, \vec{f}^{(\pi)}); M_{v_\pi \tilde{F}}) &= D_{\pi}(x, (\pi, \vec{1}), \tilde{F}^{(\pi)}; M_{v_\pi \tilde{F}}) D^{(\pi)}(x, (1,\vec{f}^{(\pi)}), \tilde{F}_{\pi}; M_{v_\pi \tilde{F}}) \\
    &= \left(1-\res{\tilde{F}^{(\pi)}}{\pi}_{\chi} x^{\deg \pi}\right)^{-b_{\pi}}D^{(\pi)}(x, \vec{f}^{(\pi)}, \tilde{F}_{\pi}; M).
\end{split}
\end{equation*}
In the second line, $b_{\pi}$ is $1$ if $n$ divides $v_{\pi}\tilde{F}$ and zero otherwise. The local series at $\pi$ is evaluated as in equation \eqref{DirichletLocalEval}, while the series away from $\pi$ simplifies because the first parameter is $1$. 

Therefore, 
\begin{equation} \label{DirichletLocalGlobal1}
    D(x, \vec{f}; M)= \varepsilon_\pi(\vec{f}; M) \left(1-\res{\tilde{F}^{(\pi)}}{\pi}_{\chi} x^{\deg \pi}\right)^{b_{\pi}} D_{\pi}(x, \vec{f}_{\pi}, \tilde{F}^{(\pi)}; M)D(x, (\pi, \vec{f}^{(\pi)}); M_{v_\pi \tilde{F}}).
\end{equation}
Since $D_{\pi}(x, \vec{f}_{\pi}, \tilde{F}^{(\pi)}; M)=D_{\pi}\left(\res{\tilde{F}^{(\pi)}}{\pi}_\chi^{1/\deg \pi} x, \vec{f}_{\pi}; M\right)$, our hypothesis on the rationality of $D_{\pi}(x, \vec{f}_{\pi}; M)$ implies that $\left(1-\res{\tilde{F}^{(\pi)}}{\pi}_{\chi} x^{\deg \pi}\right)^{b_{\pi}} D_{\pi}(x, \vec{f}_{\pi}, \tilde{F}^{(\pi)}; M)$ is a polynomial. Thus our hypothesis on the rationality of $D(x, (\pi, \vec{f}^{(\pi)}); M_{v_\pi \tilde{F}})$ implies that $D(x, \vec{f}; M)$ is a polynomial, or a rational function with denominator dividing $1-qx$ if $\tilde{F}$ is an $n$th power. 

To prove the functional equation, we would like to compare equation \eqref{DirichletLocalGlobal1} to a similar expression for $D\left(\frac{1}{qx}, \vec{f}; M'\right)$. By Lemma \ref{LemmaDirichletFactorization}, we have  
\begin{equation*}
    D\left(\frac{1}{qx}, \vec{f}; M'\right) = \varepsilon_\pi(\vec{f}; M') D_{\pi}\left(\frac{1}{qx}, \vec{f}_{\pi}, (\tilde{F}^{(\pi)})^{-1}; M'\right) D^{(\pi)}\left(\frac{1}{qx}, \vec{f}^{(\pi)}, \tilde{F}_{\pi}^{-1}; M'\right).
\end{equation*} 

The matrix $(M_{v_\pi \tilde{F}})'$ is $M'$ augmented with a zeroth row and column. If $v_\pi \tilde{F} \equiv 0 \bmod n$, the extra entries are all $0$. If $v_\pi \tilde{F} \not\equiv 0 \bmod n$, then the extra entries are $((M_{v_\pi \tilde{F}})')_{00}=v_\pi \tilde{F} + n/2$, $((M_{v_\pi \tilde{F}})')_{0i}=0$ if $M_{ir}=0$ and $((M_{v_\pi \tilde{F}})')_{0i}=v_\pi \tilde{F}+M_{ir}$ otherwise, and $((M_{v_\pi(\tilde{F}})')_{0r}=-v_\pi \tilde{F}$. By another application of Lemma \ref{LemmaDirichletFactorization}, we have
\begin{equation*}
\begin{split}
    &D\left(\frac{1}{qx}, (\pi, \vec{f}^{(\pi)}); (M_{v_\pi \tilde{F}})'\right) \\
    &= \varepsilon_\pi((\pi, \vec{f}^{(\pi)}); (M_{v_\pi\tilde{F}})') D_{\pi}\left(\frac{1}{qx}, (\pi, \vec{1}), (\tilde{F}^{(\pi)})^{-1}; (M_{-v_\pi \tilde{F}})'\right)
    D^{(\pi)}\left(\frac{1}{qx}, (1,\vec{f}^{(\pi)}), \tilde{F}_{\pi}^{-1}; (M_{-v_\pi \tilde{F}})'\right) \\
    &= \varepsilon_\pi((\pi, \vec{f}^{(\pi)}); (M_{v_\pi\tilde{F}})') \res{\pi'}{\pi}_\chi^{(v_\pi \tilde{F} +n/2)(1-b_\pi)}
    \left(1-\res{\tilde{F}^{(\pi)}}{\pi}_{\chi}^{-1} (qx)^{-\deg \pi}\right)^{-b_{\pi}}D^{(\pi)}\left(\frac{1}{qx}, \vec{f}^{(\pi)}, \tilde{F}_{\pi}^{-1}; M'\right).
\end{split}
\end{equation*}
Again, the local series at $\pi$ is evaluated as in equation \eqref{DirichletLocalEval}, while the series away from $\pi$ simplifies because the first parameter is $1$. 

Therefore, 
\begin{equation} \label{DirichletLocalGlobal2}
    \begin{split}
    D\left(\frac{1}{qx}, \vec{f}; M'\right) = &\frac{\varepsilon_\pi(\vec{f}; M')}{\varepsilon_\pi((\pi, \vec{f}^{(\pi)}); (M_{v_\pi\tilde{F}})')} \res{\pi'}{\pi}_\chi^{(v_\pi \tilde{F} +n/2)(b_\pi-1)} \left(1-\res{\tilde{F}^{(\pi)}}{\pi}_{\chi}^{-1} (qx)^{-\deg \pi}\right)^{b_{\pi}} \\
    & D_{\pi}\left(\frac{1}{qx}, \vec{f}_{\pi}, (\tilde{F}^{(\pi)})^{-1}; M'\right) D\left(\frac{1}{qx}, \left(\pi, \vec{f}^{(\pi)}\right); (M_{v_\pi \tilde{F}})'\right).
    \end{split}
\end{equation}

The known functional equation of $D_{\pi}(x, \vec{f}_{\pi}, \tilde{F}^{(\pi)}; M)$ is as follows: 
\begin{equation*}
\begin{split}
&\frac{D_{\pi}(x, \vec{f}_{\pi}, \tilde{F}^{(\pi)}; M)}{D_{\pi}\left(\frac{1}{qx}, \vec{f}_{\pi}, (\tilde{F}^{(\pi)})^{-1}; M'\right) } = \prod_{\substack{1 \leq i \leq r-1 \\ M_{ir} \neq 0}} \res{\pi'}{\pi}_{\chi}^{-v_\pi f_i(M_{ir}+n/2)} \omega_\pi(\vec{f}_\pi; M) \left(\frac{-(-g_{\chi^{v_{\pi} \tilde{F}}})^{\deg \pi}}{q^{\deg\pi}}\right)^{b_{\pi}-1}  \\
& \res{\tilde{F}^{(\pi)}}{\pi}_\chi^{v_\pi F-1} (\chi(-1)^{v_\pi\tilde{F}}qx)^{\deg F_\pi-\deg \pi}  \prod_{\substack{1 \leq i \leq r-1 \\ M_{ir} \neq 0}} \left(\frac{-(-g_{\chi^{M_{ir}}})^{\deg \pi}}{q^{\deg \pi}}\right)^{v_\pi f_i} \left( \frac{\res{\tilde{F}^{(\pi)}}{\pi}_\chi(qx)^{\deg \pi}-1}{1-\res{\tilde{F}^{(\pi)}}{\pi}_\chi x^{\deg \pi}}\right)^{b_{\pi}}.
\end{split}
\end{equation*}

The known functional equation of $D(x, (\pi, \vec{f}^{(\pi)}); M_{v_\pi \tilde{F}})$ is as follows:
\begin{equation*}
\begin{split}
\frac{D(x, (\pi, \vec{f}^{(\pi)}); M_{v_\pi \tilde{F}})}{D\left( \frac{1}{qx}, \left(\pi, \vec{f}^{(\pi)}\right); (M_{v_\pi \tilde{F}})'\right)} = & \omega((\pi, \vec{f}^{(\pi)}); M_{v_\pi \tilde{F}}) \, g_{\chi^{\deg \tilde{F}}}^{b-1} (\chi(-1)^{\deg \tilde{F}} x)^{\deg F^{(\pi)}+(1-b_\pi)\deg \pi -1} \\
&g_{\chi^{v_\pi \tilde{F}}}^{\deg \pi(1-b_\pi)} \prod_{\substack{1 \leq i \leq r-1 \\ M_{ir} \neq 0}} g_{\chi^{M_{ir}}}^{\deg f_i^{(\pi)}} \, \left( \frac{x-1}{1- qx}\right)^b
\end{split}
\end{equation*}
where the values of $\deg \tilde{F}$ mod $n$, $b$ and $b_\pi$ are the same for the pair $(\pi, \vec{f}^{(\pi)})$, $M_{v_\pi \tilde{F}}$ as they are for $\vec{f}$, $M$. 

Dividing Equation \eqref{DirichletLocalGlobal1} by \eqref{DirichletLocalGlobal2}, applying the local and global functional equations, and algebraically simplifying, we obtain:

\begin{equation*}
\begin{split}
\frac{D(x, \vec{f}; M)} {D\left(\frac{1}{qx}, \vec{f}; M'\right)} &=\frac{\varepsilon_\pi(\vec{f}; M)\varepsilon_\pi((\pi, \vec{f}^{(\pi)}); (M_{v_\pi\tilde{F}})')}{\varepsilon_\pi(\vec{f}; M')} 
\res{\tilde{F}^{(\pi)}}{F_\pi}_\chi \res{\tilde{F}^{(\pi)}}{\pi}_\chi^{b_\pi-1} \\
& \res{\pi'}{\pi}_\chi^{(n/2)(1-b_\pi-v_\pi F)} (-1)^{(\deg \pi+1)(v_\pi F+ b_{\pi}-1)} \omega_\pi(\vec{f}_\pi; M) \omega((\pi, \vec{f}^{(\pi)}); M_{v_\pi \tilde{F}})  \\
& \chi(-1)^{(\deg \tilde{F})(\deg F^{(\pi)}+(1-b_\pi)\deg \pi -1)+(\deg \tilde{F}_\pi)(v_\pi F-1)} \\
&x^{\deg F-1} g_{\chi^{\deg \tilde{F}}} ^{b-1} \prod_{\substack{1 \leq i \leq r-1 \\ M_{ir} \neq 0}} g_{\chi^{M_{ir}}} ^{\deg f_i} \, \left( \frac{x-1}{1-qx}\right)^b.
\end{split}
\end{equation*}

Using $\deg \pi \equiv (\deg \pi)^2 \bmod 2$, the total power of $\chi(-1)$ in the third line is equivalent to 
\begin{equation*}
(\deg \tilde{F})(\deg F -1)+(\deg \tilde{F}^{(\pi)})((1-b_\pi)\deg \pi-\deg F_\pi).
\end{equation*}

Using Lemma \ref{Pellet} to evaluate $\res{\pi'}{\pi}_\chi^{n/2}$, we have 
\begin{equation*}
\res{\pi'}{\pi}_\chi^{(n/2)(1-b_\pi-v_\pi F)} (-1)^{(\deg \pi+1)(b_\pi-1+v_\pi F)} = \xi(-1)^{(b_\pi-1+v_\pi F)\deg \pi(\deg \pi -1)/2}.
\end{equation*}

By comparing the entries of $M'$ and $M$, we find that
\begin{equation*}
\begin{split}
\frac{\varepsilon_\pi(\vec{f}; M)}{\varepsilon_\pi(\vec{f}; M')} &= \prod_{\substack{1 \leq i \leq r-1 \\ M_{ir} \neq 0}} \res{f_i^{(\pi)}}{(f_i)_\pi}_{\chi}^{-M_{ir}-n/2} \res{(f_i)_\pi}{f_i^{(\pi)}}_{\chi}^{-M_{ir}-n/2} \prod_{\substack{1 \leq i<j \leq r-1 \\ M_{ir}, M_{jr} \neq 0}} \res{f_i^{(\pi)}}{(f_j)_\pi}_{\chi}^{-M_{ir}-M_{jr}} \res{(f_i)_\pi}{f_j^{(\pi)}}_{\chi}^{-M_{ir}-M_{jr}} \\
&=  \res{F^{(\pi)}}{\tilde{F}_\pi}_\chi^{-1} \res{F_\pi}{ \tilde{F}^{(\pi)}}_\chi^{-1} \prod_{\substack{1 \leq i \leq r-1 \\ M_{ir} \neq 0}} \xi(-1)^{\deg (f_i)_\pi \, \deg f_i^{(\pi)}} \\
&\phantom{=} \prod_{\substack{1 \leq i<j \leq r-1 \\ M_{ir}, M_{jr} \neq 0}} \chi(-1)^{M_{ir}(\deg f_i^{(\pi)} \,\deg (f_j)_\pi + \deg (f_i)_{\pi} \,\deg f_j^{(\pi)})}.
\end{split}
\end{equation*}
The last equality is obtained using the reciprocity law and combining the product of residue symbols. Also, from the entries of $(M_{v_\pi\tilde{F}})'$, we find that 
\begin{equation*}
\varepsilon_\pi((\pi, \vec{f}^{(\pi)}); (M_{v_\pi\tilde{F}})')=\prod_{\substack{1 \leq i \leq r-1 \\ M_{ir} \neq 0}} \res{\pi}{f_i^{(\pi)}}_\chi^{(v_\pi \tilde{F} + M_{ir})(1-b_\pi)} = \res{\tilde{F}_\pi}{F^{(\pi)}}_\chi \res{\pi}{\tilde{F}^{(\pi)}}_\chi^{1-b_\pi}.
\end{equation*}

We conclude that 
\begin{equation*}
\begin{split}
\frac{D(x, \vec{f}; M)} {D\left(\frac{1}{qx}, \vec{f}; M'\right)} &=\pm g_{\chi^{\deg \tilde{F}}} ^{b-1} \, (\chi(-1)^{\deg \tilde{F}} x)^{\deg F-1} \prod_{\substack{1 \leq i \leq r-1 \\ M_{ir} \neq 0}} g_{\chi^{M_{ir}}} ^{\deg f_i} \, \left( \frac{x-1}{1-qx}\right)^b
\end{split}
\end{equation*}
where the $\pm$ symbol is given explicitly by

\begin{equation} \label{TotalSign}
\begin{split}
& \prod_{\substack{1 \leq i \leq r-1 \\ M_{ir} \neq 0}} \xi(-1)^{\deg (f_i)_\pi \, \deg f_i^{(\pi)}} \prod_{\substack{1 \leq i<j \leq r-1 \\ M_{ir}, M_{jr} \neq 0}} \chi(-1)^{M_{ir}(\deg f_i^{(\pi)} \,\deg (f_j)_\pi + \deg (f_i)_{\pi} \,\deg f_j^{(\pi)})}  \\
&\xi(-1)^{(b_\pi-1+v_\pi F)\deg \pi(\deg \pi -1)/2} \chi(-1)^{(\deg F^{(\pi)})(\deg \tilde{F}_\pi)}\\
&\prod_{\substack{1 \leq i \leq r-1 \\ M_{ir} \neq 0}} \xi(-1)^{\deg \pi (v_\pi f_i)(v_\pi f_i-1)/2} \prod_{\substack{1 \leq i<j \leq r-1 \\ M_{ir}, M_{jr} \neq 0}} \chi(-1)^{\deg \pi(v_\pi f_i)(v_\pi f_j) M_{ir} } \\
& \xi(-1)^{(\deg \pi)(\deg \pi -1)(1-b_\pi)/2} \prod_{\substack{1 \leq i \leq r-1 \\ M_{ir} \neq 0}} \chi(-1)^{\deg \pi(\deg f_i^{(\pi)})v_\pi \tilde{F}} \\
&\prod_{\substack{1 \leq i \leq r-1 \\ M_{ir} \neq 0}} \xi(-1)^{\deg f_i^{(\pi)}(\deg f_i^{(\pi)}-1)/2} \prod_{\substack{1 \leq i<j \leq r-1 \\ M_{ir}, M_{jr} \neq 0}} \chi(-1)^{(\deg f_i^{(\pi)})(\deg f_j^{(\pi)}) M_{ir} }.
\end{split}
\end{equation}

The powers of $\chi(-1)$ in the second and fourth lines cancel out, and again using $\deg \pi \equiv (\deg \pi)^2 \bmod 2$, the total power of $\chi(-1)$ in \eqref{TotalSign} is 
\begin{equation*}
\begin{split}
    &\sum_{\substack{1 \leq i<j \leq r-1 \\ M_{ir}, M_{jr} \neq 0}} M_{ir}(\deg (f_i)_\pi \, \deg (f_j)_\pi+ \deg (f_i)_\pi \, \deg f_j^{(\pi)} + \deg f_i^{(\pi)} \, \deg (f_j)_\pi + \deg f_i^{(\pi)}) \, \deg f_j^{(\pi)} \\
    &=\sum_{\substack{1 \leq i<j \leq r-1 \\ M_{ir}, M_{jr} \neq 0}} M_{ir}\deg f_i \, \deg f_j.
\end{split}
\end{equation*}

The power of $\xi(-1)$ in the fourth line cancels part of the power in the second line, and again using $\deg \pi \equiv (\deg \pi)^2 \bmod 2$, the total power of $\xi(-1)$ in \eqref{TotalSign} is 
\begin{equation*}
\begin{split}
& \frac12 \sum_{\substack{1 \leq i \leq r-1 \\ M_{ir} \neq 0}} \deg (f_i)_\pi (\deg (f_i)_\pi-\deg \pi)+\deg (f_i)_\pi(\deg \pi -1) + 2 \deg (f_i)_\pi \deg f_i^{(\pi)}+\deg f_i^{(\pi)}(\deg f_i^{(\pi)}-1) \\
& =\sum_{\substack{1 \leq i \leq r-1 \\ M_{ir} \neq 0}} \deg f_i(\deg f_i-1)/2.
\end{split}
\end{equation*}

Thus we find that the expression \eqref{TotalSign} matches $\omega(\vec{f}; M)$, completing the proof.
\end{proof}

We can now prove the functional equations of $D(x, \vec{f}; M)$  and $D_\pi(x, \vec{f}; M)$ for arbitrary $\vec{f}$.

\begin{theorem} \label{TheoremDirichletFE}
Suppose $M_{rr}=0$. For any $\vec{f}=(f_1, \ldots f_{r-1})$ the series $D(x, \vec{f}; M)$ is a rational function of $x$ with denominator dividing $1-qx$ if $\tilde{F}$ is an $n$th power in $\F_q[T]$ and a polynomial otherwise, and satisfies the global functional equation \eqref{DirichletFE}. For any $\pi$ prime and $\vec{f}=(\pi^{d_1}, \ldots \pi^{d_{r-1}})$, the series $D_\pi(x, \vec{f}; M)$ is a rational function of $x$ with denominator dividing $1-x^{\deg \pi}$ if $\tilde{F}$ is an $n$th power in $F_q[T]$ and a polynomial otherwise, and satisfies the local functional equation \eqref{DirichletLocalFE}. 
\end{theorem}
\begin{proof}
The proof is by induction on $\deg f_1+\cdots + \deg f_{r-1}$ and $v_{\pi} f_1 + \cdots + v_\pi f_{r-1}$. Note that the matrix $M$ is not fixed in the argument, though we always assume $M_{rr}=0$. At each step in the induction, the global functional equation and rationality statement are proven for all matrices $M$ and all $\vec{f}$ with $\deg f_1 + \cdots + \deg f_r$ fixed, and the local functional equation and rationality statement are proven for all matrices $M$, all $\pi$, and all $\vec{f}$ supported on $\pi$ with $v_{\pi} f_1 + \cdots + v_\pi f_{r-1}$ fixed. 

If $\deg f_1+\cdots + \deg f_{r-1}=0$ or $1$ then the global functional equation and rationality statement hold by Proposition \ref{PropDirichletFE}. If $v_{\pi} f_1 + \cdots + v_\pi f_{r-1}=0$ or $1$ then the local functional equation and rationality statement hold by Proposition \ref{PropDirichletLocalFE}. 

For the inductive step: suppose that the global functional equation and rationality statement are known for all $M$ and all $f$ with $\deg f_1+\cdots + \deg f_{r-1}<A$, and the local functional equation and rationality statement are known for all $\vec{f}$ supported on $\pi$ with $v_{\pi} f_1 + \cdots + v_\pi f_{r-1}<A$. If $\vec{f}=(f_1, \ldots f_{r-1})$ has $\deg f_1+\cdots + \deg f_{r-1}=A$ and $f_1\cdots f_{r-1}$ is squarefree, then the global functional equation and rationality statement are known for $\vec{f}$ by Proposition \ref{PropDirichletFE}. If $\pi^2 | f_1\cdots f_{r-1}$ for some prime $\pi$, we apply Lemma \ref{LocalDirichletImpliesGlobal}. The vector $(\pi, \vec{f}^{(\pi)})$ has $\deg \pi + \deg f_1^{(\pi)} + \cdots + f_{r-1}^{(\pi)}< A$, so by induction, the global functional equation and rationality statement hold for the series $D(x, (\pi, \vec{f}^{(\pi)}); M_{v_\pi{\tilde{F}}})$ with extended matrix $M_{v_\pi{\tilde{F}}}$. Moreover, unless $\pi$ is linear and each $f_i$ is a power of $\pi$, we have $v_\pi f_1+\cdots + v_\pi f_{r-1} < A$, so the local functional equation and rationality statement hold for the series $D_\pi(x, \vec{f}_\pi; M)$. Therefore by Lemma \ref{LocalDirichletImpliesGlobal}, the global functional equation and rationality statement hold for $\vec{f}$. This establishes the global functional equation and rationality statement for all $\vec{f}$ except when the $f_i$ are all powers of the same linear polynomial. 

We now consider the sum 
\begin{equation}
\sum_{\deg f_1 = d_1, \ldots \deg f_{r-1}=d_{r-1}} D(x, \vec{f}) - \sum_{\deg \pi =1} D(x, (\pi^{d_1}, \ldots \pi^{d_{r-1}})).
\end{equation}
Set $v=d_1 M_{1r} + \cdots d_{r-1} M_{(r-1) \, r}$. We know that this sum is a rational function with denominator dividing $1-qx$ if $n|v$, and a polynomial otherwise. Because the global functional equation depends only on $d_i$, not $f_i$ themselves, this sum satisfies the global functional equation \eqref{DirichletFE}. Note that, by Axiom \ref{Axiom3}, the second sum is $qD(x, (\pi^{d_1}, \ldots \pi^{d_{r-1}}))$ for an arbitrary monic linear polynomial $\pi$. 

We use Lemma \ref{LemmaDirichletFactorization} to write the second sum as 
\begin{equation*}
q D(x, (\pi^{d_1}, \ldots \pi^{d_{r-1}})) = q D_\pi(x, (\pi^{d_1}, \ldots \pi^{d_{r-1}}))D^{(\pi)}(x, \vec{1}, \pi^v).
\end{equation*}
If $n \nmid v$, then $D^{(\pi)}(x, \vec{1}, \pi^v)$ is the Dirichlet $L$-function $L(x, \res{\pi^v}{*}_\chi)$. This $L$-function with conductor $\pi$ of degree $1$ is simply equal to $1$. If $n | v$, then $D^{(\pi)}(x, \vec{1}, \pi^v)$ is the function field zeta function with the Euler factor at $\pi$ removed. This evaluates to $\frac{1-x}{1-qx}$. 

If we set $b=1$ if $n \mid v$ and $b=0$ otherwise, then the functional equation \eqref{DirichletFE}, multiplied by $(1-qx)^b q^{\frac{-A-1}{2}}$, gives the following:
\begin{equation} \label{LongDirichletFE}
\begin{split}
&(1-qx)^b q^{\frac{-A-1}{2}}\sum_{\deg f_1 = d_1, \ldots \deg f_{r-1}=d_{r-1}} D(x, \vec{f}; M) \\&
-(1-x)^b q^{\frac{-A+1}{2}} D_\pi(x, (\pi^{d_1}, \ldots \pi^{d_{r-1}}); M) \\
&=\frac{\omega g_{\chi^{v}} ^{b-1}}{\chi(-1)^v x} \prod_{\substack{1 \leq i \leq r-1 \\ M_{ir} \neq 0}} (\chi(-1)^v x g_{\chi^{M_{ir}}} )^{d_i} \left(x-1\right)^b q^{\frac{-A-1}{2}}\sum_{\deg f_1 = d_1, \ldots \deg f_{r-1}=d_{r-1}} D\left(\frac{1}{qx}, \vec{f}; M'\right) \\
&\phantom{=}-\frac{\omega g_{\chi^{v}} ^{b-1}}{\chi(-1)^v x} \prod_{\substack{1 \leq i \leq r-1 \\ M_{ir} \neq 0}} (\chi(-1)^v x g_{\chi^{M_{ir}}} )^{d_i} \left(x-\frac{1}{q}\right)^b q^{\frac{-A+1}{2}} D_\pi\left(\frac{1}{qx}, (\pi^{d_1}, \ldots \pi^{d_{r-1}}); M'\right).
\end{split}
\end{equation}
Moreover, this is an equality of polynomials in $x$. Note that the same logic of the previous paragraph also applies to $q D(x, (\pi^{d_1}, \ldots \pi^{d_{r-1}}); M')$, and the value of $v$ for this series is congruent to the negative of the previous $v$ value modulo $n$. 

Recall that we call a power series in $x$ ``sharp'' if each coefficient of $x^j$ is a linear combination of $q$-Weil numbers of weights greater than $j$ and ``flat'' if each coefficient of $x^j$ is a linear combination of $q$-Weil numbers of weights less than $j$. By Axioms \ref{Axiom4} and \ref{Axiom5}, the terms in the first line of \eqref{LongDirichletFE} are sharp, and the terms in the second line are flat. Since the first and second lines sum to a polynomial in $x$ and there is no cancellation between them, we conclude that each one is a polynomial in $x$. Again by Axioms \ref{Axiom4} and \ref{Axiom5}, the terms in the third line of \eqref{LongDirichletFE} are sharp, and the terms in the fourth line are flat. Thus we conclude that the first line equals the third line, and the second line equals the fourth line. 

Because the second line of \eqref{LongDirichletFE} is a polynomial, $(1-qx)^b D(x, (\pi^{d_1}, \ldots \pi^{d_r}); M)$ is a polynomial, which verifies the global rationality statement when $\vec{f}=(\pi^{d_1}, \ldots \pi^{d_{r-1}})$ and $\deg \pi = 1$, the last remaining case with $\deg f_i=d_i$. The equality of the second and fourth lines is equivalent to 
\begin{equation*} 
\begin{split}
& D(x, (\pi^{d_1}, \ldots \pi^{d_{r-1}}); M) =\frac{\omega g_{\chi^{v}} ^{b-1}}{\chi(-1)^v x} \prod_{\substack{1 \leq i \leq r-1 \\ M_{ir} \neq 0}} (\chi(-1)^v x g_{\chi^{M_{ir}}} )^{d_i} \left(\frac{x-1}{1- qx}\right)^b D\left(\frac{1}{qx}, (\pi^{d_1}, \ldots \pi^{d_{r-1}}); M'\right)
\end{split}
\end{equation*}
which verifies the global functional equation \eqref{DirichletFE} in this last remaining case. 

The fact that the second line of \eqref{LongDirichletFE} is a polynomial also verifies the local rationality statement for $\deg \pi=1$, $\vec{f}=(\pi^{d_1}, \ldots \pi^{d_{r-1}})$. The equality of the second and fourth lines can be written with the local series as 
\begin{equation*} 
\begin{split}
&D_\pi(x, (\pi^{d_1}, \ldots \pi^{d_{r-1}}); M) \\
&=\omega_\pi \frac{(g_{\chi^{v}}  q^{-1})^{b-1}}{\chi(-1)^v qx} \prod_{\substack{1 \leq i \leq r-1 \\ M_{ir} \neq 0}} (\chi(-1)^v x g_{\chi^{M_{ir}}} )^{d_i} \left(\frac{qx-1}{1- x}\right)^b D_\pi\left(\frac{1}{qx}, (\pi^{d_1}, \ldots \pi^{d_{r-1}}); M'\right)
\end{split}
\end{equation*}
where
\begin{equation*}
\omega_\pi= \prod_{\substack{1 \leq i \leq r-1 \\ M_{ir} \neq 0}} \xi(-1)^{d_i(d_i-1)/2} \prod_{\substack{1 \leq i<j \leq r-1 \\ M_{ir}, M_{jr} \neq 0}} \chi(-1)^{d_i d_j M_{ir} } = \pm 1.
\end{equation*}
This is the local functional equation when $\vec{f}=(\pi^{d_1}, \ldots \pi^{d_{r-1}})$ and $\deg \pi = 1$. 

To extend to $\pi$ of higher degree, we apply Axiom \ref{Axiom3}. The effect replacing $\pi$ of degree $1$ with $\pi$ of arbitrary degree is to replace $x$ by $x^{\deg \pi}$, take all $q$-Weil numbers to the power of $\deg \pi$ (with the Gauss sum $g_{\chi} $ understood as $-1$ times the Weil number $-g_{\chi} $), divide the left by $\res{\pi'}{\pi}_\chi^{\sum d_i M_{ii}}$ and the right by $\res{\pi'}{\pi}_\chi^{\sum d_i M_{ii}'}$. Note that all these operations preserve rationality, so we verify the local rationality statement for arbitrary $\pi$. After performing these operations on the functional equation and simplifying, we obtain
\begin{equation*} 
\begin{split}
&D_\pi(x, (\pi^{d_1}, \ldots \pi^{d_{r-1}}); M) =\prod_{\substack{1 \leq i \leq r-1 \\ M_{ir} \neq 0}} \res{\pi'}{\pi}_{\chi}^{-d_i (M_{ir}+n/2)} \omega_\pi \left(\frac{-(-g_{\chi^{v}} )^{\deg \pi}}{q^{\deg \pi}}\right)^{b_\pi-1} (\chi(-1)^v q x)^{-\deg \pi}\\ & \prod_{\substack{1 \leq i \leq r-1 \\ M_{ir} \neq 0}} \left((\chi(-1)^{v} q x)^{\deg \pi} \left(\frac{-( g_{\chi^{M_{ir}}} )^{\deg \pi}}{q^{\deg \pi}}\right)\right)^{d_i} \left(\frac{(qx)^{\deg \pi}-1}{1-x^{\deg \pi}}\right)^{b_\pi} D_\pi\left(\frac{1}{qx}, (\pi^{d_1}, \ldots \pi^{d_{r-1}}); M'\right)
\end{split}
\end{equation*}
with notation as in \eqref{DirichletLocalFE}. This is the desired local functional equation for arbitrary $\pi$.
\end{proof}

\subsection{Multivariable Functional Equation}\label{ss-multivariable-functional-equation}

As with the Kubota functional equations, we deduce functional equations for the multivariable series based on the single-variable functional equations. We again set $\vec{x}=(x_1, \ldots x_r) \in \C^r$, $\vec{k}=(k_1, \ldots k_r) \in \Z^r$ and let 
\begin{align}
&Z(\vec{x}, \vec{k}; q, \chi, M) = \sum_{\substack{f_1, \ldots f_r \in \F_q[T]^+ \\ \deg f_i \equiv k_i \bmod n}} a(f_1,\ldots f_r; q, \chi, M) x_1^{\deg f_1} \cdots x_r^{\deg f_r} \\
&Z_{\pi}(\vec{x}, \vec{k}; q, \chi, M) = \sum_{\substack{d_1, \ldots d_r \geq 0 \\ d_i \equiv k_i \bmod n}} a(\pi^{d_1}, \ldots \pi^{d_r}; q, \chi, M) x_1^{d_1\deg \pi} \cdots x_r^{d_r \deg \pi}.
\end{align}
We now introduce notation relevant for the multivariable functional equations. Fix $i \in \{1, \ldots r\}$, and assume that $M_{ii}=0$. Set $e_{ij}=1$ if $M_{ij} \neq 0$, $e_{ij}=0$ if $M_{ij} = 0$. For fixed $k_1, \ldots k_{i-1}, k_{i+1}, \ldots k_r$, let $v=k_1 M_{1i} + \cdots +k_{r} M_{ri}$, $v_i= k_{i+1} M_{i \, (i+1)}+ \cdots + k_r M_{ir}$, $K=k_1 e_{1i} + \cdots + k_{r} e_{ri}$, $b=1$ if $n|v$, $b=0$ otherwise, and 
\begin{equation*}
    \omega = \prod_{\substack{1 \leq j \leq r \\ M_{ji} \neq 0}} \xi(-1)^{k_j(k_j-1)/2} \prod_{\substack{1 \leq h<j \leq r \\ M_{hi}, M_{ji} \neq 0}} \chi(-1)^{k_h k_j M_{hi} }.
\end{equation*}
Define a transformation $\sigma_{i; q, \chi, M}:\C^r \to \C^r$ as follows:
\begin{equation} \label{sigma-i-dirichlet}
(\sigma_{i; q, \chi, M}(\vec{x}))_j=\left\lbrace \begin{array}{cc}
\frac{1}{qx_i} & j=i \\ (g_{\chi^{M_{ij}}}x_i)^{e_{ij}}x_j & j \neq i \end{array} \right. .
\end{equation}
Define a transformation $\tau_i$ on $r \times r$ symmetric integer matrices as follows: 
\begin{equation}
\begin{split}
&(\tau_i(M))_{ii}=M_{ii}=0 \\
&(\tau_i(M))_{ij}=-M_{ij} \text{ for all } j\neq i \\
&(\tau_i(M))_{jj} = M_{jj}+e_{ji}(M_{ji}+n/2) \text{ for all } j\neq i \\
&(\tau_i(M))_{hj} = M_{hj}+e_{hi}e_{ij}(M_{hi}+M_{ij}) \text{ for all } h\neq j \neq i .
\end{split}
\end{equation}

Recall that $S^{k,n}$ is defined as the operation which transforms a power series $D(x)$ into the sum of all terms in $D(x)$ whose powers of $x$ are congruent to $k$ mod $n$. Properties of this operation are listed before Theorem \ref{TheoremMultiFE}.

We can now prove the multivariable global functional equation.
\begin{theorem}\label{dirichlet-multivariable-global}
The vector $(Z(\vec{x}, (k_1, \ldots k_r); M)_{k_i}$ satisfies a functional equation
\begin{equation}
\begin{split}
Z(\vec{x}, (k_1, \ldots k_r); M)_{k_i} = &\omega g_{\chi^v}^{b-1}\left(\chi^{(v+v_i)(k_i+\ell_i)}(-1)\Theta_{k_i, \ell_i}(x_i, K, v) \right)_{k_i, \ell_i} \\
&Z(\sigma_i(\vec{x}), (k_1, \ldots k_{i-1}, \ell_i, k_{i+1}, \ldots k_r); \tau_i(M))_{\ell_i}
\end{split}
\end{equation}
where the $n\times n$ scattering matrix $(\Theta_{k, \ell}(x, K, v))_{k, \ell}$ has entries as described below. For $n \nmid v$, we have $\Theta_{k, \ell}(x, K, v)= x^{-1}$ if $k+\ell \equiv K-1 \bmod n$, and $\Theta_{k, \ell}(x, K, v)=0$ otherwise. 
For $n\mid v$, we have 
\begin{equation*}
\Theta_{k,\ell}(x, K, v)= x^{-1} \left(\frac{q^{n-1}x^n-1}{1-q^nx^n}\right)
\end{equation*}
if $k+\ell \equiv K-1 \bmod n$ and  
\begin{equation*}
\Theta_{k,\ell}(x, K, v)= x^{-1} \left(\frac{(q^{-1}-1)(qx)^{\left(k+\ell-K+1\right) \%n}}{1-q^nx^n}\right)
\end{equation*}
otherwise. Here $\omega$, $b$, $v$, $v_i$, and $K$ are defined as above; they depend on $M$, $i$ and $k_j$ for $j \neq i$. 
\end{theorem}
This functional equation is equivalent to \cite[Theorem 6]{HaseLiu24}, though our method of proof is analytic rather than geometric.

\begin{proof}
First we prove the result when $i=r$. We have
\begin{equation*}
\begin{split}
Z(\vec{x}, (k_1, \ldots k_r); M) =\sum_{\substack{f_1, \ldots f_{r-1} \in \F_q[T]^+ \\ \deg f_i \equiv k_i \bmod n}} x_1^{\deg f_1}\cdots x_{r-1}^{\deg f_{r-1}} S^{k_r, n} D(x_r, (f_1, \ldots f_{r-1}); M) .
\end{split}
\end{equation*}
We can then apply the global functional equation of Theorem \ref{TheoremDirichletFE} and rearrange to obtain
\begin{equation*}
\begin{split}
Z(\vec{x}, (k_1, \ldots k_r); M) =& \omega \, g_{\chi^v}^{b-1} \sum_{f_1, \ldots f_{r-1}} \prod_{i=1}^{r-1} (g_{\chi^{M_{ir}}}^{e_{ir}} x_r^{e_{ir}}x_i)^{\deg f_i} x_r^{-1} \chi^{v(K-1)}(-1) \\
&\cdot S^{k_r-K+1, n}  \left( \frac{x_r-1}{1- q x_r}\right)^b D \left(\frac{1}{q x_r}, (f_1, \ldots f_{r-1}); \tau_r(M) \right)  .
\end{split}
\end{equation*}
Suppose that $n \nmid v$, so $b=0$. We find that  
\begin{equation*}
Z(\vec{x}, (k_1, \ldots k_r); M)=\omega \, g_{\chi^v}^{-1} \chi^{v(K-1)}(-1) x_r^{-1} Z(\sigma_r(\vec{x_r}), (k_1, \ldots k_{r-1}, K-1-k_r); \tau_r(M)) 
\end{equation*}
which is the desired result since $\Theta_{k, \ell}(x_r, K, v)= x_r^{-1}$ if $k+\ell \equiv K-1 \bmod n$, and $\Theta_{k, \ell}(x_r, K, v)=0$ otherwise.

Now suppose that $n|v$, so $b=1$. This expression for $Z(\vec{x}, (k_1, \ldots k_r); M)$ becomes
\begin{equation*}
\begin{split}
&\omega \sum_{f_1, \ldots f_{r-1}} \prod_{i=1}^{r-1} (g_{\chi^{M_{ir}}}^{e_{ir}} x_r^{e_{ir}}x_i)^{\deg f_i} x_r^{-1} \sum_{\ell_r=0}^{n-1} S^{k_r+\ell_r -K+1, n}  \left( \frac{x_r-1}{1- q x_r}\right) S^{-\ell_r} D \left(\frac{1}{q x_r}, (f_1, \ldots f_{r-1}); \tau_r(M) \right) \\
&=\omega x_r^{-1} \sum_{\ell_r=0}^{n-1} S^{k_r+\ell_r -K+1, n}  \left( \frac{x_r-1}{1- q x_r}\right) Z(\sigma_r(\vec{x_r}), (k_1, \ldots k_{r-1}, \ell_r); \tau_r(M)) 
\end{split}
\end{equation*}
which is the desired result since $\Theta_{k, \ell}(x_r, K, v)= x_r^{-1} S^{k_r+\ell_r -K+1, n}  \left( \frac{x_r-1}{1- q x_r}\right)$.

This proves the theorem for $i=r$. The extension of the proof to arbitrary $i$ is exactly the same as in the proof of Theorem \ref{TheoremMultiFE}; reordering the variables simply introduces an additional $\chi^{v_i(k_i+\ell_i)}(-1)$ into the entries of the scattering matrix.
\end{proof}

We also have the analogous local functional equation.

\begin{theorem} \label{dirichlet-multivariable-local}
The vector $(Z_\pi(\vec{x}, (k_1, \ldots k_r); M)_{k_i}$ satisfies a functional equation
\begin{equation}
\begin{split}
Z_\pi(\vec{x}, (k_1, \ldots k_r); M)_{k_i} =& \omega^{\deg \pi} \left(\frac{-(-g_{\chi^v})^{\deg \pi}}{q^{\deg \pi}}\right)^{b-1} (-1)^{(\deg \pi+1)K} \res{\pi'}{\pi}_\chi^{-v-Kn/2} \\
&(\chi^{(v+v_i)(k_i+\ell_i)\deg \pi}(-1) \Theta_{\pi, k_i, \ell_i}(x_i, K, v))_{k_i, \ell_i} \\
&Z_\pi(\sigma_i(\vec{x}), (k_1, \ldots k_{i-1}, \ell_i, k_{i+1}, \ldots k_r); \tau_i(M))_{\ell_i}
\end{split}
\end{equation}
where the $n\times n$ scattering matrix $(\Theta_{k, \ell}(x, K, v))_{k, \ell}$ has entries as described below. For $n \nmid v$, we have $\Theta_{\pi, k, \ell}(x, K, v)= (qx)^{-\deg \pi}$ if $k+\ell \equiv K-1 \bmod n$, and $\Theta_{\pi, k, \ell}(x, K, v)=0$ otherwise. 
For $n\mid v$, we have 
\begin{equation*}
\Theta_{\pi,k,\ell}(x, K, v)= (qx)^{-\deg \pi} \left(\frac{q^{\deg \pi}x^{n\deg \pi}-1}{1-x^{n\deg \pi}}\right)
\end{equation*}
if $k+\ell \equiv K-1 \bmod n$ and  
\begin{equation*}
\Theta_{\pi, k,\ell}(x, K, v)=  (qx)^{-\deg \pi} \left(\frac{(q^{\deg \pi}-1)x^{\deg \pi \left(k+\ell-K+1\right) \%n}}{1-x^{n\deg \pi}}\right)
\end{equation*}
otherwise. Here $\omega$, $b$, $v$, $v_i$, and $K$ are defined as above; they depend on $M$, $i$ and $k_j$ for $j \neq i$. 
\end{theorem}

The proof is the same as that of Theorem \ref{dirichlet-multivariable-global}, using the local functional equation in Theorem \ref{TheoremDirichletFE}.

\subsection{When the Functional Equations Agree}

The functional equations in Theorems \ref{TheoremMultiFE} and \ref{dirichlet-multivariable-global} (or Theorems \ref{TheoremLocalMultiFE} and \ref{dirichlet-multivariable-local} in the local setting) appear to be quite different from each other. But there is a special situation where their hypotheses overlap, and in this case we can check that the functional equations match. We assume that $q\equiv 1 \bmod 4$, that $M_{ii}=0$, and that $n/2|M_{ij}$ for all $j$. This is the only possible case where the hypotheses on the matrix $M$ used in this section and the previous section overlap. This hypothesis will imply that the single-variable subseries in $x_i$ behave like quadratic Dirichlet $L$-functions, or Kubota $L$-series involving quadratic Gauss sums. Because of Gauss's evaluation of the quadratic Gauss sums \eqref{QuadGaussSum}, the Dirichlet and Kubota series match up to a change of variables, so they satisfy equivalent functional equations. 

In this case, we have $n_i=2$, and $e_{ij}=n_{ij}=1$ if $M_{ij}=n/2$ and $e_{ij}=n_{ij}=0$ if $M_{ij}=0$ so the definitions of $F$ and $K$ used in Sections \ref{SectionKubota} and \ref{SectionDirichlet} agree. The transformations $\sigma_i$ underlying the functional equation given in \eqref{sigma-i-kubota} and \eqref{sigma-i-dirichlet} agree:
\begin{equation*} 
\sigma_i(\vec{x})_j = \left\lbrace \begin{array}{cc} 
\frac{1}{qx_i} & j=i \\
x_j\left(\frac{q x_i}{g_\xi}\right)^{n_{ij}} & j \neq i \end{array} \right. .
\end{equation*}
This is straightforward to verify since we have $g_\xi^2=q$ in this case. It is also straightforward to check that $\tau_i(M)=M$. 

Finally, we check the equality of the global scattering matrices:
\begin{equation*}
\left( \chi^{v_i (k+\ell)}(-1) E_{n_i}^n \Gamma_{k, \ell}(x, K) \right)_{k, \ell} = \omega g_{\chi^v}^{b-1}\left(\chi^{(v+v_i)(k+\ell)}(-1)\Theta_{k, \ell}(x, K, v) \right)_{k, \ell}.
\end{equation*}
Because each $M_{ij}$ is $n/2$ or $0$, and $q\equiv 1 \bmod 4$, we have $\chi^{v_i(k+\ell)}(-1)=\chi^{(v+v_i)(k+\ell)}(-1)=\omega=1$. So it suffices to check that $E_{2}^n \Gamma_{k, \ell}(x, K) = g_{\chi^v}^{b-1} \Theta_{k, \ell}(x, K, v)$. 

Note that $v=0$, $b=1$ if $K$ is even and $v=n/2$, $b=0$ if $K$ is odd. If $K$ is even,
\begin{equation*}
(\Gamma_{k, \ell}(x, K))_{k, \ell} = \begin{pmatrix}
    \frac{1-q}{1-q^2x^2} & \frac{qx-x^{-1}}{1-q^2x^2} \\\frac{qx-x^{-1}}{1-q^2x^2} & \frac{1-q}{1-q^2x^2}
\end{pmatrix}
\end{equation*}
so $E_{2}^n \Gamma_{k, \ell}(x_i, K) = S^{k+\ell-K, n} \left(\frac{qx-x^{-1}}{1-q^2x^2}\right)$ for $k+\ell$ odd, and $E_{2}^n \Gamma_{k, \ell}(x_i, K) = S^{k+\ell-K, n} \left(\frac{1-q}{1-q^2x^2} \right)$ for $k+\ell$ even. This matches 
\begin{equation*}
    \Theta_{k, \ell}(x, K, v)=x^{-1} S^{k + \ell-K+1, n} \left( \frac{x-1}{1-qx}\right) = S^{k + \ell-K, n} \left(\frac{1-x^{-1}}{1-qx} \right)= S^{k + \ell-K, n} \left(\frac{1-q+qx-x^{-1}}{1-q^2x^2} \right).
\end{equation*}
If $K$ is odd, 
\begin{equation*}
(\Gamma_{k, \ell}(x, K))_{k, \ell} = \begin{pmatrix}
    \left(\frac{qx}{g_\xi} \right)^{-1} & 0 \\ 0 & \left(\frac{qx}{g_\xi} \right)^{-1}
\end{pmatrix}
\end{equation*}
so $E_{2}^n \Gamma_{k, \ell}(x_i, K) = \left(\frac{qx}{g_\xi} \right)^{-1}$ if $k+\ell\equiv K-1 \bmod n$, $E_{2}^n \Gamma_{k, \ell}(x_i, K) = 0$ otherwise. This matches the definition of $\omega g_{\xi}^{-1} \Theta_{k, \ell}(x, K, v)$ in this case. 

We have verified that the global functional equations of Theorems \ref{TheoremMultiFE} and \ref{dirichlet-multivariable-global} match in the overlapping case. The local functional equations can be checked similarly. All these formulas also agree with the functional equations of \cite{WhiteheadThesis}.

\section{Groupoid Structure} \label{SectionGroup}
\subsection{Arithmetic root systems and multiple Dirichlet series}

For Weyl group multiple Dirichlet series, the functional equations form a group, the Weyl group of a finite root system, which acts on $\mathbb C^r$. The functional equation corresponding to a group element can be expressed as a linear relationship between the value of the series at a point and the value at the image of that point under the group element. More generally, one must consider linear relations between the values of a tuple of series at a point and the values at its image under a group element, with the exact linear relations expressed using a scattering matrix. 

While the functional equation of Kubota type fits into this framework, the functional equation of Dirichlet type does not, since it sometimes relates the values of a multiple Dirichlet series to the values of an apparently totally different multiple Dirichlet series. Instead, the functional equations form the algebraic structure of a groupoid -- we have finitely many different series, each a holomorphic function on $\mathbb C^r$, and functional equations relating a pair of series give isomorphisms between the corresponding copies of $\mathbb C^r$, relating the values of the first series at a point and the second series at its image under the isomorphism.

The relevant groupoids were constructed, in a different context, by Istv\'an Heckenberger, as the Weyl groupoids of arithmetic root systems, forming part of the study of Nichols algebras of diagonal type. In the next section we will explain the connection between multiple Dirichlet series and Nichols algebras and explore the relationship between the functional equations and the Weyl groupoid, but for now we only need to explain the notion of an arithmetic root system at a formal level.

In a specialization of \cite[\S2]{HeckenbergerRank2}, let $\hchi \colon \mathbb Z^r \times \mathbb Z^r \to \mathbb \mu_\infty$ be a bicharacter where $\mu_{\infty}$ is the group of roots of unity of $\mathbb C$, i.e. $\hchi (a+b,c ) = \hchi(a,c) \hchi (b,c)$ and $\hchi(a, b+c) = \hchi (a,b) \hchi(a,c)$ for all $a,b,c\in \mathbb Z^r$. (In fact, one can check that the only part of $\hchi$ relevant for the construction is its restriction to the diagonal $\hchi(a,a)$, but the full bicharacter is relevant for its original application.)

Let $E= (e_1,\dots, e_r)$ be an ordered basis of $\mathbb Z^r$. We say $E$ is admissible at $i$ if $\hchi(e_i,e_i)\neq 1$ and $E$ is admissible if it is admissible at $i$ for all $i$. For each $j\in \{1,\dots,r\}$ with $j\neq i$, let $m_{ij} = \min ( m \in \mathbb Z^{\geq 0} \mid \hchi(e_i,e_i)^{m+1}=1 \textrm{ or } \hchi(e_i,e_i)^m \hchi(e_i,e_j) \hchi(e_j,e_i) =1\}$. For $E$ admissible at $i$, let $s_{i,E}\colon \mathbb Z^r \to \mathbb Z^r$ be the unique map with $s_{i,E}(e_i) = -e_i$ and $s_{i,E} (e_j)= e_j + m_{ij} e_i$ if $j\neq i$. 

For $E$ any fixed ordered basis of $\mathbb Z^r$, let $B_{\hchi,E}$ be the smallest set of ordered bases of $\mathbb Z^r$ such that $E \in B_{\hchi,E}$ and for any $E'\in B_{\hchi, E}$ if $E'$ is admissible at $i$ then $s_{i,E'}(E')\in B_{\hchi,E}$.

We now discuss the Weyl groupoid of a multiple Dirichlet series, following \cite[p. 262, third paragraph]{HeckenbergerYamabe} and not the original definition of Heckenberger. First define an equivalence relation  on $B_{\hchi,E}$ where $(e_1,\dots,e_r)$ and $(e_1',\dots,e_r')$ are equivalent if and only if $\hchi(e_i',e_i')= \hchi(e_i, e_i)$ for all $i$ and   $\hchi(e_i',e_j')\hchi(e_j',e_i')= \hchi(e_i, e_j)\hchi(e_j, e_i)$ for all $i,j$. The objects of Weyl groupoids are the equivalence classes for this equivalence relation. The morphisms from the equivalence class of $E'= (e_1',\dots, e_r')$ to the equivalence class of $E''$ are the set of $r\times r$ matrices $N$ such that $( \sum_{j=1}^n N_{1j} e_j', \dots, \sum_{j=1}^n N_{1j} e_j')$ is equivalent to $E''$. This manifestly depends only on the equivalence class of $E''$ but less obviously depends only on the equivalence class of $E'$. Composition of morphisms is given by matrix composition.


Let us consider an illustrative example.

\begin{example}\label{cartan-example} For $\mathfrak g$ a simple Lie algebra, we can identify $\mathbb Z^r$ with the lattice of roots, which admits an integer-valued pairing $\langle, \rangle$. Fixing a root of unity $v$ and letting $\qq=v^2$, we can set $\hchi (a,b) =  v^{\langle a,b \rangle}$.  Let $\alpha_1,\dots,\alpha_n$ be the simple roots of $\mathfrak g$ and let $E $ be the basis $(\alpha_1,\dots,\alpha_n)$. We have $\hchi(\alpha_i,\alpha_i) = \qq^{ \frac{\langle \alpha_i, \alpha_i \rangle }{2} } $ and \[\hchi(\alpha_i,\alpha_j) \hchi(\alpha_j,\alpha_i) = \qq^{ \langle \alpha_i,\alpha_j\rangle}\] so that  \[\hchi(\alpha_i,\alpha_i)^m \hchi(\alpha_i,\alpha_j) \hchi(\alpha_j,\alpha_i)=1\] for $m = - 2 \frac{\langle \alpha_i,\alpha_j\rangle}{\langle \alpha_i,\alpha_i \rangle}$. If $\qq$ avoids some finite list of roots of unity depending on $\mathfrak g$, we will not have $\hchi(\alpha_i,\alpha_i)^{\ell}=1$ for any $\ell \leq m$ so that in fact $m_{ij}= - 2 \frac{\langle \alpha_i,\alpha_j\rangle}{\langle \alpha_i,\alpha_i \rangle}$. Then $s_{i,E}$ is the reflection around the root $\alpha_i$ in the Weyl group of $\mathfrak g$. Using this, we see that $B_{\hchi,E}$ is the set of bases of the root lattice given by all possible sets of simple roots.

All bases in $B_{\hchi, E}$ are equivalent, so the Weyl groupoid has one object, and the isotropy group of that object is the Weyl group. 

The variable $\qq$ is denoted by $q$ in \cite{AndruskiewitschAngiono, HeckenbergerClassification}. To avoid ambiguity with $q$ the cardinality of a finite field, or, later, $q$ in the sense of quantum groups, we use a different font for $\qq$. \end{example}

Thus $B_{\hchi,E}$ is a generalization of the set of all possible collections of simple roots in the root lattice of a simple Lie algebra.

An \emph{arithmetic root system} is defined~\cite[Definition 1]{HeckenbergerRank2} to be a triple $(\mathbf \Delta, \hchi, E)$, where $\mathbf \Delta = \bigcup_{E \in  B_{\hchi,E}} E$, satisfying the assumptions that  $B_{\hchi,E}$ is finite and all bases in $B_{\hchi,E}$ are admissible.

The relationship of this to multiple Dirichlet series is as follows. For a bicharacter $\hchi$, fix a root of unity $\gene$ of even order such that every value taken by $\hchi$ is a power of $\gene$. Then define  $M^{\hchi, E, \gene}_{ij} $ to be the smallest nonnegative integer $m$ with $\gene^m = \hchi(e_i,e_j) \hchi(e_j,e_i)$ if $i \neq j$ or the smallest nonnegative integer $m$ with $\gene^m = - \hchi(e_i,e_i)$ if $i=j$. The basic theorem relating the functional equations to the reflections $s_{i,E}$ of the Weyl groupoid is as follows.

\begin{theorem}\label{translation-key-step} Let $\hchi \colon \mathbb Z^r \times \mathbb Z^r \to \mu_{\infty}$ be a bicharacter, $E$ a basis of $\mathbb Z^r$, and $\gene$ a root of unity of even order $n$ such that every value taken by $\hchi$ is a power of $\gene$. Fix $i\in \{1,\dots, r\}$ and assume $E$ is admissible at $i$.  Fix a finite field $\mathbb F_q$ and a character $\chi \colon \mathbb F_q^\times \to \mathbb C^\times$ of order $n$.

\begin{enumerate}

\item If $\hchi(e_i,a) \hchi(a,e_i) $ is a power of $\hchi(e_i,e_i)$ for all $a \in \mathbb Z^r $ and $q\equiv 1\bmod 4$, then the exponent of $x_i$ in each monomial with nonvanishing coefficient in $ \left(1 - q^{n/n_i}  \left( \frac{qx_i }{ g}\right)^n \right)Z(\vec{x}, (k_1, \ldots k_r); q, \chi, M^{\hchi,E, \gene})$ is bounded in terms of the other exponents and we have the functional equation 
\[ (Z(\vec{x}, (k_1, \ldots k_r); q, \chi, M^{\hchi,E, \gene} ))_{k_i} \] \[= \left( \chi^{v_i (k_i+\ell_i)}(-1) E_{n_i}^n \Gamma_{k_i, \ell_i}(x_i, K) \right)_{k_i, \ell_i} (Z(\sigma_i(\vec{x}),  (k_1, \ldots , k_{i-1} , \ell_i, k_{i+1}, \ldots, k_r ); q, \chi, M^{\hchi,s_{i,E}(E) , \gene} ))_{\ell_i} \]
with $K, v_i , \sigma_i, \Gamma$ and $E_{n_i}^n$ defined as in \S\ref{ss-k-multi}.

\item If $\hchi(e_i,e_i) =-1$ then the exponent of $x_i$ in each monomial with nonvanishing coefficient in $ \left(1 - q^n x_i^n \right)^n Z(\vec{x}, (k_1, \ldots k_r); q, \chi, M^{\hchi,E, \gene}$ is bounded in terms of the other exponents and we have the functional equation
\[ (Z(\vec{x}, (k_1, \ldots k_r); q, \chi, M^{\hchi,E, \gene} ))_{k_i} \] \[= \omega g_{\chi^v}^{b-1}\left(\chi^{(v+v_i)(k_i+\ell_i)}(-1)\Theta_{k_i, \ell_i}(x_i, K, v) \right)_{k_i, \ell_i}  (Z(\sigma_i(\vec{x}), (k_1, \ldots , k_{i-1} , \ell_i, k_{i+1}, \ldots, k_r ); q, \chi, M^{\hchi,s_{i,E}(E) , \gene} ))_{\ell_i}\]
with $v, v_i, K, b, \omega , \sigma_i$, and $\Theta_{k,\ell}$ defined as in \S\ref{ss-multivariable-functional-equation}.
\end{enumerate}

Moreover, in either case we have a relationship between $\sigma_i$ and $s_{i,E}$. Consider the unique bilinear pairing  $\langle , \rangle_E \colon \mathbb Z^r \times (\mathbb C^\times)^r \to \mathbb R$  with $\langle e_j, \vec{x}\rangle= \log (  \sqrt{q}|x_j| )$ for all $j \in E$. 

Then for all $\vec{x} \in (\mathbb C^\times)^r $ and $a \in \mathbb Z^r$ we have
\begin{equation}\label{sigma-s-eq} \langle s_{i,E}(a), \vec{x} \rangle_E = \langle a, \sigma_i(\vec{x} )\rangle_E. \end{equation}

\end{theorem}

\begin{proof}In case (1), the hypothesis that  $\hchi(e_i,e_j) \hchi(e_j,e_i) $ is a power of $\hchi(e_i,e_i)$ implies that   $M^{\hchi,E,\gene}_{ij}$ is an integer multiple of $M^{\hchi, E,\gene}_{ii}+ \frac{n}{2}$ for all $j$. This is the hypothesis of the functional equation proved in Theorem \ref{TheoremMultiFE}. The difference between the stated equation and the equation in Theorem \ref{TheoremMultiFE} is that Theorem \ref{TheoremMultiFE} relates $Z(\vec{x}, (k_1, \ldots k_r); q, \chi, M^{\hchi,E, \gene})$ to itself rather than to the apparently different series $Z(\vec{x}, (k_1, \ldots k_r); q, \chi, M^{\hchi,s_i(E), \gene})$. However, we will see in Lemma \ref{kubota-bicharacter-identity}(2) that in fact $M^{\hchi,s_i(E), \gene}= M^{\hchi,E, \gene}$ so these two series are the same. The boundedness of exponents follows from the rationality property of Theorem \ref{TheoremFE} since multiplying a rational function by its denominator produces a polynomial and the exponents of monomials appearing in a polynomial are bounded.

In case (2) the hypothesis that $\hchi(e_i,e_i)=-1$ implies that $M^{\hchi, E,\gene}_{ii}=0$, which is the standing assumption made in \S\ref{ss-multivariable-functional-equation}, allowing us to apply Theorem \ref{dirichlet-multivariable-global}. The difference between the stated equation and the equation of Theorem \ref{dirichlet-multivariable-global} is that Theorem \ref{dirichlet-multivariable-global} has $\tau_i(M^{\hchi,E, \gene})$ instead of  $M^{\hchi,s_{i,E}(E) , \gene} $. So we must check these two matrices are equal, which is done in Lemma \ref{dirichlet-bicharacter-identity}(2) below. The boundedness of exponents follows from the rationality property of Theorem \ref{TheoremDirichletFE}.

To check \eqref{sigma-s-eq}, it suffices to handle the case $a = e_j$ since the $e_i$ form a basis for $\mathbb Z_r$. We handle each case in turn.

If $i\neq j$, in case (1) then we will see in Lemma \ref{kubota-bicharacter-identity}(1) that $m_{ij} =n_{ij}$.  Using this, we have
\begin{equation*}\langle e_j , \sigma_i(\vec{x} )\rangle_E = \log \Bigl( \sqrt{q} \Bigl| x_j \Bigl(\frac{q x_i}{g}\Bigr)^{n_{ij}}\Bigr| \Bigr) = \log ( \sqrt{q} |x_j| (\sqrt{q} |x_i|)^{n_{ij}}) =  \log ( \sqrt{q} |x_j| (\sqrt{q} |x_i|)^{m_{ij}})  \end{equation*}
Similarly, in case (2), we will see in Lemma \ref{dirichlet-bicharacter-identity}(2) that $m_{ij}=e_{ij}$. Using this, we have
\begin{equation*}\langle e_j , \sigma_i(\vec{x} )\rangle_E = \log \Bigl( \sqrt{q} \Bigl| x_j \Bigl( (g_ {\chi^{M_{ij}}}  x_i \Bigr)^{e_{ij}}\Bigr| \Bigr) = \log ( \sqrt{q} |x_j| (\sqrt{q} |x_i|)^{e_{ij}}) = \log ( \sqrt{q} |x_j| (\sqrt{q} |x_i|)^{m_{ij}})   \end{equation*}
In either case we therefore have
\begin{equation*}\langle e_j , \sigma_i(\vec{x} )\rangle_E =  \log ( \sqrt{q} |x_j| (\sqrt{q} |x_i|)^{m_{ij}})  = \langle e_j, \vec{x} \rangle_E + m_{ij} \langle e_i, \vec{x} \rangle_E= \langle e_j+  m_{ij} e_i, \vec{x} \rangle_E = \langle s_{i,E}(e_j),\vec{x} \rangle_E  \end{equation*} as desired. 

For $j=i$, in case (1) we have \begin{equation*}\langle e_j , \sigma_i(\vec{x} )\rangle_E = \log \Bigl ( \sqrt{q} \Bigl| \frac{ g^2}{ q^2 x_i }\Bigr| \Bigr) = \log \frac{1}{ \sqrt{q} |x_i|} \end{equation*}
and in case (2) we have
\begin{equation*}\langle e_j , \sigma_i(\vec{x} )\rangle_E = \log \Bigl ( \sqrt{q} \Bigl| \frac{ 1}{ q x_i }\Bigr| \Bigr) = \log \frac{1}{ \sqrt{q} |x_i|}\end{equation*}
so in either case we have
\begin{equation*} \langle e_j , \sigma_i(\vec{x} )\rangle_E  = \log \frac{1}{ \sqrt{q} |x_i|}= - \langle e_i, \vec{x} \rangle_E = \langle -e_i, \vec{x}\rangle_E =\langle s_{i,E}(e_i) , \vec{x} \rangle_E. \end{equation*}
This completes the verification of \eqref{sigma-s-eq}.\end{proof}

Two bases $E$ and $E'$ are equivalent in the equivalence relation defining the Weyl groupoid if and only if $M^{\hchi,E,\gene} = M^{\hchi, E',\gene}$, so the objects of the Weyl groupoid correspond to choices of matrices $M$ and therefore to multiple Dirichlet series.  It follows that if $(\mathbf \Delta, \hchi, E)$ is an arithmetic root system such that each  $e\in \mathbf \Delta$ satisfies either condition (1) or (2) of Theorem \ref{translation-key-step} then each morphism in the Weyl groupoid from one series to another corresponds to a combination of functional equations that we may use to relate the two Dirichlet series. In particular, the multiple Dirichlet series $Z(s_1,\dots, s_r; q, \chi, M^{\hchi,E, \gene})$ has a finite groupoid of functional equations, relating finitely many different series, but such that each of the series has a functional equation to itself which inverts each of the variables. This turns out to be the crucial condition to obtain desired properties like meromorphic continuation.

Here we can take advantage of the classification of arithmetic root systems proven in \cite{HeckenbergerClassification}, though we find the organization of the classification in \cite{AndruskiewitschAngiono} enlightening. The following theorem shows the arithmetic root systems to which Theorem \ref{translation-key-step} applies, and thus to which our subsequent results Theorem \ref{rational-fn} and \ref{uniquely-determined} will apply.

\begin{theorem} \label{ClassificationTheorem} Every root of the following arithmetic root systems, in the notation of \cite{AndruskiewitschAngiono}, satisfies either condition (1) or (2) of Theorem \ref{translation-key-step}:

\begin{itemize}
    \item All those of Cartan type (i.e. agreeing with the root system of a simple Le algebra over $\mathbb C$): $\mathtt A_\theta,\mathtt   B_\theta, \mathtt C_\theta, \mathtt D_\theta,\mathtt E_\theta,\mathtt  F_4,$ and $\mathtt  G_2$
    \item All those of super type (i.e. agreeing with the root system of a simple super Lie algebra over $\mathbb C$): $\mathtt A(j\mid\theta-j) , \mathtt B(j\mid \theta-j), \mathtt C(j \mid \theta-j), \mathtt D(j \mid \theta-j), \mathtt D(2,1;\alpha) ,\mathtt  F(4),$ and $ \mathtt G(3) $
    \item $\mathtt{wk}(4)$, agreeing with the root system of a simple Lie algebra over a field of characteristic $2$
    \item $\mathtt{g}(1,6)$ when its parameter $\zeta$ is a third root of unity and not a sixth root of unity, $\mathtt{g}(2,3),$ $ \mathtt{g}(3,3),$ $\mathtt{g}(4,3), \mathtt{g}(3,6),$ $ \mathtt{g}(2,6) ,$ $ \mathtt{el}(5,3),$ $ \mathtt{g}(8,3), $ $\mathtt{g}(4,6) ,$ $ \mathtt{g}(6,6), $ and $\mathtt{g}(8,6)$, all agreeing with the root system of simple super Lie algebras over fields of characteristic $3$
    \item $\mathtt{brj}(2,5)$ and $\mathtt{el}(5;5)$, both agreeing with the root systems of simple super Lie algebras over fields of characteristic $5$
    \item $\mathtt{ufo}(1), \mathtt{ufo}(2), \mathtt{ufo}(6) , $ and $ \mathtt{ufo}(12)$, called UFOs because they do not correspond in any known way to Lie algebras 
    \end{itemize}

    \end{theorem}

    \begin{proof} This can be checked by examining the generalized Dynkin diagrams provided for each arithmetic root system in \cite{AndruskiewitschAngiono}. The generalized Dynkin diagram associated to a pair $(\hchi,E)$ has one node for each $i$ from $1$ to $r$. Each node is labeled by $\hchi(e_i,e_i)$. Two nodes are connected by an edge if $\hchi(e_i,e_j) \hchi(e_j,e_i)\neq 1$ and, in this case, the edge is labeled by $\hchi(e_i,e_j) \hchi(e_j,e_i)$. Thus, to check that each root satisfies (1) or (2), it suffices to examine the Dynkin diagrams and check that each node is either labeled by $-1$ or else connects only to edges with labels an integer power of its vertex label.  This is easy to do.

    The relationship of the root systems to Lie algebras are also taken from \cite{AndruskiewitschAngiono}.
    \end{proof}

The multiple Dirichlet series arising from these arithmetic root systems include many of those appearing previously in the literature, or at least $\mathbb F_q(T)$-analogues of them.

We will see shortly in Example \ref{wgmds} that the multiple Dirichlet series arising from arithmetic root systems of Cartan type give the Weyl group multiple Dirichlet series. The Weyl group multiple Dirichlet series are known to have a direct relationship to the arithmetic of the algebraic groups with the same root system, as Whittaker coefficients of Eisenstein series on metaplectic covers of those groups \cite{McNamara, PatnaikPuskas, Chen}. It would be interesting to give a similar uniform description of the multiple Dirichlet series arising from arithmetic root systems of super type involving the corresponding super Lie algebra. 

For the multiple Dirichlet series arising from arithmetic root systems of Lie algebras and super Lie algebras over fields of characteristic $2,3,$ and $5$, a challenge is that these multiple Dirichlet series fail to be defined exactly over fields $\mathbb F_q$ of characteristic matching the Lie algebra characteristic, since their definition requires multiplicative characters of $\mathbb F_q$ of order a power of the Lie algebra characteristic. So one cannot consider the Lie algebra as a Lie algebra over $\mathbb F_q(T)$ in the same way that the metaplectic Eisenstein series construction requires considering algebraic groups over the relevant global field.

To deduce properties like meromorphic continuation from the functional equations in a uniform way, one needs the following result of Heckenberger.

\begin{prop}\label{cone-union-lemma} Let $(\mathbf \Delta, \hchi, E)$ be an arithmetic root system. For each basis $E$ of $\mathbb Z^r$, let $C_E$ be the cone in $\mathbb R^r$ consisting of vectors whose dot product with each vector in $E$ is nonnegative. Then the interiors of the $C_{E'}$ are disjoint from each other and \[ \bigcup_{E' \in B_{\hchi,E}} C_{E'} = \mathbb R^r .\] \end{prop}

In the classical root system case, the $C_{E'}$ are the Weyl chambers, and the assertion that the union of the Weyl chambers fill the space is standard. The key step in the proof of the lemma is that cones related by a simple reflection share a face, which may also be familiar in the Weyl chamber case.

\begin{proof} \cite[Proposition 2]{HeckenbergerRank2} states that  $\{ v \in \mathbb R^r \mid v \cdot e \neq 0 \textrm{ for all }e\in \mathbf \Delta\}$ is the disjoint union of the sets $\{ v\in \mathbb R^r \mid v\cdot e_i>0 \textrm{ for all }i \}$, i.e. the interiors of the $C_{E'}$. Since the $C_{E'}$ are the closures of their interiors, and there are finitely many of them, their union is the closure of the union of their interiors. Since $\mathbf \Delta$ is finite, the closure of $\{ v \in \mathbb R^r \mid v \cdot e \neq 0 \textrm{ for all }e\in \mathbf \Delta\}$ is $\mathbb R^r$.\end{proof}

We are now ready to state the a result on the meromorphic continuation of multiple Dirichlet series, or, more precisely, the rationality of these series.

\begin{theorem} \label{rational-fn} Let $(\mathbf \Delta, \hchi, E)$ be an arithmetic root system such that each  $i\in \mathbf \Delta$ satisfies either condition (1) or (2) of Theorem \ref{translation-key-step}. Let $\gene$ be a root of unity such that every value taken by $\psi$ is a power of $\gene$. Fix a finite field $\mathbb F_q$ with $q \equiv 1\bmod 4$ and a character $\chi \colon \mathbb F_q^\times \to \mathbb C^\times$ of order the order of $\gene$. Then for all $E' \in B_{\hchi,E}$, the function $Z ( \vec{x}; q, \chi, M^{\hchi, E',\gene}) $ is in fact a rational function of the variables $x_1,\dots,x_r$.
\end{theorem}

Typically in the multiple Dirichlet series literature one proves that a multiple Dirichlet series is meromorphic by proving holomorphicity on some domain, using the functional equations to extend meromorphicity to a larger domain, finally use convexity to obtain meromorphicity on the whole space. In the function field case, using growth estimates, one can conclude that the function is rational.  The starting point, holomorphicity on some domain, is not at all obvious for axiomatic multiple Dirichlet series but was established by Hase-Liu~\cite[Theorem 1]{HaseLiu24} using geometric means. However, to make the argument work one must in addition show meromorphic continuation to a domain which overlaps with its translate under the functional equation, so that the functional equation gives an equality of meromorphic functions on the intersection reason. Hase-Liu checks this for functional equations of Dirichlet type but not of Kubota type. The additional verification is possible but we do not take this approach here.

Instead, we use an argument that is only possible in the function field setting, where we directly show rationality by manipulating a formal power series. This argument does not use any analysis, but morally has the same structure as the complex-analytic argument.

\begin{proof} For $v \in \mathbb R^r$ a vector, we say that a formal power series $Z(\vec{x})$ is $v,E'$-bounded if there exists a polynomial $F$ in $x_1,\dots, x_r$ such that the set of sums $\sum_{i=1}^r d_i v \cdot e_i'$ such that  $\prod_{i=1}^r x_i^{d_i}$ appears with nonzero coefficient in $F(\vec{x})Z ( \vec{x})$ is bounded below.

For $v \in C_{E'}$ so that $v \cdot e_i' \geq 0$ for all $i$, the power series $Z ( \vec{x}; q, \chi, M^{\hchi, E',\gene})$ is automatically $v,E'$-bounded, taking $F=1$, since only monomials of nonnegative degrees appear in it.

For $v$ in the interior of $- C_{E'}$, so that $v \cdot e_i'<0$ for all $i$, if $Z ( \vec{x}; q, \chi, M^{\hchi, E',\gene})$ is $v,E'$-bounded then $F(\vec{x})Z ( \vec{x}; q, \chi, M^{\hchi, E',\gene})$ is a polynomial and hence $Z ( \vec{x}; q, \chi, M^{\hchi, E',\gene})$ is a rational function.

We can always take $v$ in the interior of $-C_{E'}$ and then, by Proposition \ref{cone-union-lemma}, $v$ will lie in $C_{E^*}$ for some $E^*\in B_{\hchi,E}$. By the previous two observations, it suffices to prove that if $Z ( \vec{x}; q, \chi, M^{\hchi, E^*,\gene})$ is $v,E^*$-bounded for one $E^* \in B_{\hchi,E}$ then $Z ( \vec{x}; q, \chi, M^{\hchi, E',\gene})$ is $v,E'$-bounded for all $E' \in B_{\hchi,E}$. Since, by definition of $B_{\hchi,E}$, $E'$ is mapped to $E^*$ by a finite composition of simple reflections, it suffices by induction to prove that if  $Z ( \vec{x}; q, \chi, M^{\hchi, s_{i,E'}(E'),\gene})$ is $v,s_{i,E'}(E')$-bounded then $Z ( \vec{x}; q, \chi, M^{\hchi, E',\gene})$ is $v,E'$-bounded.

If  $Z ( \vec{x}; q, \chi, M^{\hchi, s_{i,E'}(E'),\gene})$ is  $v,s_{i,E'}(E')$-bounded then, replacing $F(\vec x)$ by $ \prod_{ \gamma_1,\ldots,\gamma_r \in \langle \gene\rangle} F(\gamma_1x_1,\ldots, \gamma_r x_r)$, the same is true for  $Z ( \vec{x}, \vec{k}; q, \chi, M^{\hchi, s_{i,E'}(E'),\gene})$ for all $\vec{k}$.

Now we check that $Z ( \sigma_i( \vec{x}), \vec{k}; q, \chi, M^{\hchi, s_{i,E'}(E'),\gene})$ is $ (v, E')$-bounded.  For integers $d_1,\dots, d_r$, the composition of the monomial $\prod_{j=1}^{d_r} (x_j)$ with $\sigma_i$ is a scalar multiple of the monomial \[ x_i^{-d_i} \prod_{j \neq i} (x_j  x_i^{n_{ji}})^{d_j} = x_i^{-d_i +  \sum_{j\neq i} n_{ji}} \prod_{j \neq i} x_j^{d_j}.\] Since the map on the set of monomials that sends $\prod_{j=1}^{d_r} (x_j)$ to $x_i^{-d_i +  \sum_{j\neq i} n_{ji}} \prod_{j\neq i} x_j^{d_j}$ is a bijection, the monomial $\prod_{j=1}^r x_j^{d_j}$ has nonzero coefficient in $ F(\vec{x})Z ( \vec{x}, \vec{k}; q, \chi, M^{\hchi, s_{i,E'}(E'),\gene})$ if and only if the monomial  $x_i^{-d_i +  \sum_{j\neq i} n_{ji} d_j} \prod_{i\neq j} x_j^{d_j}$ has nonzero coefficient in  $F(\sigma_i(\vec{x}))Z ( \sigma_i(\vec{x}), \vec{k}; q, \chi, M^{\hchi, s_{i,E'}(E'),\gene})$.  For the set of tuples $d_1,\dots, d_r$ satisfying either of these two equivalent conditions, we have 
\[ (-d_i +  \sum_{j\neq i} n_{ji} d_j)  v \cdot e_i' + \sum_{j \neq i}  d_j v \cdot e_{j'} = d_i v\cdot (-e_i') + \sum_{j\neq i} d_j v \cdot (e_{j'} + m_{ij} e_{i'}) = \sum_{j=1}^r d_j v\cdot s_{i,E'}(e_i') \] 
and the quantities $\sum_{j=1}^r d_j v\cdot s_{i,E'}(e_i')$ are bounded below since $Z ( \vec{x}, \vec{k}; q, \chi, M^{\hchi, s_{i,E'}(E'),\gene})$ is $v,s_{i,E'}(E')$-bounded, hence the quantities $(-d_i +  \sum_{j\neq i} n_{ji} d_j)  v \cdot e_i' + \sum_{j \neq i}  d_j v \cdot e_{j'}$ are bounded below and thus $Z ( \sigma_i( \vec{x}), \vec{k}; q, \chi, M^{\hchi, s_{i,E'}(E'),\gene} )$ is $ (v, E')$-bounded.

It is straightforward to check that multiplication by a rational function and summation both preserve the property of being $v,E'$-bounded. (For the rational function case, we multiply $F$ by the denominator of the rational function. This is the only step in the argument for which the inclusion of $F$ in the definition of $v,E'$-bounded is necessary.)

Because multiplication by a rational function and summation preserve the property of being $v,E'$-bounded, it follows from the functional equation of Theorem \ref{translation-key-step} that $Z ( \vec{x}, \vec{k}; q, \chi, M^{\hchi, E',\gene})$ is $v,E'$-bounded. Summing over $\vec{k}$, we see that $Z ( \vec{x}; q, \chi, M^{\hchi, E',\gene})$ is $v,E'$-bounded, completing the induction step.

(We could have simplified the proof of the induction step by making the induction hypothesis that $Z ( \vec{x}, \vec{k}; q, \chi, M^{\hchi, E',\gene})$ is $(v,E')$-bounded for all $\vec{k}$, at the cost of complicating the explanation of the structure of the induction argument.) \end{proof}

Now we state the theorem that our multiple Dirichlet series are uniquely determined by their functional equations.

\begin{theorem}\label{uniquely-determined} Let $(\mathbf \Delta, \hchi, E)$ be an arithmetic root system such that each  root in $ \mathbf \Delta$ satisfies either condition (1) or (2) of Theorem \ref{translation-key-step}. Let $\gene$ be a root of unity of even order $n$ such that every value taken by $\psi$ is a power of $\gene$. Fix a finite field $\mathbb F_q$ with $q \equiv 1 \bmod 4$ and a character $\chi \colon \mathbb F_q^\times \to \mathbb C^\times$ of order $n$.

There exists a unique collection of formal power series $\tilde{Z} (\vec{x},\vec{k}; E')$ for each $ E'\in B_{\hchi,E}$ and $\vec{k} \in \{0,\dots, n-1\}^r$, satisfying all the following conditions:

\begin{enumerate}
    \item For each $\vec{k}, E'=(e_1,\dots,e_r)$, and $i\in\{1,\dots, r\}$ such that $\hchi(e_i,a) \hchi(a,e_i)$ is a power of $\hchi(e_i,e_i)$ for all $a\in \mathbb Z^r$, the exponent of $x_i$ in each monomial appearing with nonvanishing coefficient in $ \left(1 - q  \left( \frac{qx_i }{g}\right)^n \right)\tilde{Z} (\vec{x},\vec{k}; E')$ (where $g=g_{\xi \chi^{M_{ii}^{\hchi, E', \gene}}}$) is bounded in terms of the other exponents  and we have the functional equation 
    \[ (\tilde{Z} (\vec{x}, (k_1, \ldots k_r);E' )_{k_i} \] \[= \left( \chi^{v_i (k_i+\ell_i)}(-1) E_{n_i}^n \Gamma_{k_i, \ell_i}(x_i, K) \right)_{k_i, \ell_i} ( \tilde{Z} (\sigma_i(\vec{x}), (k_1, \ldots , k_{i-1} , \ell_i, k_{i+1}, \ldots, k_r );  s_{i, E'}(E')  ))_{\ell_i} \]
    \item For each $\vec{k}, E'=(e_1,\dots,e_r)$, and $i\in\{1,\dots, r\}$ such that $\hchi(e_i,e_i)=-1$, the exponent of $x_i$ in each monomial appearing with nonvanishing coefficient in $ \left(1 - q^n x_i^n \right)\tilde{Z} (\vec{x},\vec{k}; E')$ is bounded in terms of the other exponents and we have the functional equation 
\[ ( \tilde{Z} (\vec{x}, (k_1, \ldots k_r); E' )_{k_i} \] \[= \omega g_{\chi^v}^{b-1}\left(\chi^{(v+v_i)(k_i+\ell_i)}(-1)\Theta_{k_i, \ell_i}(x_i, K, v) \right)_{k_i, \ell_i}  ( \tilde{Z} (\sigma_i(\vec{x}), (k_1, \ldots , k_{i-1} , \ell_i, k_{i+1}, \ldots, k_r) ; s_{i, E'}(E')  ))_{\ell_i}\]
 \item $Z( \vec{x}, (0,\dots,0), E)$ has constant term $1$.
 \end{enumerate} 
Those power series are given by
\begin{equation}\label{uniqueness-eq} \tilde{Z} (\vec{x},\vec{k}; E')= Z(\vec{x}, \vec{k} ; q, \chi, M^{\hchi,E',\gene}).\end{equation} 
\end{theorem}  

Before proving Theorem \ref{uniquely-determined}, we prove some lemmas that involve a single functional equation relating a pair of multiple Dirichlet series. Lemmas \ref{uniquely-determined-sv-d} and \ref{uniquely-determined-sv-k} will involve a single-variable functional equation while Lemma \ref{uniquely-determined-step}, proven using the previous two, will involve a multivariable functional equation. In Lemmas \ref{uniquely-determined-sv-d} and Lemmas \ref{uniquely-determined-sv-k} we will not state the single-variable functional equations as we did when we originally proved them, but rather in a form that is easier to derive from the multivariable functional equations.

We could add to Theorem \ref{uniquely-determined} the condition that the exponent of $x_i$ in each monomial with non-vanishing coefficient in $\tilde{Z}(\vec{x}, \vec{k}; E')$ is congruent to $k_i$ mod $n$, a property that is satisfied, by definition, by $ Z(\vec{x}, \vec{k} ; q, \chi, M^{\hchi,E',\gene})$. This would simplify the proof, but only slightly, so we prefer to avoid imposing this condition.

\begin{lemma}\label{uniquely-determined-sv-d}  Fix a finite field $\mathbb F_q$ and a character $\chi \colon \mathbb F_q^\times \to \mathbb C^\times$ of even order $n$. Let $v, v_i,$ and $K$ be integers and $\omega \in \mathbb C^\times$. Set $b=1$ if $n\mid v$ and $b=0$ otherwise.  Let $\tilde{d}$ be an integer. Let $(S(x,k))_k$ and $(S'(x,k))_k$ be $n$-dimensional vectors of power series in $x$ such that $(1-q^n x^n) S(x,k)$ and $(1-q^n x^n) S'(x,k)$ are polynomials in $x$. Suppose that
\begin{equation}\label{ud-sv-d-fe} (S(x,k))_k = x^{\tilde{d}} \omega g_{\chi^v}^{b-1}\left(\chi^{(v+v_i)(k+\ell)}(-1)\Theta_{k, \ell}(x, K, v) \right)_{k, \ell}  \left(S'(\frac{1}{qx} ,\ell)\right)_\ell.\end{equation} 
Let $d^{\mathrm{min}}$ be the minimum integer $d$ such that the coefficient of $x^d$ in $S(x,k)$ is nonzero for at least one $k$. Let $d^{'\mathrm{min}}$ be the minimum integer $d$ such that the coefficient of $x^d$ in $S'(x,k)$ is nonzero for at least one $k$.  
If $d^{\mathrm{min}}<\infty$ then \begin{equation} \label{di-bound-d} d^{\mathrm{min}}+d^{'\mathrm{min}} \leq \tilde{d}. \end{equation} \end{lemma}

\begin{proof} Since $d^{\mathrm{min}}<\infty$,  $S(x,k) $ is nonzero for at least one $k$, so by the invertibility of the Fourier transform there is $\eta\in \mu_n$ such that 
\begin{equation}\label{d-Snonzero}  \sum_{k =0}^{n-1} \eta^{k} S(x,k) \neq 0.\end{equation}

The functional equation \eqref{ud-sv-d-fe} together with the formula for $\Theta_{k,\ell}(x, K, v)$ gives
\[\sum_{k =0}^{n-1} \eta^{k} S(x,k)  =\sum_{\ell=0}^{n-1} S'(\frac{1}{qx} ,\ell) \sum_{k=0}^{n-1} \eta^k x^{\tilde{d}}   \omega g_{\chi^v}^{b-1} \chi^{(v+v_i)(k+\ell)}(-1)  \Theta_{k,\ell}(x, K, v) \] \[ =   \sum_{\ell =0}^{n-1} \eta^{-\ell}S' (\frac{1}{qx} ,\ell)  x^{ \tilde{d}}  \omega \times \begin{cases} \eta^{K-1}  g_{\chi^v}^{-1} \chi^{(v+v_i)(K-1) }(-1)x^{-1} & \mathrm{if } n\nmid v\\ \eta^{K-1}   \chi^{(v+v_i) (K-1)} (-1) x^{-1} \frac{1 -\eta \chi^{v+v_i}(-1)  x}{ 1- \eta \chi^{v+v_i}(-1)  qx } & \mathrm{if } n \mid v\end{cases}.  \]
Let $\eta' = \eta \chi^{v+v_i}(-1)$. 

 Now we know that $(1-q^n x^n) \sum_{\ell =0}^{n-1} \eta^{-\ell}S' (x,\ell)$ is a polynomial in $x$. Hence $(1-x^{-n} ) \sum_{\ell =0}^{n-1} \eta^{-\ell}S' (\frac{1}{qx} ,\ell)$ is a polynomial in $x^{-1}$. Since we obtain  $\sum_{k =0}^{n-1} \eta^{k} S(x,k)$ by multiplying by a polynomial in $x,x^{-1}$ and possibly dividing by $(1- \eta' qx)$, we conclude that $(1-x^{-n})(1-\eta' qx) \sum_{k =0}^{n-1} \eta^{k} S(x,k)$ is a polynomial in $x, x^{-1}$. On the other hand, we also know that $(1-q^n x^n)  \sum_{k =0}^{n-1} \eta^{k} S(x,k)$ is a polynomial in $x,x^{-1}$. It follows that $(1- \eta' q x)  \sum_{k =0}^{n-1} \eta^{k} S(x,k)$ is a polynomial in $x,x^{-1}$, where $(1-\eta' q x)$ is the greatest common divisor of $(1-x^{-n})(1-\eta' qx) $ and $(1-q^nx^n)$ in the ring $\mathbb C[x,x^{-1}]$. 

Hence we obtain an identity of polynomials in $x,x^{-1}$:
\[ (1-\eta' qx) \sum_{k =0}^{n-1} \eta^{k} S(x,k) =  \sum_{\ell =0}^{n-1} \eta^{-\ell}S' (\frac{1}{qx} ,\ell)  x^{ \tilde{d}}  \omega \times \begin{cases} \eta^{K-1}  g_{\chi^v}^{-1} \chi^{(v+v_i)(K-1) }(-1) x^{-1} (1- \eta' q x)  & \textrm{if } n\nmid v\\ \eta^{K-1}   \chi^{(v+v_i)  (K-1)}(-1) x^{-1} (1 -\eta' x) & \textrm{if } n \mid v\end{cases}.\]
Since the coefficient of $x^d$ in $S'(x,\ell)$ is zero for $d< d^{'\mathrm{min}}$, the coefficient of $x^d$ in  $\sum_{\ell =0}^{n-1} \eta^{-\ell}S' (\frac{1}{qx} ,\ell)$ is zero for $d> - d^{'\mathrm{min}}$ and thus the coefficient of $x^d$ in  $\sum_{\ell =0}^{n-1} \eta^{-\ell}S' (\frac{1}{qx} ,\ell)  x^{ \tilde{d}}$ is zero for $d> \tilde{d}- d^{'\mathrm{min}}$. Since both $x^{-1} (1 -\eta' x)$ and $x^{-1} (1 -\eta' q x)$ have degree $\leq 0$ in $x$, the coefficient of $x^d$ in $(1-\eta' qx) \sum_{k =0}^{n-1} \eta^{k} S(x,k) $ is zero for $d > \tilde{d}- d^{'\mathrm{min}}$.

However, \eqref{d-Snonzero} implies that the coefficient of $d$ in $(1-\eta' qx) \sum_{k =0}^{n-1} \eta^{k} S(x,k) $ is nonzero for some $d$. By assumption, we have $d \geq d^{\mathrm{min}}$, and we have just seen that $d \leq \tilde{d}- d^{'\mathrm{min}}$. Combining these, we obtain \eqref{di-bound-d}.
\end{proof}

\begin{lemma}\label{uniquely-determined-sv-k}  Fix a finite field $\mathbb F_q$ and a character $\chi \colon \mathbb F_q^\times \to \mathbb C^\times$ of even order $n$. Let $n_i$ be a divisor of $n$ and $M_{ii}$ an integer such that $\gcd( M_{ii} + \frac{n}{2}, n) =\frac{n}{n_i}$.  Let $K$ and $v$ be integers. Let $c\in \mathbb C^\times$ and let $\tilde{d}$ be an integer. Let $(S(x,k))_k$ and $(S'(x,k))_k$ be $n$-dimensional vectors of power series in $x$ such that $\left(1- q^{n/n_i} \left( \frac{q x}{g}\right)^{n} \right) S(x,k)$ and $(1- q^{n/n_i} \left( \frac{q x}{g}\right)^{n} ) (S'(x,k))$ (where $g=g_{\xi \chi^{M_ii}}$) are polynomials in $x$. Suppose that
\begin{equation}\label{ud-sv-k-fe} (S(x,k))_k = c x^{\tilde{d}} \left( \chi^{v (k+\ell)}(-1) E_{n_i}^n \Gamma_{k, \ell}(x, K) \right)_{k, \ell}   \Bigl(S'(\frac{g^2}{q^2x} ,\ell)\Bigr)_\ell.\end{equation} 
Let $d^{\mathrm{min}}$ be the minimum integer $d$ such that the coefficient of $x^d$ in $S(x,k)$ is nonzero for at least one $k$. Let $d^{'\mathrm{min}}$ be the minimum integer $d$ such that the coefficient of $x^d$ in $S'(x,k)$ is nonzero for at least one $k$.  
If $d^{\mathrm{min}}<\infty$ then \begin{equation} \label{di-bound-k} d^{\mathrm{min}}+d^{'\mathrm{min}} \leq \tilde{d}. \end{equation} \end{lemma}

\begin{proof} Since $d^{\mathrm{min}}<\infty$,  $S(x,k) $ is nonzero for at least one $k$, so by the invertibility of the Fourier transform there is $\eta\in \mu_n$ and $\kappa \in \{0,\dots,n_i-1\}$ such that 
\begin{equation}\label{k-Snonzero}  \sum_{k \equiv \kappa \bmod n_i}  \eta^{k} S(x,k) \neq 0.\end{equation}
The functional equation \eqref{ud-sv-k-fe} together with the definition of $E_{n_i}^n$ gives
\begin{equation*}
\begin{split}
\left(\sum_{k \equiv \kappa \bmod n_i}  \eta^{k} S(x,k) \right)_\kappa &= c x^{\tilde{d}} \Bigl(\sum_{\lambda=0}^{n_i-1} \sum_{\substack{k \equiv \kappa \bmod n_i \\ \ell \equiv \lambda \bmod n_i}} \eta^{k} \chi^{v(k+\ell)}(-1) S^{k+\ell-K, n} \Gamma_{\kappa, \lambda}(x,K) S'(\frac{g^2}{q^2x} ,\ell) \Bigr)_\kappa\\
&=c x^{\tilde{d}} \eta^{'K}   x^{\tilde{d}} \left( \Gamma_{\kappa, \lambda}( \eta'  x, K) \right) _{\kappa, \lambda } \Bigl( \sum_{\ell \equiv \lambda \bmod n_i} \eta^{-\ell} S'(\frac{g^2}{q^2x} ,\ell)\Bigr)_\lambda
\end{split}
\end{equation*}
where $\kappa$ and $\lambda$ range from $0$ to $n_i-1$ and $\eta' = \chi^v(-1) \eta$ is also an $n$th root of unity. For each $\lambda$, we know $(1- q^{n/n_i} \left( \frac{q x}{g}\right)^{n} )  \sum_{\ell \equiv \lambda \bmod n_i} \eta^{-\ell} S'(x ,\ell)$ is a polynomial in $x$, so  $\label{entries-to-obtain} (1- q^{n/n_i} \left( \frac{g}{qx }\right)^{n} ) \sum_{\ell \equiv \lambda \bmod n_i} \eta^{-\ell} S'(\frac{g^2}{q^2x} ,\ell)$ is a polynomial in $x^{-1}$.  Since $\sum_{k \equiv \kappa \bmod n_i}  \eta^{k} S(x,k)$ is obtained from these values by applying a matrix whose entries are polynomials in $x,x^{-1}$ divided by $1- q \left(\frac{ \eta'  qx}{g} \right)^{n_i}$, we can conclude that $(1- q^{n/n_i} \left( \frac{g}{qx }\right)^{n} )  (1- q \left(\frac{ \eta' qx}{g} \right)^{n_i}) \sum_{k \equiv \kappa \bmod n_i}  \eta^{k} S(x,k)$ is a polynomial in $x, x^{-1}$. On the other hand, we also know that $(1- q^{n/n_i} \left( \frac{qx}{g }\right)^{n} )  \sum_{k \equiv \kappa \bmod n_i}  \eta^{k} S(x,k)$ is a polynomial in $x$. Therefore, it follows that $(1 - q \left( \frac{ \eta' qx}{g} \right)^{n_i} )   \sum_{k \equiv \kappa \bmod n_i}  \eta^{k} S(x,k)$ is a polynomial in $x,x^{-1}$, where $(1 - q \left( \frac{\eta'  qx}{g} \right)^{n_i} )$ is the greatest common divisor of $(1- q^{n/n_i} \left( \frac{g}{qx }\right)^{n} )  (1- q \left(\frac{ \eta'  qx}{g} \right)^{n_i})$  and $(1 - q \left( \frac{qx}{g} \right)^{n_i} ) $ in the ring $\mathbb C[x,x^{-1}]$.

Thus we obtain an identity of vectors of polynomials
\[  \left(1 - q \left( \frac{\eta'  qx}{g} \right)^{n_i} \right)      \Bigl( \sum_{k \equiv \kappa \bmod n_i}  \eta^{k} S(x,k) \Bigr)_\kappa\] \[= \left(1 - q \left( \frac{\eta' qx}{g} \right)^{n_i} \right)    c x^{\tilde{d}} \eta^{'K} (  \Gamma_{\kappa,\lambda }(  \eta' x, K))_{\kappa, \lambda } \Bigl( \sum_{\ell \equiv \lambda \bmod n_i} \eta^{-\ell} S'(\frac{g^2}{q^2x} ,\ell)\Bigr)_\lambda.\]
We know the coefficient of $x^d$ in $S'(x,\ell)$ vanishes for $d< d^{'\mathrm{min}}$. Hence the coefficient of $x^d$ in $S'(\frac{g^2}{q^2x} ,\ell)$ vanishes for $d>- d^{'\mathrm{min}}$. After multiplying by $1- q \left( \frac{  \eta' q x}{g}\right)^{n_i}$, the entries of the scattering matrix $\Gamma(\eta' x,K)$ are polynomials in $x, x^{-1} $ of degree $\leq 1$ in $x$. Combined with the $x^{\tilde{d}}$ factor, it follows that the coefficient of $x^d$ in $\left(1 - q \left( \frac{ \eta' qx}{g} \right)^{n_i} \right)      \sum_{k \equiv \kappa \bmod n_i}  \eta^{k} S(x,k)$ vanishes for $d> \tilde{d}+1 -d^{'\mathrm{min}}$. If \eqref{di-bound-k} fails so that $\tilde{d}- d^{'\mathrm{min}}<d^{\mathrm{min}}$, we conclude that the coefficient of $x^d$ in $\left(1 - q \left( \frac{\eta' qx}{g} \right)^{n_i} \right)      \sum_{k \equiv \kappa \bmod n_i}  \eta^{k} S(x,k)$  vanishes for $d \geq d^{\mathrm{min}}+1$. However, we already know that the coefficient of $x^d$ vanishes for $d< d^{\mathrm{min}}$. Hence this coefficient is nonzero only for $d= d^{\mathrm{min}}$, and by \eqref{k-Snonzero} must in fact be nonzero for this $d$.

We now use the inverse matrix to express $ \sum_{\ell \equiv \lambda \bmod n_i} \eta^{-\ell} S'(x,\ell)$ in terms of  $\sum_{k \equiv \kappa \bmod n_i}  \eta^{k} S(x,k)$. By Lemma \ref{scattering-matrix-inverse}, the inverse matrix $(\Gamma(\eta' x, K))^{-1}$ is $\Gamma( \frac{g^2}{q^2 \eta'  x}, K)$. This implies 
\[ \Bigl( \sum_{\ell \equiv \lambda \bmod n_i} \eta^{-\ell} S'(\frac{g^2}{q^2x} ,\ell)\Bigr)_\lambda = c^{-1} x^{-\tilde{d}} (\eta')^{-K} \left( \Gamma_{\lambda,\kappa}(  \frac{g^2}{q^2 \eta' x}, K) \right) _{\lambda, \kappa }  \left(\sum_{k \equiv \kappa \bmod n_i}  \eta^{k} S(x,k) \right)_{\kappa} .\]
which after the substitution $x \mapsto \frac{g^2}{q^2 x}$ gives
\[  \Bigl( \sum_{\ell \equiv \lambda \bmod n_i} \eta^{-\ell} S'(x ,\ell)\Bigr)_\lambda = c^{-1} \left( \frac{q^2}{g^2} \right)^{\tilde{d}} x^{\tilde{d}} (\eta')^{-K} \left( \Gamma_{\lambda,\kappa}( (\eta')^{-1} x, K) \right) _{\lambda, \kappa }  \left(\sum_{k \equiv \kappa \bmod n_i}  \eta^{k} S(\frac{q^2}{g^2x},k) \right)_{\kappa} .\]



Using that the coefficient of  $x^d $ in  $\left( 1 - q \left( \frac{ \eta' qx }{g }\right)^{n_i} \right) \sum_{\ell \equiv \lambda \bmod n_i} \eta^{-\ell}S(x,\ell) $ is nonzero only for $d = d^{\mathrm{min}}$, we see that the coefficient of $x^d$ in  $\left( 1 - q \left( \frac{g }{q (\eta')^{-1} x }\right)^{n_i} \right) \sum_{\ell \equiv \lambda \bmod n_i} \eta^{-\ell} S(\frac{g^2}{q^2x} ,\ell)$ is nonzero only for $d = -d^{\mathrm{min}}$ (and is in fact nonzero for this $d$ for some $\lambda$).

By assumption, $ \left(\left(1- q^{n/n_i} \left( \frac{\eta' q x}{g}\right)^{n} \right)  \sum_{k \equiv \kappa \bmod n_i}  \eta^{-k} S'(x,k) \right)_\kappa $ is a vector of polynomials in $x$.

By the functional equation and the above analysis of $\left( 1 - q  \left( \frac{ \eta' qx }{g }\right)^{n_i} \right) \sum_{\ell \equiv \lambda \bmod n_i} \eta^{-\ell}S(x,\ell)$, this vector is a nontrivial $\mathbb C$-linear combination of the columns of $\Gamma((\eta')^{-1} x,K)$ times $$ x^{ \tilde{d}- d^{\mathrm{min}} } \frac{ 1 - q^{n/n_i} \left( \frac{ qx}{g }\right)^n }{ 1 - q \left( \frac{g }{q (\eta')^{-1} x}\right)^{n_i} } .$$ To check this is impossible, it suffices to check that a nontrivial $\mathbb C$-linear combination of the columns of $\Gamma((\eta')^{-1} x,K)$ is nontrivial mod $1 - q\left( \frac{g }{q (\eta')^{-1}x}\right)^{n_i}$, as then multiplying by a polynomial relatively prime to $1 - q\left( \frac{g }{(\eta')^{-1} q x}\right)^{n_i}$ and then dividing by $1 - q\left( \frac{g }{ (\eta')^{-1}q x}\right)^{n_i}$  will always produce a non-polynomial rational function.

The $\kappa,\lambda$'th entry of $\Gamma(x,K)$ is a sum of powers of $x$ congruent to $k +\ell-K$ modulo $n_i$, which implies different columns of the scattering matrix will not cancel, even modulo $1 - q\left( \frac{g }{(\eta')^{-1} qx_i}\right)^{n_i}$, so it suffices to check that each column of the scattering matrix is nonzero modulo $1 - q\left( \frac{g }{(\eta')^{-1} qx_i}\right)^{n_i}$, which follows from the explicit descriptions of, for example, the diagonal entries of $\Gamma((\eta')^{-1} x,K)$, which are all nonzero modulo $1 - q\left( \frac{g }{(\eta')^{-1}qx_i}\right)^{n_i}$.\end{proof}

\begin{lemma}\label{uniquely-determined-step} Let $E= (e_1,\dots, e_r)$ be a basis of $\mathbb Z^r$, $\hchi\colon \mathbb Z^r \times \mathbb Z^r \to \mu_\infty$ a bicharacter, and $i \in \{1,\dots, r\}$ such that $e_i$ satisfies either condition (1) or (2) of Theorem \ref{translation-key-step}. Let $\gene$ be a root of unity of even order $n$ such that every value taken by $\psi$ is a power of $\gene$. Fix a finite field $\mathbb F_q$ and a character $\chi \colon \mathbb F_q^\times \to \mathbb C^\times$ of order $n$.

Let $\tilde{Z} (\vec{x}, \vec{k}; E) $ and $\tilde{Z} (\vec{x}, \vec{k}; s_{i,E}(E))$ be two tuples of formal power series indexed by $\vec{k} \in \{0,\dots, n-1\}^r$ satisfying the functional equation stated in Theorem \ref{uniquely-determined}.


Let $d_1,\dots, d_{i-1}, d_{i+1} ,\dots, d_r$ be a tuple of nonnegative integers. Let $d^{\mathrm{min}}$ be the minimum integer $d$ such that the coefficients of $x_i^d \prod_{j \neq i}  x_j^{d_j}$ in $\tilde{Z} (\vec{x}, \vec{k}; E) $ and $Z(\vec{x}, \vec{k} ,q, \chi, M^{\hchi,E,\gene})$ differ for at least one $\vec{k}$. Let $d^{'\mathrm{min}}$ be the minimum integer $d'$ such that the coefficients of  of $x_i^{d'} \prod_{j \neq i}  x_j^{d_j}$ in $\tilde{Z} (\vec{x}, \vec{k}; s_i(E) ) $ and $Z(\vec{x}, \vec{k} ,q, \chi, M^{\hchi,s_{i,E}(E),\gene})$ differ for at least one $\vec{k}$. 

If $d^{\mathrm{min}}<\infty$ then \begin{equation} \label{di-bound} d^{\mathrm{min}}+d^{'\mathrm{min}} \leq \sum_{j\neq i }  d_j m_{ij}.\end{equation}
\end{lemma}

Lemma \ref{uniquely-determined-step} implies two weaker statements that we will also use. The first is that if $d^{\mathrm{min}}<\infty$ then $d^{'\mathrm{min}}<\infty$. The second is that if $\tilde{Z} (\vec{x}, \vec{k}; s_{i,E}(E) ) =Z(\vec{x}, \vec{k} ;q, \chi, M^{\hchi,s_i(E),\gene})$ then $\tilde{Z} (\vec{x}, \vec{k}; E) = Z(\vec{x}, \vec{k}; q, \chi, M^{\hchi,E,\gene})$. Both of these have shorter proofs, but we will prove them all together as the strongest statement is needed.

\begin{proof} After subtracting  $Z(\vec{x}, \vec{k} , q, \chi, M^{\hchi,E,\gene})$ from $\tilde{Z} (\vec{x}, \vec{k}; E) $, and the same for $s_{i,E}(E)$, since the functional equation is linear and is known to hold for the multiple Dirichlet series $Z$, the functional equation still holds for the new values of $\tilde{Z} (\vec{x}, \vec{k}; E) $ and $\tilde{Z} (\vec{x}, \vec{k}; s_{i,E}(E))$. Furthermore $d^{\mathrm{min}}$ becomes the minimum integer $d$ such that the coefficient of $x_j^d \prod_{j \neq i}  x_j^{d_j}$ in $\tilde{Z} (\vec{x}, \vec{k}; s_{i,E}(E) ) $ is nonzero for at least one $\vec{k}$ and $d^{'\mathrm{min}}$ becomes the minimum integer $d'$ such that the coefficient of $x_j^{d'} \prod_{j \neq i}  x_j^{d_j}$ in $\tilde{Z} (\vec{x}, \vec{k}; s_{i,E}(E) ) $ is nonzero for at least one $\vec{k}$. 

We will work with this definition of $d^{\mathrm{min}}$ and $d^{'\mathrm{min}}$ and show \eqref{di-bound}.

To do this, we fix $k_1,\dots, k_{i-1}, k_{i+1},\dots, k_r$ such that the coefficient of $x_j^{d_{\mathrm{min}}} \prod_{j \neq i}  x_j^{d_j}$ in $\tilde{Z} (\vec{x},\vec{k}; s_{i,E}(E) ) $ is nonzero for at least one $k_i$. We view \[ \tilde{Z} (\vec{x}, k_1,\dots,k_{i-1},k,k_{i+1},\dots, k_r ; E) \] as a power series in $x_1,\dots,x_{i-1} ,x_{i+1},\dots,x_r$ with coefficients power series of $x_i$ and set $S(x_i, k)$ to be the coefficient of $ \prod_{j \neq i}  x_j^{d_j}$ in \[\tilde{Z} (\vec{x}, k_1,\dots,k_{i-1},k,k_{i+1},\dots, k_r ; E).\] Similarly, we set $S'(x_i,k)$ to be the coefficient of $ \prod_{j \neq i}  x_j^{d_j}$ in \[ \tilde{Z} (\vec{x}, k_1,\dots,k_{i-1},k,k_{i+1},\dots, k_r ; s_{i,E}(E) ).\]

We check that $S$ and $S'$ satisfy the assumptions of Lemma \ref{uniquely-determined-sv-k} in case (1) and satisfy the assumptions of Lemma \ref{uniquely-determined-sv-d} in case (2), taking $\tilde{d} = \sum_{j\neq i }  d_j m_{ij}$. Having done this, \eqref{di-bound-k} or \eqref{di-bound-d} respectively will imply  \eqref{di-bound}, since restricting attention to the fixed values $k_1,\dots, k_{i-1}, k_{i+1},\dots, k_r$ preserves $d^{\mathrm{min}}$ and can only increase $d^{'\mathrm{min}}$. 

The boundedness of terms with nonvanishing coefficients assumption in the functional equation exactly implies the polynomiality assumption in Lemma \ref{uniquely-determined-sv-k} or \ref{uniquely-determined-sv-d}. To verify the functional equation itself, a key is that $\sigma_i$ sends $x_i$ to $\frac{ g^2}{q^2x_i}$ in case (1), matching the functional equation in Lemma \ref{uniquely-determined-sv-k}, and $\sigma_i$ sends $x_i$ to $\frac{1}{qx_i}$ in case (2), matching the functional equation in Lemma \ref{uniquely-determined-sv-d}. Furthermore, $\sigma_i$ multiplies each $x_j$ variable for $j\neq i$ by $x_i^{m_{ij}}$ times a constant. The $x_i^{m_{ij}}$ terms give a multiplicative factor of $x_i^{ \sum_{j\neq i }  d_j m_{ij}} = x_i^{\tilde{d}}$ while the constant may be absorbed into $c$ in Lemma \ref{uniquely-determined-sv-k} or $\omega$ in Lemma \ref{uniquely-determined-sv-d}. So we indeed obtain the functional equations from these lemmas. \end{proof}




\begin{proof}[Proof of Theorem \ref{uniquely-determined}] The fact that $Z(\vec{x}, \vec{k}; q, \chi, M^{\hchi,E,\gene})$ satisfy these equations is the content of Theorem \ref{translation-key-step}, so we need only prove uniqueness.  In other words, we fix a tuple of formal power series $ \tilde{Z} (\vec{x},\vec{k}; E)$, assume the functional equations, and must prove \eqref{uniqueness-eq}.

A natural strategy of proof, previously used in the multiple Dirichlet series literature, is an induction on the degree of monomials: We check that if the coefficients of all monomials of total degree $<d$ in the two sides of \eqref{uniqueness-eq} agree then the same is true for monomials of total degree $\leq d$. For a variant, we can consider a weighted total degree where different variables are assigned different weights. The induction could be equivalently viewed as an argument by contradiction where we examine the monomial whose coefficients in the two sides of \eqref{uniqueness-eq} differ of the least total degree and derive a contradiction.

Our proof is philosophically similar but more indirect: Rather than considering the least total degree we define a function on $\mathbb R^r$ whose value at each point can be interpreted as the minimum of a weighted total degree. Examining the minimizers at each point, we will derive a contradiction by checking that this function cannot be concave and also must be concave.

Assume for contradiction that \eqref{uniqueness-eq} is false for some $ E'\in B_{\hchi,E}, \vec{k} \in \{0,\dots, n-1\}^r$. It is straightforward to see that for each $E'\in B_{\hchi,E}$, there is some $\vec{k} \in \{0,\dots, n-1\}^r$ for which \eqref{uniqueness-eq} fails, since by Lemma \ref{uniquely-determined-step}, if \eqref{uniqueness-eq} fails for $s_{i,E'}(E)$ for some $\vec{k}$, then it also fails for $E'$, for some $\vec{k}$.

For each $E'=(e_1,\dots,e_r)\in B_{\hchi,E}$, we define a function $F_{E'}$ on $C_{E'}$: For $v \in C_{E'} $, define $F_{E'}(v)$ to be the minimum value of $\sum_{j=1}^r d_j  v\cdot e_j $ over the set of pairs of $\vec{k} \in \{0,\dots, n-1\}^r$, $ \vec{d} \in \mathbb Z^r$ such that the coefficients of $\prod_{i=1}^r x_i^{d_i}$ in $\tilde{Z} (\vec{x},\vec{k}; E')$ and $Z(\vec{x}, \vec{k} ;q, \chi, M^{\hchi,E',\gene}) $ are not equal.

 There are at most finitely many values of $\vec{d}$ that can contribute to the minimization in the definition of $F_{E'}$, since every infinite subset of $\mathbb N^r$ contains two different elements, one of which has all entries greater than or equal to the corresponding entries of the other.  Because \eqref{uniqueness-eq} fails for each $E'$, the set minimized over is not empty. Hence $F_{E'}$ is a well-defined piecewise linear function, and, in particular, is continuous. Since the minimum of an arbitrary set of linear functions is concave, $F_{E'}$ is a concave function. 

We will show that there is a single continuous function $F$ on $\mathbb R^r$ whose restriction to $C_{E'}$ is $F_{E'}$ for each $E'$. We will show $F$ cannot be concave and then derive a contradiction from this by showing that $F$ must be concave. 

Fix two adjacent cones $C_{E'}$ and $C_{ s_{i,E'}(E')}$. Let us check that $F_{E'}$ and $F_{s_{i,E'}(E')}$ agree on the intersection of $C_{E'}$ and $C_{ s_{i,E'}(E')}$. Fix a point $v$ in the boundary. Assume the minimum in the definition of $F_{E'}(v)$ is obtained by some pair $\vec{k},\vec{d}$.  Then it follows from Lemma \ref{uniquely-determined-step} that the coefficients of some monomial $\prod_{j=1}^r x_j^{d_j'}$ in $\tilde{Z} (\vec{x},\vec{k}'; s_{i,E'}(E'))$ and $Z(\vec{x}, \vec{k}' ; q, \chi, M^{\hchi, s_{i,E'}(E'),\gene}) $  are different, with $d_j'=d_j$ for all $j\neq i$ (because $d^{\mathrm{min}}<\infty$ and so $d^{'\mathrm{min}}<\infty$). Then since $v$ lies on the boundary, $v \cdot e_i=0$ and thus $v \cdot e_j = v \cdot s_{i,E'} (e_j)$ for all $j$.  Hence  \[ F_{ s_{i,E'}(E')} (v) \leq \sum_{j=1}^r d_j' v \cdot e_j = \sum_{j=1}^r d_j v \cdot e_j  = F_{E'}(v) .\] A symmetrical argument gives $ F_{E'}(v) \leq F_{ s_{i,E'}(E')} (v) $ and hence $ F_{E'}(v) = F_{ s_{i,E'}(E')} (v) $.

Now consider two cones $C_{E'}$ and $C_{E^*}$ that intersect. Again we check that $F_{E'}$ and $F_{E^*}$ agree on the intersection of $C_{E'}$ and $C_{E^*}$. To do this, observe that their interesction is a lower-dimensional face. Since a neighborhood of that face is connected, we can find a sequence of cones, each adjacent to the next, starting with $C_{E'}$ and ending with $C_{E^*}$ that all contain that face. Since each function agrees with the next on that face, $F_{E'}$ and $F_{E^*}$ agree as well.

Since each $v \in \mathbb R^r$ is in $C_{E'}$ for some $E'$ by Proposition \ref{cone-union-lemma}, and the value $F_{E'}(v)$ is independent of the chosen $E'$, there exists a unique function $F$ with $F (v) = F_{E'}(v)$ for all $v\in C_{E'}$. Since $F$ is continuous restricted to each $C_{E'}$, and the $C_{E'}$ are finitely many closed sets that cover the space, $F$ is continuous.

We saw above that $F_{E'}(v)\geq 0$ for all $v$ and $F_{E'}(0)=0$.  From Proposition \ref{cone-union-lemma} it follows that $F(v) \geq 0$ for all $v$ and $F(0)=0$. On the other hand, since we assumed that the constant coefficient of $\tilde{Z} (\vec{x},\vec{k}; E)$ is $1$, which agrees with the constant coefficient of  $Z(\vec{x}, \vec{k} q, \chi, M^{\hchi,E,\gene}) $, we have $F(v)>0$ for any $v$ in the interior of the cone $C_{E}$. S This implies that $F$ is not a concave function, by considering $0 = \frac{1}{2} v+ \frac{1}{2} (-v)$. 

The final step will be similar to the argument that $F$ is continuous, but more subtle. Since $F$ is not concave, there must be two points where the inequality for concave functions fails. We can perturb those points so the line segment between them does not intersect the codimension $\geq 2$ faces of the $C_{E'}$ and the inequality still fails. Since we have already checked concavity in the interior of each cone, the line segment cannot lie in the interior of a cone and thus must intersect a boundary between cones. If the line segment intersects multiple boundaries, we can pass to a shorter line segment intersecting a single boundary where the inequality fails. Thus there are points $v$ in the interior of a cone $C_{E'}$ and $w$ in the interior of an adjacent cone $C_{s_i(E')}$ such that $F( \lambda v + (1-\lambda ) w) < \lambda F(v) + (1-\lambda) F(w)$ for some $\lambda \in [0,1]$. Without loss of generality, $\lambda v + (1-\lambda) w \in C_{E'}$. We will derive a contradiction from this.

Let the minimum in the definition of $F(\lambda v + (1-\lambda) w) )$ be attained at $\vec{k}, \vec{d}$. We clearly have $F(v) \leq \sum_{i=1}^r d_i  v \cdot e_i$. Certainly also $d_i = d^{\mathrm{min}}$ in the sense of Lemma \ref{uniquely-determined-step}. Thus by Lemma \ref{uniquely-determined-step}, the coefficients in in $\tilde{Z} (\vec{x},\vec{k}'; s_{i,E'}(E'))$ and $Z(\vec{x}, \vec{k}' ;q, \chi, M^{\hchi, s_{i,E'}(E'),\gene}) $ of the monomial $x_i^{ d^{'\mathrm{min}}} \prod_{j \neq i} x_i^{d_i}$ differ for $d^{'\mathrm{min}} \leq \prod_{j\neq i} d_{j} m_{ij} -d_i$. 

Thus we have
\[ F(w) \leq  d^{'\mathrm{min}} w \cdot s_{i,E} (e_i) + \sum_{j \neq i}  d_j w \cdot s_{i, E} (e_j) = - d^{'\mathrm{min}}  w \cdot e_i + \sum_{j \neq i} d_j w \cdot (e_j + m_{ij} e_i)\] \[ = (\sum_{j \neq i} d_j m_{ij}  - d^{' \textrm {min}} ) w\cdot e_i + \sum_{j\neq i } d_j w \cdot e_j  \leq  d_i  w \cdot e_i + \sum_{j\neq i } d_j w \cdot e_j = \sum_{j=1}^r d_j w \cdot e_j \]
since  $\sum_{j \neq i} d_j m_{ij}  - d^{' \textrm {min}} \geq d_i $ by  Lemma \ref{uniquely-determined-step} and $w\cdot e_i \leq 0$.

It follows that
\[\lambda F(v) + (1-\lambda) F(w) \leq \lambda \sum_{j=1}^r d_j v \cdot e_j + (1-\lambda) \sum_{j=1}^r d_j w \cdot e_j = \sum_{j=1}^r d_j  ( \lambda v + (1-\lambda )w ) \cdot e_j =F(\lambda v + (1-\lambda) w)\] giving the desired contradiction and completing the proof. \end{proof}

\subsection{Examples, prior work, and moments}

In this subsection, we list some examples of axiomatic multiple Dirichlet series for which our functional equations imply meromorphicity, focusing on those that appeared already in prior work (or for which a number field analogue appeared), or those potentially useful for moment computations.

The moments of $L$-functions one hopes to study with multiple Dirichlet series are of the form \[\sum_{\substack{ f_{t+1},\dots, f_r \in \mathbb F_q[T]^+\\ \deg f_i = d_i }} \prod_{i=1}^{t} L(s_i, \chi^i_{ f_{t+1},\dots,f_r}) \] where $\chi^i_{f_{t+1},\dots,f_r}$ is a Dirichlet character depending on the variables $f_{T+1},\dots, f_r$, defined using some combination of power residue symbols. One studies these using the series
\[\sum_{\substack{ f_{t+1},\dots, f_r \in \mathbb F_q[T]^+}} \prod_{i=1}^{t} L(s_i, \chi^i_{ f_{t+1},\dots,f_r}) \prod_{i=t+1}^r q^{ - s_i \deg f_i} \] \[ =\sum_{\substack{ f_{t+1},\dots, f_r} \in \mathbb F_q[T]^+}\Bigl( \sum_{ f_{1},\dots, f_t \in \mathbb F_q[T]^+ }\prod_{i=1}^{t} \chi^i_{f_{t+1},\dots, f_r} ( f_i ) q^{ - s_i \deg f_i} \Bigr) \prod_{i=t+1}^r q^{ - s_i \deg f_i}  \] \[= \sum_{\substack{ f_1,\dots, f_r \in \mathbb F_q[t]^+}} \prod_{i=1}^{t} \chi^i_{f_{t+1},\dots, f_r} ( f_i )  \prod_{i=1}^r q^{ - s_i \deg f_i}\]
which can be analyzed by comparison with a multiple Dirichlet series whose coefficients agree with $\prod_{i=1}^{t} \chi^i_{f_{t+1},\dots, f_r} ( f_i )$ for $f_1,\dots, f_r$ relatively prime. These arise from axiomatic multiple Dirichlet series associated to a matrix $M$ such that $M_{ij}=0$ if $i,j \leq t$ or $i,j >t$. Such series have coefficients of the form $\prod_{ i=1}^t \prod_{j=t+1}^r \left( \frac{f_i}{f_j} \right)_\chi^{M_{ij}}$ so that we can take $\chi^i_{f_{t+1},\dots,f_r} (f) = \prod_{j=t+1}^r \left(\frac{f}{f_i}\right)_{\chi}^{M_{ij}}$.

Thus, to find axiomatic multiple Dirichlet series of relevance to moments, we can examine the classification of arithmetic root systems for Dynkin diagrams that form a bipartite graph with all vertices labeled with $-1$, as this ensures that the matrix $M$ satisfies the vanishing condition.

    \begin{example}\label{brubaker-example} We consider the multiple Dirichlet series arising from the arithmetic root system $\mathtt{g}(2,3)$. According to \cite[\S8.3.3]{AndruskiewitschAngiono}, this root system has bases with four different generalized Dynkin diagrams. 
\begin{align}\label{eq:dynkin-g(2,3)}
\begin{aligned}
&\xymatrix{ \overset{-1}{\underset{\ }{\circ}}\ar  @{-}[r]^{\overline{\zeta}}  &
	\overset{\zeta}{\underset{\ }{\circ}} \ar  @{-}[r]^{\zeta}  & \overset{-1}{\underset{\
		}{\circ}}}
& &\xymatrix{ \overset{-1}{\underset{\ }{\circ}}\ar  @{-}[r]^{\zeta}  &
	\overset{-1}{\underset{\ }{\circ}} \ar  @{-}[r]^{\zeta}  & \overset{-1}{\underset{\
		}{\circ}}}
\\
&\xymatrix{ \overset{-1}{\underset{\ }{\circ}}\ar  @{-}[r]^{\overline{\zeta}}  &
	\overset{-\overline{\zeta}}{\underset{\ }{\circ}} \ar  @{-}[r]^{\overline{\zeta}}  & \overset{-1}{\underset{\ }{\circ}}}
&
&\xymatrix@R8pt{ & \overset{\zeta}{\underset{\ }{\circ}} \ar@{-}[rd]^{\overline{\zeta}} &
	\\
	\overset{\zeta}{\underset{\ }{\circ}} \ar@{-}[rr]^{\overline{\zeta}} \ar@{-}[ru]^{\overline{\zeta}} & & \overset{-1}{\underset{\ }{\circ}}}.
\end{aligned}
\end{align}
Here $\zeta$ is a primitive third root of unity. Of greatest interest to us is the top-right diagram, as this one has each node labeled with $-1$. Set $\gene$ to be a primitive sixth root of unity whose square is $\zeta$. The matrix associated to the top-right diagram is \[ M_1= \begin{pmatrix} 0 & 2 & 0 \\ 2 & 0 & 2 \\ 0 & 2 & 0 \end{pmatrix}. \] Taking $q \equiv1 \bmod 12$ and letting  $\chi$ be a character of $\mathbb F_q^\times$ of order $6$, so that $\chi^2$ is a character of order $3$, the coefficients $a(f_1,f_2,f_3;q,\chi, M_1)$ agree for $f_1,f_2,f_3$ relatively prime with \begin{equation}\label{compare-Z1} \left( \frac{f_1}{f_2} \right)_{\chi^2} \left( \frac{f_3}{f_2} \right)_{\chi^2}\end{equation} and hence should be applicable to the moment of cubic Dirichlet $L$-functions \[\sum_{\substack { f_2 \in \mathbb F_q[T]^+\\ \deg f_2 =d}} L\left(s, \left( \frac{}{f_2} \right)_{\chi^2} \right)^2 .\]

A similar moment of cubic Dirichlet $L$-functions, over $\mathbb Q(\zeta_3)$, was considered by Brubaker in \cite{BrubakerThesis}. The series associated to $\mathtt{g}(2,3)$ is a direct function field analogue of a multiple Dirichlet series studied by Brubaker. In fact, Brubaker defines five  multiple Dirichlet series $Z_1,Z_3, Z_4, Z_5, Z_6$, and functional equations relating them, with $Z_1$ of immediate relevance to moments of cubic Dirichlet $L$-functions.

Correspondingly, \cite[\S8.3]{AndruskiewitschAngiono} divides the bases of the arithmetic root system $\mathtt{g}(2,3)$ into five types and describes the reflections in the Weyl groupoid relating them. Given a basis $E'=(e_1,\dots,e_r)$, the bicharacter $\hchi$ can be expressed as a matrix $\mathfrak q$ with $\mathfrak q_{ij} = \hchi(e_i,e_j)$. Such a matrix determines the generalized Dynkin diagram. \cite[\S8.3.3]{AndruskiewitschAngiono} considers four matrices $\mathfrak q^{(i)}$, corresponding to the Dynkin diagrams in \eqref{eq:dynkin-g(2,3)} numbered left to right and top to bottom, so $\mathfrak q^{(1)}$ is the top-left, $\mathfrak q^{(2)}$ is the top-right, $\mathfrak q^{(3)}$ is the bottom-left, and $\mathfrak q^{(4)}$ is the bottom-right.  There is also the matrix $\tau( \mathfrak q^{(1)})$ arising from $q^{(1)}$ by reversing the order of the rows and columns. Equivalently, this arises from the top-left Dynkin diagram, but with the nodes labeled from right to left instead of from left to right. The five types of bases producing these five possible matrices are labeled $a_1,\ldots, a_5$, with $a_1$ corresponding to $\mathfrak q^{(1)}$, $a_2$ corresponding to $\mathfrak q^{(2)}$, $a_3$ corresponding to $\mathfrak q^{(4)}$, $a_4$ corresponding to $\mathfrak q^{(3)}$, and $a_5$ corresponding to $\tau(\mathfrak q^{(1)})$. (Note that there is a misprint in \cite[\S8.3.3]{AndruskiewitschAngiono} which states that $a_3$ corresponds to $\mathfrak q^{(3)}$ and $a_4$ to $\mathfrak q^{(4)}$.)

The relations between these bases are described by the following diagram based on \cite[\S8.3.1]{AndruskiewitschAngiono}.

\begin{center}
	\begin{tabular}{c c c c c c c c c c c c}
		&&
		& $\overset{Z_4}{\underset{a_1}{\bullet}}$
		& \hspace{-5pt}\raisebox{3pt}{$\overset{1}{\rule{40pt}{0.5pt}}$}\hspace{-5pt}
		& $\overset{Z_1}{ \underset{a_2}{\bullet}}$
		& \hspace{-5pt}\raisebox{3pt}{$\overset{2}{\rule{40pt}{0.5pt}}$}\hspace{-5pt}
		& $\overset{Z_3} {\underset{a_3}{\bullet}}$ & &
		\\
		&& & {\scriptsize 3} \vline\hspace{5pt} & & {\scriptsize 3} \vline\hspace{5pt} & & &&
		\\
		&& &  $\overset{Z_6} {\underset{a_4}{\bullet}}$
		&\hspace{-5pt}\raisebox{3pt}{$\overset{1}{\rule{40pt}{0.5pt}}$}\hspace{-5pt}
		& $\overset{Z_5}{\underset{a_5}{\bullet}}$
		& & & &
	\end{tabular}
\end{center}

Here the edge labeled $1$ between the vertices labeled $a_1$ and $a_2$ denotes that for a basis $E'$ of type $a_1$, $s_{1,E'}(E') $ has type $a_2$, and for a basis $E'$ of type $a_2$, $s_{1,E'}(E') $ has type $a_1$. The same is true for the other edges. Vertices without an edge labeled by a particular $j \in \{1,2,3\}$ indicate that, for $E'$ of the corresponding type, $s_{j,E'}(E')$ is a basis of the same type. For the multiple Dirichlet series induced by the arithmetic root system, an edge indicates a functional equation of Dirichlet type while the lack of an edge indicates a functional equation of Kubota type. We have placed labels above the vertices stating the corresponding series from \cite{BrubakerThesis}. For example, the $Z_1$ above $a_2$ indicates that a function field analogue of Brubaker's $Z_1$ arises from bases of type $a_2$.   

To see why particular series correspond to particular diagrams, we first review some notation from \cite[\S2.1]{BrubakerThesis}. Brubaker considers a cubic character $\left(\frac{r}{d} \right)_3$, as well as a quadratic character $\left(\frac{r}{d} \right)_2$ and their product, a sextic character, $\left(\frac{r}{d} \right)_6=\left(\frac{r}{d} \right)_3\left(\frac{r}{d} \right)_2$. From our perspective, the cubic character will correspond to $\left(\frac{r}{d} \right)_{\chi^2}$ and the quadratic character to $\left(\frac{r}{d} \right)_{\chi^3}$ so the sextic character corresponds to their product $\left(\frac{r}{d} \right)_{\chi^5}$. Brubaker also defines Gauss sums $g(m,d)= g_3(m,d)$ and $g_6(m,d)$ by summing these characters against an additive character, so that $g_3(m,d)$ corresponds to $g_{\chi^2}(m,d)$ and $g_6(m,d)$ corresponds to $g_{\chi^5}(m,d)$.  He writes $\chi_d(r)$ as an alternate form for $\left(\frac{r}{d} \right)_3$ but we will avoid this because of the conflict of notation with our $\chi$.

Also, note that Brubaker's series involve characters $\psi_1,\psi_2$ with conductor dividing $9$. These exist to deal with primes lying above $3$. In the context of $\mathbb F_q[T]$ with $q \equiv 1\bmod 12$, those primes do not exist, so the characters $\psi_1,\psi_2$ are not necessary, and thus we simply ignore them when they appear in Brubaker's formulas.

The series $Z_1(s_1,s_2,w)$ which has the heuristic form \cite[\S1.3]{BrubakerThesis}  \[ \sum_{d,m,n} \frac{ \left(\frac{mn}{d} \right)_3  }{ \mathbb N d^w \mathbb Nm^{s_1} \mathbb N n^{s_2} }\]   which matches \eqref{compare-Z1} after sending $f_1$ to $m$, $f_3$ to $n$, and $f_2$ to $d$. As we already saw, \eqref{compare-Z1} arises from the top-right diagram of \eqref{eq:dynkin-g(2,3)}, i.e. from vertices of type $a_2$.

The series $Z_6(s_1,s_2,w)$ has the heuristic form  \[ \sum_{d,m,n} \frac{\overline{G_6(1,d)}  \overline{ \left(\frac{mn}{d} \right)_3 }  }{ \mathbb N d^w \mathbb Nm^{s_1} \mathbb N n^{s_2} }\] where the numerator $\overline{G_6(1,d)}  \overline{  \left(\frac{mn}{d} \right)_3 }$  corresponds to $\overline{ g_{\chi^5} (1, f_2) } \overline{ \left(\frac{f_1 f_3 }{f_2} \right)_{\chi^2}} $. By Lemma \ref{GaussSumLifting}, $\overline{ g_{\chi^5} (1, f_2) }$ is proportional to  $ \left(\frac{f_2'}{f_2}\right)_{\chi^{5+3}}= \left(\frac{f_2'}{f_2}\right)_{\chi^{2}}$ and its complex conjugate to  $\left(\frac{f_2'}{f_2}\right)_{\chi^{-2}}= \left(\frac{f_2'}{f_2}\right)_{\chi^{4}}$.  Furthermore  $\overline{ \left(\frac{f_1 f_3 }{f_2} \right)_{\chi^2}}=  \left(\frac{f_1 f_3 }{f_2} \right)_{\chi^4}$.
Thus $Z_6$ corresponds to a matrix \[ M_6 = \begin{pmatrix} 0 & 4 & 0 \\ 4 & 4 & 4 \\ 0 & 4 & 0 \end{pmatrix}\] where the off-diagonal $4$s arise from $\left(\frac{f_1 f_3 }{f_2} \right)_{\chi^4}$ and the diagonal $4$ from $\left(\frac{f_2'}{f_2}\right)_{\chi^{4}}$. This matches the bottom-left diagram of \eqref{eq:dynkin-g(2,3)} since $\overline{\zeta}=\gene^4$, i.e. matches bases of type $a_4$.

The series $Z_3(s_1,s_2,w)$ has the heuristic form \[ \sum_{d,m,n} \frac{G_3(d,mn) }{ \mathbb N d^w \mathbb Nm^{s_1} \mathbb N n^{s_2} }.\] The numerator $G_3(d,mn)$ matches in the function field setting $g_{\chi^2} (f_2,f_1f_3)$ which by Lemma \ref{GaussSumLifting} is proportional to  \[\left( \frac{ (f_1f_3)'}{f_1f_3} \right)_{\chi^{2+3}}  \left( \frac{ f_2}{ f_1 f_3} \right)_{\chi^2}^{-1} =\left( \frac{ (f_1f_3)'}{f_1f_3} \right)_{\chi^{5}}  \left( \frac{ f_2}{ f_1 f_3} \right)_{\chi^4} \] 
\[=\left( \frac{ f_1' f_3 + f_1 f_3' }{f_1} \right)_{\chi^{5}} \left( \frac{ f_1' f_3 + f_1 f_3' }{f_3 } \right)_{\chi^{5}} \left( \frac{ f_2}{ f_1} \right)_{\chi^4} \left( \frac{ f_2}{  f_3} \right)_{\chi^4} \] 
\[= \left(\frac{f_1'}{f_1} \right)_{\chi^5}  \left(\frac{f_3}{f_1} \right)_{\chi^5}  \left(\frac{f_1}{f_3} \right)_{\chi^5}  \left(\frac{f_3'}{f_3} \right)_{\chi^5}  \left( \frac{ f_2}{ f_1} \right)_{\chi^4} \left( \frac{ f_2}{  f_3} \right)_{\chi^4}  \]
\[ = \left(\frac{f_1'}{f_1} \right)_{\chi^5}  \left(\frac{f_1}{f_3} \right)_{\chi^{10}}  \left(\frac{f_3'}{f_3} \right)_{\chi^5}  \left( \frac{ f_1}{ f_2} \right)_{\chi^4} \left( \frac{ f_2}{  f_3} \right)_{\chi^4} \] since $q \equiv 1\bmod 12$. This corresponds to the matrix 
\[ M_3 =\begin{pmatrix} 5 & 4 & 4 \\ 4 & 0 & 4 \\ 4 & 4 & 5 \end{pmatrix}\] since $\chi^{10}=\chi^4$. This matches the bottom-right diagram in \eqref{eq:dynkin-g(2,3)} since $\gene^4 = \overline{\zeta}$ and $-\gene^5=\zeta$. 

The series $Z_4$ \cite[(3.4)]{BrubakerThesis} has the heuristic form 
\[ \sum_{d,m,n} \frac{G_3(mn^2, d) }{ \mathbb N d^w \mathbb Nm^{s_1} \mathbb N n^{s_2} }.\] The numerator $G_3( mn^2,d)$ matches in the function field setting $g_{\chi^2} ( f_1 f_3^2, f_2)$ which by Lemma \ref{GaussSumLifting} is proportional to \[ \left( \frac{f_2'}{f_2}\right)_{\chi^{2+3}} \left( \frac{ f_1 f_3^2}{f_2} \right)_{\chi^2}^{-1} = \left( \frac{f_2'}{f_2}\right)_{\chi^{5}} \left( \frac{ f_1 }{f_2} \right)_{\chi^4}\left( \frac{ f_3 }{f_2} \right)_{\chi^2} .\]
This corresponds to the matrix \[ M_4 =\begin{pmatrix} 0 & 4 & 0 \\ 4 & 5 & 2 \\ 0 & 2 & 0 \end{pmatrix}\] which matches the top-left diagram in \eqref{eq:dynkin-G-super} since $\gene^4=\overline{\zeta}, \gene^2=\zeta$, and $-\gene^5= \zeta$. The series $Z_5$ is identical with $n$ and $m$ switched, which corresponds to switching $f_1$ and $f_3$ and thus reversing the order of the rows and columns.

\cite[\S8.3.1]{AndruskiewitschAngiono} indicates the presence of ten functional equations relating the function field analogues of $Z_1,Z_3,Z_4,Z_5,Z_6$, five of Dirichlet type and five of Kubota type. Brubaker considers most of these functional equations. In~\cite[\S1.4]{BrubakerThesis} he considers four particularly important functional equations $A,B,C,D$ relating $Z_1, Z_3, Z_6$. In our setting, $B$ corresponds to the functional equation of Kubota type in the variable $f_2$ relating $a (f_1,f_2,f_3; q, \chi, M_6)$ to itself and $C$ corresponds to the functional equation of Dirichlet type in the variable $f_2$ relating $a (f_1,f_2,f_3; q, \chi, M_1)$ to $a (f_1,f_2,f_3; q, \chi, M_3)$ (i.e. the edge between $a_2$ and $a_3$ in the graph). The others are a bit more complicated: $A$ corresponds to the composition of the functional equation of Dirichlet type in the variable $f_1$ with the functional equation of Dirichlet type in the variable $f_3$ which together relate $a (f_1,f_2,f_3; q, \chi, M_1)$ to $a (f_1,f_2,f_3; q, \chi, M_6)$ (i.e. either of the two paths from $a_2$ to $a_4$ by edges labeled $1$ and $3$). Finally, $D$ corresponds to the composition of the functional equation of Kubota type in the variable $f_1$ with the functional equation of Kubota type in the variable $f_3$ which each relate $a (f_1,f_2,f_3; q, \chi, M_6)$ to itself. 

Other functional equations appear elsewhere. \cite[(3.4),(3.5),(3.6)]{BrubakerThesis} relate (by definition) the series $Z_4, Z_5, Z_6$ to $Z_1$. The equations \cite[(3.4) and (3.5)]{BrubakerThesis} relating $Z_4$ and $Z_5$ respectively to $Z_1$ correspond to functional equations of Dirichlet type in $f_1$ and $f_3$ respectively, while combining \cite[(3.6)]{BrubakerThesis} relating $Z_6$ to $Z_1$ with equations \cite[(3.4)]{BrubakerThesis} relating $Z_1$ and $Z_4$ gives a relation between $Z_6$ and $Z_4$ corresponding to the Dirichlet functional equation in $f_3$. A symmetrical construction relates $Z_6$ and $Z_5$ via a Dirichlet functional equation in $f_1$. Combined with $C$, this gives all the expected functional equations of Dirichlet type.

\cite[Proposition 6.2]{BrubakerThesis} gives a functional equation for $Z_4$ in the $w$ variable that corresponds to a functional equation of Dirichlet type in $f_2$. \cite[\S6.3]{BrubakerThesis} gives the analogue for $Z_5$. \cite[Proposition 6.9]{BrubakerThesis} gives a functional equation for $Z_6$ in the $w$ variable, corresponding to the final functional equation of Kubota type in the variable $f_2$. The remaining functional equations of Kubota type are in the variables $f_1$ and $f_3$ and apply to the series corresponding to $M_3$, i.e. these should give functional equations of $Z_3$ in the $s_1$ and $s_2$ variables whose composition is $D$. In \cite[\S6.5.1]{BrubakerThesis}, Brubaker makes progress towards constructing these functional equations but does not fully derive them, instead obtaining only a complicated formula \cite[(6.20)]{BrubakerThesis} for $Z_3$ which is sufficient to prove the desired meromorphic continuation.  Together, these give all the functional equations that appear in \cite{BrubakerThesis}, so all these functional equations may be predicted from the arithmetic root system.  \end{example}

 \begin{example}\label{wgmds} One can check that the multiple Dirichlet series arising from the arithmetic root systems of Cartan type $\mathtt A_\theta, \mathtt B_\theta, \mathtt C_\theta, \mathtt D_\theta, \mathtt E_\theta, \mathtt F_4,$ and $ \mathtt G_2$ are the (untwisted) Weyl group multiple Dirichlet series. Weyl group multiple Dirichlet series over function fields were defined for arbitrary root systems by Holley Friedlander in \cite{HFriedlander23}, after prior work on specific root systems and in the number field case. It is possible to check that Friedlander's series, in the untwisted case, agree with the multiple Dirichlet series we have defined after a simple change of variables. 

For $\hchi, E$ arising from an arithmetic root system of Cartan type (i.e. arising from a root system and parameter $v$ via the construction of Example \ref{cartan-example}), we have an identity of series in variables $s_1,\dots,s_r$
\begin{equation}\label{comparison-to-friedlander} Z^* (\vec{s}; \vec{1} ) = \sum_{f_1,\dots, f_r \in \mathbb F_q[T]^+ } a( f_1,\dots, f_r; q, \chi, M^{\hchi, E})  \prod_{i=1}^r ( g_{ \chi^{M_{ii}} \xi} q^{-s_i})^{ \deg f_i} \end{equation}
where $Z^*(\vec{s}; \vec{1} )$ are the normalized untwisted Weyl group multiple Dirichlet series defined in \cite[p. 347] {HFriedlander23}. Here $\vec{1}$ is used to represent a vector of $r$ ones, and we set the twisting parameter $\vec{m}$ of \cite{HFriedlander23}  to $\vec{1}$ to get the untwisted series. 

More precisely, the series $Z^*(\vec{s}; \vec{1})$ depends on some data, which we fix as follows: Take $n$ to be the order of $\mathfrak q=v^2$, which is always a divisor of our $n$.  If $\mathfrak q = \gene^d$ then the embedding $\epsilon\colon \mu_n(\mathbb F_q)\to \mathbb C^\times$ of \cite[\S2]{HFriedlander23} needs to be chosen so that the character $\chi: \mathbb F_q^\times \to \mu_n(\mathbb F_q)^\times \to \mathbb C^\times$ of \cite[\S2]{HFriedlander23} agrees with $\chi^e$. We sketch the computations needed to verify this below.

One can also relate the coefficients $H(\vec{f};\vec{1} )$ defined in \cite{HFriedlander23} to our $ a( \vec{f} ; q, \chi, M^{\hchi, E})$, but this relation is a little complicated as $H(\vec{f} ;\vec{1})$ are not the coefficients of the normalized series $Z^* (\vec{s}; \vec{1})$ but instead another series that differs from it by some zeta factors.  Doing this involves checking that the twisted multiplicativity relation \cite[(16)]{HFriedlander23} agrees with our twisted multiplicativity axiom, and using the local functional equation \cite[(15)]{HFriedlander23} where the argument we sketch below uses the global functional equation.

These conventions make the $n$th power residue symbol $\left( \frac{f}{g} \right)$ of \cite{HFriedlander23} agree with our $\left(\frac{f}{g}\right)_{\chi^e}$. Note that \[ \gene^{M_{ii}} = - \hchi(e_i,e_i) = -\mathfrak q^{\frac{\langle \alpha_i,\alpha_i \rangle }{2}} = \gene^{ d \frac{\langle \alpha_i,\alpha_i \rangle}{2} + \frac{n}{2}}\] so that \[\chi^{M_{ii}} \xi = (\chi^e)^{ \frac{ \langle \alpha_i,\alpha_i \rangle}{2}} =\epsilon^{\frac{ \langle \alpha_i,\alpha_i \rangle}{2}} =\epsilon^{||\alpha_i||^2 }\]  where one adopts as in  \cite{HFriedlander23} the convention that short roots have length $1$ so that $||\alpha_i||^2=\frac{ \langle \alpha_i,\alpha_i\rangle}{2}$. The quantity $n(\alpha_i)$ of \cite{HFriedlander23} matches our $n_i$.

For integers $I_1,\dots I_r$ with $0 \leq I_i < n_i$ for all $i$, \cite{HFriedlander23} defines a series $Z^* ( s_1,\dots,s_r;1,\dots,1, I_1,\dots,I_r)$ by taking all terms in $Z^* ( s_1,\dots,s_r;1,\dots,1)$ with power of $q^{-s_i}$ congruent to $I_i$ mod $n_i$. Thus \eqref{comparison-to-friedlander} is equivalent to
\begin{equation}\label{comparison-to-friedlander-2}  Z^* ( \vec{s} ;\vec{1} , \vec{I} )= \sum_{ \substack{ k_1,\dots, k_r \in \{0,\dots,n-1\} \\ k_i \equiv I_i \bmod n_i}} Z ( g_{ \chi^{M_{11}} \xi} q^{-s_1}, \dots, g_{ \chi^{M_{rr}} \xi} q^{-s_r}; \vec{k} ) .\end{equation}   
To check \eqref{comparison-to-friedlander-2}, the key observation is that the functional equation \cite[(20)]{HFriedlander23} is equivalent to the functional equation using our scattering matrix:
\begin{equation}\label{comparison-vfe} (Z^* ( \vec{s} ;\vec{1} ,\vec{I} ))_{I_i} = ( \Gamma_{I_i, I_i'}( g_{\chi^{M_ii}\xi} q^{-s_i}, J_i (1,\dots,1,\vec{I} )))_{I_i, I_i'}  (Z^*( \sigma_i \vec{s}; \vec{1}, I_1,\dots, I_{i-1}, I_i' , I_{i+1},\dots, I_r))_{I_{i'} }\end{equation}
where we adopt the following notation from \cite{HFriedlander23}: $c(j,i)$ is the $j,i$ coefficient of the Cartan matrix of the root system, $ J_i (1,\dots,1,\vec{I} )=-\sum_{j \neq i} c(j,i) I_j$, and $(\sigma_i \vec{s})_j =s_j - c(j,i) s_i-1$. 

We can check \eqref{comparison-to-friedlander-2} using \eqref{comparison-vfe}. Indeed, the functional equation with the $n_i \times n_i$ scattering matrix implies a functional equation with an $n\times n$ scattering matrix for the series obtained from  $Z^* ( s_1,\dots,s_r;1,\dots,1, I_1,\dots,I_r)$ by taking all terms where the power of $q^{-s_i}$ is congruent to $k_i$ mod $n$. This exactly agrees with our multivariate functional equation of Kubota type  after the change of variables $x_i \mapsto g_{ \chi^{M_{ii}} \xi} q^{-s_i}$. To check this agreement, one must note that the $\chi^{ v_i(k_i+\ell_i)}(-1)$ factor in our multivariate functional equation may be ignored in this case by the congruence assumption on $q$, that $J_i(1,\dots,1,\vec{I} )$ is congruent to $-K$ mod $n_i$ since $c(j,i)$ agrees with $-n_{ji}$, and that the change of variables interchanges our $\sigma_i$ and the $\sigma_i$ of \cite{HFriedlander23}. The multivariate functional equation this way may be used to apply Proposition \ref{uniquely-determined} to the series obtained from $Z^*$ by taking all terms where the power of $q^{-s_i}$ is congruent to $k_i$ mod $n$ after an inverse change of variables, thereby proving \eqref{comparison-to-friedlander-2}. 

To check \eqref{comparison-vfe}, observe that \cite[(20)]{HFriedlander23} has two terms, with the $P$ terms corresponding to the diagonal entries of the scattering matrix, the $Q$ terms corresponding to the entries with $\ell \equiv 1+K-k \bmod n_i$, and the sum of these two terms corresponding to the special entries with $k\equiv \ell \equiv 1+K-k\bmod n_i$. The formula for $P$ agrees with our formula for the diagonal entry, the formula for $Q$ agrees with our formula for the off-diagonal entry, and the sum of these agrees with our formula for the special entries. The $m_i$ terms in \cite[(20)]{HFriedlander23} may be ignored as we have set these variables to $1$.



 For these series to be related to moments, it is necessary to have $\qv^{ \frac{\langle \alpha_i, \alpha_i \rangle}{2}}=-1$ for all $i$. Since we always have $\langle \alpha_i, \alpha_i \rangle=2$ for short roots $\alpha_i$, this gives $\qv^2=-1$. In the $\mathtt B_\theta,\mathtt C_\theta$ case we have $\langle \alpha_i, \alpha_i \rangle =4$ for long roots $\alpha_i$, so these have $q^{ \frac{\langle \alpha_i, \alpha_i \rangle}{2}}=1\neq -1$ and are not directly helpful for moments, and the $\mathtt G_2$ arithmetic root system reduces to $\mathtt A_2$ if $\qv^2=-1$, so only the root systems $\mathtt A_\theta, \mathtt D_\theta, E_\theta$ are relevant.

 Since all values taken by $\psi$ are powers of $\qv$, we may take $\gene=\qv$ which has order $4$, so $M_{ij}$ is either $0$ or $4$ for all $i,j$, equal to $1$ if $i$ and $j$ are connected by an edge in the Dynkin diagram so that $\hchi(e_i,e_j) \hchi(e_j,e_i) = \qv^{2 \langle \alpha_i,\alpha_j\rangle}= \qv^2$ and $0$ if $i$ and $j$ are not connected by an edge (including if $i=j$). Since $\gene$ has order $2$, we may take $\chi$ to be a quadratic character, and the relevant moments are moments of quadratic $L$-functions. In particular, the $\mathtt A_2, \mathtt A_3,$ and $\mathtt D_4$ Dynkin diagrams correspond to the first, second, and third moments of quadratic $L$-functions respectively. These multiple Dirichlet series, and their applications to moments, have been well studied in prior work \cite{DiaconuGoldfeldHoffstein, Diaconu, DiaconuWhitehead}.\end{example}

\begin{example}\label{super-An} The multiple Dirichlet series arising from the arithmetic root systems $\mathtt A(j\mid j)$ and $\mathtt A(j \mid j-1)$ are relevant to moments of Dirichlet $L$-functions of arbitrary order. In this case, the root system depends on two parameters $\theta$ and $j$ with $j \leq \frac{\theta+1}{2}$, in addition to a root of unity $\qq\neq \pm 1$,  and is written $\mathtt A(j\mid \theta-j)$.

The relevant Dynkin diagrams are defined in \cite[\S5]{AndruskiewitschAngiono} in terms of a subset $\mathbb J$ of $\{1,\dots, \theta\}$ to consist of nodes $1,\dots, \theta$ with the node $i$ labeled by $q_{ii}$ and edges between nodes $i$ and $i+1$ labeled by $\tilde{q}_{i(i+1)}$ where $q_{ii}$ and $\tilde{q}_{i(i+1)}$ satisfy certain rules. Specifically, we have $q_{ii}=-1$ and $\tilde{q}_{(i-1)i} = \tilde{q}_{i(i+1)}^{-1}$ if $i\in \mathbb J$ and $\tilde{q}_{(i-1)i } = \tilde{q}_{ii}^{-1}= \tilde{q}_{i(i+1)}$ if $i\notin\mathbb J$. These rules fix all $q_{ii}$ and $\tilde{q}_{i(i+1)}$ by backwards induction once we add the initial condition  $\tilde{q}_{(\theta-1)\theta}= \qq$ if $\theta\in \mathbb J$ and $q_{\theta\theta}=\tilde{q}_{(\theta-1)\theta}^{-1} = \qq$ if $\theta\notin \mathbb J$.

The set of Dynkin diagrams associated to $\mathtt A(j\mid \theta+1-j)$ are those arising from $\mathbb J  \subseteq \{1,\dots, \theta\}$ such that, writing $\mathbb J = \{i_1,\dots, i_k\}$ with $i_1< \dots < i_k$, we have $ \left| \sum_{l=1}^k (-1)^l i_l\right|$ equal to either $j$ or $\theta+1-j$.

In particular, the rules ensure that all edges are labeled by either $\qq$ or $\qq^{-1}$ and a node $i$ is labeled by $-1$ if $i\in \mathbb J$ and either $\qq$ or $\qq^{-1}$ if $i\notin \mathbb J$.  This can be used to verify that each root satisfies condition (1) or (2) of \ref{translation-key-step}. For moments, we are interested in the case that $\mathbb J = \{1,\dots,\theta\}$ in which case \[  \sum_{l=1}^k (-1)^l i_l = \sum_{l=1}^{\theta} (-1)^l l = (-1)^\theta \left\lceil \frac{\theta}{2} \right\rceil \] which occurs as one of the diagrams associated to $\mathtt A(j\mid \theta+1-j)$ if and only if $j = \lceil \frac{\theta}{2} \rceil $.

In this case, the edge labels $\tilde{q}_{(i-1) i}$ alternate between the values $\qq$ and $\qq^{-1}$ and the vertex labels are all $-1$. If we set $\gene=\qq$ then the matrix $M$ has the form \[ \begin{pmatrix} 0 & 1 \\ 1 & 0 \end{pmatrix} \textrm{if } \theta=2\]\[ \begin{pmatrix} 0 & n-1 & 0  \\ n-1 & 0 & 1 \\ 0 & 1 & 0  \end{pmatrix} \textrm{if } \theta=3\]
\[ \begin{pmatrix} 0 & 1 & 0 & 0 \\ 1 & 0 & n-1 & 0 \\ 0 & n-1 & 0 & 1 \\ 0 & 0 & 1 & 0  \end{pmatrix} \textrm{if } \theta=4\]
and so on. (If $\qq$ has odd order, we can set $\gene$ to a square root of $\qq$ and double the matrix entries).

For $\theta \leq 2$ the moments of $L$-functions to which the axiomatic multiple Dirichlet series with these values of $M$ relate are standard, but they become stranger as $\theta$ grows. For $f_1,\dots, f_\theta$ relatively prime we have
\[a(f_1,\dots,f_\theta; q, \chi, M) = \prod_{i=1}^{\theta-1} \left( \frac{f_i}{f_{i+1}} \right)_{\chi}^{ (-1)^{\theta-1-i}}\]
which up to a possible sign is
\[ \prod_{i=1}^{\theta-1} \left( \frac{f_{ i+ \frac{ 1 + (-1)^i}{2}}}{f_{i+\frac{ 1 - (-1)^i}{2}}} \right)_{\chi}^{ (-1)^{\theta-1-i}}\] where we have swapped the two polynomials appearing in the power residue symbol whenever $i$ is even to ensure that the one with even index always appears on the bottom. These coefficients occur when calculating the moment of $L$-functions
\[ \sum_{\substack{f_2,f_4,\dots, f_{2\lfloor \frac{\theta}{2} \rfloor} \in \mathbb F_q[t]^+\\ \deg f_i =d_i}} \prod_{i \in \{1,3,\dots,2\lfloor \frac{\theta-1}{2}\rfloor  +1\}}  L \left( s_i, \left( \frac{*}{f_{i-1} }\right)_{\chi^{(-1)^{\theta-i}}} \left( \frac{*}{f_{i+1}}\right)_{\chi^{(-1)^{\theta-1-i}}}\right) \] where we adopt the convention that $f_0 = f_{ 2\lfloor \frac{\theta}{2} \rfloor+2}=1$ so these terms can be ignored when they appear.  

For $\theta=2$ this is the first moment
\[ \sum_{\substack{f_2\in \mathbb F_q[t]^+\\ \deg f_2= d_2}} L \left(s_1, \left( \frac{*}{f_2} \right)_{\chi} \right).\]
For $\theta=3$ this is the second moment
\[ \sum_{\substack{f_2\in \mathbb F_q[t]^+\\ \deg f_2= d_2}} L \left(s_1, \left( \frac{*}{f_2} \right)_{\chi^{-1} } \right)L \left(s_3, \left( \frac{*}{f_2} \right)_{\chi}\right)\] which if $s_1=s_3$ specializes to the second moment of the absolute value 
\begin{equation}\label{second-absolute-moment} \sum_{\substack{f_2\in \mathbb F_q[t]^+\\ \deg f_2= d_2}} \left| L \left(s_1, \left( \frac{*}{f_2} \right)_{\chi^{-1} } \right)\right|^2.\end{equation}
For $\theta=4$ this is the second moment
\[ \sum_{\substack{f_2,f_4\in \mathbb F_q[t]^+\\ \deg f_2= d_2 \\ \deg f_4=d_4 }} L \left(s_1, \left( \frac{*}{f_2} \right)_{\chi } \right)L \left(s_3, \left( \frac{*}{f_2} \right)_{\chi^{-1}}  \left( \frac{*}{f_4} \right)_{\chi} \right)\] and for $\theta=5$ this is the third moment
\[ \sum_{\substack{f_2,f_4\in \mathbb F_q[t]^+\\ \deg f_2= d_2 \\ \deg f_4=d_4 }} L \left(s_1, \left( \frac{*}{f_2} \right)_{\chi^{-1} } \right)L \left(s_3, \left( \frac{*}{f_2} \right)_{\chi}  \left( \frac{*}{f_4} \right)_{\chi^{-1} } \right) L \left( s_5, \left( \frac{*}{f_4} \right)_{\chi}\right).\]

The special case $\theta=2$, i.e. $\mathtt A(1\mid 1)$, is the matrix $M$ discussed in the Introduction. Its associated axiomatic multiple Dirichlet series was already observed in \cite[Proposition 4.1]{s-amds} to agree with a multiple Dirichlet series defined by Chinta and Mohler~\cite{ChintaMohler}. The series of \cite{ChintaMohler} was itself a function field analogue of a series defined by Friedberg, Hoffstein, and Lieman~\cite{FriedbergHoffsteinLieman} to study the first moments of Dirichlet $L$-functions.

A number field analogue of the moment \eqref{second-absolute-moment} was studied by Diaconu in \cite{Diaconu2004} using a multiple Dirichlet series. The paper concerns a two-variable Dirichlet series, but briefly mentions \cite[Remark on p. 677]{Diaconu2004} that this series is a specialization of a three-variable Dirichlet series. It is reasonable to guess that this three-variable series is a number field analogue of the axiomatic multiple Dirichlet series arising from the root system $\mathtt A(2\mid 1)$ (i.e. the case $\theta=3$), satisfying corresponding functional equations, similar to the calculations in Example \ref{brubaker-example} for the Dirichlet series studied in \cite{BrubakerThesis}. However, since the functional equations of the three-variable series are not described in \cite{Diaconu2004}, it would not be possible to verify this without deriving these functional equations ourselves, which we will not attempt to do here. \end{example}

The next three examples do not seem to have appeared in the literature before:

\begin{example} \label{new-example-1} We now consider the multiple Dirichlet series arising from the arithmetic root system $\mathtt{g}(4,3)$. According to \cite[\S8.5.3]{AndruskiewitschAngiono} this root system has ten Dynkin diagrams, of which one is

\begin{align}\label{eq:dynkin-g(4,3)}
\begin{aligned}
 \xymatrix{ \overset{-1}{\underset{\ }{\circ}}\ar  @{-}[r]^{\overline{\zeta}}  &
	\overset{-1}{\underset{\ }{\circ}} \ar  @{-}[r]^{\zeta}  & \overset{-1}{\underset{\
		}{\circ}}
	\ar  @{-}[r]^{\zeta}  & \overset{-1}{\underset{\ }{\circ}}}
\end{aligned}
\end{align}
where $\zeta$ is a primitive third root of unity. If we let $\gene$ be a primitive sixth root of unity with $\gene^2=\zeta$, then this Dynkin diagram gives a matrix $M$ of the form
\[ M = \begin{pmatrix} 0 & -2 & 0 & 0 \\ -2 & 0  & 2 & 0 \\ 0 & 2 & 0 & 2 \\ 0 & 0 & 2 & 0 \end{pmatrix} .\]
Taking $q \equiv1 \bmod 12$ and letting  $\chi$ be a character of $\mathbb F_q^\times$ of order $6$, so that $\chi^2$ is a character of order $3$, the coefficients $a(f_1,f_2,f_3,f_4;q,\chi, M)$ agree for $f_1,f_2,f_3,f_4$ relatively prime with \[\left( \frac{f_1}{f_2} \right)_{\overline{\chi}^2} \left( \frac{f_3}{f_2} \right)_{\chi^2} \left( \frac{f_3}{f_4} \right)_{\chi^2}\] and hence should be applicable to the moment of cubic Dirichlet $L$-functions \[\sum_{\substack { f_2,f_4 \in \mathbb F_q[T]^+\\ \deg f_2 =d_2 \\ \deg f_4=d_4 }} L\left(s_1, \left( \frac{*}{f_2} \right)_{\bar{\chi}^2} \right) L\left(s_3, \left( \frac{*}{f_2 f_4} \right)_{\chi^2}  \right)  .\]
 Equivalently by reciprocity, the coefficients agree with  \[\left( \frac{f_2}{f_1} \right)_{\overline{\chi}^2} \left( \frac{f_2}{f_3} \right)_{\chi^2} \left( \frac{f_4}{f_3} \right)_{\chi^2}\]and hence should be applicable to the moment of cubic Dirichlet $L$-functions \[\sum_{\substack { f_1,f_3 \in \mathbb F_q[T]^+\\ \deg f_1 =d_1 \\ \deg f_3=d_3 }} L\left(s_2, \left( \frac{*}{f_3 } \right)_{\chi^2}  \right)  L\left(s_4, \left( \frac{*}{f_1^2 f_3} \right)_{\overline{\chi}^2}\right) .\]
\end{example}

\begin{example} \label{new-example-2}
    We consider multiple Dirichlet series arising from the arithmetic root system $\mathtt G(3)$. According to \cite[\S5.6.4]{AndruskiewitschAngiono}, this root system has four generalized Dynkin diagrams, of which one is
    
    \begin{align}\label{eq:dynkin-G-super}
\begin{aligned}
a_2 \longmapsto&\xymatrix{ \overset{-1}{\underset{1}{\circ}} \ar  @{-}[r]^{\qq}  &
\overset{-1}{\underset{2}{\circ}} \ar  @{-}[r]^{\qq^{-3}}  & \overset{\,\, \qq^3}{\underset{3}{\circ}},}
\end{aligned}
\end{align}
Here $\qq$ is any root of unity other than $\pm 1$. In this diagram, two nodes are labeled by $-1$ and one is labeled with $\qq^{3}$. If $\qq$ is a primitive sixth root of unity so that $\qq^3=-1$ then all nodes are labeled with $-1$. In this case, if we set $\gene=\qq$ then we have \[ M= \begin{pmatrix} 0 & 1 & 0 \\ 1 & 0 & 3 \\ 0 & 3 & 0 \end{pmatrix}.\]
Taking $q\equiv1 \bmod 12$ and letting  $\chi$ be a character of $\mathbb F_q^\times$ of order $6$, the coefficients $a(f_1,f_2,f_3;q,\chi, M)$ agree for $f_1,f_2,f_3$ relatively prime with $\left( \frac{f_1}{f_2} \right)_{\chi} \left( \frac{f_3}{f_2} \right)_{\chi}^3$ and hence should be applicable to the moment \[\sum_{\substack { f_2 \in \mathbb F_q[T]^+\\ \deg f_2 =d}} L\left(s_1, \left( \frac{*}{f_2} \right)_{\chi} \right) L\left(s_3, \left( \frac{*}{f_2} \right)_{\chi}^3 \right)  \] which is a joint moment combining $L$-functions of Dirichlet characters of order $2$ and $6$.

\end{example}

\begin{example} \label{new-example-3} The multiple Dirichlet series arising from the arithmetic root systems $\mathtt D(j \mid j)$ and $\mathtt D(j \mid j+1)$ are relevant to moments in the special case where the parameter $\qq$ is a primitive fourth root of unity. In this case, the root system depends on two parameters $\theta$ and $j$ with $j \leq \theta$, in addition to a root of unity $\qq\neq \pm 1$,  and is written $\mathtt D(j\mid \theta-j)$.

The relevant Dynkin diagrams build on those described earlier in Example \ref{super-An}, and are defined in \cite[\S5.3.9 and \S5.3.10]{AndruskiewitschAngiono} in terms of a subset $\mathbb J$ of $\{1,\dots, \theta-1\}$ and a symbol $c$, $\tilde{c}$, or $d$. If $j = \theta/2$ and the symbol is $d$, a line is put over the set $\mathbb J$ for reasons not relevant to us. When the symbol is $c$, the Dynkin diagram consists of the Dynkin diagram on nodes $1,\dots, \theta-1$ associated to the set $\mathbb J$ as defined in Example \ref{super-An} with an additional node $\theta$ labeled by $\qq^2$ and an edge between $\theta-1$ and $\theta$ labeled by $\qq^{-2}$. When the symbol is $\tilde{c}$, the Dynkin diagram is exactly the same but with the roles of $\theta-1$ and $\theta$ switched. (This corresponds to switching two variables in the multiple Dirichlet series and thus is not relevant to understanding which moments can be studied by this arithmetic root system). The Dynkin diagrams with symbol $d$ are described in \cite[(5.25) and (5.26)]{AndruskiewitschAngiono} and are not relevant to us except that it is possible to check in these cases that all roots satisfy the conditions (1) or (2) of Theorem \ref{translation-key-step}.

The set of Dynkin diagrams associated to $\mathtt D(j\mid \theta+1-j)$ are those arising from $\mathbb J  \subseteq \{1,\dots, \theta-1\}$ and a symbol $c, \tilde{c}$, or $d$ such that, writing $\mathbb J = \{i_1,\dots, i_k\}$ with $i_1< \dots < i_k$, if the symbol is $c$ or $\tilde{c}$ then we have  $ \left| \sum_{l=1}^k (-1)^l i_l\right|=j$ and if the symbol is $d$ then we have $ \left| \sum_{l=1}^k (-1)^l i_l\right|=\theta-j$.

 Just as in Example \ref{super-An}, the case relevant to moments will be when  $\mathbb J = \{1,\dots,\theta-1\}$  in which case \[  \sum_{l=1}^k (-1)^l i_l = \sum_{l=1}^{\theta-1} (-1)^l l = (-1)^{\theta-1} \left\lceil \frac{\theta-1 }{2} \right\rceil \] which occurs as one of the diagrams if and only if $j = \lceil \frac{\theta-1}{2} \rceil $.

In this case, the edge labels $\tilde{q}_{(i-1) i}$ alternate between the values $\qq$ and $\qq^{-1}$, with a final edge label of $\qq^{-2}$ and a final node label of $\qq^2$. To be relevant to moments, the last node label must also be $-1$, which occurs when $\qq$ is a primitive $4$th root of unity. If we set $\gene=\qq$ then the matrix $M$ has the form \[ \begin{pmatrix} 0 & 1 & 0 \\ 1 & 0 & 2 \\ 0 & 2 & 0 \end{pmatrix} \textrm{if } \theta=3\]\[ \begin{pmatrix} 0 & 3 & 0 & 0  \\ 3 & 0 & 1 & 0  \\ 0 & 1 & 0 &  2 \\ 0 & 0 & 2 & 0 \end{pmatrix} \textrm{if } \theta=4\]
and so on.

Fixing a character $\chi$ of order $4$, we observe that for $\theta =3$, i.e. the arithmetic root system $\mathtt D(1 \mid 2)$, the $a(f_1,f_2,f_3;q,\chi, M)$ agree for $f_1,f_2,f_3$ relatively prime with $\left( \frac{f_1}{f_2} \right)_{\chi} \left( \frac{f_3}{f_2} \right)_{\chi}^2$ and hence should be applicable to the moment \[\sum_{\substack { f_2 \in \mathbb F_q[T]^+\\ \deg f_2 =d}} L\left(s_3, \left( \frac{*}{f_2} \right)_{\chi} \right) L\left(s_1, \left( \frac{*}{f_2} \right)_{\chi}^2 \right)  \] which is a joint moment combining $L$-functions of Dirichlet characters of order $2$ and $4$. Higher values of $\theta$ will give more unusual moments, though all combining characters of orders $2$ and $4$.
\end{example}

\subsection{Bicharacter and scattering matrix calculations}

In this subsection are some elementary calculations that were used earlier but not proven until now.

\begin{lemma}\label{kubota-bicharacter-identity} Let $\psi\colon \mathbb Z^r \times \mathbb Z^r \to \mu_\infty$ be a bicharacter, $E$ a basis of $\mathbb Z^r$, and $\gene$ a root of unity such that every value taken by $\psi$ is a power of $\gene$. Fix $i \in \{1,\dots, r\}$ and assume that $E$ is admissible at $i$.  If $\psi(e_i,a) \psi(a,e_i)$ is a power of $\psi(e_i,e_i)$ for all $a\in \mathbb Z^r$. Then

\begin{enumerate}
    \item  For all $j$, $m_{ij}$ is the least nonnegative integer value of $m$ such that $\hchi(e_i,e_i)^m \hchi(e_i,e_j) \hchi(e_j,e_i) =1$. In particular, $m_{ij}=n_{ij}(M^{\hchi, E, \gene})$.

    \item We have $M^{\hchi,s_{i,E}(E) , \gene} = M^{\hchi,E, \gene}$.

\end{enumerate}

\end{lemma}

\begin{proof} To prove part (1), note that the least value of $m$ such that $\hchi(e_i,e_i)^m \hchi(e_i,e_j) \hchi(e_j,e_i) =1$ is necessarily between $0$ and the order of $\hchi(e_i,e_i)$ minus $1$, and thus is less than or equal to the least value of $m$ such that $\hchi(e_i,e_i)^{m+1}=1$, and hence is equal to $m_{ij}$. By definition we have $ \hchi(e_i,e_i) = - \gene^{M_{ii}} = \gene^{ M_{ii}+\frac{n}{2}} $ and $\hchi(e_i,e_j) = \gene^{M_{ij}}$ so the least value of $m$ such that $\hchi(e_i,e_i)^m \hchi(e_i,e_j) \hchi(e_j,e_i) =1$  is the least value of $m$ such that $m ( M_{ii} + \frac{n}{2}) + M_{ij} \equiv 0 \bmod n$ which is $n_{ij}$ by definition (as long as such a value exists, which it does by assumption).

Part (1) implies in particular that we have $\hchi(e_i,e_i)^{m_{ij}} \hchi(e_i,e_j) \hchi(e_j,e_i) =1$. Now observe that
\begin{equation}\label{bi-ii} \hchi( s_{i,E}(e_i), s_{i,E} (e_i)) =\hchi( -e_i,-e_i)=\hchi(e_i,-e_i)^{-1} = \hchi(e_i,e_i) \end{equation} and thus $M_{ii}^{\hchi,s_{i,E}(E) , \gene}=M_{ii}^{\hchi,E, \gene}$.
For $j \neq i$ we have
\[ \hchi( s_{i,E}(e_i), s_{i,E} (e_j))= \hchi(-e_i, e_j + m_{ij}e_i)= \hchi(e_i,e_j+m_{ij} e_i)^{-1} = \hchi(e_i,e_j)^{-1} \hchi(e_i,e_i)^{- m_{ij}} = \hchi(e_j,e_i)\]
and symmetrically
\[ \hchi( s_{i,E}(e_j), s_{i,E} (e_i)) = \hchi(e_i,e_j) \]
so that
\[ \hchi( s_{i,E}(e_i), s_{i,E} (e_j)) \hchi( s_{i,E}(e_j), s_{i,E} (e_i))=  \hchi(e_i,e_j)\hchi(e_j,e_i)\]
and hence $M_{ij}^{\hchi,s_{i,E}(E) , \gene}=M_{ij}^{\hchi,E, \gene}$. Finally for $j,k \neq i$ we have
\[ \hchi( s_{i,E}(e_j), s_{i,E} (e_k))=\hchi(e_j + m_{ij}e_i,e_k + m_{ik}e_i) = \hchi(e_j,e_k) \hchi(e_i,e_k)^{m_{ij} }  \hchi(e_j,e_i)^{m_{ik}} \hchi(e_i,e_i)^{m_{ij}m_{ik}}  \]
which if $j=k$ implies
\[  \hchi( s_{i,E}(e_j), s_{i,E} (e_j)) =\hchi(e_j,e_j) \hchi(e_i,e_j)^{m_{ij} }\hchi(e_j,e_i)^{m_{ij}} \hchi(e_i,e_i)^{m_{ij}^2} = \hchi(e_j,e_j) \]
so that $M_{jj}^{\hchi,s_{i,E}(E) , \gene}=M_{jj}^{\hchi,E, \gene}$ while if $j \neq k$ we have
\[ \hchi( s_{i,E}(e_j), s_{i,E} (e_k))\hchi( s_{i,E}(e_k), s_{i,E} (e_j))\] \[=\hchi(e_j,e_k)  \hchi(e_k,e_j) \hchi(e_i,e_k)^{m_{ij} }  \hchi(e_k,e_i)^{m_{ij} } \hchi(e_j,e_i)^{m_{ik}} \hchi(e_i,e_j)^{m_{ik} }\hchi(e_i,e_i)^{2m_{ij}m_{ik}} = \hchi(e_j,e_k)  \hchi(e_k,e_j)\]
so that $M_{jk}^{\hchi,s_{i,E}(E) , \gene}=M_{kj}^{\hchi,E, \gene}$, completing part (2). \end{proof}

\begin{lemma}\label{dirichlet-bicharacter-identity} Let $\psi\colon \mathbb Z^r \times \mathbb Z^r \to \mu_\infty$ be a bicharacter, $E$ a basis of $\mathbb Z^r$, and $\gene$ a root of unity such that every value taken by $\psi$ is a power of $\gene$. Fix $i \in \{1,\dots, r\}$ and assume that $\psi(e_i,e_i)=-1$. Then

\begin{enumerate}
    \item  For all $j$ we have
    \begin{equation*} m_{ij} = \begin{cases} 0 & \textrm{if }M_{ij}^{\hchi,E,\gene} =0 \\ 1 & \textrm{if }M_{ij}^{\hchi,E,\gene} \neq 0 \end{cases} =e_{ij}.\end{equation*}

    \item We have
    \[  M^{\hchi,s_{i,E}(E) , \gene} \equiv \tau_i ( M^{\hchi,E, \gene}) \bmod n \]
with $\tau_i$ defined as in \S\ref{ss-multivariable-functional-equation}.
    
    
\end{enumerate}\end{lemma}

\begin{proof}To check part (1), $\psi(e_i,e_i)^2=1$ ensures that we have $m_{ij}$ equal to either $0$ or $1$ for all $j$, with $m_{ij}=0$ if and only if $\hchi(e_i,e_j)\hchi(e_j,e_i)=1$, i.e. if and only if $M_{ij}^{\hchi,E,\gene} =0$, and $m_{ij}=1$ otherwise. This agrees with the definition of $e_{ij}$.

We now calculate $M^{\hchi,s_{i,E}(E) , \gene}$ to verify part (2). We have $M_{ii}^{\hchi,s_{i,E}(E) , \gene}=M_{ii}^{\hchi,E, \gene}$ by a calculation identical to \eqref{bi-ii}. If $i\neq j$ and $M_{ij}=0$ we have
\[ \hchi( s_{i,E}(e_i), s_{i,E} (e_j))= \hchi(-e_i, e_j )= \hchi(e_i,e_j)^{-1}  \] 
and symmetrically $ \hchi( s_{i,E}(e_j), s_{i,E} (e_i))=\hchi(e_j,e_i)^{-1}$ so
\[ \hchi( s_{i,E}(e_i), s_{i,E} (e_j)) \hchi( s_{i,E}(e_j), s_{i,E} (e_i)) = \hchi(e_i,e_j)^{-1}\hchi(e_i,e_j)^{-1} =1\] and thus $M^{\hchi,s_i(E),\gene}_{ij}=0$.
If $i\neq j$ and $M_{ij}^{\hchi,E,\gene} \neq 0$ then 
\[ \hchi( s_{i,E}(e_i), s_{i,E} (e_j))= \hchi(-e_i, e_j +e_i) = \hchi(e_i,e_i)^{-1} \hchi(e_i,e_j)^{-1} \]
and symmetrically $\hchi( s_{i,E}(e_j), s_{i,E} (e_i))= \hchi(e_i,e_i)^{-1} \hchi(e_j,e_i)^{-1}$ so that
\[ \hchi( s_{i,E}(e_i), s_{i,E} (e_j)) \hchi( s_{i,E}(e_j), s_{i,E} (e_i)) = \hchi(e_i,e_i)^{-2}  \hchi(e_i,e_j)^{-1}\hchi(e_j,e_i)^{-1}=\hchi(e_i,e_j)^{-1}\hchi(e_j,e_i)^{-1}\]
so $M^{\hchi,s_i(E),\gene}_{ij} \equiv - M_{ij}^{\hchi,E,\gene}\bmod n$. For $j \neq i$ we have
\[  \hchi( s_{i,E}(e_j), s_{i,E} (e_j)) =\hchi(e_j,e_j) \hchi(e_i,e_j)^{m_{ij} }\hchi(e_j,e_i)^{m_{ij}} \hchi(e_i,e_i)^{m_{ij}^2}\]
and this gives $M^{\hchi,s_i(E),\gene}_{jj} = M_{jj}^{\hchi,E,\gene}$ if $M_{ij}^{\hchi,E,\gene}=0$ so that $m_{ij}=0$ or  $$M^{\hchi,s_i(E),\gene}_{jj}  \equiv M_{jj}^{\hchi,E,\gene} + M_{ij}^{\hchi,E,\gene}+ \frac{n}{2} \bmod n $$ if $M_{ij}^{\hchi,E,\gene}\neq 0$ so that $m_{ij}=1$.

For $i \neq j \neq k$ we have 
\[ \hchi( s_{i,E}(e_j), s_{i,E} (e_k))\hchi( s_{i,E}(e_k), s_{i,E} (e_j))\] \[=\hchi(e_j,e_k)  \hchi(e_k,e_j) \hchi(e_i,e_k)^{m_{ij} }  \hchi(e_k,e_i)^{m_{ij} } \hchi(e_j,e_i)^{m_{ik}} \hchi(e_i,e_j)^{m_{ik} }\hchi(e_i,e_i)^{2m_{ij}m_{ik}} \]
\[=\hchi(e_j,e_k)  \hchi(e_k,e_j) \hchi(e_i,e_k)^{m_{ij} }  \hchi(e_k,e_i)^{m_{ij} } \hchi(e_j,e_i)^{m_{ik}} \hchi(e_i,e_j)^{m_{ik}} \]
which gives $M^{\hchi,s_i(E),\gene}_{jk} = M_{jk}^{\hchi,E,\gene}$ if either $M_{ij}^{\hchi,E,\gene}=0$ or $M_{ik}^{\hchi,E,\gene}=0$ since then every factor past the first two is trivial or 
\[M^{\hchi,s_i(E),\gene}_{jk} \equiv  M_{jk}^{\hchi,E,\gene}+  M_{ik}^{\hchi,E,\gene}+  M_{ij}^{\hchi,E,\gene} \bmod n \] if $M_{ij}^{\hchi,E,\gene},M_{ik}^{\hchi,E,\gene}\neq 0$.

In all cases, this agrees with the definition of $\tau_i(M)$.\end{proof}

\section{Topology and Quantum Algebra} \label{SectionTopology}

\subsection{Overview}

Nichols algebras of diagonal type are certain non-commutative algebras parameterized by a basis $E$ of $\mathbb Z^r$ and a bicharacter $\hchi \colon \mathbb Z^r \times \mathbb Z^r \to \mathbb C^\times$ (though we will be interested only in the case where $\hchi$ is valued in $\mu_\infty$). The concepts of Weyl groupoid and arithmetic root system was developed to understand the structure of Nichols algebras of diagonal type. In particular, the classification of arithmetic root systems gives the classification of finite-dimensional Nichols algebras of diagonal type. We will briefly review the general notion of Nichols algebra in \S\ref{nichols-intro}, but for now we state that for a bicharacter $\chi$, the Nichols algebra $\mathfrak{B} (V_{E,\hchi})$ is the quotient of the free algebra $\mathbb C \langle x_1,\dots, x_r \rangle$ on generators $x_1,\dots, x_r$ by the kernel of the unique $\mathbb C$-linear map $\mathbb C \langle x_1,\dots, x_r \rangle\to \mathbb C \langle x_1,\dots, x_r \rangle$ that for any sequence $s_1,\dots,s_n \in \{1,\dots,r\}$ sends $x_{s_1} \dots x_{s_n}$ to \[ \sum_{\tau \in S_n}    x_{s_{\tau^{-1}(1)}}  \dots  x_{s_{\tau^{-1}(n)}}\prod_{\substack{ 1\leq i < j \leq n \\ \tau(i)>\tau(j)}} \hchi( e_{s_i}, e_{s_j}) .\]
While this map is not an algebra homomorphism, its kernel is a two-sided ideal, so the quotient algebra is well-defined. (In fact, it is a homomorphism for a modified algebra structure on the target, the quantum shuffle product, explaining why its kernel is a two-sided ideal).

An important special case of Nichols algebras are closely related to quantum groups. Specifically, for $E, \hchi$ as in Example \ref{cartan-example}, i.e. $E$ the set of simple roots of a simple Lie algebra $\mathfrak g$ and $\hchi ( \alpha_i,\alpha_j) = v^{\langle \alpha_i,\alpha_j \rangle}$, the associated Nichols algebra is the positive part of Lusztig's small quantum group with parameter $\qv$, i.e. the algebra generated by the positive roots~\cite[Theorem 15.2]{Rosso1998}. (The parameter $\qv$ is most commonly denoted by $q$ in the quantum groups literature, but is sometimes called $\qv$. In the case of $\mathfrak{sl}_2$, the literature sometimes uses the notation $q$ for $\qv^2$. To avoid confusion with the finite field size $q$ and $\qq=\qv^2$ in the literature on Nichols algebras of diagonal type, we use the notation $\qv$.) Quantum groups themselves are deformations of the universal enveloping algebras of Lie algebras, so the positive part is (part of) a deformation of the algebra generated by the positive roots inside the universal enveloping algebra of $\mathfrak g$. Thus, in this section, quantum groups will be used as a representative case and the algebra generated by the positive roots of a Lie algebra will be used as a toy model.

Earlier, we saw that there is a formal relationship between the reflections in a Weyl groupoid and the functional equations of an axiomatic multiple Dirichlet series, and using this, we were able to construct meromorphic multiple Dirichlet series from certain arithmetic root systems. This raises the question: Is there a direct relationship between the functional equations of a multiple Dirichlet series arising from an arithmetic root system and the Weyl groupoid of a Nichols algebra arising from the same arithmetic root system?

We give a partial answer to this question in the following way.

First, we show in \S\ref{ss:connections-cohomology} that the coefficients of the axiomatic multiple Dirichlet series are the trace of Frobenius acting on certain cohomology groups. These cohomology groups have multiple descriptions, including topological ones, and, following Kapranov and Schechtman, as graded components of the $\Ext$ groups of the corresponding Nichols algebra.

Second, we show in Theorem \ref{homological-relation-theorem} that the functional equations relating the coefficients imply that certain equations relating the dimensions of these cohomology groups hold. In particular, this implies the following linear relations between the dimensions of different graded components of the $\Ext$ groups of the Nichols algebra.

\begin{cor}\label{nichols-relation-theorem} Let $r$ be a natural number, $\hchi\colon \mathbb Z^r \times \mathbb Z^r \to\mu_\infty$ a bicharacter, and $E=(e_1,\dots,e_r)$ an ordered basis of $\mathbb Z^r$.

Fix $i$ from $1$ to $r$. Assume $E$ is admissible at $i$. Fix nonnegative integers $d_0,\dots, d_{i-1}, d_{i+1},\dots, d_r$. Let $\tilde{d}= \sum_{k \neq i} d_k m_{ik}$. Let $ \mathfrak B( V_{\hchi,E} )$ be the Nichols algebra of the diagonal braided vector space defined using $\hchi$ and $E$ as in \S\ref{nichols-intro}.
Then
\begin{enumerate}
    \item If $\hchi(e_i,a)\hchi(a,e_i)$ is a power of $\hchi(e_i,e_i) $ for all $a\in \mathbb Z^r$, let
    \[ h^j_{d_i} =  \dim \Gr^{ -\sum_{i=1}^r d_i e_i}  \operatorname{Ext}^{j } _{ \mathfrak B( V_{\hchi,E} )} (\mathbf 1, \mathbf 1) .\]
Let $n_i$ be the order of $\hchi(e_i,e_i)$. Recall that $m \% n_i$ denotes the smallest nonnegative integer congruent to $m$ modulo $n_i$.  Then we have 
\begin{equation}\label{nichols-kubota-general} h^j_{d_i} - h^{j-1}_{d_i - (2d_i-1-\tilde{d})\%n_i} = h^{j}_{\tilde{d}+1-d_i - (\tilde{d}+1-2d_i)\% n_i} - h^{j-1}_{\tilde{d}+1-d_i-n_i} \end{equation}
    unless $2d_i \equiv 1 + \tilde{d}\bmod n_i$, in which case 
    \begin{equation}\label{nichols-kubota-special} h^j_{d_i} = h^j_{\tilde{d}+1-d_i-n_i}.\end{equation}

    \item If $\hchi(e_i,e_i)=-1$, let 
    \[ h^{j,E}_{d_i} = \dim \Gr^{ -\sum_{i=1}^r d_i e_i}  \operatorname{Ext}^{j } _{ \mathfrak B( V_{\hchi,E} )} (\mathbf 1, \mathbf 1) \] and
        \[ h^{j, s_{i,E}(E)}_{d_i}= \dim \Gr^{ -\sum_{i=1}^r d_i e_i}  \operatorname{Ext}^{j } _{ \mathfrak B( V_{\hchi,s_{i,E}(E) } )} (\mathbf 1, \mathbf 1) .\] 
Then
    \begin{equation}\label{nichols-dirichlet-general} h^{j,E}_{d_i}= h^{j,s_{i,E}(E)}_{\tilde{d}-d_i-1} \end{equation} unless $\prod_{k\neq i}( \hchi(e_i,e_k) \hchi(e_k,e_i))^{d_k}=1$, in which case
\begin{equation}\label{nichols-dirichlet-special} h^{j,E}_{d_i}- h^{j-1,E}_{d_i-1} = h^{j,s_{i,E}(E)}_{\tilde{d}-d_i} - h^{j-1,s_{i,E}(E)}_{\tilde{d}-d_i-1}  .\end{equation}
    
\end{enumerate}

\end{cor}

In Theorem \ref{nichols-determination-theorem}, we show that for many Nichols algebras, Corollary \ref{nichols-relation-theorem} uniquely determines the Betti numbers $\dim \Gr^{ -\sum_{i=1}^r d_i e_i}  \operatorname{Ext}^{j } _{ \mathfrak B( V_{\hchi,E} )} (\mathbf 1, \mathbf 1)$, giving a combinatorial algorithm to compute those Betti numbers, though not necessarily an explicit formula for them.

This includes the positive part of Lusztig's small quantum group. Let us comment on the relationship of this part of Corollary \ref{nichols-relation-theorem}. In that case, these Betti numbers may be computed as a consequence of work of Drupieski, Nakano, and Ngo~\cite[Theorem 5.6.1]{Drupieski2012}, which computes these cohomology groups under the assumption that the order $\ell$ of $v^2$ is odd, greater than the Coxeter number of $\mathfrak g$, coprime to $n+1$ if $\mathfrak g \cong \mathfrak{sl}_{n+1}$, and coprime to $3$ if $\mathfrak g$ has type $E_6$ or $G_2$. It is easy to deduce Corollary \ref{nichols-relation-theorem} from their formulas. Corollary \ref{nichols-relation-theorem} and Theorem \ref{nichols-determination-theorem} generalize this result since they work for arbitrary $\ell$, but they are weaker in that they do not compute the algebra structure on the cohomology and apply only to the Ext groups of the trivial module. It would likely be difficult to generalize the methods of \cite[Theorem 5.6.1]{Drupieski2012} to handle arbitrary $\ell$. On the other hand, generalizing the method of Corollary \ref{nichols-relation-theorem} to the Ext groups of arbitrary irreducible representations might be possible if a theory of axiomatic twisted Weyl group multiple Dirichlet series were developed. Computing the algebra structure would require a geometric perspective on the functional equations.

 Various similar cohomology groups in the setting of quantum groups have also been computed, such as the cohomology of the associated graded of a certain filtration on the positive part of Lusztig's small quantum group~\cite[Proposition 2.3.1]{Ginzburg1993} or the cohomology of the positive part of the De Koncini-Kac integral form~\cite{Vigre2010}. The cohomology of the full quantum group has been computed also for roots of unity of small order~\cite{BKPP}.

Corollary \ref{nichols-relation-theorem} is a result purely on the homological algebra of an algebra with explicit generators and relations. Our proof of it, however, relies on the functional equations for axiomatic multiple Dirichlet series proven in the remainder of the paper, and thus indirectly relies on a huge amount of machinery: the theory of \'{e}tale cohomology, including perverse sheaves and Deligne's proof of the Weil conjectures, and also the theory of metaplectic Eisenstein series.

This, of course, raises the question of whether there exists a proof of Corollary \ref{nichols-relation-theorem} that is purely algebraic, and does not reduce to an arithmetic situation. There are two possible approaches to the proof.

First, one could compute the Betti numbers $\dim \Gr^{ -\sum_{i=1}^r d_i e_i}  \operatorname{Ext}^{j } _{ \mathfrak B( V_{\hchi,E} )} (\mathbf 1, \mathbf 1)$ by an independent argument, perhaps giving generators and relations for this cohomology algebra, and use these computations to verify Corollary \ref{nichols-relation-theorem}. For example, in certain quantum group cases this can be done using \cite[Theorem 5.6.1]{Drupieski2012}, but beyond the case of quantum groups, there are few Nichols algebras whose cohomology has been succesfully computed.

Second, one could directly prove Corollary \ref{nichols-relation-theorem} -- in particular, in case (2) this would involve proving a relation between cohomology groups of two different algebras without computing the cohomology groups of either. This would have the apparent disadvantage that it does not lead to a calculation of the Betti numbers, but in fact Theorem \ref{nichols-determination-theorem} shows that Corollary \ref{nichols-relation-theorem} does uniquely determine the Betti numbers in the case of an arithmetic root system where every root satisfies one of the two hypotheses. This would have multiple advantages: First, it would give some information even for those $\hchi$ where it is not feasible to compute all the Betti numbers, and second, it would give some additional conceptual understanding of how Corollary \ref{nichols-relation-theorem}, and thus the functional equations, relate to the Weyl group.

We give both positive and negative results on this second strategy by considering a certain toy model.   This model consists of the $\Ext$ groups of the subalgebra of the universal enveloping algebra $U(\mathfrak g)$ of a simple Lie algebra $\mathfrak g$ generated by the positive roots, or, equivalently, the Lie algebra cohomology of the Lie algebra $\mathfrak n$ of positive roots. The Lie algebra has a grading arising from the action of the Cartan subalgebra which induces a grading on the Lie algebra cohomology. This Lie algebra cohomology was computed by Kostant, and Kostant's formula has an evident symmetry under the Weyl group, giving a proof of a statement similar to Corollary \ref{nichols-relation-theorem} in this case by the first strategy. This leaves open the question of finding a proof by the second strategy, i.e. of finding a direct explanation for this Weyl group symmetry.

This symmetry could be easily explained if the Weyl group had an action on the Lie algebra $\mathfrak n$ of positive roots, as then the Weyl group would act on the cohomology. This action would be compatible with the grading and thus imply that, for $w$ in the Weyl group, the dimension of the $\alpha$-graded component of $H^i$ would equal the dimension of the $w(\alpha)$-graded component of $H^i$. However, while the Weyl group acts on $\mathfrak g$, it does not stabilize $\mathfrak n$, and hence does not act on $\mathfrak n$.

However, we are able to obtain relations between the dimensions of different cohomology groups using the Weyl group in a more subtle way. Elements of the Weyl group send the Lie algebra $\mathfrak n$ to other Lie subalgebras of $\mathfrak g$. In particular, the reflection around a simple root sends $\mathfrak n$ to a Lie algebra $\mathfrak n'$ whose intersection with $\mathfrak n$ has codimension $1$ in either. We can relate the Lie algebra cohomology of $\mathfrak n$ and $\mathfrak n'$ indirectly, via the Lie algebra cohomology of their intersection. This relation, combined with the direct relation between the cohomology of $\mathfrak n$ with $\mathfrak n'$ obtained from the isomorphism between them arising from the reflection in the Weyl group, gives an ``action" of the Weyl group on the Lie algebra cohomology of $\mathfrak n$ and implies certain equalities between dimensions of cohomology groups. 

The objects needed for this strategy appear also in the setting of quantum groups, where one has a braid group acting on the quantum group deforming the Weyl group action on the Lie algebra. A standard generator of the braid group sends the subalgebra generated by the positive roots to the subalgebra generated by all but one positive root and one negative root. 

In the setting of Nichols algebras of diagonal type, again the same objects appear. The Weyl groupoid in \cite{HeckenbergerWeyl} is defined to be generated by simple refelctions that relate a pair of Nichols algebras which can be together embedded into a larger algebra and share a large common subalgebra.

However, to make the argument work in this setting, a certain lemma relating the cohomology of the positive and negative parts of $u_v (\mathfrak{sl}_2)$ with coefficients in a complex of $u_v (\mathfrak{sl}_2)$-modules generalizing a relation that holds for complexes of finite-dimensional representations of $\mathfrak{sl}_2$. The obstruction is simply that this lemma is false for some complexes of finite-dimensional modules over $u_v (\mathfrak{sl}_2),$ as we show in Example \ref{counter-example}.  However, after taking graded Euler characteristics, i.e. alternating sums of dimensions of graded components of cohomology groups, a similar relation becomes true for a complex of finite-dimensional representations satisfying a mild finiteness hypothesis, and this allows us to prove algebraically a relation between Euler characteristics of cohomology of Nichols algebras.

It is believable that an additional algebraic idea will lead to an algebraic proof of Corollary \ref{nichols-relation-theorem}. However, it is possible that this obstruction occurs because the functional equations demonstrate some geometric structure that is not visible purely algebraically.

In addition to Corollary \ref{nichols-relation-theorem}, our results imply similar relations among cohomology groups in other settings, which are defined in different ways but turn out to be isomorphic. These include cohomology groups of certain perverse sheaves on symmetric powers of the affine line (or complex plane) arising from local systems on the configuration space of points on the plane. These local systems are associated to representations of the braid group that are induced representations of one-dimensional representations from a certain subgroup, the colored braid group. One could equally well ask for a direct proof of Corollary \ref{nichols-relation-theorem} using the topological or algebraic-geometric theory of perverse sheaves, but avoiding arithmetic arguments.

Such a proof would also have advantages. It might clarify the meaning of Corollary \ref{nichols-relation-theorem} in terms of the topology of configuration spaces. It also might upgrade the relation of Betti numbers to an explicit isomorphism of cohomology groups. If done in the algebraic geometry setting, this isomorphism could even be Frobenius-equivariant, which would mean it implies the original functional equation (which a non-Frobenius-equivariant isomorphism does not).

Progress on this exists in case (2), where Hase-Liu~\cite{HaseLiu24} gave a geometric proof of the functional equation -- it should be possible to derive Corollary \ref{nichols-relation-theorem} directly from geometric statements in his paper, without passing through the functional equation. It would be interesting to find an analogous geometric proof in case (1).

\subsection{Nichols algebras and their Weyl groups}\label{nichols-intro}

A \emph{braided vector space} is a vector space $V$ (always for us over the field $\mathbb C$) with an invertible linear transformation $\sigma \colon V \otimes V \to V \otimes V$, the \emph{braiding}, satisfying \[ (\sigma \otimes \operatorname{id})  \circ (\operatorname{id} \otimes \sigma) \circ (\sigma \otimes \operatorname{id})= (\operatorname{id} \otimes \sigma) \circ (\sigma \otimes \operatorname{id}) \circ (\operatorname{id} \otimes \sigma). \]

The braided vector spaces of interest to us are the \emph{diagonal braided vector spaces}, which may be parameterized by a basis $E=(e_1,\dots, e_n)$ for $\mathbb Z^r$ and a bicharacter $\hchi \colon \mathbb Z^r \times \mathbb Z^r \to \mathbb C$. The diagonal braided vector space $V_{E,\hchi}$ is the vector space with basis $x_1,\dots, x_n$ together with the transformation $\sigma$ given by $\sigma(x_i \otimes x_j) = \hchi(e_i,e_j) x_j \otimes x_i$. 

We also need the diagonal vector space $V'_{E,\hchi}$ where the braiding is negated, i.e. in $V'_{E,\hchi}$ we have $\sigma(x_i \otimes x_j) = - \hchi(e_i,e_j) x_j \otimes x_i$.

A braided vector space is \emph{graded} if the underlying vector space is graded and $\sigma$ respects the induced grading on $V \otimes V$. The diagonal braided vector spaces are $\mathbb Z^r$-graded, with $x_i$ given grade $e_i$. (Graded braided vector spaces are a special case of the more general theory of braided Yetter-Drinfeld modules, of which we only need the $\mathbb Z^r$-gradings.)

The \emph{braid group} on $n$ strands $B_n$ is the group with generators $\sigma_1,\dots, \sigma_{n-1}$ and relations $\sigma_i \sigma_{i+1} \sigma_i=\sigma_{i+1} \sigma_i \sigma_{i+1}$ for $i=1\dots, n-2$ and $\sigma_i \sigma_j = \sigma_j \sigma_i$ for $|i-j|>1$. Elements of the braid group are represented visually as braids, where $\sigma_i$ is represented by the $i$th strand passing over the $i+1$st strand. This can also be used to realize the braid group as the fundamental group of the space parameterizing configurations of $n$ unordered points in the plane. There is a natural homomorphism from the braid group to the symmetric group $S_n$, sending $\sigma_i$ to the transposition $(i,i+1)$.

Braided vector spaces define representations of the braid group. Specifically, $V^{\otimes n}$ is a representation of the braid group on $n$ strands where $\sigma_i$ acts by $\operatorname{id}^{\otimes (i-1)} \otimes \sigma \otimes \operatorname{id}^{\otimes (n-i-1)} $. 

The homomorphism $B_n \to S_n$ admits a set-theoretic section, the \emph{Matsumoto lift}. This may be defined by writing $\tau \in S_n$ as a product of transpositions $(i,i+1)$ in an expression of minimal length, and then replacing each copy of the transposition $(i,i+1)$ with the braid $\sigma_i$ to obtain the lift $\tilde{\tau} \in B_n$. Algebraically, one can check that this is well-defined. Visually, this corresponds to a braid where the $i$th strand crosses over the $j$th strand if and only if $i<j$ but $\tau(i)> \tau(j)$, and crosses exactly once in this case, and one can see topologically that all such braids are equivalent.

The \emph{quantum symmetrizer} is a map $V^{\otimes n} \to V^{\otimes n}$ that sends a vector $v \in V^{\otimes n}$ to $\sum_{\tau \in S_n} \tilde{\tau}(v)$. For each $n$ the kernel of the quantum symmetrizer is a subspace $K_n$ of $V^{\otimes n}$, and it is possible to check that if $v \in V^{\otimes n}$ is in $K_n$ and $w \in V^{\otimes m}$ then $v\otimes w$ and $w \otimes v$ both lie in $K_{n+m}$. (The proof proceeds by factoring each permutation in $S_{n+m}$ into a permutation in $S_n$ composed with a permutation that does not change the order of the first $n$ numbers, and checking that the Matsumoto lift of the composition is in this case the composition of the Matsumoto lifts.) It follows that $\bigoplus_{n=0}^{\infty} K_n$ is an ideal in the tensor algebra $T(V) = \bigoplus_{n=0}^{\infty} V^{\otimes n}$, with multiplication given by tensor product. The quotient $\bigoplus_{n=0}^{\infty}  V^{\otimes n}/ K_n$ is known as the \emph{Nichols algebra} $\mathfrak{B}(V)$ of the braided vector space $V$. It has multiple equivalent definitions.

If $V$ is a graded  braided vector space then $V^{\otimes n}$ is graded as a representation of $B_n$, so in particular the quantum symmetrizer respects the grading, inducing a grading on the Nichols algebra $\mathfrak B(V)$. We will be interested in the cohomology of the Nichols algebra in the sense of the $\operatorname{Ext}$ groups $\operatorname{Ext}^j ( \mathbb C, \mathbb C)$ (for $\mathbb C$ the trivial module). These $\operatorname{Ext}$ groups admit a grading induced by the grading of the Nichols algebra.

In the special case of a diagonal braided vector space $V_{E,\hchi}$, the braid $\sigma_i$ sends \[x_{s_1} \otimes \dots \otimes x_{s_{i-1}} \otimes  x_{s_i} \otimes x_{s_{i+1} } \otimes x_{s_{i+2}} \otimes \dots \otimes x_{s_n}\] to 
\[\hchi(e_{s_i}, e_{s_{i+1}}) x_{s_1} \otimes \dots \otimes x_{s_{i-1}} \otimes  x_{s_{i+1} } \otimes x_{s_{i} } \otimes x_{s_{i+2}} \otimes \dots \otimes x_{s_n}\] from which it follows that the Matsumoto lift $\tilde{\tau}$ sends $x_{s_1} \otimes \dots \otimes x_{s_n}$ to \[ x_{s_{\tau^{-1}(1)}} \otimes \dots \otimes x_{s_{\tau^{-1}(n)}} \prod_{\substack{ 1\leq i < j \leq n \\ \tau(i)>\tau(j)}} \hchi( e_{s_i}, e_{s_j})  .\] In other words $\mathfrak{B}(V_{E,\hchi})$ may be presented as the free algebra on $x_1,\dots, x_r$ modulo the kernel of the map 
\[x_{s_1} \dots x_{s_n}\mapsto \sum_{\tau \in S_n}    x_{s_{\tau^{-1}(1)}}  \dots  x_{s_{\tau^{-1}(n)}}\prod_{\substack{ 1\leq i < j \leq n \\ \tau(i)>\tau(j)}} \hchi( e_{s_i}, e_{s_j}) .\]

\begin{example} Continuing Example \ref{cartan-example}, let $E$ be the set of simple roots of a simple Lie algebra $\mathfrak g$ and $\hchi ( \alpha_i,\alpha_j) = v^{\langle \alpha_i,\alpha_j \rangle}$. Then $\mathfrak{B}(V_{E,\hchi})$ is the positive part of Lusztig's small quantum group with parameter $v$ \cite[Theorem 15.2]{Rosso1998}.\end{example}

With these algebras defined, the relevance of the reflections $s_{i,E}$ to them can now be explained. Heckenberger~\cite[\S4]{HeckenbergerWeyl} defined for each $i$ with $e_i$ admissible an algebra $( \mathfrak{B}(V_{\hchi,E}^{\mathrm{op}}) \# H_i^{\mathrm{cop}})^{\mathrm{op}}$ containing $\mathfrak{B}(V_{\hchi,E})$ as a subalgebra, and proved that it also contains $\mathfrak {B}(V_{\hchi,s_{i,E}(E)})$~\cite[Proposition 1 on p. 180]{HeckenbergerWeyl}. This result was the motivation for the definition of $s_{i,E}$ and the Weyl groupoid.

\subsection{Multiple Dirichlet series, perverse sheaves, and Nichols algebras}\label{ss:connections-cohomology}

We first review the geometric construction of axiomatic multiple Dirichlet series from \cite[\S1]{s-amds}.

Let $d_1,\dots,d_{r}$ be natural numbers and $q$ a prime power so that $\mathbb F_q$ is a finite field. View $\mathbb A^{d_i}$ over $\mathbb F_q$ as the moduli space of monic polynomials of degree $d_i$, so that $\prod_{i=1}^{r} \mathbb A^{d_i} $ is a moduli space of tuples $(f_1,\dots,f_{r})$ of monic polynomials. On $\prod_{i=1}^{r} \mathbb A^{d_i}$, define the polynomial function \[F_{d_1,\dots,d_r}= \prod_{i=1}^r \Res (f_i', f_i) ^{M_{ii} }  \prod_{1\leq i< j \leq r}  \Res(f_i, f_j)^{ M_{ij}}.\]

 Write $\hat{d}$ for $d_1+\dots+d_r$. so that $\mathbb A^{\hat{d}}= \prod_{i=1}^{r} \mathbb A^{d_i}$.

Let $U$ be the open set of tuples $f_1,\dots, f_r$ of squarefree coprime polynomials, so that $F$ is nonzero on $U$, and let $u$ be the open immersion of $U$ into $\prod_{i=1}^{r} \mathbb A^{d_i} $. Then we define the Kummer sheaf $\mathcal L_\chi( F_{d_1,\dots,d_r})$ using the Kummer exact sequence and the character $\chi \colon \mu_{q-1} \to \mathbb C^\times \cong \overline{\mathbb Q}_\ell^\times,$ and then take the intermediate extension $u_{*!}$, obtaining
\[ K_{d_1,\dots,d_r} = u_{!*} (\mathcal L_\chi(F_{d_1,\dots,d_r}) [\hat{d}]) [-\hat{d}].\]

By \cite[Theorem 1.1 ]{s-amds}, the multiple Dirichlet coefficients $a(f_1,\dots, f_r;q,\chi,M)$ are given by the trace of Frobenius acting on the stalk of the complex $K_{d_1,\dots, d_r}$. It follows from \cite[Lemma 2.16]{s-amds} that the stalk of the complex $K_{d_1,\dots,d_r}$ at $T^{d_1},\dots, T^{d_r}$ is  $H^*( \mathbb A^{d_1}  \times \dots \times \mathbb A^{ d_r }_{\overline{\mathbb F}_q}, K_{d_1,\dots, d_r} )$ so that
\[  a(T^{d_1},\dots, T^{d_r}; q,\chi, M) =  \sum_j (-1)^j \tr( \Frob_q, H^j( \mathbb A^{d_1}  \times \dots \times \mathbb A^{ d_r }_{\overline{\mathbb F}_q}, K_{d_1,\dots, d_r}) ).\]

It also follows from the Lefschetz fixed point formula that
\[ \sum_{\substack{f_1,\dots, f_{r} \in \F_q[T]^+ \\ \deg f_i = d_i}}   a(\vec{f}; q,\chi, M) =  \sum_j (-1)^j \tr( \Frob_q, H^j_c( \mathbb A^{d_1} \times   \dots \times \mathbb A^{d_r }_{\overline{\mathbb F}_q}, K_{d_1,\dots, d_r}) ).\]

In \S\ref{ss:dimensions-cohomology} we will use the functional equations for the coefficients $ a(T^{d_1},\dots, T^{d_r}; q,\chi, M)$ to extract information about the cohomology groups. We could also use the functional equations for the full series to extract information about the compactly-supported cohomology groups. Poincaré duality makes these two pieces of information equivalent, and which one we use is a matter of taste.

In this section, we relate the cohomology groups 
\begin{equation}\label{cohomology-by-definition} H^j( \mathbb A^{d_1 }\times  \dots \times \mathbb A^{ d_r }_{\overline{\mathbb F}_q}, K_{d_1,\dots, d_r})= H^{j-\hat{d}} ( \mathbb A^{d_1 } \times  \dots \times \mathbb A^{ d_r }_{\overline{\mathbb F}_q}, u_{!*} (\mathcal L_{\chi}(F_{d_1,\dots,d_r}) [\hat{d}])) \end{equation} defined using algebraic geometry in characteristic $p$, successively to cohomology groups of different types: First, an analogous construction in the setting of algebraic geometry over $\mathbb C$, then a purely topological construction, and then a purely algebraic construction using Nichols algebras. These relations are expressed in \eqref{cohomology-by-definition}, Lemma \ref{0-to-p}, \eqref{0-to-top}, Lemma \ref{zero-to-top}, \eqref{top-to-KS}, and Lemma \ref{ks-to-nichols}.

We begin by observing that the construction of $K_{d_1,\dots,d_r}$ makes sense over any field containing the $q-1$st roots of unity in which $\ell$ is invertible, and in particular makes sense over $\mathbb C$, since its ingredients, the polynomial $F_{d_1,\dots, d_r}$, Kummer sheaves, and intermediate extension, all make sense over such fields. Thus, we need an argument to relate the cohomology of perverse sheaves in characteristic $p$ with the cohomology of perverse sheaves in characteristic $0$. Such an argument is provided by \cite[Appendix]{Chang}, in the setting of intermediate extensions of lisse sheaves on the configuration space of points in $\mathbb A^1$. Thus, we first relate $u_{!*} (\mathcal L_{\chi}(F) [\hat{d}])$ to a lisse sheaf on the configuration space (which we also need to do later to make the connection to Nichols algebras).

Let $\pi \colon \mathbb A^{d_1} \times \dots \times \mathbb A^{d_r}\to \mathbb A^{\hat{d}}$ be the map given by multiplication of polynomials, a finite morphism. Let $\operatorname{Conf}_{\hat{d}} \subseteq \mathbb A^{\hat{d}}$ be the open set of squarefree polynomials, equivalently, the moduli space of unordered tuples of $\hat{d}$ points in $\mathbb A^1$. Let $\overline{u}$ be the open immersion $\operatorname{Conf}_{\hat{d}} \to \mathbb A^{\hat{d}}$.

The inverse image of $\operatorname{Conf}_{\hat{d}}$ under $\pi$ is $U$. Let $\pi^\circ$ be the restriction of $\pi$ to $U$, a finite \'{e}tale morphism $U \to \operatorname{Conf}_{\hat{d}}$.

\begin{lemma}\label{0-to-p} We have
\[ H^{j-\hat{d}} ( \mathbb A^{d_1 } \times  \dots \times \mathbb A^{ d_r }_{\overline{\mathbb F}_q}, u_{!*} (\mathcal L_{\chi}(F_{d_1,\dots,d_r}) [\hat{d}])) \] \[= H^{j-\hat{d}} ( \mathbb A^{\hat{d} }_{\overline{\mathbb F}_q},  \overline{u} _{!*}  \pi^\circ_* (\mathcal L_{\chi}(F_{d_1,\dots,d_r}) [\hat{d}])).\]
 \[\cong H^{j-\hat{d}} ( \mathbb A^{\hat{d} }_{\mathbb C},  \overline{u} _{!*}  \pi^\circ_* (\mathcal L_{\chi}(F_{d_1,\dots,d_r}) [\hat{d}]))\]
\[ = H^{j-\hat{d}} ( \mathbb A^{d_1 } \times \dots \times \mathbb A^{d_r }_{\mathbb C},   {u} _{!*}   (\mathcal L_{\chi}(F_{d_1,\dots,d_r}) [\hat{d}])).\] \end{lemma}

Note that, since $\pi^\circ$ is finite \'etale, $\pi^\circ_* \mathcal L_{\chi}(F_{d_1,\dots,d_r})$ is a lisse sheaf on $\operatorname{Conf}_{\hat{d}}$.

\begin{proof} The first equality follows from
\[ H^{j-\hat{d}} ( \mathbb A^{\hat{d} }_{\overline{\mathbb F}_q}, u_{!*} (\mathcal L_{\chi}(F_{d_1,\dots,d_r}) [\hat{d}])) \] \[= H^{j-\hat{d}} ( \mathbb A^{\hat{d} }_{\overline{\mathbb F}_q},  \pi_* u_{!*} (\mathcal L_{\chi}(F_{d_1,\dots,d_r}) [\hat{d}])) \] \[= H^{j-\hat{d}} ( \mathbb A^{\hat{d} }_{\overline{\mathbb F}_q},  \overline{u} _{!*}  \pi^\circ_* (\mathcal L_{\chi}(F_{d_1,\dots,d_r}) [\hat{d}])).\]
which relies on the Leray spectral sequence and the finiteness of $\pi$.

The last equality follows from an identical argument in characteristic $0$.

For the second equality, i.e. to relate the characteristic $0$ and characteristic $p$ situations, we need a lisse sheaf on $\operatorname{Conf}_{\hat{d}, \overline{\mathbb Z}_p}$. The construction of $\pi^\circ_* \mathcal L_{\chi}(F_{d_1,\dots,d_r})$ makes sense over the ring $\overline{\mathbb Z}_p$ because the polynomial $F$, the Kummer sheaf construction, and the map $\pi^\circ$ may all be defined over this setting. We identify the residue field $\overline{\mathbb F}_p$ with $\overline{\mathbb F}_q$ and we embed the fraction field $\overline{\mathbb Q_p}$ into $\mathbb C$. These choices give an identification between $\mu_{q-1} \subset \mathbb F_q$ and $\mu_{q-1} \subset \mathbb C$. We can always choose this isomorphism to be compatible with data mentioned earlier: $\chi \colon \mu_{q-1} (\mathbb F_q) \to \mathbb C^\times$ composed with the identification $\mu_{q-1}(\mathbb F_q)\to \mu_{q-1}(\mathbb C)$ should send the standard generator $e^{2 \pi i \frac{1}{q-1}}$ of $\mu_{q-1}(\mathbb C)$ to the fixed root of unity $\gene$. The sheaf $\pi^\circ_* \mathcal L_{\chi}(F_{d_1,\dots,d_r})$ is lisse because $\pi^\circ_*$ is finite \'{e}tale.

Having done this, we apply \cite[Theorem A.2]{Chang} and obtain
\[H^{j-\hat{d}} ( \mathbb A^{\hat{d} }_{\overline{\mathbb F}_q},  \overline{u} _{!*}  \pi^\circ_* (\mathcal L_{\chi}(F_{d_1,\dots,d_r}) [\hat{d}])) \cong H^{j-\hat{d}} ( \mathbb A^{\hat{d} }_{\mathbb C},  \overline{u} _{!*}  \pi^\circ_* (\mathcal L_{\chi}(F_{d_1,\dots,d_r}) [\hat{d}])).\qedhere \] \end{proof} 

We are now free to use the comparison of \'{e}tale and singular cohomology to express $H^{j-\hat{d}} ( \mathbb A^{d_1 } \times \dots \times \mathbb A^{d_r }_{\mathbb C},   {u} _{!*}   (\mathcal L_{\chi}(F_{d_1,\dots,d_r}) [\hat{d}]))$ as singular cohomology groups of an intersection cohomology complex in the topological sense. When we do this, it is natural to express things in a more topological language. Since there is a natural bijection between polynomials and their multisets of roots, we can express $\mathbb A^{d_1 } \times \dots \times \mathbb A^{d_r }(\mathbb C)$ as $\operatorname{Sym}^{d_1} (\mathbb C) \times \dots \times \operatorname{Sym}^{d_r}(\mathbb C)$. Similarly the space $U(\mathbb C)$ may be represented as $\operatorname{Conf}_{d_1,\dots, d_r} (\mathbb C)$, the space parameterizing configurations of $\hat{d}$ points in the plane with a partition of the points into parts of size $d_1,\dots, d_r$. Visually, this can be viewed as the space of configurations of points of $r$ colors in the plane, with $d_i$ points of the $i$th color.

Local systems on $\operatorname{Conf}_{d_1,\dots, d_r} (\mathbb C) $ are equivalent to representations of its fundamental group, which is the colored braid group: Group elements are braids with $d_i$ strands of the $i$th color, where the sequence of colors of strands at the starting and ending configuration are both some fixed sequence. Equivalently, the colored braid group is the inverse image of $S_{d_1} \times \dots \times S_{d_r}$ under the natural homomorphism from the braid group on $\hat{d}$ strands to $S_{\hat{d}}$.  For $W$ a representation of the colored braid group, write $\mathcal L\{W\}$ for the corresponding local system.

The local system $\mathcal L_{\chi}(F_{d_1,\dots,d_r})$ has rank one and thus corresponds to a one-dimensional representation of the colored braid group, which must factor through its abelianization. The abelianization of the colored braid group is $\mathbb Z^{r (r+1)/2}$, generated by the braids where a strand of the $i$th color makes a loop around a strand of the $j$th color for $i<j$ and the braids where one strand of the $i$th color passes another strand of the $i$th color. (More precisely, if some of the $d_i$ are $\leq 1$, the rank of the abelianization is lower as we drop some of the generators. When specifying a representation by specifying the action of each generator, we simply ignore the generators that do not appear for our choice of $d_1,\dots,d_r$.)

\begin{lemma}\label{colored-braid-calculation} Let $W$ be the unique one-dimensional representation of the colored braid group on  $d_i$ strands of the $i$th color for $i$ from $1$ to $r$ whose value on the braid where a strand of the $i$th color makes a loop around a strand of the $j$th color for $i<j$ is $\gene^{M_{ij}}$ and whose value on the braid where one strand of the $i$th color passes another strand of the $i$th color is $\gene^{M_{ii}}$.

Then $\mathcal L_\chi (F_{d_1,\dots, d_r}) \cong \mathcal L\{W\}$. \end{lemma}

\begin{proof} The action of a generator of the abelianization of the braid group on the representation corresponding to the local system $\mathcal L_{\chi}(F_{d_1,\dots, d_r})$, is the monodromy of $\mathcal L_{\chi}(F_{d_1,\dots, d_r})$ over the corresponding paths in $\operatorname{Conf}_{d_1,\dots, d_r} (\mathbb C) $.

Following a path in $\operatorname{Conf}_{d_1,\dots, d_r} (\mathbb C) $ where a point of the $i$th color loops clockwise around a point of the $j$th color, the resultant $\Res(f_i,f_j)$ winds one time around the origin: The resultant is the product of differences of the roots of $f_i$ corresponding to the points of color $i$ and the roots of $f_j$ corresponding to the points of color $j$. The difference between the specific points that loop around each other has winding number $1$ around the origin, while the other differences have winding number $0$, so the product has winding number $1$. Similarly, the other resultants and discriminants in the definition of $F_{d_1,\dots, d_r}$ have winding number $0$, so $F_{d_1,\dots,d_r}$ has winding number $M_{ij}$ on this path. The $q-1$st root of $F_{d_1,\dots, d_r}$ is then multiplied by $e^{2\pi i \frac{M_{ij}}{q}} $ after following this path, and so the local system has monodromy $\chi ( e^{2\pi i \frac{M_{ij}}{q}}) =\gene^{M_{ij}}$.

Following a path in $\operatorname{Conf}_{d_1,\dots, d_r} (\mathbb C) $ where a point of the $i$th color passes another point of the $i$th color, the discriminant $\Res(f_i',f_i)$ has winding number $1$, since it is (up to sign) the product of the squares of the differences of each pair of points of the $i$th color, and the pair of points that pass each other and thus have their difference negated give the square of their difference a winding number of $1$, while all other squared differences have winding number zero. By the same reasoning, the local system $\mathcal L_\chi( F_{d_1,\dots, d_r})$ has monodromy $\gene^{M_{ii}}$. \end{proof}

It follows from Lemma \ref{colored-braid-calculation} that
\begin{equation}\label{0-to-top} \begin{aligned}  &H^{j-\hat{d}} ( \mathbb A^{d_1 } \times \dots \times \mathbb A^{d_r }_{\mathbb C},   {u} _{!*}   (\mathcal L_{\chi}(F_{d_1,\dots,d_r}) [\hat{d}])) \\ =& H^{j-\hat{d}}  ( \operatorname{Sym}^{d_1} (\mathbb C) \times \dots \times \operatorname{Sym}^{d_r}(\mathbb C), u_{!*} \mathcal L\{W\} [\hat{d}]  ) \end{aligned} \end{equation} with $W$ the representation defined in Lemma \ref{colored-braid-calculation} and $u_{!*}$ the intermediate extension from $\operatorname{Conf}_{d_1,\dots, d_r}(\mathbb C)$ to $ \operatorname{Sym}^{d_1} (\mathbb C) \times \dots \times \operatorname{Sym}^{d_r}(\mathbb C)$

Similarly we have
\begin{lemma}\label{zero-to-top} \[ H^{j-\hat{d}} ( \mathbb A^{d_1 } \times \dots \times \mathbb A^{d_r }_{\mathbb C},   {u} _{!*}   (\mathcal L_{\chi}(F_{d_1,\dots,d_r}) [\hat{d}]))\] \[ = H^{j-\hat{d}}  ( \operatorname{Sym}^{\hat{d} } (\mathbb C),  \overline{u}_{!*} \mathcal L\{\operatorname{ind}_{B_{d_1,\dots,d_r}}^{B_{\hat{d}}} W\} [\hat{d}])  \] where $W$ is the representation defined in Lemma \ref{colored-braid-calculation}, $\operatorname{Ind}_{B_{d_1,\dots,d_r}}^{B_{\hat{d}}}$  denotes the induced representation from the colored braid group to the full braid group, and $\overline{u}_{!*}$ is the intermediate extension from $\operatorname{Conf}_{\hat{d}}(\mathbb C)$ to $\operatorname{Sym}^{\hat{d} } (\mathbb C)$. \end{lemma}
\begin{proof} This follows from Lemma \ref{colored-braid-calculation} because pushforward of local systems along covering spaces corresponds to induced representations of their fundamental groups. \end{proof}

\begin{lemma}\label{representation-isomorphism} Suppose now that $M = M^{ \hchi,E,\gene}$ for some basis $E$ of $\mathbb Z^r$ and bicharacter $\hchi$. Then the diagonal braided vector space $V'_{\hchi,E}$ defines a representation $V^{ '\otimes (\hat{d})}_{\hchi ,E} $ of the braid group. Let $ \Gr^{ \sum_{i=1}^r d_i e_i} V^{ '\otimes (\hat{d})}_{\hchi ,E}$ be its associated graded component with grade $\sum_{i=1}^r d_i e_i$.

We have an isomorphism of braid group representations
\[  \Gr^{ \sum_{i=1}^r d_i e_i} V^{ '\otimes (\hat{d})}_{\hchi ,E} \cong  \operatorname{Ind}_{B_{d_1,\dots,d_r}}^{B_{\hat{d}}} W. \] \end{lemma}

\begin{proof} The vector space $V^{ '\otimes (\hat{d})}_{\hchi ,E}$ admits a basis consisting of tensor products of $\hat{d}$-tuples of the $x_i$ with $x_i$ graded as $e_i$. Thus the component $ \Gr^{ \sum_{i=1}^r d_i e_i} V^{ '\otimes (\hat{d})}_{\hchi ,E}$ admits a basis consisting of tensor products where $x_i$ appears exactly $d_i$ times (recalling that the $e_i$ are a basis and hence linearly independent). Each basis vector generates a one-dimensional subspace, whose direct sum is the full vector space, and the action of the braid group sends each one-dimensional subspace to another one-dimensional subspace, with the action on the subspaces given by the homomorphism $B_n \to S_n$ composed with the action of $S_n$ on $n$-tuples of symbols $x_1,\dots, x_r$ where the $i$th symbol occurs $d_i$ times. This condition, that a representation is a sum of subspaces permuted by the elements of a group, is exactly the condition for the representation to be an induced representation: Specifically, the representation is $\operatorname{Ind}_H^G W$ where $W$ is one subspace, $H$ is the stabilizer of that subspace, and $G$ is the full group.

The stabilizer of the subspace generated by $x_1^{ \otimes d_1} \otimes \dots\otimes x_r^{\otimes d_r}$ is the colored braid group on $\hat{d}$ strands, with $d_i$ strands of the $i$th color. A braid where a strand of the $i$th color passes another strand of the $i$th color acts on this subspace as $-\hchi(e_i,e_i) = \gene^{M_{ii}}$. A braid where a strand of the $i$th color loops around a strand of the $j$th color can be expressed as a series of simple braids to bring those two strands next to each other, a simple braid where the strand of the $i$th color crosses over the strand of the $j$th color, a simple braid where the strand of the $j$th color now crosses over the strand of the $i$th color, followed by the inverse of the original sequence of simple braids. Each simple braid induces a scalar multiplication, with the initial sequence and final sequence cancelling, and the two middle steps contributing $-\hchi(e_i,e_j)$ and $-\hchi(e_j,e_i)$ respectively, so this braid acts by multiplication by \[(-\hchi(e_i,e-j)) (-\hchi(e_j,e_i)) = \hchi(e_i,e_j)\hchi(e_j,e_i) = \gene^{M_{ij}}.\]

This identifies the one-dimensional character with the character $W$. \end{proof} 

It follows from Lemma \ref{representation-isomorphism} that

\begin{equation}\label{top-to-KS} \begin{aligned} 
& H^{j-\hat{d}}  ( \operatorname{Sym}^{\hat{d} } (\mathbb C),  \overline{u}_{!*} \mathcal L\{\operatorname{ind}_{B_{d_1,\dots,d_r}}^{B_{\hat{d}}} W\} [\hat{d}]) \\
= & H^{j-\hat{d}}  ( \operatorname{Sym}^{\hat{d} } (\mathbb C),  \overline{u}_{!*} \mathcal L\{ V^{ '\otimes (\hat{d})}_{\hchi ,E}[\sum_{i=1}^r d_i e_i]  \} [\hat{d}] )\\
 = & H^{j-\hat{d}}  ( \operatorname{Sym}^{\hat{d} } (\mathbb C),  \overline{u}_{!*} \mathcal L\{ V^{ '\otimes (\hat{d})}_{\hchi ,E} \} [\hat{d}]) [\sum_{i=1}^r d_i e_i] \end{aligned}\end{equation} since taking the sheaf associated to a local system, intermediate extension, and cohomology all commute with taking a graded component.

The cohomology of the intermediate extension was essentially computed by Kapranov and Schechtman and this leads to the following lemma:

\begin{lemma}\label{ks-to-nichols} We have (where $\mathbf 1$ is the trivial one-dimensional module) \[  \Gr^{ \sum_{i=1}^r d_i e_i}  H^{j-\hat{d}}  ( \operatorname{Sym}^{\hat{d} } (\mathbb C),  \overline{u}_{!*} \mathcal L\{ V^{ '\otimes (\hat{d})}_{\hchi ,E} \} [\hat{d}]) 
\cong  \Biggl(  \Gr^{ -\sum_{i=1}^r d_i e_i}  \operatorname{Ext}^{\hat{d}-j } _{ \mathfrak B( V_{\hchi,E} )} (\mathbf 1, \mathbf 1)  \Biggr)^\vee. \]\end{lemma}

\begin{proof} The proof is very similar to \cite[Corollary 3.3.4]{KapranovSchechtman}.

Kaparanov and Schechtman define \cite[Theorem 3.3.1]{KapranovSchechtman} a series of functors $L_n$ from primitive bialgebras to perverse sheaves on $\operatorname{Sym}^n (\mathbb C)$ which combine into a functor $L$ to certain tuples of perverse sheaves. It follows from \cite[Theorem 3.3.3(c)]{KapranovSchechtman} that $ L_{\hat{d}} ( \mathfrak B(V_{\hchi,E} )) =  \overline{u}_{!*} \mathcal L\{ V^{ '\otimes (\hat{d})}_{\hchi ,E} \} [\hat{d}] $.

The scaling action of $\mathbb G_m$ on $\mathbb C$ induces a scaling action of $\mathbb G_m$ on $\mathbb C$. The perverse sheaf $ \overline{u}_{!*} \mathcal L \{ V^{ '\otimes (\hat{d})}_{\hchi ,E} \}[\hat{d}] $ is equivariant under this action because it is locally constant on the diagonal stratification which is $\mathbb G_m$-invariant, so the cohomology $H^{j-\hat{d}}  ( \operatorname{Sym}^{\hat{d} } (\mathbb C),  \overline{u}_{!*} \mathcal L\{ V^{ '\otimes (\hat{d})}_{\hchi ,E} \} [\hat{d}] )$ of this complex is isomorphic to the stalk at $0$ by localization.

By \cite[Theorem 3.3.1(d)]{KapranovSchechtman}, the stalk of $\overline{u}_{!*} \mathcal L\{ V^{ '\otimes (\hat{d})}_{\hchi ,E} \}[\hat{d}] $ at $0$ is isomorphic to the cohomology of the bar complex computing $ \Gr^{ \hat{d}} \operatorname{Tor}_* ^{ \mathfrak B( V_{\hchi,E} )} (\mathbf 1, \mathbf 1)$, normalized so that $\operatorname{Tor}_j$ occurs in degree $-j$.

Matching degrees, \[ H^{j-\hat{d}}  ( \operatorname{Sym}^{\hat{d} } (\mathbb C),  \overline{u}_{!*} \mathcal L\{ V^{ '\otimes (\hat{d})}_{\hchi ,E} \} )= \Gr^{\hat{d}} \operatorname{Tor}_{\hat{d}-j} ^{ \mathfrak B( V_{\hchi,E} )} (\mathbf 1, \mathbf 1).\]

The relevant $\mathbb Z$-grading is related to the $\mathbb Z^r$-grading by sending a degree $\sum_{i=1}^r f_i e_i \in \mathbb Z^r$ to a degree $\sum_{i=1}^r f_i \in \mathbb Z$. In particular, it sends the degree $ \sum_{i=1}^r d_i e_i$ to $\hat{d}$. Thus, taking this Tor group,  \[  \Gr^{ \sum_{i=1}^r d_i e_i} \Gr^{ \hat{d} } \operatorname{Tor}_{\hat{d}-j} ^{ \mathfrak B( V_{\hchi,E} )} (\mathbf 1, \mathbf 1)=  \Gr^{ \sum_{i=1}^r d_i e_i}\operatorname{Tor}_{\hat{d}-j} ^{ \mathfrak B( V_{\hchi,E} )} (\mathbf 1, \mathbf 1).\]

Hence \[  \Gr^{ \sum_{i=1}^r d_i e_i} H^{j-\hat{d}}  ( \operatorname{Sym}^{\hat{d} } (\mathbb C),  \overline{u}_{!*} \mathcal L\{ V^{ '\otimes (\hat{d})}_{\hchi ,E} \} [\hat{d}]) \cong  \Gr^{ \sum_{i=1}^r d_i e_i} \operatorname{Tor}_{\hat{d}-j} ^{ \mathfrak B( V_{\hchi,E} )} (\mathbf 1, \mathbf 1) \] \[
\cong \Biggl(  \Gr^{ -\sum_{i=1}^r d_i e_i} \operatorname{Ext}^{\hat{d}-j } _{ \mathfrak B( V_{\hchi,E} )} (\mathbf 1, \mathbf 1)  \Biggr)^\vee \]
by the duality between $\operatorname{Ext}$ and $\operatorname{Tor}$ (arising from the fact that $\otimes \mathbf 1$ and $\operatorname{Hom}( -, \mathbf 1)$ are equivalent as functors from modules to vector spaces after composing with the dual vector space functor). \end{proof}

\begin{corollary}\label{K-to-nichols}There exists an isomorphism
\[ H^j( \mathbb A^{d_1 }\times  \dots \times \mathbb A^{ d_r }_{\overline{\mathbb F}_q}, K_{d_1,\dots, d_r}) \cong   \Biggl(  \Gr^{ -\sum_{i=1}^r d_i e_i}  \operatorname{Ext}^{\hat{d}-j } _{ \mathfrak B( V_{\hchi,E} )} (\mathbf 1, \mathbf 1)  \Biggr)^\vee\]
\end{corollary}

\begin{proof}
    This follows by combining \eqref{cohomology-by-definition}, Lemma \ref{0-to-p}, \eqref{0-to-top}, Lemma \ref{zero-to-top}, \eqref{top-to-KS}, and Lemma \ref{ks-to-nichols}.
\end{proof}

\subsection{Relations between cohomology groups from the functional equations}\label{ss:dimensions-cohomology} 

\begin{theorem}\label{homological-relation-theorem} Fix a finite field $\mathbb F_q$ and a character $\chi \colon \mathbb F_q^\times \to \mathbb C^\times$ of order $n$. Let $r$ be a natural number and $M$ an $r\times r$ symmetric matrix of integers in $[0,n)$.  Let $K_{d_1,\dots,d_r}^{M}$ be the shifted perverse sheaf on $\mathbb A^{d_1} \times \dots \times \mathbb A^{d_r} = \mathbb A^{\hat{d}}$ defined using the matrix $M$.

Fix $i$ from $1$ to $r$ and nonnegative integers $d_0,\dots, d_{i-1}, d_{i+1},\dots, d_r$.
Then
\begin{enumerate}
    \item If $\gcd(n, M_{ii} + M_{ij}/2)\mid M_{ij}$ for all $j$, let $n_i = \frac{n}{ \gcd(n, M_{ii} + \frac{n}{2})}$ and for all $j$ let $n_{ij}$ be the unique integer in $[0,n_i)$ with $-M_{ij} \equiv n_{ij} (M_{ii}+\frac{n}{2} ) \bmod n$. Let 
    \[ h^j_{d_i} = \dim H^{\hat{d}-j}(\mathbb A^{\hat{d}}_{\overline{\mathbb F}_q}, K^{M}_{d_1,\dots, d_r})\]  and $\tilde{d} =\sum_{j \neq i} d_j n_{ij}$.
    Then we have 
    \begin{equation}\label{hom-k-g-state} h^j_{d_i} - h^{j-1}_{d_i - (2d_i-1-\tilde{d})\%n_i} = h^{j}_{\tilde{d}+1-d_i - (\tilde{d}+1-2d_i)\% n_i} - h^{j-1}_{\tilde{d}+1-d_i-n_i} \end{equation}
    unless $2d_i \equiv 1 + \tilde{d}\bmod n_i$, in which case 
    \begin{equation}\label{hom-k-s-state} h^j_{d_i} = h^j_{\tilde{d}+1-d_i-n_i}.\end{equation}

    \item If $\hchi(e_i,e_i)=-1$, let \[ \tilde{d} = \sum_{\substack{j \neq i\\ M_{ij}\neq 0 }} d_j .\] let
    \[ h^{j,M}_{d_i} = \dim H^{\hat{d}-j}(\mathbb A^{\hat{d}}_{\overline{\mathbb F}_q}, K^{M}_{d_1,\dots, d_r})\] and
        \[ h^{j, \tau_i(M)}_{d_i}= \dim H^{\hat{d}-j}(\mathbb A^{\hat{d}}_{\overline{\mathbb F}_q}, K^{\tau_i(M)}_{d_1,\dots, d_r})\] with $\tau_i(M)$ defined as in \S\ref{ss-multivariable-functional-equation}.
Then
    \begin{equation}\label{hom-d-g-state} h^{j,M}_{d_i}= h^{j,\tau_i(M)}_{\tilde{d}-d_i-1} \end{equation} unless $\sum_{k\neq i} d_k M_{ik} \equiv 0 \bmod n$, in which case
\begin{equation}\label{hom-d-s-state} h^{j,M}_{d_i}- h^{j-1,M}_{d_i-1} = h^{j,\tau_i(M)}_{\tilde{d}-d_i} - h^{j-1,\tau_i(M)}_{\tilde{d}-d_i-1}  .\end{equation}
    
\end{enumerate}

\end{theorem}

\begin{proof} Applying a permutation of $\{1,\dots, r\} $ preserves everything except that the polynomial $F_{d_1,\dots,d_r}$ may be multiplied by $-1$ which has the effect of taking the tensor product of $K_{d_1,\dots,d_r}$ with a one-dimensional Galois representation. This does not affect the Betti numbers. Hence we may without loss of generality assume $i=r$. We now handle each case in turn.

In case (1), consider the local series, specializing to $\pi=T $,
\begin{equation} D_T ( x, T^{\vec{d}}, k ) = \sum_{\substack{ d_r \geq 0 \\ d_r \equiv k \bmod n}} a ( T^{d_1},\dots, T^{d_{r}})  x^{d_r } =\sum_{\substack{ d_r \geq 0 \\ d_r \equiv k \bmod n}} \sum_{j} (-1)^j \operatorname{tr}( \operatorname{Frob}_q, H^j( \mathbb A^{\hat{d}}_{\overline{\mathbb F_q}}, K_{d_1,\dots,d_{r}})x^{d_r}.\end{equation}

Then the local functional equation \eqref{LocalFE}, established in Theorem \ref{TheoremFE}, states that  $D_T^* ( x, T^{\vec{d}}, k )$ is a rational function of $x$ with denominator dividing $1-q^{-1 } \left( \frac{qx}{g(1,T)}\right)^{n_r }$ and 
\begin{equation} 
\left( D_T ( x, T^{\vec{d}}, k ) \right)_k =  \left( \frac{qx}{g} \right)^{\tilde{d}} \left( \Gamma_{T, k, \ell}(x, \tilde{d}) \right)_{k, \ell}  \left(  D_T^* ( \frac{g^2}{q^2x}, T^{\vec{d}}, \ell )\right)_\ell 
\end{equation}
where $ \left( \Gamma_{T, k, \ell}(x, \tilde{d}) \right)_{k, \ell}$ is the $n_r \times n_r$ scattering matrix with
\begin{equation*}
\begin{split}
\Gamma_{T, k,k} =& \left( \frac{qx}{g}\right)^{ 1- (\tilde{d}+1-2k)\% n_r } \frac{1-q^{-1}}{1-q^{-1}\left( \frac{q x}{g}\right)^{n_r}}
\end{split}
\end{equation*}
\begin{equation*}
\begin{split}
\Gamma_{T, k,\ell} =& (\xi \chi^{M_{rr}}) ^{\ell(1+\tilde{d})}(-1)\frac{g_{(\xi \chi^{M_{rr}} )^{2k-\tilde{d} -1}}}{q} \\
 &\left( \frac{qx}{g}\right) \frac{1-\left(\frac{q x}{g}\right)^{-n_r}}{1-q^{-1}\left(\frac{q x}{g}\right)^{n_r}}
 \end{split}
\end{equation*}for $\ell \equiv 1+\tilde{d} -k \bmod n$, and all other entries equal to $0$. In addition, there are special entries in the matrix when $k \equiv \ell \equiv 1+\tilde{d}-k \bmod n$. In this case 
\begin{equation*}
\Gamma_{k,\ell}=\left(\frac{qx}{g}\right)^{1-n_r}
\end{equation*}

We may clear denominators by multiplying both sides by  $1-q^{-1 } \left( \frac{qx}{g}\right)^{n_r} $, which makes the left side, and hence the right side, a polynomial in $x$. After doing this, we have an equality of coefficients of each power of $x$.

Each coefficient of a power of $x$ in the expression on both sides is a signed sum of Weil numbers. Furthermore, the analogous coefficients for the Dirichlet series over $\mathbb F_{q^e}$ are obtained by taking the signed sums of the $e$th powers of the Weil numbers. Since the same identity holds over $\mathbb F_{q^e}$ for each $e$, it follows that the same Weil numbers, with the same multiplicities, must appear in the coefficient of each power of $x$ on both sides. We now introduce a formal variable $u$ and consider the weight-extraction operation that replaces a Weil number of absolute value $q^{ w/2}$ by $(-u)^w$. Weight-extraction is a ring homomorphism from the ring of sets of ordered pairs of Weil numbers and integers to the ring of polynomials in $u$ with integer coefficients.

The cohomology group $H^j( \mathbb A^{\hat{d} }_{\overline{\mathbb F}_q}, K_{d_1,\dots, d_r})$ is pure, i.e. all Frobenius eigenvalues on it have size $q^{j/2}$. This is because $K_{d_1,\dots, d_r}$ is pure so the eigenvalues of Frobenius on its stalk in degree $j$ have size $\leq q^{j/2}$, while the eigenvalues of Frobenius on its global cohomology have size $\geq q^{j/2}$, but these are equal by \cite[Lemma 2.16]{s-amds} and thus all have size exactly $q^{j/2}$.\footnote{This observation was pointed out to one of the authors by Kevin Chang.}

Hence applying weight-extraction to the set of Frobenius eigenvalues on $H^j( \mathbb A^{\hat{d} }_{\overline{\mathbb F}_q}, K_{d_1,\dots, d_r})$, i.e. the set of Weil numbers contributing to $\tr( \Frob_q , H^j( \mathbb A^{\hat{d} }_{\overline{\mathbb F}_q}, K_{d_1,\dots, d_r}))$, produces $(-u)^j \dim H^j( \mathbb A^{\hat{d} }_{\overline{\mathbb F}_q}, K_{d_1,\dots, d_r}))$. Thus applying weight-extraction to the set of ordered pairs of Weil numbers and integers computing $\sum_j (-1)^j \tr( \Frob_q , H^j( \mathbb A^{\hat{d} }_{\overline{\mathbb F}_q}, K_{d_1,\dots, d_r})) $ produces $\sum_j u^j \dim H^j_c( \mathbb A^{\hat{d} }_{\overline{\mathbb F}_q}, K_{d_1,\dots, d_r}))$, i.e., the Poincar\'e polynomial.

Also, weight-extraction sends $q$ to $u^2 $, $g$ to $u$, and $\chi(-1)$ to $1$. Hence the Dirichlet series
\[ W(x,u, \vec{d},k) = \sum_{\substack{ d_r \in \mathbb Z^{\geq 0} \\ d_r \equiv k \bmod n }} \sum_j u^j \dim H^j( \mathbb A^{\hat{d} }_{\overline{\mathbb F}_q}, K_{d_1,\dots, d_r})) x^{d_r} \]
satisfies the functional equation
\begin{equation} 
\left( \left(1-u^{-2}  \left( ux \right)^{n_r}\right)  W (x,u ,  \vec{d}, k) \right)_k = (ux)^{\tilde{d} } \left( \Omega_{k, \ell}(x, u, \tilde{d}) \right)_{k, \ell} \left( W\left(\frac{1}{u^2x},u,  \vec{d}, \ell\right) \right)_{\ell}
\end{equation}
where $\left( \Omega_ {k, \ell}(x, u, \tilde{d}) \right)_{k, \ell}$ is the $n_r \times n_r$ scattering matrix with
\begin{equation*}
\Omega_{k,k} = \left( ux \right)^{1-(\tilde{d}+1-2k)\% n_r }  ( 1-u^{-2} )
\end{equation*}
\begin{equation*}
\begin{split}
\Omega_{k,\ell}=& x\left( 1-\left(ux \right)^{-n_r}\right) \end{split}
\end{equation*}
for $\ell \equiv 1+\tilde{d} -k \bmod n$, and all other entries equal to $0$. In addition, there are special entries in the matrix when $k \equiv \ell \equiv 1+\tilde{d} -k \bmod n$. In this case 
\begin{equation*}
\Omega_ {k,\ell}= \left(1-u^2 \left( ux \right)^{n_r}\right)   \left( ux\right)^{1-n_r } 
\end{equation*}

We can simplify the functional equation further. The renormalized series
\[ \overline{W} (x,u, \vec{d},k) = \sum_{\substack{ d_r \in \mathbb Z^{\geq 0} \\ d_r \equiv k \bmod n }} \sum_j u^j \dim H^j( \mathbb A^{\hat{d} }_{\overline{\mathbb F}_q}, K_{d_1,\dots, d_r})) (u^{-1} x) ^{d_r}, \]
where we have substituted $u^{-1}x$ for $x$, satisfies the functional equation
\begin{equation} 
\left( \left(1-u^{-2} x^{n_r} \right)  \overline{W} (x,u ,  \vec{d}, k) \right)_k = \left( \overline{\Omega}_{k, \ell}(x, u, \tilde{d}) \right)_{k, \ell} \left( \overline{W}  \left(x^{-1} ,u,  \vec{d}, \ell\right) \right)_{\ell}
\end{equation}
where $\left( \overline{\Omega}_ {k, \ell}(x, u, \tilde{d}) \right)_{k, \ell}$ is the $n \times n$ scattering matrix with
\begin{equation*}
\overline{\Omega}_{k,k} =x^{ \tilde{d}+1 - (\tilde{d}+1-2k)\% n_r }  \left( 1-u^{-2} \right)
\end{equation*}
\begin{equation*}
\begin{split}
\overline{\Omega}_{k,\ell}=&u^{-1} x^{\tilde{d}+1}\left( 1-x^{-n_r }\right) \end{split}
\end{equation*}
for $\ell \equiv 1+\tilde{d} -k \bmod n$, and all other entries equal to $0$. In addition, there are special entries in the matrix when $k \equiv \ell \equiv 1+\tilde{d} -k \bmod n$. In this case 
\begin{equation*}
\overline{\Omega}_ {k,\ell}= \left(1-u^{-2} x^{n_r}\right)   x^{\tilde{d}+1-n_r} 
\end{equation*}

Extracting the coefficient of $u^{ d_1 + \dots + d_{r-1}-j}  x^{d_r}$ on each side and recalling that $h^j_{d_r} =  \dim H^{\hat{d}-j} ( \mathbb A^{\hat{d} }_{\overline{\mathbb F}_q}, K_{d_1,\dots, d_r}) $, we see that
\begin{equation}\label{hom-k-g-pre}  h^j_{d_r} - h^{j-2}_{d_r-{n_r}}  =  h^{j-1 }_{\tilde{d} +1 -d_r} -  h^{j-1}_{\tilde{d}+1 -d_r-n_r} + h^j_{ \tilde{d}+1-d_r - ( \tilde{d}+1-2d_r)\%n_r}  -  h^{j-2} _{\tilde{d}+1-d_r - ( \tilde{d}+1-2d_r)\%n_r} \end{equation}
unless $2d_{r} \equiv 1 + \tilde{d} \bmod n_r$, in which case
\begin{equation}\label{hom-k-s-pre} h^j_{d_r} - h^{j-2}_{d_r-n_r} = h^j_{\tilde{d}+1 -d_r-n_r} - h^{j-2}_{ \tilde{d}+1-d_r }. \end{equation}
Both of these equations may be simplified by a telescoping sum. 
If we substitute $j-2k$ for $j$ and $d_r-n_rk$ for $d_r$ into \eqref{hom-k-s-pre} we obtain 
\begin{equation}\label{hom-k-s-2} h^{j-2k}_{d_r-kn_r} - h^{j-2(k+1) }_{d_r-(k+1) n_r} = h^{j-2k} _{\tilde{d}+1 -d_r+ (k-1) n_r} - h^{j-2(k+1) }_{ \tilde{d}+1-d_r+ kn_r } \end{equation}
and then summing \eqref{hom-k-s-2} over $k$ from $0$ to $\infty$ we obtain
\begin{equation}\label{hom-k-s-final} \begin{aligned} h^j_{d_r}=&  ( h^j_{d_r} - h^{j-2}_{d_r-n_r}) + ( h^{j-2}_{d_r-n_r} - h^{j-4}_{d_r-2n_r}) + \dots \\=&(h^j_{\tilde{d}+1 -d_r-n_r} - h^{j-2}_{ \tilde{d}+1-d_r }) + (h^{j-2}_{\tilde{d}+1 -d_r} - h^{j-4}_{ \tilde{d}+1-d_r +n_r  }) + \dots   = h^j_{\tilde{d}+1-d_r-n_r} .\end{aligned}\end{equation}
Note that all $h^j_d$ with $j<0$ vanish so the sum has only finitely many nonvanishing terms and thus our rearrangement is justified. \eqref{hom-k-s-final} gives \eqref{hom-k-s-state}.

If we set $d_{r}' = d_r - (2d_r -1 -\tilde{d})\% n_r$ then we see that $d_r + (\tilde{d}+1 -2d_r) = d_r' + n_r $ and that \[d_r'' = d_r - (2d_r -1 -\tilde{d})\% n_r - (2d_r' -1 -\tilde{d})\% n_r=d_r - (2d_r -1 -\tilde{d})\% n_r - ( 1 +\tilde{d}-2d_r )\% n_r= d_r -n_r \] so we can rewrite \eqref{hom-k-g-pre} as
\begin{equation}\label{hom-k-g-0}  h^j_{d_r} - h^{j-2}_{d_r-{n_r}}  =  h^{j-1 }_{\tilde{d} +1 -d_r} -  h^{j-1}_{\tilde{d}+1 -d_r-n_r} + h^j_{ \tilde{d}+1-d_r' -n_r}  -  h^{j-2} _{\tilde{d}+1-d_r' -n_r} .\end{equation}
Substituting $j-2k$ for $j$ and $d_r-n_r k$ for $d_r$ into \eqref{hom-k-g-0} we obtain 
\begin{equation}\label{hom-k-g-1}  h^{j-2k}_{d_r-kn_r} - h^{j-2k-2 }_{d_r-(k+1) n_r}  =  h^{j-2k-1 }_{\tilde{d} +1 - d_r +kn_r} -  h^{j-2k-1}_{\tilde{d}+1 -d_r+(k-1) n_r} + h^{j-2k}_{ \tilde{d}+1-d_r' +(k-1) n_r}  -  h^{j-2k-2} _{\tilde{d}+1-d_r' +(k-1) n_r} .\end{equation}
Substituting $j-2k-1$ for $j$ and $d_r'-n_r k$ for $d_r$ into \eqref{hom-k-g-0} we obtain 
\begin{equation}\label{hom-k-g-2}  h^{j-2k-1}_{d_r'-kn_r} - h^{j-2k-3 }_{d_r'-(k+1) n_r}  =  h^{j-2k-2 }_{\tilde{d} +1 -d_r' +k n_r} -  h^{j-2k-2}_{\tilde{d}+1 -d_r'+ (k-1) n_r} + h^{j-2k-1}_{ \tilde{d}+1-d_r +k n_r}  -  h^{j-2k-3} _{\tilde{d}+1-d_r +k n_r} .\end{equation}
Summing \eqref{hom-k-g-1} over $k$ from $0$ to $\infty$ together with minus \eqref{hom-k-g-2} over $k$ from $0$ to $\infty$ we obtain
\begin{equation}\label{hom-k-g-final} \begin{aligned} &h^j_{d_r} - h^{j-1}_{d_r'} \\
= &( h^j_{d_r} - h^{j-2}_{d_r-n_r})- (h^{j-1}_{d_r'}- h^{j-3}_{d_r' -n_r}) + (h^{j-2}_{d_r-n_r} - h^{j-4}_{d_r-4n_r}) - ( h^{j-3}_{d_r'+n_r}- h^{j-5}_{d_r'+2n_r}) + \dots \\
= & (h^{j-1 }_{\tilde{d} +1 -d_r} -  h^{j-1}_{\tilde{d}+1 -d_r-n_r} + h^j_{ \tilde{d}+1-d_r' -n_r}  -  h^{j-2} _{\tilde{d}+1-d_r' -n_r}) 
-   ( h^{j-2 }_{\tilde{d} +1 -d_r' } -  h^{j-2}_{\tilde{d}+1 -d_r'- n_r} + h^{j-1}_{ \tilde{d}+1-d_r }  -  h^{j-3} _{\tilde{d}+1-d_r })\\
+ &  (h^{j-3 }_{\tilde{d} +1 - d_r +n_r} -  h^{j-3}_{\tilde{d}+1 -d_r} + h^{j-2}_{ \tilde{d}+1-d_r' }  -  h^{j-4} _{\tilde{d}+1-d_r' }) 
-   ( h^{j-4}_{\tilde{d} +1 -d_r' + n_r} -  h^{j-4}_{\tilde{d}+1 -d_r'} + h^{j-3}_{ \tilde{d}+1-d_r + n_r}  -  h^{j-5} _{\tilde{d}+1-d_r + n_r})+\dots \\
&= -h^{j-1}_{\tilde{d}+1-d_r-n_r} + h^j_{\tilde{d}+1-d_r'-n_r}.
\end{aligned}\end{equation}
Again only finitely many summands are nonzero so rearrangement is justified, and now \eqref{hom-k-g-final} implies \eqref{hom-k-g-state}.

We now consider the second case. Note that the definition of $M'$ matches $\tau_r(M)$. 

Consider the local series
\begin{equation} D_T ( x, T^{\vec{d}}; M  ) = \sum_{d_r=0}^\infty  a ( T^{d_1},\dots, T^{d_{r}})  x^{d_r } ; q, \chi, M) =\sum_{d_r=0}^\infty \sum_{j} (-1)^j \operatorname{tr}( \operatorname{Frob}_q, H^j( \mathbb A^{\hat{d}}_{\overline{\mathbb F_q}}, K^{M}_{d_1,\dots,d_{r}})x^{d_r}.\end{equation}
and similarly
\begin{equation} D_T ( x, T^{\vec{d}}; M'  ) = \sum_{d_r=0}^\infty  a ( T^{d_1},\dots, T^{d_{r}})  x^{d_r } ; q, \chi, M') =\sum_{d_r=0}^\infty \sum_{j} (-1)^j \operatorname{tr}( \operatorname{Frob}_q, H^j( \mathbb A^{\hat{d}}_{\overline{\mathbb F_q}}, K^{M'}_{d_1,\dots,d_{r}})x^{d_r}.\end{equation}

Then the local functional equation \eqref{DirichletLocalFE}, established by Theorem \ref{DirichletFE},  states that
\begin{equation} 
\begin{split}
D_{T}(x, T^{\vec{d}}; M  ) &=  \omega_T \left(\frac{g_{\chi^{\sum_{i=1}^{r-1} d_i M_{ir}}}}{q}\right)^{b_{T}-1} (\chi(-1)^{\sum_{i=1}^{r-1} d_i M_{ir}}qx)^{\tilde{d}-1 } \\
&\prod_{\substack{1 \leq i \leq r-1 \\ M_{ir} \neq 0}} \left(\frac{g_{\chi^{M_{ir}}}}{q}\right)^{d_i} \left( \frac{qx-1}{1-x}\right)^{b_{T}} D_{T}(\frac{1}{qx}, T^{\vec{d}}; M')
\end{split}
\end{equation}
where $b_{T}=1$ if $n \mid \sum_{i=1}^{r-1} d_i M_{ir} $ and $0$ otherwise, and 
\begin{equation} 
\omega_T =\omega_T(T^{\vec{d}} ; M)= \prod_{\substack{1 \leq i \leq r-1 \\ M_{ir} \neq 0}} (-1)^{ d_i (d_i-1)(q-1)/4} \prod_{\substack{1 \leq i<j \leq r-1 \\ M_{ir}, M_{jr} \neq 0}} \chi(-1)^{d_i d_j M_{ir} } = \pm 1.
\end{equation}
The functional equation also states that $D_{T}(x, T^{\vec{d}}; M  )$ and $D_{T}(x, T^{\vec{d}}; M'  )$ are rational functions with denominator dividing $(1-x)^{b_T}$. Hence we may multiply both sides by $(1-x)^{b_T}$ to obtain an equation bewteen polynomials in $x$. After doing this, we have an equality of coefficients of each power of $x$.

We now define the Dirichlet series \[W(x, u, \vec{d}; M) = \sum_{d_r=0}^\infty \sum_j u^j \dim H^j ( \mathbb A_{\overline{\mathbb F}_q}^{\hat{d}}, K_{d_1,\dots,d_r}^M x^{d_r})\] 
Since the above functional equation holds over $\mathbb F_{q^e}$ for each $e$, we may apply the weight-extraction operation, exactly as in the previous argument, to obtain the functional equation
\[ (1-x)^{b_T}  W(x, u, \vec{d}; M) =  \left( \frac{u}{u^2} \right)^{b_{T}-1} (u^2 x)^{\tilde{d}-1 } 
\prod_{\substack{1 \leq i \leq r-1 \\ M_{ir} \neq 0}} \left(\frac{u}{u^2}\right)^{d_i} \left( u^2 x-1\right)^{b_{T}} W( \frac{1}{u^2 x}, u, \vec{d}; M') \]
which further simplifies to
\[ (1-x)^{b_T}  W(x, u, \vec{d}; M) =   u^{-b_T}  (u x)^{\tilde{d}-1 } d_i
 \left( u^2 x-1\right)^{b_{T}} W(\frac{1}{u^2 x}, u, \vec{d}; M') .\]
For the renormalized Dirichlet series
 \[\overline{W}(x, u, \vec{d}; M) = \sum_{d_r=0}^\infty \sum_j u^j \dim H^j ( \mathbb A_{\overline{\mathbb F}_q}^{\hat{d}}, K_{d_1,\dots,d_r}^M (u^{-1} x) ^{d_r})\] 
we obtain
\[ (1-u^{-1} x)^{b_T}  \overline{W} (x, u, \vec{d}; M) =   u^{-b_T}  x^{\tilde{d}-1 } 
  \left( u x-1\right)^{b_{T}} \overline{W} (\frac{1}{x} , u, \vec{d}; M') .\]
Extracting the coefficient of $u^{d_1+\dots + d_{r-1} - j} x^{d_r}$ on each side and recalling that $h^{j,M}_{d_r} =  \dim H^{\hat{d}-j} ( \mathbb A^{\hat{d} }_{\overline{\mathbb F}_q}, K_{d_1,\dots, d_r}^M ) $ and $h^{j,M' }_{d_r} =  \dim H^{\hat{d}-j} ( \mathbb A^{\hat{d} }_{\overline{\mathbb F}_q}, K_{d_1,\dots, d_r}^{M'} ) $, we obtain
\[ h^{j,M}_{d_r}=  h^{j, M'}_{ \tilde{d}-1-d_r}\]
if $b_T=0$ and
\[ h^{j,M}_{d_r} - h^{j-1,M}_{d_r-1} = h^{j, M'}_{ \tilde{d}-d_r} -h^{j-1, M'}_{ \tilde{d}-1- d_r}\]
otherwise. These give \eqref{hom-d-g-state} and \eqref{hom-d-s-state} respectively.
\end{proof}

\begin{proof}[Proof of Corollary \ref{nichols-relation-theorem}] This follows immediately from Theorem \ref{homological-relation-theorem}, taking $M= M^{ \hchi, E,\gene}$, by the following reasoning: In case 1, Corollary \ref{K-to-nichols} shows that $h^j_{d_i}$ as defined here and in Theorem \ref{homological-relation-theorem} agree, and Lemma \ref{kubota-bicharacter-identity} $m_{ik}=n_{ik}$ so that the definitions of $\tilde{d}$ agree.  In case 2, Lemma \ref{dirichlet-bicharacter-identity} shows that $ \tau_i(M^{\hchi,E,\gene}) = (M^{\hchi,E,\gene})'$ so that Corollary \ref{K-to-nichols} gives $h^{j,M}_{d_i}= h^{j,E}_{d_i}$ and $h^{j,\tau_i(M)}_{d_i}= h^{j,s_i(E)}_{d_i}$, and $\tilde{d}$ agrees by Lemma \ref{dirichlet-bicharacter-identity}.\end{proof}

The results of \S\ref{ss:connections-cohomology} allow us to give several other equivalent definitions of $h^j_{d_j}$ and thus several other variants of Theorem \ref{homological-relation-theorem}, but we do not state all of these.

We also raise the question of whether a generalized relation holds.

\begin{question}\label{nichols-relation-question} Let $r$ be a natural number, $\hchi\colon \mathbb Z^r \times \mathbb Z^r \to\mu_\infty$ a bicharacter, and $E=(e_1,\dots,e_r)$ an ordered basis of $\mathbb Z^r$.

Fix $i$ from $1$ to $r$ assume $E$ is admissible at $i$. Fix nonnegative integers $d_0,\dots, d_{i-1}, d_{i+1},\dots, d_r$. Let $\tilde{d}= \sum_{k \neq i} d_k m_{ik}$. Let $B \mathfrak B( V_{\hchi,E} )$ be the Nichols algebra of the diagonal braided vector space defined using $\hchi$ and $E$ as in \S\ref{nichols-intro}. Let
\[ h^{j,E}_{d_i} = \dim \Gr^{ -\sum_{i=1}^r d_i e_i}  \operatorname{Ext}^{j } _{ \mathfrak B( V_{\hchi,E} )} (\mathbf 1, \mathbf 1) \] and
        \[ h^{j, s_{i,E}(E)}_{d_i}= \dim \Gr^{ -\sum_{i=1}^r d_i e_i}  \operatorname{Ext}^{j } _{ \mathfrak B( V_{\hchi,s_{i,E}(E) } )} (\mathbf 1, \mathbf 1) .\] 

Does the following system of equations always hold? We have

   \begin{equation} h^{j,E}_{d_i} = h^{j, s_{i,E}(E)}_{\tilde{d}+1-d_i-n_i},\end{equation}
   unless there exists an integer $e$ with $\prod_{k\neq i}( \hchi(e_i,e_k) \hchi(e_k,e_i))^{d_k} = \hchi(e_i,e_i)^e$, in which case we have
    \begin{equation}h^{j,E}_{d_i} - h^{j-1,E}_{d_i - (2d_i+e -1)\%n_i} = h^{j,s_{i,E}(E)}_{\tilde{d}+1-d_i - (1-e-2d_i)\% n_i} - h^{j-1,s_{i,E}(E)}_{\tilde{d}+1-d_r-n_r},  \end{equation}
unless $2d_i +e \equiv 1 \bmod n_i$, in which case we again have
\begin{equation} h^{j,E}_{d_i} = h^{j, s_{i,E}(E)}_{\tilde{d}+1-d_i-n_i}.\end{equation}

\end{question}

Finally, we show that Corollary \ref{nichols-relation-theorem} uniquely determines the graded dimensions of cohomology groups of Nichols algebras. It will be clear from Theorem \ref{nichols-determination-theorem} that there exists a purely combinatorial algorithm to compute these Betti numbers.

\begin{theorem}\label{nichols-determination-theorem} Let $(\mathbf \Delta, \hchi, E)$ be an arithmetic root system such that each root $e$ in $ \mathbf \Delta$ satisfies either condition (1) or (2) in Corollary \ref{nichols-relation-theorem}. 

For each $E'=(e_1,\dots,e_r)\in B_{\hchi, E}$, let $h_{E'}(j,d_1,\dots, d_r)$ be a function from $r+1$-tuples of integers to nonnegative integers. Assume that $h_{E'}(j,d_1,\dots, d_r)=0$ if any argument is negative, i.e. if $j<0$ or $d_i<0$ for any $i$. 

For each $E' \in B_{\hchi,E}$, $i$ from $1$ to $r$, and nonnegative integers $d_1,\dots, d_{i-1}, d_{i+1},\dots, d_r$, let 
\[ h^{j,E'}_{d_i}= h_{E'} (j, d_1,\dots,d_r)  \textrm{ and } h^{j,s_{i,E'}(E')}_{d_i}= h_{s_{i,E'}(E')} (j, d_1,\dots,d_r) \]
and assume that conditions (1) and (2) below hold.
\begin{enumerate}
    \item If $\hchi(e_i,a)\hchi(a,e_i)$ is a power of $\hchi(e_i,e_i) $ for all $a\in \mathbb Z^r$, let \[ h^j_{d_i}= h_{E'} (j, d_1,\dots,d_r) .\] Then for all $d_i$ we have 
    \begin{equation}\label{asymmetric-kubota-general} h^{j,E'}_{d_i} - h^{j-1,E'}_{d_i - (2d_i-1-\tilde{d})\%n_i} = h^{j,s_{i,E'}(E')}_{\tilde{d}+1-d_i - (\tilde{d}+1-2d_i)\% n_i} - h^{j-1,s_{i,E'}(E')}_{\tilde{d}+1-d_i-n_i} \end{equation}
    unless $2d_i \equiv 1 + \tilde{d}\bmod n_i$, in which case 
    \begin{equation}\label{asymmetric-kubota-special} h^{j,E'}_{d_i} = h^{j,E'}_{\tilde{d}+1-d_i-n_i}.\end{equation}
\item If $\hchi(e_i,e_i)=-1$, then for all $d_i$ we have \eqref{nichols-dirichlet-general} or \eqref{nichols-dirichlet-special} as appropriate.
\end{enumerate}
Finally, assume that $h^{E'}(j,0,\dots,0)=\begin{cases} 1 & \textrm{if } j=0 \\ 0 & \textrm{if } j \neq 0 \end{cases}$ for all $E'$.

Then we have \[ h_{E'}(j, d_1,\dots, d_r) =  \dim \Gr^{ -\sum_{i=1}^r d_i e_i}  \operatorname{Ext}^{j } _{ \mathfrak B( V_{\hchi,E'} )} (\mathbf 1, \mathbf 1) .\]

Moreover, the converse holds, i.e. $\dim \Gr^{ -\sum_{i=1}^r d_i e_i}  \operatorname{Ext}^{j } _{ \mathfrak B( V_{\hchi,E'} )} (\mathbf 1, \mathbf 1)$ satisfies all the above assumptions. \end{theorem}
  
\begin{proof} That $\dim \Gr^{ -\sum_{i=1}^r d_i e_i}  \operatorname{Ext}^{j } _{ \mathfrak B( V_{\hchi,E'} )} (\mathbf 1, \mathbf 1)$ satisfies these relations is the content of \eqref{nichols-relation-theorem}, once one observes that in case (1) we have $h^{j,s_{i,E'}(E')} = h^{j, E'}$ as both of these have the same values of $\hchi(e_i,e_i)$ and $\hchi(e_i,e_j)\hchi(e_j,e_i)$ for all $i,j$ by Lemma \ref{kubota-bicharacter-identity}. Furthermore, that \[\dim \Gr^{ 0}  \operatorname{Ext}^{j } _{ \mathfrak B( V_{\hchi,E'} )} (\mathbf 1, \mathbf 1) = \begin{cases} 1 & \textrm{if } j=0 \\ 0 & \textrm{if } j \neq 0 \end{cases}\] is clear from the fact that $\operatorname{Ext}^0$ is $1$ and that the Ext groups are computed by the bar complex, from which we may see that all higher Ext groups have nonzero grades. Thus $\dim \Gr^{ -\sum_{i=1}^r d_i e_i}  \operatorname{Ext}^{j } _{ \mathfrak B( V_{\hchi,E'} )} (\mathbf 1, \mathbf 1)$ satisfies all the assumptions, and so it suffices to show that a solution is unique. 

We show that each value of $h_{E'}(j,d_1,\dots, d_r)$ is uniquely determined by the conditions by induction on $j$. In other words, we may assume that $h_{E'}(j-1,d_1,\dots, d_r)$ is uniquely determined for all $d_1,\dots,d_r$. (We take $j=-1$ to be the base case, where the unique determination holds by assumption.) Examining each of the relations \eqref{asymmetric-kubota-general}, \eqref{asymmetric-kubota-special}, \eqref{nichols-dirichlet-general}, \eqref{nichols-dirichlet-special}, we see that each one involves a single $h^j$ on the left hand side, a single $h^j$ on the right hand side, possibly some $h^{j-1}$s, and no other terms. Since we have assumed that the $h^{j-1}$ terms are uniquely determined, we see that the $h^j$ term on the left hand side is uniquely determined if and only if the $h^j$ term on the right hand side is, since one term in a linear equation is uniquely determined if all the other terms are.

Thus we may form a graph whose vertices are the set of tuples $\{ E', d_1,\dots,d_r\}$ with $E'\in B_{\hchi,E}$ and $d_1,\dots, d_r$ integers. Two vertices $\{ E', d_1,\dots,d_{i-1},d_i,d_{i+1},\dots,d_r\}$, $\{ s_{i,E'}(E'),d_1,\dots,d_{i-1}, d_i',d_{i+1},\dots,d_r \}$ with $E = (e_1,\dots,e_r)$ are connected for $d_i'$ given by:
\begin{enumerate}
    \item $\hchi(e_i,a)\hchi(a,e_i)$ is a power of $\hchi(e_i,e_i)$ for all $a\in \mathbb Z^r$ and, letting $n_i$ be the order of $\chi(e_i,e_i)$, we have $d_i' =\tilde{d}+1-d_i -(\tilde{d}+1-2d_i) \%n_i $, unless $2d_i \equiv 1 + \tilde{d}\bmod n_i$, which case $d_i'= \tilde{d}+1-d_i - n_i$. (Or, more simply, we may put $d_i' = \tilde{d}-d_i - (\tilde{d}-2d_i)\% n_i$.)
    \item $\hchi(e_i,e_i)=-1$, then $d_i'= \tilde{d}-d_i-1$, unless $\prod_{k\neq i} (\hchi(e_i,e_k) \hchi(e_k,e_i))^{d_k}=1$, in which case $d_i' = \tilde{d}-d_i$.
\end{enumerate}

It suffices to prove that each vertex of the graph is connected to some vertex $(E', d_1,\dots,d_r)$ where either some $d_i$ is negative or all $d_i$ vanish, as these values of $h_{E'}$ are uniquely determined.  In fact, the only thing we will use about these formulas is that $d_i'\leq \tilde{d}-d_i$. 

To do this, fix a vertex $(E', d_1,\dots,d_r)$. If one of $d_1,\dots,d_r$ is negative or  $d_1,\dots,d_r$ are all zero, we are done. Otherwise, choose a point $x$ in the interior of the corresponding cone $C_{E'}$. Choose a point $y$ such that the nonzero linear form $\sum_{k=1}^r d_k v\cdot e_k$ on $\mathbb R^r$ takes a negative value at $y$, and the line between $x$ and $y$ does not intersect the codimension $\geq 2$ faces of the cones. The line between $x$ and $y$ passes through some sequence of cones $E_1=E',\dots, E_m$. Since $E_t$ is adjacent to $E_{t+1}$, it must be connected by some face, say the $i_t$'th face. Define $(d_1^t,\dots,d_r^r)$ inductively by $(d_1^0,\dots, d_r^0)= ( d_1,\dots,d_r)$ and  \[ (d_1^{t+1},\dots, d^r_{t+1})= (d_1^t, \dots, d_{i_t-1}^t, (d_{i_t}^t)', d_{i_t-1}^t,\dots, d_r^t).\]

Consider the piecewise linear function $F$ on the line segment from $x$ to $y$ which is equal to $v\mapsto \sum_{k=1}^r d_{k}^t v\cdot e_k^t $ in the region where the line segment intersects the cone $E_t= (e_1^t,\dots, e_r^t)$. Let us check that $F$ is concave. Since $F$ agrees with the linear form $\sum_{k=1}^r d_k v\cdot e_k$ in the first part of the line segment, this will imply that $F$ bounded above by that linear form in the whole line segment, and in particular negative at $y$, so some $d_k^m$ is negative, giving the desired claim since each $ (E^t, d_1^t,\dots, d_r^r)$ is connected to the next by construction.

To check $F$ is concave, it suffices to check locally at a neighborhood of each point where $F$ fails to be linear, i.e. in a neighborhood of the boundary between the cones $C_{E_t}$ and $C_{E_{t+1}}$ for each $t$.  On $ C_{E_t}$,  $F(v)$ is given by the formula 
$$ F(v) = \sum_{k=1}^r d_{k}^t v\cdot e_k^t$$
and we have $ e_{k}^{t+1} = s_{i, E_t} ( e_k^t)$ so that on $C_{E_{t+1}}$, $F(v)$ is given by the formula

$$ F(v) = \sum_{k=1}^r d_{k}^{t+1} v\cdot e_k^{t+1} = \sum_{k=1}^r d_{k}^{t+1} v\cdot s_{i, E_t} ( e_k^t) = d_i^{t'}v\cdot s_{i, E_t} ( e_i^t) + \sum_{k\neq i}d_{k}^{t} v\cdot s_{i, E_t} ( e_k^t)   = -  d_i^{t'}v\cdot e_i^t  + \sum_{k\neq i}d_{k}^{t} v\cdot (e_k^t + m_{ik} e_i)$$ $$ = (-d_i^{t'} + \sum_{k\neq i} m_{ik } d_k^t)  v \cdot e_k^t +  \sum_{k\neq i}d_{k}^{t} v\cdot e_k^t  . $$

Since $d_i^{t'} \leq \tilde{d}^t - d_i^t = \sum_{k\neq i} m_{ik } d_k^t  - d_i^t$ and $ v \cdot e_k^t \leq 0$ for $v \in C_{E_{t+1}}$,  we get $$F(v) \leq \sum_{k=1}^r d_{k}^t v\cdot e_k^t$$ on $C_{E_{t+1}}$. On the boundary between $C_{E_t}$ and $C_{E_{t+1}}$, we have $v \cdot e_k=0$ so both formulas agree.

Hence, in a neighborhood of the boundary between the cones $C_{E_t}$ and $C_{E_{t+1}}$, $F$ has two linear segments which meet at the boundary point, and one linear segement is above the line formed by extending the other. This means $F$ is concave in this neighborhood. Since this works for an arbitrary boundary point, $F$ is concave everywhere, as desired.\end{proof}

\subsection{A toy model with Lie algebras}\label{ss:Kostant}

Let $\mathfrak g$ be a semisimple Lie algebra over the complex numbers with rank $r$. Fix a Cartan and a Borel subalgebra. Let $\Lambda$ be the root lattice, a rank $r$ lattice. For $\alpha \in \Lambda$, let $\mathfrak g_\alpha$ be the subspace of $\mathfrak g$ with eigenvalue $\alpha$ (which is $1$-dimensional if $\alpha$ is a root, $r$-dimensional if $\alpha=0$, and $0$-dimensional otherwise). Let $\alpha_1,\dots, \alpha_r$ be the simple roots.

 Let $\mathfrak n$ be the nilpotent part of the Borel subalgebra, i.e. the subspace spanned by $\mathfrak g_\alpha$ for all positive roots $\alpha$, equivalently, the sub-Lie-algebra generated by $\mathfrak g_{\alpha_1},\dots, \mathfrak g_{\alpha_r}$. The Lie algebra $\mathfrak n$ is naturally $\Lambda$-graded. Kostant gave the following formula for the Lie algebra cohomology $\LieH^*_{\mathfrak n} ( \mathbf 1) = \Ext^*_{\mathfrak n } (\mathbf 1, \mathbf 1)$ in terms of the Weyl group of $\mathfrak g$.

 \begin{theorem}\label{kostant}(Kostant) We have $ \Gr^\beta \LieH^{k}_{\mathfrak n} ( \mathbf 1) =0 $ unless $\beta = w(\rho)-\rho$ for $\rho$ half the sum of the positive roots and $k= \operatorname{length}(w)$ for some $w$ in the Weyl group, and $\LieH^{k}_{\mathfrak n} ( \mathbf 1) [\beta]$ is $1$-dimensional in that case. \end{theorem}

It follows from Theorem \ref{kostant} that the $\beta$-graded piece of the cohomology of $\mathfrak n$ admit the following symmetry under a shifted action of the Weyl group on $\Lambda$:

 \begin{prop}\label{kostant-fe} For  $\langle \beta, \alpha_i  \rangle > - \langle \alpha_i, \alpha_i \rangle $ there exists an isomorphism
 \[ \Gr^\beta \LieH^{k}_{\mathfrak n} ( \mathbf 1)  \to   \Gr^{w_i(\beta)-\alpha_i} \LieH^{k+1}_{\mathfrak n} ( \mathbf 1) \]
 and for $\langle \beta, \alpha_i \rangle < 0$ there exists an isomorphism 
\[ \Gr^{\beta} \LieH^{k+1}_{\mathfrak n} ( \mathbf 1)   \to \Gr^{ w_i(\beta)- \alpha_i}  \LieH^{k}_{\mathfrak n} ( \mathbf 1).    \]

Furthermore if $- \langle \alpha_i, \alpha_i \rangle <\langle \beta, \alpha_i \rangle < 0 $ then both sides of these isomorphisms vanish.

\end{prop}

We give a proof of Theorem \ref{kostant}  in two steps. First, we prove Proposition \ref{kostant-fe} by producing an explicit isomorphism. We then show there is a unique vector space graded by $\Lambda$ and cohomological degree which admits these symmetries, and that it is given by Theorem \ref{kostant}.

In fact Kostant proved a more general formula, which calculates the cohomology of $\mathfrak n$ with coefficients in an irreducible representation of $\mathfrak g$. The analogue of Proposition \ref{kostant-fe} can still be proven in this context by the same argument, but no longer contains enough information to deduce the analogue of Theorem \ref{kostant}.

Thus the main purpose of the proof is to give context for the algebraic meaning of Corollary \ref{nichols-relation-theorem}. It is possible that this proof gives a template for a purely algebraic proof of Corollary \ref{nichols-relation-theorem} but we will show in the next section that there is an obstruction to a direct translation of this proof.

The method is as follows. Fix a simple root $\alpha_i$. Let $\mathfrak n'$ be the translate of $\mathfrak n$ by the reflection around $\alpha_i$ in the Weyl group. This is the subalgebra of $\mathfrak g$ spanned by $\mathfrak g_{\alpha}$ for all positive roots $\alpha$ except $\alpha_i$ together with $\mathfrak g_{-\alpha_i}$. The reflection around $\alpha_i$ gives an isomorphism $\mathfrak n \to \mathfrak n'$ which is compatible with the grading in the sense that it sends a vector graded by $\beta$ to the vector $w_i(\beta)$ where $w_i$ is the simple reflection associated to $\alpha_i$ in the Weyl group acting on the root lattice $\Lambda$. This induces an isomorphism $\LieH^*_{\mathfrak n} ( \mathbf 1) \to \LieH^*_{\mathfrak n'} ( \mathbf 1)$ which is compatible in the same sense. We will construct another isomorphism $\bigoplus_j \LieH^j_{\mathfrak n} ( \mathbf 1) \to \bigoplus_j  \LieH^j_{\mathfrak n'} ( \mathbf 1)$. The composition of these two will give us an ``action" of the reflection around $\alpha_i$ on $\bigoplus_j \LieH^j_{\mathfrak n} ( \mathbf 1) $. Combining these actions for all $i$ gives enough information to calculate $\bigoplus_j \LieH^j_{\mathfrak n} ( \mathbf 1) $.

To do this, let $\mathfrak n_0$ be the intersection of $\mathfrak n$ and $\mathfrak n'$, i.e. the subalgebra spanned by $\mathfrak g_{\alpha}$ for all positive roots $\alpha$ except $\alpha_i$. Let $\mathfrak p$ be the subalgebra generated by $\mathfrak n$ and $\mathfrak n'$, i.e. the subalgebra spanned by $\mathfrak g_{\alpha}$ for all positive roots $\alpha$ together with $\mathfrak g_{-\alpha_i}$ and $[\mathfrak g_{\alpha_i}, \mathfrak g_{-\alpha_i}]$.

\begin{lemma}\label{lie-algebra-relations}  With $\mathfrak n, \mathfrak n', \mathfrak n_0, \mathfrak p$ as above, $\mathfrak n_0$ is an ideal in $\mathfrak p$. Furthermore, $\mathfrak p/\mathfrak n_0$ is isomorphic to $\mathfrak{sl}_2$, with $\mathfrak n/\mathfrak n_0$ the space generated by the unique positive root $E \in \mathfrak{sl}_2$ and $\mathfrak n'/\mathfrak n_0$ the space generated by the unique negative root $F\in \mathfrak{sl}_2$. \end{lemma}

\begin{proof} First we check that $\mathfrak n_0$ is an ideal in $\mathfrak n$ with quotient $\mathfrak g_{\alpha_i}$. This is because $[\mathfrak g_{\alpha_i}, \mathfrak g_{\alpha}] \subseteq \mathfrak g_{\alpha_i + \alpha} \subseteq \mathfrak n_0$ for any positive root $\alpha$, since $\alpha$ positive and $\alpha_i$ simple implies $\alpha+\alpha_i$ positive but $\alpha+\alpha_i \neq \alpha_i$. Since the roots of $\mathfrak n'$ are the positive roots for a different linear form on $\Lambda$ (lying in an adjacent Weyl chamber) the same argument shows $\mathfrak n_0$ is an ideal in $\mathfrak n'$, with quotient $\mathfrak g_{-\alpha_i}$.

Since $\mathfrak p$ is generated by $\mathfrak n$ and $\mathfrak n'$, it immediately follows that $\mathfrak n_0$ is an ideal in $\mathfrak p$, with quotient generated by $\mathfrak g_{\alpha_i}$ and $\mathfrak g_{-\alpha_i}$. It is standard that $\mathfrak g_{\alpha_i} $ and $\mathfrak g_{-\alpha_i}$ generate a Lie algebra isomorphic to $\mathfrak{sl}_2$, with $\alpha_i$ the unique positive root and $-\alpha_i$ the unique negative root.\end{proof} 

\begin{lemma}\label{hypercohomology-expression} 

There is an isomorphism between $\LieH^{j}_{\mathfrak n} ( \mathbf 1) $ and the hypercohomology $\LieH^j_ {\langle E \rangle} (  \LieH^*_{\mathfrak n_0} ( \mathbf 1 ))$ and similarly an isomorphism between $\LieH^{j}_{\mathfrak n'} ( \mathbf 1) $ and the hypercohomology $\LieH^j_{\langle F \rangle} (  \LieH^*_{\mathfrak n_0} ( \mathbf 1 ))$. These isomorphisms respect the $\Lambda$-grading.

\end{lemma}

\begin{proof} Whenever we have a lie algebra $\mathfrak a$ with ideal $\mathfrak b$, we have an isomorphism \[ \LieH^j_{\mathfrak a} ( \mathbf 1) \cong \LieH^j_{ \mathfrak a/\mathfrak b} (  \LieH^*_{\mathfrak b} (\mathbf 1)).\] Applying this isomorphism to $\mathfrak a= \mathfrak n$ or $\mathfrak n'$ and $\mathfrak b= \mathfrak b$ yields, by Lemma \ref{lie-algebra-relations}, the desired isomorphisms. These are natural in the sense that they are compatible with automorphisms of the Lie algebra, which gives the compatibility with grading as the grading arises from a family of automorphisms parameterized by $\mathbb G_m^{ r}$. \end{proof}

Thus we will need to understand how to relate  the hypercohomology groups $\LieH^j_{\langle E \rangle} (M) $ and $\LieH^j_{ \langle F \rangle}  ( M )$ for a complex of $\mathfrak{sl}_2$-representations $M$. Such a complex splits as a sum of representations, but we do not want to take such a splitting immediately. It is more convenient to first write down natural maps in the setting of a general graded complex, then show they are isomorphisms by calculating in the case of a  finite-dimensional representation.

We have the hypercohomology spectral sequences
\[  \LieH^j_{\langle E \rangle} (   \mathcal H^k( M) ) \to \LieH^{j+k}_{\langle E \rangle} ( M)  \]
and
\[  \LieH^j_{\langle F \rangle} (  \mathcal H^k(M)) ) \to \LieH^{p+q}_{\langle E \rangle} ( M)  .\]

The Lie algebras $\langle E \rangle $ and $\langle F \rangle$, being one-dimensional, have cohomology of any representation concentrated in degrees $0$ and $1$, with the cohomology in degree $0$ isomorphic to the invariants and the cohomology in degree $1$ isomorphic to the coinvariants, so the spectral sequences collapse to short exact sequences
\[ 0 \to ( \mathcal H^{k-1} (M)  )_E \to   \LieH^{k}_{\langle E \rangle} ( M) \to   ( \mathcal H^{k} (M)  )^E \to 0 \]
\[ 0 \to ( \mathcal H^{k-1} (M)  )_F \to   \LieH^{k}_{\langle F \rangle} ( M) \to   ( \mathcal H^{k} (M)  )^E \to 0 \]
where superscripts $E,F$ denote invariants and subscripts $E,F$ denote coinvariants.

The natural maps from invariants to the underlying vector space and from the underlying vector space to coinvariants give natural maps
\begin{equation}\label{first-composition}  \LieH^{k}_{\langle E \rangle} ( M) \to   ( \mathcal H^{k} (M)  )^E \to \mathcal H^k(M) \to  ( \mathcal H^{k} (M)  )_F \to   \LieH^{k+1}_{\langle F \rangle} ( M) \end{equation}
\begin{equation}\label{second-composition}  \LieH^{k}_{\langle F \rangle} ( M) \to   ( \mathcal H^{k} (M)  )^F \to \mathcal H^k(M) \to  ( \mathcal H^{k} (M)  )_E \to   \LieH^{k+1}_{\langle E \rangle} ( M)\end{equation}

We say a complex of $\mathfrak{sl}_2$-modules is $\Lambda$-graded if each module in the context admits a $\Lambda$-grading as a vector space, with the action of $\mathfrak {sl}_2$ compatible with the $\Lambda$-grading, where  $\mathfrak{sl}_2$ is $\Lambda$-graded so that $E$ has grade $\alpha_i$ and $F$ has grade $-\alpha_i$.

\begin{lemma}\label{simple-shift}  Let $M$ be a complex of $\Lambda$-graded $\mathfrak{sl}_2$ modules. 
\begin{enumerate}

\item The natural maps $\LieH^{k}_{\langle E \rangle} ( M) \to   ( \mathcal H^{k} (M)  )^E$ and   $ \LieH^{k}_{\langle F \rangle} ( M) \to   ( \mathcal H^{k} (M)  )^E$ respect the $\Lambda$-grading.

\item The map $( \mathcal H^{k-1} (M)  )_E \to   \LieH^{k}_{\langle E \rangle} ( M) $ subtracts $\alpha_i$ from the grading while $( \mathcal H^{k-1} (M)  )_F \to   \LieH^{k}_{\langle F \rangle} ( M) $ adds $\alpha_i$ to the grading.

\end{enumerate}

\end{lemma}

\begin{proof} The cohomology of the one-dimensional Lie algebra $\langle E \rangle$ with coefficients in the representation $\mathcal H^{k}(M)$ can be computed from the Chevally-Eilenberg complex $\mathcal H^k( M) \to  \operatorname{Hom}(\langle E\rangle, \mathcal H^{k}(M))$ with the map sending $v \in \mathcal H^k(M)$ to the unique linear form that takes value $[E,v]$ on $E$. The kernel is isomorphic to the invariants in a way that respects the grading. On the other hand, the isomorphism $M \to   \operatorname{Hom}(\langle E\rangle, \mathcal H^{k}(M))$ that sends $w$ to the unique linear form that takes value $w$ on $E$ shifts the grading by minus the grade of $E$, i.e. by $-\alpha$, so the isomorphism from the coinvariants of $M$ to the cokernel also shifts the grade by $-\alpha_i$.

Symmetrically, the same isomorphisms in the case of $\langle F\rangle$ shift the grade by $0$ and $+\alpha_i$. \end{proof}

\begin{lemma}\label{grading-shift} Let $M$ be a complex of $\Lambda$-graded $\mathfrak{sl}_2$ modules. The map \eqref{first-composition} shifts the grading by $+\alpha_i$ while the map \eqref{second-composition} shifts the grading by $-\alpha_i$. \end{lemma}

\begin{proof}  By Lemma \ref{simple-shift}, in \eqref{first-composition}, the last arrow adds $\alpha_i$ to the grading while the first three arrows preserve the grading, so the composition adds $\alpha_i$ to the grading. In \eqref{second-composition}, the last arrow subtracts $\alpha_i$ from the grading while the first three arrows preserve the grading, so the composition subtracts $\alpha_i$ from the grading. \end{proof}

We use $\Gr^\beta V$ to denote the $\beta$-graded piece of a $\Lambda$-graded vector space $V$. It follows from Lemma \ref{grading-shift} that we have natural maps
\begin{equation}\label{first-magic-map}  \LieH^{k}_{\langle E \rangle} ( M) [\beta] \to   \LieH^{k+1}_{\langle F \rangle} ( M) [\beta+\alpha_i]  \end{equation}
\begin{equation}\label{second-magic-map}  \LieH^{k}_{\langle F \rangle} ( M)[\beta+ \alpha_i]  \to   \LieH^{k+1}_{\langle E \rangle} ( M) [\beta] \end{equation}
\begin{lemma}\label{sl2-isos} Let $M$ be a bounded complex of $\Lambda$-graded $\mathfrak{sl}_2$-modules with finite-dimensional cohomology objects. Assume the eigenvalue of $H\in \mathfrak sl_2$ on $\mathcal H^{i} (M)[\beta]$ is equal to $2\langle \beta, \alpha_i \rangle / \langle \alpha_i, \alpha_i \rangle $

For each $\beta$, \eqref{first-magic-map}  is an isomorphism for $\langle \beta, \alpha_i  \rangle > - \langle \alpha_i, \alpha_i \rangle $ and \eqref{second-magic-map} is an isomorphism for $\langle \beta, \alpha_i \rangle < 0$. \end{lemma} 

\begin{proof} Since $M$ is a sum of shifts of irreducible representations, it suffices to prove this for $M$ a shift of an irreducible representation, and we we may as well assume the shift is $0$.  The $k+1$-dimensional irreducible representation $V_k$ is generated by a vector $v$ such that $F \cdot v=0$, and has a basis $v, E \cdot v,\dots, E^k \cdot v$. Furthermore $v$ is an eigenvector of $H$ with eigenvalue $-k$, and so the grade $\beta_0$ of $v$ satisfies $\langle \beta_0, \alpha _i \rangle= - k \langle \alpha_i,\alpha_i \rangle/2$. 

 The $E$-invariants are generated by $E^k v$ and the $E$-coinvariants by $v $. Since $\LieH^0_{\langle E \rangle} $ is the $E$-invariants and $\LieH^1_{\langle E \rangle}$ is the $E$-coinvariants shifted in grade by $\alpha_i$ (Lemma \ref{simple-shift}),  we have $\LieH^0_{\langle E \rangle}(V_k)  = \mathbb C$ in grade $\beta_0 +k \alpha_i$ and $\LieH^1_{\langle E \rangle} (V_k) = \mathbb C$ in grade $\beta_0 - \alpha_i$ Similarly, the $F$-invariants are generated by $v$ and the $F$-coinvariants by $E ^k v$, so $\LieH^0_{\langle F \rangle}(V_k)  = \mathbb C$ in grade $\beta_0$  and $\LieH^1_{\langle F \rangle} (V_k) = \mathbb C$ in grade $\beta_0+ (k+1) \alpha_i$.

The map $ \LieH^{k}_{\langle E \rangle} ( V_k) [\beta] \to   \LieH^{k+1}_{\langle F \rangle} ( V_k) [\beta+\alpha_i] $ is trivially an isomorphism whenever both sides are $0$. This holds unless $ \beta = \beta_0+ k \alpha_i , k=0$ or $\beta= \beta_0 -\alpha_i , k=1$ or $\beta= \beta_0- \alpha_i, k=-1$. The second two cases have \[ \langle \beta, \alpha_i \rangle = - (k+2) \langle \alpha_i, \alpha_i \rangle/2 \leq - \langle \alpha_i, \alpha_i \rangle  \]  so these can be ignored as they do not satisfy our hypothesis. In the first case, both sides of the map are one-dimensional, generated by $E^k v$. Since the map is defined by taking a cohomology class, viewing as an invariant, then as an element, then as a coinvariant, and then as another class, it sends the class represented by $E^k v$ to the class represented by $E^k v$ and hence is an isomorphism.

Similarly, $ \LieH^{k}_{\langle F \rangle} ( V_k)( [\beta+ \alpha_i]  \to   \LieH^{k+1}_{\langle E \rangle} ( V_k)  [\beta]$ is trivially an isomorphism whenever both sides are zero. This holds unless $\beta =  \beta_0 + k\alpha_i  , k=-1$ or $ \beta = \beta_0-\alpha_i , k=0$ or  $\beta= \beta_0 + k \alpha_i , k=1$. The first and last cases have \[\langle \beta, \alpha_i \rangle = k \langle \alpha_i , \alpha_i \rangle /2  \geq 0\] so these can be ignored as they do not satisfy our hypothesis. In the middle case, both sides are one-dimensional, generated by $v$. Again, this forces the map to be an isomorphism in the remaining case.\end{proof}

\begin{cor}\label{applied-sl2-isos} The natural map
\[ \Gr^\beta \LieH^{k}_{\mathfrak n} ( \mathbf 1)  \to   \Gr^{\beta+\alpha_i}\LieH^{k+1}_{\mathfrak n'} ( \mathbf 1) \]
is an isomorphism  for $\langle \beta, \alpha_i  \rangle > - \langle \alpha_i, \alpha_i \rangle $ and the natural map
\[ \Gr^{\beta+ \alpha_i}\LieH^{k}_{\mathfrak n'} ( \mathbf 1)  \to  \Gr^\beta  \LieH^{k+1}_{\mathfrak n} ( \mathbf 1)  \]
 is an isomorphism for $\langle \beta, \alpha_i \rangle < 0$.\end{cor}
 
 \begin{proof} This follows from Lemmas \ref{hypercohomology-expression} and  \ref{sl2-isos}, taking $M= \LieH^*_{\mathfrak n_0} ( \mathbf 1 )$ once we check that $\LieH^*_{\mathfrak n_0} ( \mathbf 1 )$ satisfies the compatibilities. The compatibility between the $\Lambda$-grading of $\mathfrak{sl}_2$ and the $\Lambda$-grading of $\LieH^*_{\mathfrak n_0} ( \mathbf 1 )$ is clear because $\mathfrak p$ and $\mathfrak n_0$ are both $\Lambda$-graded subspaces of $\mathfrak g$ and thus the action of $\mathfrak p/\mathfrak n_0 \cong \mathfrak{sl}_2$ on the $\Lambda$-graded Chevalley-Eilenberg complex computing $\LieH^*_{\mathfrak n_0} ( \mathbf 1 )$ is compatible with the grading. 
  

 For the compatibility between the grading and the action of $H \in \mathfrak {sl}_2$, we first check the same compatibility for the action of $\mathfrak {sl}_2$ on $\mathfrak n_0$, where it follows from the definition of the weight space grading on the Lie algebra.  Then the same compatibility will hold for the Chevalley-Eilenberg complex and thus its cohomology. \end{proof}

\begin{proof}[Proof of Proposition \ref{kostant-fe}]  The isomorphism $\mathfrak n \to \mathfrak n'$ gives an isomorphism between  $ \Gr^{\beta+\alpha_i }\LieH^{k}_{\mathfrak n'} ( \mathbf 1)  $ and  $ \Gr^{w_i (\beta+\alpha_i )}\LieH^{k}_{\mathfrak n} ( \mathbf 1)  $ and we have $w_i(\beta+ \alpha_i) = w_i (\beta) -\alpha_i$. Composing this isomorphism with the first isomorphism of Corollary \ref{applied-sl2-isos} gives the first isomorphism above, and composing with the second isomorphism of Corollary \ref{applied-sl2-isos} and reversing the order gives the second isomorphism above.
 
When both isomorphisms hold, their composition implies an isomorphism $ \Gr^\beta \LieH^{k}_{\mathfrak n} ( \mathbf 1)  \cong   \Gr^\beta \LieH^{k+2}_{\mathfrak n} ( \mathbf 1)  $ which since Ext groups are concentrated in nonnegative degrees implies by induction that $ \Gr^\beta \LieH^{k}_{\mathfrak n} ( \mathbf 1) =0$ for all $k$. \end{proof}

 \begin{lemma}\label{negative-positive} We have $\Gr^\beta \LieH^{k}_{\mathfrak n} ( \mathbf 1) =0$ unless $\beta$ is a negative integer combination of exactly $k$ positive roots. \end{lemma}
 
 \begin{proof} This follows from the description of the Chevalley-Eilenberg complex, which writes $\LieH^{k}_{\mathfrak n} ( \mathbf 1) $ as a subquotient of $\wedge^k (\mathfrak n)^\vee$, since $\mathfrak n$ is spanned by the weight spaces of the positive roots so that $\wedge^k (\mathfrak n)$ is spanned by wedge products of weight spaces of $k$ positive roots. \end{proof}
 
 \begin{lemma}\label{reduced-word-criterion} Let $w_{i_1} \dots w_{i_n}$ be a reduced word in a Weyl group. Let $\gamma$ be a vector in the positive Weyl chamber. Then $ \langle  w_{i_1} \dots w_{i_n}\gamma, \alpha_{i_1} \rangle \leq 0$. \end{lemma}
 
 \begin{proof} Suppose for contradiction that  $ \langle  w_{i_1} \dots w_{i_n}\gamma, \alpha_{i_1} \rangle > 0$. Then  \[ \langle  \gamma, w_{i_n} \dots w_{i_2}  \alpha_{i_1} \rangle =- \langle  \gamma, w_{i_n} \dots w_{i_1}  \alpha_{i_1} \rangle =-  \langle  w_{i_1} \dots w_{i_n}\gamma, \alpha_{i_1} \rangle < 0\] so $w_{i_n} \dots w_{i_2}  \alpha_{i_1} $ is a negative root. Since $\alpha_i$ is a positive root, there must be some $j\geq 1$ such that $ w_{i_j} \dots w_{i_2} \alpha_{i_1} $ is a positive root but $w_{i_{j+1}} \dots w_{i_2} \alpha_{i_1}$ is a negative root. Since each reflection through a simple root moves the positive Weyl chamber to an adjacent Weyl chamber and thus changes only one positive root, i.e. that simple root, to a negative root, we must have $w_{i_j} \dots w_{i_2} \alpha_{i_1} = \alpha_{i_{j+1}} $ so  \[ w_{i_{j+1}} = w_{i_j} \dots w_{i_2} w_{i_1} w_{i_2}^{-1} \dots w_{i_j}^{-1} =w_{i_j}^{-1} \dots w_{i_2}^{-1} w_{i_1} w_{i_2} \dots w_{i_j}  \] (since each reflection is its own inverse) which implies \[ w_{i_1} \dots w_{i_n} = w_{i_2} \dots w_{i_j} w_{i_{j+1}  }  w_{i_{j+1}} w_{i_{j+2}} \dots w_{i_n} = w_{i_2} \dots w_{i_j} w_{i_{j+2}} \dots w_{i_n},\] contradicting the assumption that $w_{i_1} \dots w_{i_n}$  is reduced. \end{proof}

\begin{proof}[Proof of Theorem \ref{kostant}] 


Fix any $\beta$, and let $w$ be an element of the  Weyl group such that $\langle w^{-1}( \beta+\rho) , \alpha_i \rangle  \geq  0$ for all $i$. This is possible since the set of vectors satisfying these conditions form the positive Weyl chamber, and we can choose an element of the Weyl group that sends any vector into the positive Weyl chamber. Equivalently, we have  $\langle w^{-1}( \beta+\rho) - \rho  , \alpha_i \rangle  \geq  - \langle \alpha_i , \alpha_i \rangle /2 $ 

It suffices to prove that  $ \Gr^\beta \LieH^{k}_{\mathfrak n} ( \mathbf 1)=0 $ unless $\beta = w(\rho)-\rho$ for $\rho$ half the sum of the positive roots and $k= \operatorname{length}(w)$ for this value of $w$, and is $1$-dimensional in that case. Indeed, if the condition holds for this value of $w$ then it clearly holds for some value of $w$, and if it holds for some value of $w$ then $w$ is the unique $w'$ such that $\langle w'^{-1}( \beta+\rho) , \alpha_i \rangle  \geq  0$, since $\beta+\rho = w(\rho)-\rho+\rho=w(\rho)$, and $w(\rho)$ lies in the interior of its Weyl chamber so there is a unique element $w'=w$ sending it to the positive Weyl chamber.

We prove this claim by induction on $\operatorname{length}(w)$.

For the base case where $w$ is the identity element, we assume $\langle \beta + \rho , \alpha_i \rangle \geq 0$ for all $i$, which implies $\langle \beta,\alpha_i\rangle \geq - \langle \alpha_i,\alpha_i \rangle /2$ for all $i$.  By Lemma \ref{negative-positive} if $\Gr^\beta \LieH^{k}_{\mathfrak n} ( \mathbf 1) \neq0$ then $\beta$ is a negative combination of exactly $k$ roots. If $k=0$ it follows that $\beta=0 = w(\rho)-\rho$ for $w$ the identity element, and in this case $\LieH^{k}_{\mathfrak n} ( \mathbf 1) [\beta]=\mathbf 1$ is the $\mathfrak n$-invariants of $\mathbf 1$. If $k>0$ then it follows that $\langle \beta, \alpha_i \rangle<0$ for some $i$ since otherwise $\beta$ would lie in the positive cone and hence could not be a nontrivial combination of negative roots, and thus $$- \langle \alpha_i,\alpha_i \rangle < - \langle \alpha_i,\alpha_i \rangle /2 \leq \langle \beta, \alpha_i \rangle<0$$ so $\LieH^{k}_{\mathfrak n} ( \mathbf 1) [\beta]=0 $ by Proposition \ref{kostant-fe}.

For the induction step, note that we can write $w= w_i w'$ for some $i$ and some $w'$ in the Weyl group where $\operatorname{length} (w') = \operatorname{length}(w)-1$. Let
\[\beta' = w_i^{-1} (\beta+\rho) - \rho= w_i(\beta+\rho)-\rho =w_i(\beta) + w_i(\rho) - \rho= w_i (\beta)-\alpha_i\]
Observe that, by the induction hypothesis applied to $\beta', w'$, since \[w^{'-1} ( \beta'+\rho) = w^{'-1} ( w_i^{-1} (\beta+\rho) ) = w^{-1} (\beta+\rho) \] is in the positive Weyl chamber, the group $\Gr^{\beta'}  \LieH^{k-1}_{\mathfrak n} ( \mathbf 1) $ vanishes unless $\beta'= w'(\rho)-\rho$ and $k-1= \operatorname{length}(w')= \operatorname{length}(w)-1 $, which is equivalent to \[ \beta= w_i(\beta'+\rho) - \rho= w_i ( w'(\rho))-\rho= w(\rho)-\rho\] and $k= \operatorname{length}(w)$, and is one-dimensional in that case. So it suffices to check
\[\Gr^\beta  \LieH^{k}_{\mathfrak n} ( \mathbf 1)  = \Gr^{\beta'} \LieH^{k-1}_{\mathfrak n} ( \mathbf 1) \]
which follows from Proposition \ref{kostant-fe} as long as $\langle \beta, \alpha_i \rangle <0$.

By Lemma \ref{reduced-word-criterion} applied to $\gamma = w^{-1} (\beta+\rho) $ in the positive Weyl chamber and any reduced word for $w$ beginning with $w_i$,  we must have $\langle \beta +\rho, \alpha_i \rangle \leq 0$, which since $\langle \rho,\alpha_i \rangle >0$ implies $\langle \beta, \alpha_i \rangle <0$. \end{proof}

\subsection{Beyond the toy model}\label{ss:Hopf}

In this subsection, we use the notation $M \{\alpha\}$ for a $\Lambda$-graded module $M$ to denote the module $M$ with the grading shifted by $\alpha$, i.e. each homogenous element of grade $\beta$ in $M$ becomes a homogeneous element of grade $\beta+\alpha$ in $M\{\alpha\}$.

One could try to generalize the proof of Proposition \ref{kostant-fe} to the setting of quantum groups. Much of the argument transfers, but there is one crucial obstruction that we now explain.

Here, $\mathfrak n$ is replaced by the positive part of Lusztig's small quantum group $u_{\qv}(\mathfrak g)$. We replace the Lie algebra cohomology with the Ext group of the trivial representation with itself over this algebra. (Since the Lie algebra cohomology of the trivial representation is simply the Ext group of the trivial representation with itself, this is a direct generalization.)
 We replace the Weyl group with the braid group, so $\mathfrak n'$ is replaced by the image of $\mathfrak n'$ under the action of the $i$th standard generator of the braid group under Lusztig's braid group action.  We can replace $\mathfrak p$ with the subalgebra generated by $\mathfrak n$ and $\mathfrak n'$, and $\mathfrak n_0$ by the intersection of these subalgebras.  Lemma \ref{lie-algebra-relations} could be replaced by short exact sequences relating these Hopf algebras to $u_{\qv}(\mathfrak sl_2)$ and its positive and negative parts. The spectral sequences could be replaced by the spectral sequences associated to the short exact sequences of Hopf algebras. 

 A suitable analogue of Lemma \ref{sl2-isos} would be the following statement \ref{key-non-theorem}, which unfortunately is false.

 Let $\Lambda$ be a finitely generated abelian group (which we will take to be a root lattice) and $\alpha_i \in \Lambda$ an element (which we will take to be a root). Let $u_{\qv}(\mathfrak{sl}_2)$ be Lusztig's small quantum group, endowed with a $\Lambda$-grading where $E$ has grade $\alpha_i$, $F$ has grade $-\alpha_i$, and $K$ has grade $0$. We say a complex of $u_{\qv}(\mathfrak{sl}_2)$ models is $\Lambda$-graded if it is endowed with a $\Lambda$-grading as a complex of vector spaces for which the action of $u_q(\mathfrak{sl}_2)$ is compatible with the grading.

 \begin{non-theorem}\label{key-non-theorem} Let ${\qv}$ be a root of unity whose square has order $n_i$. Let $M$ be a complex of $\Lambda$-graded $u_{\qv}(\mathfrak{sl}_2)$-modules with finite-dimensional cohomology objects. Assume the eigenvalue of $K\in \mathfrak u_{\qv}(\mathfrak{sl}_2)$ on $\mathcal H^{i} (M)[\beta]$ is equal to ${\qv}^{ 2\langle \beta, \alpha_i \rangle / \langle \alpha_i, \alpha_i \rangle }$. Let $\langle E \rangle_{\qv}$ and $\langle F \rangle_{\qv}$ be the subalgebras of $u_{\qv}(\mathfrak{sl}_2)$ generated by $E$ and $F$ respectively. Fix $\beta\in \Lambda$ and let $c \in \{1,\dots, n_i\}$ be congruent to $ \frac{2 \langle \beta, \alpha_i\rangle }{ \langle \alpha_i, \alpha_i\rangle} +1$ modulo $n_i$. Fix $j\in \mathbb Z$. If $c\neq n_i$ then we have  
 \[  \dim \Gr^\beta \Ext^j_{ \langle E\rangle_{\qv}} (\mathbf 1, M )  - \dim \Gr^{\beta + (n-c) \alpha_i} \Ext^{j-1}_{ \langle E\rangle_{\qv}} (\mathbf 1, M )  \] \[ =   \dim \Gr^{\beta+ \alpha_i 
 - c\alpha_i  }\Ext^{j}_{ \langle F\rangle_{\qv}} (\mathbf 1, M )  - \dim \Gr^{\beta  + (1-n_i)  \alpha_i} \Ext^{j-1}_{ \langle F\rangle_{\qv}} (\mathbf 1, M )    \] 
 and if $c=n_i$ we have 
 \[\dim  \Gr^\beta \Ext^j_{ \langle E\rangle_{\qv}} (\mathbf 1, M ) [\beta]= \dim  \Gr^{ \beta+ (1-n_i) \alpha_i} \Ext^j_{ \langle F\rangle_{\qv}} (\mathbf 1, M ) .\]
 \end{non-theorem}

 A counterexample is provided by the following:

\begin{example}\label{counter-example} We consider the special case when $M$ is a module, viewed as a complex by placing in degree $0$, and $j=0$. Since cohomology groups in negative degree vanish,  the degree $0$ cohomology with respect to $E$ are the $E$-invariants, and the same is true for $F$, the putative equation becomes

 \[  \dim \Gr^\beta M^E   =   \dim \Gr^{\beta+ \alpha_i 
 - c\alpha_i  }M^F  \] 

The map $\beta \mapsto \beta + \alpha_i - c_\beta \alpha_i$ is a bijection, with inverse $\beta \mapsto \beta  + (n_i - c_\beta-1) \alpha_i$ so summing over $\beta$ we obtain $\dim M^E = \dim M^F$.

But this is absurd. A Verma module will always have a one-dimensional space of invariants for $F$ but typically a two-dimensional space of invariants for $E$.

Explicitly, Let $\qv$ be a root of unity of odd order $n_i$. Then $u_{\qv}(\mathfrak{sl}_2)$ is the algebra with generators $E, F, K, K^{-1} $ and relations $EF - FE = \frac{ K- K^{-1}}{\qv-\qv^-1} , KE = \qv^2 EK, KF= \qv^{-2} FK , E^{n_i} =0, F^{n_i}=0, K^{2n_i}=1$, as well as $K K^{-1} = K^{-1}K=1$.

The Verma module $M$ can be presented as the module freely generated by a vector $x$ with relations $E x =0 $ and $K x = \qv^s x$ for some integer $s \in \{0,\dots, n_i-2\}$.

Even more explicitly, this is the module with basis $x, Fx, \dots, F^{n_i-1} x$ where 
\[ F \cdot F^t x= F^{t+1}x \textrm{ for }t<n_i-1\] \[  F \cdot F^{n_i-1} x=0,\]  
\[K \cdot F^t x = \qv^{s-2t} F^t x,\]
\[ E F^t x = \frac{ \qv^{s+1} + \qv^{-1-s } - \qv^{s+1-2t }  - \qv^{ 2t-s-1}   }{ (\qv- \qv^{-1})^2}   F^{t-1} x   \textrm{ for } t>0\]
\[ F x=0.\]One can verify directly that the relations are satisfied for this module.

We give the generator $x$ a grade $\beta_0 \in \Lambda$ satisfying $\langle \beta_0 ,\alpha_i \rangle = \frac{s}{2} \langle \alpha_i, \alpha_i \rangle$.  Then $M$ is a complex of $\Lambda$-graded $u_{\qv}(\mathfrak{sl}_2)$-modules with finite-dimensional cohomology objects and satisfies the compatibility with the eigenvalue of $K$.

The space of $F$-invariants is generated by $F^{n_i-1} x$. On the other hand, we have $E F^t=0$ if and only if  $ \qv^{s+1} + \qv^{-1-s } - \qv^{s+1-2t }  - \qv^{ 2t-s-1} =0$ which occurs if and only if $t \equiv 0 \bmod n_i$ or $i \equiv s+1 \bmod n_i$, since $\qv$ has odd order so that $ \qv^{s+1} $ and $\qv^{-1-s }$ can't cancel so that $ \qv^{s+1} $ must cancel one of the other two terms.  Since $s \in \{0,\dots,n_i-2\}$ this occurs exactly when $t=0$ or $t=s+1$, so that $M$ has a two-dimensional space of $E$-invariants, giving the desired contradiction.

\end{example}

Therefore it seems one cannot prove the functional equation based solely on the hypercohomology description, without further information on the cohomology of $\mathfrak n_0$.

One claim that can be checked directly is a functional equation for Euler characteristics.


\begin{lemma}\label{Euler-comparison}Let $\qv$ be a root of unity whose square has order $n_i>1$. Let $M$ be $\Lambda$-graded complex of $u_{\qv}(\mathfrak{sl}_2)$-modules such that for each $\beta \in u_{\qv}(\mathfrak{sl}_2)$, the $\beta$-graded part of the cohomology $\Gr^\beta \mathcal H^k(M) $ is finite-dimensional for all $k$ and zero for all but finitely many $k$. Then
\[ \sum_{k \in \mathbb Z} (-1)^k \dim \Gr^\beta \Ext^k_{\langle E \rangle_{\qv}} (\mathbf 1, M ) - \sum_{k \in \mathbb Z} (-1)^k \dim \Gr^{\beta+n_i\alpha_i} \Ext^k_{\langle E \rangle_{\qv}} (\mathbf 1, M )  \] \[= \sum_{k \in \mathbb Z} (-1)^k \dim \Gr^\beta  \mathcal H^k (M)  - \sum_{k\in \mathbb Z} (-1)^k \dim \Gr^{\beta+\alpha_i} \mathcal H^k (M) \]
\end{lemma}

\begin{proof} The key player will be the complex $U$ of $\langle E \rangle_{\qv}$-modules defined by $U^0 = \langle E \rangle_{\qv}$,  $U^1 = \langle E \rangle_{\qv}\{-\alpha_i\} $, $U^k = 0 $ for $k\neq 0,1$ and where $d^0 \colon U^0 \to U^1$ is multiplication by $E$. This is a $\Lambda$-graded complex of $\langle E \rangle_{\qv}$-modules, so  $\operatorname{RHom}_{\langle E \rangle_{\qv}} ( U, M) $ carries a $\Lambda$-grading. We calculate \[ \Gr^{\beta+\alpha_i}  \sum_{k \in \mathbb Z}(-1)^k  \dim \operatorname{RHom}_{\langle E \rangle_{\qv}} ( U, M) \]  in two ways. On the one hand, by the spectral sequence associated to the filtration with associated graded pieces $U^0$ and $U^1$ we have
\[  \sum_{k \in \mathbb Z}(-1)^k  \dim \Gr^{\beta+\alpha_i} \operatorname{RHom}_{\langle E \rangle_{\qv}} ( U, M)  \] \[=  \sum_{k \in \mathbb Z}(-1)^k  \dim \Gr^{\beta+\alpha_i}\operatorname{RHom}_{\langle E \rangle_{\qv}} (\langle E \rangle_{\qv} , M)  -  \sum_{k \in \mathbb Z}(-1)^k  \dim \Gr^{\beta+\alpha_i} \operatorname{RHom}_{\langle E \rangle_{\qv}} (\langle E \rangle_{\qv})\{-\alpha_i \}  , M)  \] \[= \sum_{k \in \mathbb Z} (-1)^k \dim  \Gr^{\beta+\alpha_i} \mathcal H^k (M)  - \sum_{k\in \mathbb Z} (-1)^k \dim  \Gr^{\beta+\alpha_i} \mathcal H^k (M \{\alpha_i\} )  \] \[= \sum_{k \in \mathbb Z} (-1)^k \dim \Gr^{\beta+\alpha_i} \mathcal H^k (M) - \sum_{k\in \mathbb Z} (-1)^k \dim \Gr^\beta \mathcal H^k (M) .\]

On the other hand, by the spectral sequence associated to the filtration with associated graded pieces $\ker d^0 $ and $\coker d^0  $ we have
\[  \sum_{k \in \mathbb Z}(-1)^k  \dim  \Gr^{\beta+\alpha_i} \operatorname{RHom}_{\langle E \rangle_{\qv}}  ( U, M)  \] \[ =  \sum_{k \in \mathbb Z}(-1)^k  \dim  \Gr^{\beta+\alpha_i}\operatorname{RHom}_{\langle E \rangle_{\qv}} (\ker d^0 , M)   -  \sum_{k \in \mathbb Z}(-1)^k  \dim  \Gr^{\beta+\alpha_i}\operatorname{RHom}_{\langle E \rangle_{\qv}} (\coker d^0   , M)   \]

Now $\coker d^0 =\langle E \rangle_{\qv}/ E \langle E \rangle_{\qv} \{ -\alpha_i\}  = \mathbf 1\{-\alpha_i\}$ and thus  \[ \operatorname{RHom}_{\langle E \rangle_{\qv}} (\coker d^0   , M)  =   \operatorname{RHom}_{ \langle E \rangle_{\qv}} ( \mathbf 1, M )\{\alpha_i\} \]  so that \[  \sum_{k \in \mathbb Z}(-1)^k  \dim  \Gr^{\beta+\alpha_i}\operatorname{RHom}_{\langle E \rangle_{\qv}} (\coker d^0   , M)   = \sum_{k\in \mathbb Z} (-1)^k  \dim  \Gr^\beta \operatorname{RHom}_{ \langle E \rangle_{\qv}} ( \mathbf 1, M ) \]

Since $\langle E\rangle_{\qv} = \mathbb C[E]/E^{n_i}$, the kernel $\ker d^0$ of multiplication by $E$ is generated by $E^{n-1}$ and thus is isomorphic to $\mathbf 1 \{ (n_i-1)\alpha_i\}$ so that
 \[  \sum_{k \in \mathbb Z}(-1)^k  \dim   \Gr^{\beta+\alpha_i} \operatorname{RHom}_{\langle E \rangle_{\qv}} (\ker d^0   , M) [\beta+\alpha_i]   = \sum_{k\in \mathbb Z} (-1)^k  \dim \Gr^{\beta+ n_i\alpha_i } \operatorname{RHom}_{ \langle E \rangle_{\qv}} ( \mathbf 1, M ) . \]

Combining the equations, and subtracting, gives the formula. \end{proof}

\begin{lemma}\label{Euler-fe} Let $v$ be a root of unity whose square has order $n_i>1$. Let $M$ be $\Lambda$-graded complex of $u_{\qv}(\mathfrak{sl}_2)$-modules such that for each $\beta \in u_{\qv}(\mathfrak{sl}_2)$, the $\beta$-graded part of the cohomology $\mathcal H^k(M) [\beta]$ is finite-dimensional for all $k$ and zero for all but finitely many $k$. Then
\[ \sum_{k \in \mathbb Z} (-1)^k \dim \Gr^\beta \Ext^k_{\langle E \rangle_{\qv}} (\mathbf 1, M )  - \sum_{k \in \mathbb Z} (-1)^k \dim \Gr^{\beta+   n_i \alpha_i}\Ext^k_{\langle E \rangle_{\qv}} (\mathbf 1, M )  \] \[ =  \sum_{k \in \mathbb Z} (-1)^k \dim \Gr^{\beta+(1-n_i) \alpha_i }  \Ext^k_{\langle F \rangle_{\qv}} (\mathbf 1, M )  - \sum_{k \in \mathbb Z} (-1)^k \dim \Gr^{\beta+\alpha_i} \Ext^k_{\langle F \rangle_{\qv}} (\mathbf 1, M ) \]
\end{lemma}

\begin{proof} The algebra $u_{\qv}(\mathfrak{sl}_2)$ is symmetric under exchanging $E$ and $F$ and $K$ and $K^{-1}$. This has the effect of negating the grading. Thus, we may exchange $E$ and $F$ in Lemma \ref{Euler-comparison}, also exchanging $\alpha_i$ and $-\alpha_i$, to obtain
\[ \sum_{k \in \mathbb Z} (-1)^k \dim  \Gr^\beta \Ext^k_{\langle F \rangle_{\qv}} (\mathbf 1, M )  - \sum_{k \in \mathbb Z} (-1)^k \dim  \Gr^{\beta-n_i\alpha_i} \Ext^k_{\langle F \rangle_{\qv}} (\mathbf 1, M )\] \[= \sum_{k \in \mathbb Z} (-1)^k \dim  \Gr^\beta \mathcal H^k (M) - \sum_{k\in \mathbb Z} (-1)^k \dim \Gr^{\beta- \alpha_i} \mathcal H^k (M) \]
and shift by $\alpha_i$ to get
\[ \sum_{k \in \mathbb Z} (-1)^k \dim \Gr^{\beta+\alpha_i}  \Ext^k_{\langle F \rangle_{\qv}} (\mathbf 1, M ) - \sum_{k \in \mathbb Z} (-1)^k \dim 
 \Gr^{\beta+ (1-n_i)\alpha_i} \Ext^k_{\langle F \rangle_{\qv}} (\mathbf 1, M )\] \[= \sum_{k \in \mathbb Z} (-1)^k \dim \Gr^{\beta+\alpha_i} \mathcal H^k (M)  - \sum_{k\in \mathbb Z} (-1)^k \dim \Gr^\beta \mathcal H^k (M) \]
Negating, and combining with Lemma \ref{Euler-comparison}, we get the stated equation.\end{proof}

It should be possible to use Lemma \ref{Euler-fe} to prove Euler characteristic analogues of Corollary \ref{applied-sl2-isos} and Proposition \ref{kostant-fe}, and prove the Euler characteristic variant of the functional equation.



Even more generally than quantum groups, one could work with Nichols algebras of diagonal type. Here the Weyl group would be replaced by the Weyl groupoid defined by Heckenberger. In the notation of \cite[p. 180]{HeckenbergerWeyl}, we replace $\mathfrak n$ by a Nichols algebra $\mathcal B(V)$ of diagonal type, replace $\mathfrak p$ by $ (\mathcal B(V)^{\mathrm{op}} \# H_i^{ \mathrm{cop}} )^{\mathrm{op}}$, replace $\mathfrak n'$ by the subalgebra $\mathcal B_i $ (which is another Nichols algebra of diagonal type, associated to a possibly-different vector space), and replace $\mathfrak n_0$ by $\operatorname{ker} y_i^L$. 
It seems likely that these Hopf algebras lie in exact sequences relating them to $\mathfrak{sl}_2(q_{ii})$ and its positive and negative parts, in which case one can apply the spectral sequences associated to a short exact sequence of Hopf algebras.

Again, Lemma \ref{Euler-fe} should be applicable here, and give analogues of Corollary \ref{applied-sl2-isos} and Proposition \ref{kostant-fe}. 

\printbibliography

\end{document}